\documentclass[a4j,10pt,showkeys]{article}

\usepackage{amsmath,amsthm,amssymb}
\usepackage{geometry}
\usepackage{hyperref}
\usepackage[capitalize]{cleveref}
\usepackage{cancel}
\usepackage{mathtools}
\usepackage{bm}
\usepackage{color}
\usepackage{mathrsfs}
\usepackage[dvipsnames]{xcolor}
\hypersetup{
    colorlinks=true,
    citecolor=MidnightBlue,
    linkcolor=Red,
    urlcolor=OliveGreen,
}
\usepackage{tikz}
\usetikzlibrary{calc}
\tikzset{set label/.style={fill=white}}

\theoremstyle{plain}
\newtheorem{theo}{Theorem}[section]
\newtheorem{lemm}{Lemma}[section]
\newtheorem{prop}{Proposition}[section]
\newtheorem{cor}{Corollary}[section]

\theoremstyle{definition}
\newtheorem{defi}{Definition}[section]
\newtheorem{rem}{Remark}[section]
\newtheorem{exam}{Example}[section]

\newtheorem{assum}{Assumption} 

\makeatletter
\let\c@lemm\c@theo
\let\c@prop\c@theo
\let\c@cor\c@theo
\let\c@defi\c@theo
\let\c@rem\c@theo
\let\c@exam\c@theo
\let\c@prob\c@theo
\makeatother

\crefname{theo}{Theorem}{Theorems}
\crefname{lemm}{Lemma}{Lemmas}
\crefname{prop}{Proposition}{Propositions}
\crefname{cor}{Corollary}{Corollaries}
\crefname{claim}{Claim}{Claims}
\crefname{defi}{Definition}{Definitions}
\crefname{rem}{Remark}{Remarks}
\crefname{assum}{Assumption}{Assumptions}
\crefname{exam}{Example}{Examples}
\crefname{prob}{Problem}{Problems}

\newcommand{\bE}{\mathbb{E}}
\newcommand{\bF}{\mathbb{F}}

\newcommand{\bN}{\mathbb{N}}

\newcommand{\bP}{\mathbb{P}}

\newcommand{\bR}{\mathbb{R}}

\newcommand{\bW}{\mathbb{W}}

\newcommand{\cB}{\mathcal{B}}
\newcommand{\cC}{\mathcal{C}}

\newcommand{\cE}{\mathcal{E}}
\newcommand{\cF}{\mathcal{F}}
\newcommand{\cG}{\mathcal{G}}

\newcommand{\cK}{\mathcal{K}}

\newcommand{\cP}{\mathcal{P}}

\newcommand{\cS}{\mathcal{S}}
\newcommand{\cT}{\mathcal{T}}

\newcommand{\cW}{\mathcal{W}}

\newcommand{\sC}{\mathscr{C}}

\newcommand{\fm}{\mathfrak{m}}
\newcommand{\fp}{\mathfrak{p}}

\newcommand{\ep}{\varepsilon}
\newcommand{\diff}{\mathrm{d}}
\newcommand{\Law}{\mathrm{Law}}
\newcommand{\LP}{\mathrm{LP}}
\newcommand{\TV}{\mathrm{TV}}
\newcommand{\KL}{\mathrm{KL}}

\newcommand{\BL}{\mathrm{BL}}
\newcommand{\Lip}{\mathrm{Lip}}

\newcommand{\op}{\mathrm{op}}
\newcommand{\relmiddle}[1]{\mathrel{}\middle#1\mathrel{}}
\newcommand{\1}{\mbox{\rm{1}}\hspace{-0.25em}\mbox{\rm{l}}}
\newcommand{\weaksol}{(X,W,\Omega,\cF,\bF,\bP)}
\newcommand{\EM}{(X^\pi(t_k))^m_{k=0}}
\newcommand{\growth}{\mathrm{G}}
\newcommand{\conti}{\mathrm{C}}
\newcommand{\ellip}{\mathrm{E}}
\newcommand{\dist}{\mathrm{dist}}
\newcommand{\poly}{\fp^\pi}
\newcommand{\hX}{\widehat{X}}
\newcommand{\supp}{\mathrm{supp}\,}
\newcommand{\data}{(D,\mu_0,b,\sigma)}

\providecommand{\keywords}[1]{\textbf{Keywords:} #1}
\makeatletter

\@addtoreset{equation}{section}
\makeatother
\def\widebar{\accentset{{\cc@style\underline{\mskip10mu}}}}
\numberwithin{equation}{section}
\allowdisplaybreaks

\title{A generalized coupling approach for the weak approximation of stochastic functional differential equations}

\author{
Yushi Hamaguchi\footnote{Department of Mathematics, Graduate School of Science, Kyoto University, Kyoto 606-8502, Japan. Email: \href{mailto:hamaguchi@math.kyoto-u.ac.jp}{hamaguchi@math.kyoto-u.ac.jp}}
\quad and \quad
Dai Taguchi\thanks{
Department of Mathematics,
Faculty of Engineering Science,
Kansai University,
Suita,
Osaka,
564-8680,
Japan.
Email: \href{mailto:taguchi@kansai-u.ac.jp}{taguchi@kansai-u.ac.jp}}
}

\begin{document}
\maketitle


\begin{abstract}
In this paper, we study functional type weak approximation of weak solutions of stochastic functional differential equations by means of the Euler--Maruyama scheme. Under mild assumptions on the coefficients, we provide a quantitative error estimate for the weak approximation in terms of the L\'evy--Prokhorov metric of probability laws on the path space. The weak error estimate obtained in this paper is sharp in the topological and quantitative senses in some special cases. We apply our main result to ten concrete examples appearing in a wide range of science and obtain a weak error estimate for each model. The proof of the main result is based on the so-called generalized coupling of probability measures.
\end{abstract}


\keywords
Stochastic functional differential equation; weak approximation; generalized coupling; Euler--Maruyama scheme.


\textbf{2020 Mathematics Subject Classification}: 34K50; 65C30; 60F17; 65C20.



{
\hypersetup{linkcolor=black}
\tableofcontents
}


\section{Introduction}

In this paper, we consider the following stochastic functional differential equation (SFDE) driven by Brownian motion $W$:
\begin{equation}\label{eq_SFDE}
	\diff X(t)=b(t,X)\,\diff t+\sigma(t,X)\,\diff W(t),\ \ t\in[0,\infty),
\end{equation}
where the coefficients $b(t,X)$ and $\sigma(t,X)$ are allowed to depend on the path of the solution $X=(X(s))_{s\in[0,\infty)}$ up to time $t$. The aim of this paper is to show a functional type weak approximation of a weak solution of the SFDE \eqref{eq_SFDE} by means of the Euler--Maruyama scheme, providing a quantitative error estimate for the weak convergence. More precisely, given a weak solution $\weaksol$ to the SFDE \eqref{eq_SFDE}, for each partition $\pi=(t_k)^m_{k=0}$ of a bounded interval $[0,T]$ such that $0=t_0<t_1<\cdots<t_m=T$, we provide an estimate of
\begin{equation}\label{eq_LP}
	d_\LP\big(\Law_\bP(X_T),\Law_{\bP^\pi}(\poly[X^\pi])\big)
\end{equation}
in terms of the mesh size $|\pi|:=\max_{k\in\{0,\dots,m-1\}}(t_{k+1}-t_k)$. Here, $d_\LP(\cdot,\cdot)$ denotes the so-called L\'evy--Prokhorov metric between probability measures on the space of continuous functions stopped at time $T$, $\Law_\bP(X_T)$ denotes the law of the stopped path $X_T:=X(\cdot\wedge T)=(X(s\wedge T))_{s\in[0,\infty)}$ of the given weak solution, $X^\pi=(X^\pi(t_k))^m_{k=0}$ denotes the Euler--Maruyama scheme defined on a probability space $(\Omega^\pi,\cF^\pi,\bP^\pi)$, and $\Law_{\bP^\pi}(\poly[X^\pi])$ denotes the law of the linear interpolation $\poly[X^\pi]$ of $X^\pi$ at each point of the partition $\pi$. The convergence to zero of the L\'evy--Prokhorov metric \eqref{eq_LP} as $|\pi|\downarrow0$ is equivalent to the weak convergence of $\poly[X^\pi]$ to $X_T$ on the path space. Precise definitions of the weak solution, Euler--Maruyama scheme and L\'evy--Prokhorov metric are summarized in \cref{section_pre} below. Under mild assumptions on the coefficients, we obtain an error estimate for \eqref{eq_LP} by using the ``generalized coupling approach'', which has been used in the literature of the ergodicity of infinite dimensional Markovian systems.

In the literature of the weak approximations of numerical schemes for stochastic differential equations (SDEs), most existing works such as \cite{BaTa96,ClKoLa06,GoLa08,Ho24,KoMe17,MiPl91} consider the convergence order of the weak error measured by
\begin{equation}\label{eq_one-dim-error}
	\big|\bE_\bP[f(X(T))]-\bE_{\bP^\pi}[f(X^\pi(T))]\big|
\end{equation}
for some fixed $\bR$-valued test function $f$ on the Euclidean space and fixed time $T\in(0,\infty)$. This kind of convergence analysis does not provide a weak convergence on the path space. In contrast, the weak error estimate in terms of the L\'evy--Prokhorov metric \eqref{eq_LP} is related to the estimate of
\begin{equation*}
	\big|\bE_\bP[f(X_T)]-\bE_{\bP^\pi}[f(\poly[X^\pi])]\big|
\end{equation*}
for bounded and Lipschitz continuous test functions $f$ on the path space; see \cref{section_pre} for more details. Alfonsi, Jourdain and Kohatsu-Higa \cite{AlJoKH14} study a functional type weak error estimate for one-dimensional Markovian SDEs in terms of the ($L^1$-)Wasserstein metric:
\begin{equation}\label{eq_Wasserstein-AlJoKH14}
	\cW_1\big(\Law_\bP(X_T),\Law_{\bP^\pi}(\widetilde{X}^\pi_T)\big),
\end{equation}
where $\widetilde{X}^\pi$ is the continuous time extension of the Euler--Maruyama scheme $X^\pi=(X^\pi(t_k))^m_{k=0}$ by means of the interpolation by the Brownian path. By the Kantorovich--Rubinstein theorem (cf.\ \cite[Theorem 11.8.2]{Du02}), the error estimate with respect to the Wasserstein metric \eqref{eq_Wasserstein-AlJoKH14} is related to the estimate of
\begin{equation}\label{eq_functional-error-BMinterpolation}
	\big|\bE_\bP[f(X_T)]-\bE_{\bP^\pi}[f(\widetilde{X}^\pi_T)]\big|
\end{equation}
in terms of (not necessarily bounded) Lipschitz continuous test functions $f$ on the path space. Assuming that the coefficients are bounded, sufficiently smooth (with bounded derivatives) and uniformly elliptic, they obtain in \cite[Theorem 3.2]{AlJoKH14} a convergence order of almost $|\pi|^{3/4}$ for the error \eqref{eq_Wasserstein-AlJoKH14}, which is faster than the order $\sqrt{|\pi|}$ obtained by a direct application of the well-known results \cite{Ka88,KlPl95} on the strong error analysis. Also, Ngo and Taguchi \cite{NgTa18} study the error in the form of \eqref{eq_functional-error-BMinterpolation} for Markovian SDEs with irregular drift and constant diffusion. It is worth mentioning that, besides the choices of the metrics, there is a significant difference in the ways of continuous time extensions of the Euler--Maruyama scheme $X^\pi$ between the objective \eqref{eq_LP} in the present paper and the ones \eqref{eq_Wasserstein-AlJoKH14} and \eqref{eq_functional-error-BMinterpolation} in \cite{AlJoKH14,NgTa18}. Our error analysis using the linear interpolation of the Euler--Maruyama scheme is in the spirit of the Donsker-type functional central limit theorem (cf. \cite[Chapter 2, Theorem 4.20]{KaSh91}), and the results in \cite{AlJoKH14,NgTa18} using the interpolation by the Brownian path do not provide an error estimate for such a purpose. There are only a few works on the Donsker-type functional central limit theorem for SDEs. Ankirchner, Kruse and Urusov \cite{AnKrUr17} show the functional central limit theorem for one-dimensional Markovian SDEs without drift and with irregular diffusion coefficient, but they do not provide a quantitative error estimate. Lototsky \cite{Lo22} provides a quantitative error estimate for the functional central limit theorem for Brownian motion in terms of the linear interpolation of Gaussian random walks and applies it to one-dimensional Markovian SDEs with Lipschitz continuous drift coefficient and constant diffusion coefficient. In the latter paper \cite{Lo22}, it is shown that the sharp convergence order in terms of the Wasserstein metric (which is stronger than the L\'evy--Prokhorov metric) is $\sqrt{|\pi|\log\frac{1}{|\pi|}}$. Similarly, it turns out that the weak convergence speed of linear interpolations of any random variables (with respect to the uniform partitions $\pi$ of $[0,T]$) to Brownian motion cannot be faster than $\sqrt{|\pi|\log\frac{1}{|\pi|}}$ in terms of the L\'evy--Prokhorov metric; see \cref{prop_lower-bound}. Hence, the order $\sqrt{|\pi|\log\frac{1}{|\pi|}}$ can be seen as a benchmark for the convergence order of the objective \eqref{eq_LP} in this paper.

The above mentioned works on the weak convergence of numerical methods require the coefficients to be either Markovian, bounded or smooth with bounded derivatives. In addition, they typically require the uniform ellipticity for the diffusion coefficient. However, these assumptions are restrictive in view of applications. Most of stochastic models appearing in science have unbounded or even super-linearly growing coefficients together with non-uniformly elliptic diffusion coefficient, and they are typically defined only on a proper subset of the Euclidean space. Gy\"ongy and Krylov \cite{GyKr96} show the convergence in probability of the Euler--Maruyama scheme under a setting including these situations. However, they are still restricted to the Markovian case, and do not provide a convergence rate when the coefficients are super-linearly growing. There are many important non-Markovian models involving strongly nonlinear coefficients as well as path-dependent structures such as the time-delay of the system. This motivates us to study SFDEs \eqref{eq_SFDE} defined only on a domain $D$ under mild assumptions on the coefficients. Specifically, the framework in the present paper allows the coefficients to be path-dependent, locally bounded, locally H\"{o}lder continuous and non-uniformly elliptic; see \cref{assum} below for more details. In particular, our framework includes not only the case of super-linearly growing coefficients but also the case of the coefficients involving strongly nonlinear terms such as $x^{-1}$. Under this quite general setting, in our main result (\cref{theo_main}), we provide a quantitative estimate for the functional type weak approximation in terms of the L\'evy--Prokhorov metric \eqref{eq_LP}. Furthermore, as a by-product of the weak convergence result, we obtain a novel result on uniqueness in law for the SFDE \eqref{eq_SFDE}. Our main result can be applied to many examples appearing in science which are beyond the existing works on numerical approximations. In \cref{section_example}, we apply \cref{theo_main} to ten concrete examples of Markovian and non-Markovian models appearing in physics, chemistry, mathematical finance, economics, population biology, and so on. Interestingly, the path-dependence of the coefficients enables us to apply our result not only to SDEs with delay but also to some non-standard Markovian SDEs such as reflected SDEs and stochastic oscillator models after appropriate transformations; see \cref{subsec_remark-specialSFDE} for more details.

The significance of the main result of this paper lies not only in the generality of the framework as mentioned above but also in the sharpness of the error estimate. Indeed, in a special case of \cref{theo_main}, we reach the optimal convergence order $\sqrt{|\pi|\log\frac{1}{|\pi|}}$ for the L\'evy--Prokhorov metric \eqref{eq_LP}; see \cref{cor_rate-polynomial} and \cref{prop_lower-bound}. Furthermore, it is worth mentioning that the topology of the weak convergence (which is nothing but the topology generated by the L\'evy--Prokhorov metric) obtained in \cref{theo_main} cannot be improved to the topology generated by the Wasserstein metric in general; see \cref{exam_Wasserstein-divergence} for such an example. Indeed, Hutzenthaler, Jentzen and Kloeden \cite{HuJeKl11,HuJeKl15} show that, for any one-dimensional, non-degenerate and Markovian SDEs (defined on the whole space $\bR$) with super-linearly growing coefficients, the $p$-th moment $\bE_{\bP^\pi}[|X^\pi(T)|^p]$ of the (standard) Euler--Maruyama scheme $X^\pi$ diverges as $|\pi|\downarrow0$ for any $p\in(0,\infty)$ and $T\in(0,\infty)$. This implies the divergence of the error in terms of the Wasserstein metric even for the one-dimensional distribution at each time $T\in(0,\infty)$. Recently, due to this difficulty in the strongly nonlinear coefficients case, several authors have considered modified versions of the Euler--Maruyama scheme and showed their strong convergence. For example, in the Markovian setting, \cite{HuJeKl12,Sa16} consider the tamed Euler--Maruyama scheme, \cite{Al13,HiMaSt02,NeSz14} consider the backward Euler--Maruyama scheme, and \cite{ChJaMi16} consider the truncated Euler--Maruyama scheme. Moreover, Guo, Mao and Yue \cite{GuMaYu18} apply the truncated Euler--Maruyama scheme to the case of SDEs with delay. Unlike this recent trend, in the present paper, we do not consider any modifications of the Euler--Maruyama scheme but focus on the standard one. Even for the case of multi-dimensional, path-dependent and super-linearly growing coefficients case, our results ensure the functional type weak convergence of the standard Euler--Maruyama scheme together with its quantitative error estimate.

Typical methods adopted in the literature to establish quantitative weak error estimates for the Euler--Maruyama scheme are based on the ``PDE approach'', which relies on the analysis of the corresponding Kolmogorov equation; see, for example, \cite{BaTa96,GoLa08,KoMe17,MiPl91}. Such methods fundamentally rely on the Markovian structure of SDEs. Recently, Holland \cite{Ho24} combined the PDE approach with the stochastic sewing lemma of L\^e \cite{Le20} to improve the weak convergence rate; see also \cite{BuDaGe19,DaGeLe23} for applications of the stochastic sewing lemma to strong convergence analysis. Although the stochastic sewing lemma itself does not rely on the Markovian structure, Holland's weak error analysis still relies on the Feynman--Kac formula and the corresponding Kolmogorov equation. Extending this approach to the path-dependent SFDE \eqref{eq_SFDE} would require the analysis of an infinite-dimensional Kolmogorov equation and is therefore quite difficult. In the case of SDEs with delay, Cl\'ement, Kohatsu-Higa and Lamberton \cite{ClKoLa06} develop the so-called ``duality approach'' based on Malliavin calculus, for which smoothness of the coefficients is crucial. Ngo and Taguchi \cite{NgTa18} adopt the ``Girsanov transform approach'' for Markovian SDEs with an irregular drift coefficient and a constant diffusion coefficient, and Bao and Shao \cite{BaSh25} extend this approach to SDEs with delayed drift coefficients, still assuming a constant diffusion coefficient. These approaches mainly concern weak convergence of finite-dimensional distributions, measured by \eqref{eq_one-dim-error} at each fixed time $T$. They require the coefficients to be path-independent, bounded, or sufficiently smooth. Moreover, in all the above-mentioned works, boundedness and uniform ellipticity of the diffusion coefficient play crucial roles. The present paper aims to substantially weaken these restrictive assumptions.

In this paper, inspired by the work by Kulik and Scheutzow \cite{KuSc20} on weak well-posedness and ergodicity of SFDEs, we demonstrate a ``\emph{generalized coupling approach}'', which is a stochastic control-type approach based on the analysis of the so-called \emph{generalized coupling} of probability measures. A (true) coupling of probability measures is defined as a probability measure on the product space with prescribed marginal laws. It is known that the L\'evy--Prokhorov metric has a dual representation by means of couplings; see \eqref{eq_LP-coupling} in \cref{subsec_LP} below. Hence, it is important to construct a reasonable true coupling in order to give a good estimate for the L\'evy--Prokhorov metric \eqref{eq_LP}. However, the construction and estimate of a reasonable true coupling are quite difficult in general. In contrast, a generalized coupling allows the marginal to have a mild deviation bound instead of the exact coincidence with respect to the prescribed probability measure, and the construction of the generalized coupling satisfying a desired property is typically easier than the construction of the true coupling. The general idea of the generalized coupling approach, which is also known as the Control-and-Reimburse strategy, is ``\emph{to apply a stochastic control in order to improve the system, and then to take into account the impact of the control}'' \cite{KuSc20}. In the present paper, we introduce a \emph{controlled Euler--Maruyama scheme} $\widehat{X}^\pi$ which ``improves'' the (true) Euler--Maruyama scheme $X^\pi$ using additional control parameters. The joint distributions of the pair of $X_T$ and (the linear interpolation of) $\widehat{X}^\pi$ can be seen as a generalized coupling between $\Law_\bP(X_T)$ and $\Law_{\bP^\pi}(\poly[X^\pi])$. The error between a (given) weak solution $X$ and the controlled Euler--Maruyama scheme $\widehat{X}^\pi$ can be estimated by means of the control parameters. Then, taking into account the deviation from the true Euler--Maruyama scheme $X^\pi$ to the controlled one $\widehat{X}^\pi$, which corresponds to the ``impact of the control'', we can formulate a ``stochastic control problem''. Solving the control problem, we can obtain a desired weak error estimate for the L\'evy--Prokhorov metric \eqref{eq_LP}, showing our main result (\cref{theo_main}). See \cref{section_proof} for more detailed idea of the generalized coupling approach, construction of the controlled Euler--Maruyama scheme $\widehat{X}^\pi$ and the proof of \cref{theo_main}. The idea of the generalized coupling investigated in \cite{KuSc20} originates from Hairer's work \cite{Ha02} on exponential mixing properties of stochastic partial differential equations and has been used in ergodic theory of Markov processes; for example, it is used in \cite{HaMaSc11} to show a general form of Harris' theorem, in \cite{BuKuSc20} to show ergodicity of various infinite dimensional Markov processes that may lack the strong Feller property, in \cite{Wa11} to show the dimension-free Harnack inequality for SDEs, and in \cite{BaWaYu19} to show the asymptotic log-Harnack inequality for SDEs with infinite delay, among others. To the best of our knowledge, the present paper is the first to apply the generalized coupling approach to theory of numerical approximations for stochastic processes.

The rest of this paper is organized as follows: \cref{section_pre} is a preliminary section, where we introduce the notations which we use throughout the paper, summarize basic properties of the L\'evy--Prokhorov metric and define the concepts of weak solutions of SFDEs as well as the Euler--Maruyama scheme. In \cref{section_main}, we state the main result (\cref{theo_main}) and show its immediate consequence (\cref{cor_rate-polynomial}). In \cref{section_remark}, we make remarks on the main result. In particular, we show in \cref{prop_lower-bound} a lower bound for the weak approximation in terms of the L\'evy--Prokhorov metric and discuss in \cref{subsec_remark-specialSFDE} some special classes of SFDEs included in our framework. \cref{section_proof} is devoted to the proof of \cref{theo_main}. In the introductory part of this section, we demonstrate the idea of the generalized coupling approach in detail and introduce the controlled Euler--Maruyama scheme. In \cref{section_example}, we apply our main result to ten concrete examples appearing in a wide range of scientific fields. In \hyperref[appendix]{Appendix}, we prove some auxiliary results used in this paper. In particular, the results in \cref{appendix_prob} provide some fundamental estimates for stochastic processes, which are important in their own right.


\section{Preliminaries}\label{section_pre}

In this section, we summarize the notations which we use throughout this paper, recall the definition and well-known properties of the L\'evy--Prokhorov metric, define the concept of weak solutions of SFDEs and introduce the Euler--Maruyama scheme.


\subsection{Notation}\label{subsec_notation}

We denote by $|\cdot|$ the standard Euclidean norm on the space $\bR^n$ of $n$-dimensional (column) vectors or the Frobenius norm on the space $\bR^{n\times d}$ of $n\times d$-matrices for each $n,d\in\bN$. $I_{n\times n}\in\bR^{n\times n}$ denotes the identity matrix. For each $\xi\in\bR^n$ and $A\subset\bR^n$, we define $\dist(\xi,A):=\inf_{\eta\in A}|\xi-\eta|$, where we set $\dist(\xi,\emptyset):=\infty$. For each set $A$, $\1_A$ is the corresponding indicator function, and $A^\complement$ is the complement relative to a given set which is clear from the context.

For a measurable space $(S,\cS)$, we denote by $\cP(S)$ the set of all probability measures on $(S,\cS)$. When $S$ is a topological space, we always choose $\cS$ as the corresponding Borel $\sigma$-algebra $\cB(S)$, and we denote by $\supp\mu$ the support of a Borel probability measure $\mu\in\cP(S)$. For two probability measures $\mu,\nu\in\cP(S)$, we denote by $\sC(\mu,\nu)\subset\cP(S\times S)$ the set of all couplings between $\mu$ and $\nu$, that is, probability measures $\fm$ on the product measurable space $(S\times S,\cS\otimes\cS)$ such that $\fm(\cdot\times S)=\mu$ and $\fm(S\times\cdot)=\nu$. For each $s\in S$, $\delta_s\in\cP(S)$ denotes the Dirac measure at the point $s$. For each random variable $\xi$ on a probability space $(\Omega,\cF,\bP)$ with values in a measurable space $(S,\cS)$, we denote by $\Law_\bP(\xi):=\bP\circ\xi^{-1}\in\cP(S)$ the law of $\xi$ under the probability measure $\bP$. The expectation under a probability measure $\bP$ is denoted by $\bE_\bP[\cdot]$; we sometimes denote it by $\bE[\cdot]$ when the underlying probability $\bP$ is clear from the context.

Let $n\in\bN$. We denote by $\cC:=C([0,\infty);\bR^n)$ the set of continuous functions from $[0,\infty)$ to $\bR^n$. For each $x\in\cC$ and $t\in[0,\infty)$, $x(t)\in\bR^n$ denotes the value of $x\in\cC$ at time $t\in[0,\infty)$, and $x_t\in\cC$ denotes the stopped function defined by $x_t(s):=x(t\wedge s)$ for each $s\in[0,\infty)$. We denote by $\cC_t$ the set of all functions $x\in\cC$ such that $x=x_t$. The set $\cC_t$ becomes a separable Banach space with the norm $\|x\|_\infty:=\sup_{s\in[0,t]}|x(s)|$. It is known that the Borel $\sigma$-algebra $\cB(\cC_t)$ coincides with the $\sigma$-algebra generated by the family of evaluation maps $\cC_t\ni x\mapsto x(s)\in\bR^n$ with $s\in[0,t]$. The family $\{\cG_t\}_{t\in[0,\infty)}$ of $\sigma$-algebras $\cG_t:=\{\{x\in\cC\,|\,x_t\in A\}\,|\,A\in\cB(\cC_t)\}$ with $t\in[0,\infty)$ defines a filtration on $\cC$. We denote $\Lambda:=[0,\infty)\times\cC$. We say that a set $\Gamma\subset\Lambda$ is progressively measurable if $\Gamma\cap([0,t]\times\cC)$ is $\cB([0,t])\otimes\cG_t$-measurable for any $t\in[0,\infty)$. Furthermore, we say that a map $\varphi:\Gamma\to S$ from a progressively measurable set $\Gamma\subset\Lambda$ to a measurable space $(S,\cS)$ is progressively measurable if the set $\{(t,x)\in\Gamma\,|\,\varphi(t,x)\in A\}$ is progressively measurable for any $A\in\cS$. If $\varphi:\Gamma\to S$ is progressively measurable and $\cS$ separates points of $S$, then
\begin{equation*}
	\varphi(t,x)=\varphi(t,x_t)\ \ \text{for all $(t,x)\in\Gamma$}.
\end{equation*}
Note that a set $\Gamma\subset\Lambda$ is progressively measurable if and only if the map $\1_\Gamma:\Lambda\to\bR$ is progressively measurable. The family of progressively measurable sets is a $\sigma$-algebra on $\Lambda$.

Let $D_1$ and $D_2$ be two subsets of $\bR^n$. We define
\begin{equation*}
	\Lambda[D_1,D_2]:=\left\{(t,x)\in\Lambda\relmiddle|x(0)\in D_1\ \text{and}\ x(s)\in D_2\ \text{for all $s\in[0,t]$}\right\}.
\end{equation*}
Clearly, if $(t,x)\in\Lambda[D_1,D_2]$ and $s\in[0,t]$, then $(s,x)\in\Lambda[D_1,D_2]$. Furthermore, if $D_1$ is a Borel subset of $\bR^n$ and if $D_2$ is an open subset of $\bR^n$, then the set $\Lambda[D_1,D_2]\subset\Lambda$ is progressively measurable. This follows from the right-continuity of $[0,\infty)\ni t\mapsto\1_{\Lambda[D_1,D_2]}(t,x)$ for each $x\in\cC$ and the $\cG_t$-measurability of $\cC\ni x\mapsto\1_{\Lambda[D_1,D_2]}(t,x)$ for each $t\in[0,\infty)$. For simplicity of notation, we abbreviate $\Lambda[D,D]$ to $\Lambda[D]$ when $D_1=D_2=D\subset\bR^n$. Note that $\Lambda[\bR^n]=\Lambda[\bR^n,\bR^n]=\Lambda$.

For each $(t,x)\in\Lambda$ and $\delta\in[0,\infty)$, we denote by $\varpi(t,x;\delta)$ the modulus of continuity of $x\in\cC$ on $[0,t]$ with length $\delta$, that is,
\begin{equation*}
	\varpi(t,x;\delta):=\sup_{\substack{0\leq r\leq s\leq t\\s-r\leq\delta}}|x(s)-x(r)|.
\end{equation*}
Clearly, $\varpi(t,x;0)=\varpi(0,x;\delta)=0$, and the function $\delta\mapsto\varpi(t,x;\delta)$ is finite, non-decreasing and continuous. For each $t\in[0,\infty)$ and $\delta\in[0,\infty)$, the map $\cC\ni x\mapsto\varpi(t,x;\delta)$ is $\cG_t$-measurable. Also, for each $x\in\cC$, and $\delta\in[0,\infty)$, it holds that
\begin{equation*}
	0\leq\varpi(t,x;\delta)-\varpi(s,x;\delta)\leq\varpi\big(t,x;(t-s)\wedge\delta\big),\ \ 0\leq s\leq t<\infty,
\end{equation*}
and hence $t\mapsto\varpi(t,x;\delta)$ is non-decreasing and continuous. In particular, the map $\Lambda\ni(t,x)\mapsto\varpi(t,x;\delta)$ is progressively measurable for each $\delta\in[0,\infty)$.

For each $n$-dimensional vector $\xi=(\xi_1,\dots,\xi_n)^\top\in(0,\infty)^n$, we define $\xi^{-1}:=(\xi_1^{-1},\dots,\xi_n^{-1})^\top\in(0,\infty)^n$. Also, for each $(t,x)\in\Lambda[(0,\infty)^n]$, we define $x_t^{-1}\in\cC_t$ by $x_t^{-1}(s):=x_t(s)^{-1}=x(t\wedge s)^{-1}\in(0,\infty)^n$ for each $s\in[0,\infty)$.

Throughout the paper, the natural numbers $n$ and $d$ typically represent the dimensions of the solution of the SFDE \eqref{eq_SFDE} and Brownian motion, respectively.


\subsection{Basic properties of the L\'evy--Prokhorov metric}\label{subsec_LP}

Let $(S,d_S)$ be a complete separable metric space. The L\'evy--Prokhorov metric $d_\LP(\mu,\nu)$ between two probability measures $\mu,\nu\in\cP(S)$ is defined by
\begin{equation}\label{eq_LPdef}
	d_\LP(\mu,\nu):=\inf\left\{\ep>0\relmiddle|\mu(A)\leq\nu(A^\ep)+\ep\ \text{and}\ \nu(A)\leq\mu(A^\ep)+\ep\ \text{for all}\ A\in\cB(S)\right\},
\end{equation}
where
\begin{equation*}
	A^\ep:=\left\{x\in S\relmiddle|\text{there exists $y\in A$ such that $d_S(x,y)<\ep$}\right\}.
\end{equation*}
It is well-known that $(\cP(S),d_\LP)$ is a metric space (cf.\ \cite[Theorems 11.3.1]{Du02}), and the weak convergence of probability measures on $(S,\cB(S))$ is equivalent to the convergence with respect to the L\'evy--Prokhorov metric $d_\LP$ (cf.\ \cite[Theorem 11.3.3]{Du02}). Also, for each $\mu,\nu\in\cP(S)$, the L\'evy--Prokhorov metric $d_\LP(\mu,\nu)$ can be represented by means of couplings between $\mu$ and $\nu$ as follows:
\begin{equation}\label{eq_LP-coupling}
	d_\LP(\mu,\nu)=\inf_{\fm\in\sC(\mu,\nu)}\inf\left\{\ep>0\relmiddle|\fm\big(\left\{(x,y)\in S\times S\relmiddle|d_S(x,y)>\ep\right\}\big)<\ep\right\}
\end{equation}
(cf.\ \cite[Corollary 11.6.4]{Du02}). In terms of $S$-valued random variables $\xi$ and $\eta$ defined on a (common) probability space $(\Omega,\cF,\bP)$, the above relation yields
\begin{equation}\label{eq_LP-KF}
	d_\LP\big(\Law_\bP(\xi),\Law_\bP(\eta)\big)\leq\inf\left\{\ep>0\relmiddle|\bP\big(d_S(\xi,\eta)>\ep\big)<\ep\right\}.
\end{equation}
The right-hand side above is called the Ky Fan metric between $\xi$ and $\eta$, which induces the topology of convergence in probability of $S$-valued random variables on the prescribed probability space $(\Omega,\cF,\bP)$ (cf.\ \cite[Theorem 9.2.2]{Du02}).

In the literature of numerical analysis of SDEs, the weak rate of convergence of an approximation scheme is often measured by means of a functional with respect to a class $\Xi$ of measurable test functions $f:S\to\bR$. For such a purpose, it is natural to introduce the following quantity:
\begin{equation*}
	d_\Xi(\mu,\nu):=\sup_{f\in\Xi}\left|\int_Sf(x)\,\mu(\diff x)-\int_Sf(x)\,\nu(\diff x)\right|
\end{equation*}
for each $\mu,\nu\in\cP(S)$ such that $\int_S|f(x)|\mu(\diff x)<\infty$ and $\int_S|f(x)|\nu(\diff x)<\infty$ for any $f\in\Xi$. In terms of $S$-valued random variables $\xi$ and $\eta$ defined on a probability space $(\Omega,\cF,\bP)$ such that $\bE[|f(\xi)|]<\infty$ and $\bE[|f(\eta)|]<\infty$ for any $f\in\Xi$, we have
\begin{equation*}
	d_\Xi\big(\Law_\bP(\xi),\Law_\bP(\eta)\big)=\sup_{f\in\Xi}\left|\bE\big[f(\xi)\big]-\bE\big[f(\eta)\big]\right|,
\end{equation*}
where $\bE[\cdot]$ denotes the expectation with respect to $\bP$. The L\'{e}vy--Prokhorov metric is related to the metric $d_\Xi$ induced by the class $\Xi=\BL$ of bounded and Lipschitz continuous functions, that is,
\begin{equation}\label{eq_BL}
	\BL:=\left\{f:S\to\bR\relmiddle|\sup_{x\in S}|f(x)|+\sup_{\substack{x,y\in S\\x\neq y}}\frac{|f(x)-f(y)|}{d_S(x,y)}\leq1\right\}.
\end{equation}
Indeed, it is known that
\begin{equation}\label{eq_LP-BL}
	\frac{2d_\LP(\mu,\nu)^2}{2+d_\LP(\mu,\nu)}\leq d_\BL(\mu,\nu)\leq2d_\LP(\mu,\nu)\ \ \text{for any $\mu,\nu\in\cP(S)$},
\end{equation}
and both the above two inequalities are sharp (cf.\ \cite[Sections 11.3 and 11.6]{Du02}).
The above relations show that the topologies induced by the metrics $d_\BL$ and $d_\LP$ are the same, which coincide with the weak convergence topology. However, we should be careful when we treat the order of weak convergence. The above relations mean that, for each $\mu,\mu^\ep\in\cP(S)$ with $\ep\in(0,1]$ and $r\in(0,\infty)$, as $\ep\downarrow0$,
\begin{itemize}
\item
the convergence order $d_\LP(\mu,\mu^\ep)=O(\ep^r)$ with respect to the L\'{e}vy--Prokhorov metric implies the same convergence order $d_\BL(\mu,\mu^\ep)=O(\ep^r)$ with respect to the metric induced by the class $\BL$ of test functions, while
\item
the convergence order $d_\BL(\mu,\mu^\ep)=O(\ep^r)$ implies only the convergence order $d_\LP(\mu,\mu^\ep)=O(\ep^{r/2})$ in general.
\end{itemize}

We refer the reader to \cref{appendix_LP} for additional properties of the L\'{e}vy--Prokhorov metric that will be utilized in this paper.

In this paper, we are mainly interested in the convergence order with respect to the L\'{e}vy--Prokhorov metric $d_\LP$ associated with the metric space $(S,d_S)=(\cC_T,\|\cdot-\cdot\|_\infty)$. From the above observations, we see that such quantitative results imply the same convergence order with respect to $d_\BL$ corresponding to $(S,d_S)=(\cC_T,\|\cdot-\cdot\|_\infty)$.


\begin{rem}
It is important to keep in mind that the quantitative results concerning the L\'evy--Prokhorov metric in this paper depend not only on the topology of $S$ but also on the choice of the metric $d_S$ itself.
\end{rem}


\subsection{Weak solutions of SFDEs}\label{subsec_weaksol}

We are concerned with the SFDE \eqref{eq_SFDE} and its weak solution taking values on a subset of the Euclidean space. First, we provide a precise notion of the \emph{data} of the SFDE \eqref{eq_SFDE}.


\begin{defi}
By \emph{data} of the SFDE \eqref{eq_SFDE}, we mean a tuple $\data$ consisting of:
\begin{itemize}
\item
a non-empty open set $D\subset\bR^n$, called a \emph{domain};
\item
a probability measure $\mu_0\in\cP(\bR^n)$ such that $\mu_0(D)=1$, called an \emph{initial distribution};
\item
progressively measurable maps $b:\Lambda[\supp\mu_0,D]\to\bR^n$ and $\sigma:\Lambda[\supp\mu_0,D]\to\bR^{n\times d}$, called \emph{coefficients}.
\end{itemize}
\end{defi}

We define the concepts of weak solutions and uniqueness in law for the SFDE \eqref{eq_SFDE} as follows.


\begin{defi}
We say that a tuple $\weaksol$ is a \emph{weak solution} of the SFDE \eqref{eq_SFDE} associated with data $\data$ if the following hold:
\begin{itemize}
\item
$(\Omega,\cF,\bP)$ is a complete probability space, and $\bF=(\cF_t)_{t\geq0}$ is a filtration satisfying the usual conditions;
\item
$W$ is an $\bR^d$-valued Brownian motion on $(\Omega,\cF,\bP)$ relative to $\bF$;
\item
$X$ is an $\bR^n$-valued continuous $\bF$-adapted process on $(\Omega,\cF,\bP)$ such that $\bP\circ X(0)^{-1}=\mu_0$,
\begin{equation*}
	\bP\left(\text{$X(t)\in D$ for any $t\in[0,\infty)$}\right)=1,
\end{equation*}
and
\begin{equation*}
	\bP\left(\text{$\int^T_0\big\{|b(t,X)|+|\sigma(t,X)|^2\big\}\,\diff t<\infty$ for any $T\in[0,\infty)$}\right)=1;
\end{equation*}
\item
$\bP$-a.s., it holds that
\begin{equation*}
	X(t)=X(0)+\int^t_0b(s,X)\,\diff s+\int^t_0\sigma(s,X)\,\diff W(s)
\end{equation*}
for any $t\in[0,\infty)$.
\end{itemize}
We say that \emph{uniqueness in law} holds for the SFDE \eqref{eq_SFDE} associated with data $\data$ if, for any two weak solutions $(X^i,W^i,\Omega^i,\cF^i,\bF^i,\bP^i)$, $i=1,2$, of the SFDE \eqref{eq_SFDE} associated with data $\data$, it holds that
\begin{equation*}
	\Law_{\bP^1}(X^1)=\Law_{\bP^2}(X^2).
\end{equation*}
Here, for $i=1,2$, $\Law_{\bP^i}(X^i)$ denotes the law of $X^i$ on $\cC$ under the probability measure $\bP^i$.
\end{defi}


\begin{rem}
We emphasize that the above definition requires a weak solution
of the SFDE \eqref{eq_SFDE} associated with data $\data$ to satisfy
$\bP(\text{$X(t)\in D$ for any $t\in[0,\infty)$})=1$. The concept of the uniqueness in law is also considered in this manner. Hence, the case where $X$ takes values in $\bR^n\setminus D$ with positive probability is a priori excluded from the definition. Also, notice that uniqueness in law holds if there is no weak solution. In this paper, we focus on the weak approximation of a given weak solution as well as uniqueness in law of the SFDE \eqref{eq_SFDE} associated with data $\data$ satisfying mild growth, continuity and ellipticity conditions specified in \cref{assum} below.
\end{rem}


\subsection{Euler--Maruyama scheme}\label{subsec_EM}

For each $T\in(0,\infty)$, let $\Pi_T$ be the collection of all partitions of $[0,T]$, that is,
\begin{equation*}
	\Pi_T:=\left\{\pi=(t_k)^m_{k=0}\relmiddle|m\in\bN,\ t_0=0,\ t_k<t_{k+1}\ \text{for any $k\in\{0,\dots,m-1\}$, and}\ t_m=T\right\}.
\end{equation*}
For each $\pi=(t_k)^m_{k=0}\in\Pi_T$, denote $|\pi|:=\max_{k\in\{0,\dots,m-1\}}(t_{k+1}-t_k)$. Furthermore, for each vector $x=(x(t_k))^m_{k=0}\in(\bR^n)^{m+1}$, define $\poly[x]\in\cC_T$ as the linear interpolation of $(x(t_k))^m_{k=0}$ with respect to $\pi$, that is,
\begin{equation}\label{eq_poly}
	\poly[x](t):=\sum^{m-1}_{k=0}\left\{\frac{t_{k+1}-t}{t_{k+1}-t_k}x(t_k)+\frac{t-t_k}{t_{k+1}-t_k}x(t_{k+1})\right\}\1_{[t_k,t_{k+1})}(t)+x(t_m)\1_{[t_m,\infty)}(t)
\end{equation}
for $t\in[0,\infty)$. Clearly, the map $\poly:(\bR^n)^{m+1}\to\cC_T$ is continuous and linear. Notice that, for each $k\in\{0,\dots,m-1\}$, the stopped function $\poly[x]_{t_k}\in\cC_{t_k}$ is given by
\begin{equation}\label{eq_poly-t_k}
	\poly[x]_{t_k}(t)=\sum^{k-1}_{j=0}\left\{\frac{t_{j+1}-t}{t_{j+1}-t_j}x(t_j)+\frac{t-t_j}{t_{j+1}-t_j}x(t_{j+1})\right\}\1_{[t_j,t_{j+1})}(t)+x(t_k)\1_{[t_k,\infty)}(t)
\end{equation}
for $t\in[0,\infty)$, which depends only on $(x(t_j))^k_{j=0}$.
Note also that, for each Borel set $D_1\subset\bR^n$ and open and convex set $D_2\subset\bR^n$, $(t_k,\poly[x])\in\Lambda[D_1,D_2]$ if and only if $x(t_0)\in D_1$ and $x(t_j)\in D_2$ for all $j\in\{0,\dots,k\}$.

Let data $\data$ be given. We extend the domain $\Lambda[\supp\mu_0,D]$ of the coefficients $b$ and $\sigma$ to the whole space $\Lambda$ by setting $b(t,x)=0$ and $\sigma(t,x)=0$ for each $(t,x)\in\Lambda\setminus\Lambda[\supp\mu_0,D]$. Then, $b$ and $\sigma$ can be seen as progressively measurable maps from $\Lambda$ to $\bR^n$ and $\bR^{n\times d}$, respectively. For each $T\in(0,\infty)$ and $\pi=(t_k)^m_{k=0}\in\Pi_T$, let $(\Omega^\pi,\cF^\pi,\bP^\pi)$ be a complete probability space supporting independent random variables $\xi^\pi$ and $Z^\pi_k$ for $k\in\{0,\dots,m-1\}$, where $\xi^\pi$ is an $\bR^n$-valued random variable distributed according to the given initial distribution $\mu_0$, and each $Z^\pi_k$ is a $d$-dimensional random variable distributed according to the Gaussian distribution $N(0,(t_{k+1}-t_k)I_{d\times d})$ with mean zero and covariance matrix $(t_{k+1}-t_k)I_{d\times d}$. Without loss of generality, we may assume that $\xi^\pi(\omega)\in\supp\mu_0$ for all $\omega\in\Omega^\pi$. Under this setting, define the Euler--Maruyama scheme $X^\pi=(X^\pi(t_k))^m_{k=0}$ inductively as follows:
\begin{equation}\label{eq_EM}
	\begin{dcases}
	X^\pi(t_0)=\xi^\pi,\\
	X^\pi(t_{k+1})=X^\pi(t_k)+b\big(t_k,\poly[X^\pi]\big)(t_{k+1}-t_k)+\sigma\big(t_k,\poly[X^\pi]\big)Z^\pi_k,
	\ k\in\{0,\dots,m-1\}.
	\end{dcases}
\end{equation}
Thanks to \eqref{eq_poly-t_k} and $\varphi(t_k,\poly[X^\pi])=\varphi(t_k,\poly[X^\pi]_{t_k})$ for $\varphi\in\{b,\sigma\}$, the Euler--Maruyama scheme \eqref{eq_EM} is an explicit scheme, and we have $X^\pi=\Psi^\pi(\xi^\pi,Z^\pi_0,\dots,Z^\pi_{m-1})$ for some measurable map $\Psi^\pi:\bR^n\times(\bR^d)^m\to(\bR^n)^{m+1}$.


\begin{rem}\label{rem_domain}
If, for any $\bar{x}\in D$, $\sigma(0,\bar{x}\1_{[0,\infty)}(\cdot))\in\bR^{n\times d}$ is non-degenerate in the sense that all eigenvalues of $\sigma(0,\bar{x}\1_{[0,\infty)}(\cdot))\sigma(0,\bar{x}\1_{[0,\infty)}(\cdot))^\top$ are positive, then $\sigma(0,\bar{x}\1_{[0,\infty)}(\cdot))Z^\pi_0$ is a non-degenerate Gaussian random variable on $\bR^n$ for each $\bar{x}\in D$. Thus, when $D$ is a proper subset of $\bR^n$, even if $X^\pi(t_0)=\xi^\pi\in D$ a.s., the random variable $X^\pi(t_1)$ defined by the first step of the Euler--Maruyama scheme \eqref{eq_EM} takes values in $\bR^n\setminus D$ with positive probability. This is why we need to enlarge the domain $\Lambda[\supp\mu_0,D]$ of $b$ and $\sigma$ to the whole space $\Lambda$. Since $D$ is assumed to be convex and $X^\pi(t_0)=\xi^\pi\in\supp\mu_0$, we have $(t_k,\poly[X^\pi])\in\Lambda[\supp\mu_0,D]$ if and only if $X^\pi(t_j)\in D$ for all $j\in\{0,\dots,k\}$. By virtue of the convention that $b(t,x)=0$ and $\sigma(t,x)=0$ for each $(t,x)\in\Lambda\setminus\Lambda[\supp\mu_0,D]$, we have $X^\pi(t_\ell)=X^\pi(t_k)$ for any $\ell\in\{k+1,\dots,m\}$ on the event $\{X^\pi(t_k)\notin D\}$ for each $k\in\{0,\dots,m-1\}$.
\end{rem}

The purpose of this paper is to give an error estimate for the weak approximation of SFDE \eqref{eq_SFDE} by the Euler--Maruyama scheme \eqref{eq_EM} under mild regularity conditions for the (path-dependent) coefficients $b$ and $\sigma$. More precisely, for each $T\in(0,\infty)$ and $\pi\in\Pi_T$, we estimate the L\'evy--Prokhorov metric $d_\LP(\Law_\bP(X_T),\Law_{\bP^\pi}(\poly[X^\pi]))$ between the law of the path of the weak solution $\weaksol$ of the SFDE associated with data $\data$ (if any) and the law of the linear interpolation of the Euler--Maruyama scheme $X^\pi=\EM$ having the same initial distribution as $X$.


\section{Main result: Weak approximation of SFDEs}\label{section_main}

We consider the following assumption on the data $\data$ of the SFDE \eqref{eq_SFDE}.


\begin{assum}\label{assum}
The domain $D$ is a non-empty, open and convex subset of $\bR^n$, and the coefficients $b:\Lambda[\supp\mu_0,D]\to\bR^n$ and $\sigma:\Lambda[\supp\mu_0,D]\to\bR^{n\times d}$ satisfy the following conditions:
\begin{itemize}
\item[(G)](The growth condition.)
For any $T\in(0,\infty)$, there exist constants $K_{\growth,b;T},K_{\growth,\sigma;T}\in[1,\infty)$ and increasing families $\{D_{\growth,b;T}(R)\}_{R\in[1,\infty)}$ and $\{D_{\growth,\sigma;T}(R)\}_{R\in[1,\infty)}$ of open and convex subsets of $D$ with $\bigcup_{R\in[1,\infty)}D_{\growth,b;T}(R)=D$ and $\bigcup_{R\in[1,\infty)}D_{\growth,\sigma;T}(R)=D$ such that
\begin{equation*}
	|\varphi(t,x)|\leq K_{\growth,\varphi;T}R
\end{equation*}
for any $(t,x)\in\Lambda\big[\supp\mu_0,D_{\growth,\varphi;T}(R)\big]$ with $t\leq T$, $R\in[1,\infty)$ and $\varphi\in\{b,\sigma\}$.
\item[(C)](The local H\"{o}lder continuity condition.)
For any $T\in(0,\infty)$, there exist constants $K_{\conti,b;T},K_{\conti,\sigma;T}\in(0,\infty)$, $\alpha_b\in(0,1]$, $\alpha_\sigma\in(\frac{1}{2},1]$, and increasing families $\{D_{\conti,b;T}(R)\}_{R\in[1,\infty)}$ and $\{D_{\conti,\sigma;T}(R)\}_{R\in[1,\infty)}$ of open and convex subsets of $D$ with $\bigcup_{R\in[1,\infty)}D_{\conti,b;T}(R)=D$ and $\bigcup_{R\in[1,\infty)}D_{\conti,\sigma;T}(R)=D$ such that
\begin{equation}\label{eq_C1}
	|\varphi(t,x)-\varphi(t,y)|\leq K_{\conti,\varphi;T}R\|x_t-y_t\|^{\alpha_\varphi}_\infty
\end{equation}
and
\begin{equation}\label{eq_C2}
	|\varphi(t,x)-\varphi(s,x)|\leq K_{\conti,\varphi;T}R\Big\{\varpi(t,x;t-s)+(t-s)^{1/2}\Big\}^{\alpha_\varphi}
\end{equation}
for any $(t,x),(t,y)\in\Lambda\big[\supp\mu_0,D_{\conti,\varphi;T}(R)\big]$ with $t\leq T$, $s\in[0,t]$, $R\in[1,\infty)$ and $\varphi\in\{b,\sigma\}$.
\item[(E)](The local ellipticity condition.)
For any $T\in(0,\infty)$, there exists a constant $K_{\ellip;T}\in[1,\infty)$ and an increasing family $\{D_{\ellip;T}(R)\}_{R\in[1,\infty)}$ of open and convex subsets of $D$ with $\bigcup_{R\in[1,\infty)}D_{\ellip;T}(R)=D$ such that
\begin{equation*}
	\langle\sigma(t,x)\sigma(t,x)^\top\xi,\xi\rangle\geq\frac{|\xi|^2}{K_{\ellip;T}^2R^2}
\end{equation*}
for any $\xi\in\bR^n$, $(t,x)\in\Lambda\big[\supp\mu_0,D_{\ellip;T}(R)\big]$ with $t\leq T$ and $R\in[1,\infty)$.
\end{itemize}
\end{assum}

Under \cref{assum}, for each $T\in(0,\infty)$ and $\vec{R}=(R_{\growth,b},R_{\growth,\sigma},R_{\conti,b},R_{\conti,\sigma},R_\ellip)\in[1,\infty)^5$, denote
\begin{equation*}
	D_T(\vec{R}):=D_{\growth,b;T}(R_{\growth,b})\cap D_{\growth,\sigma;T}(R_{\growth,\sigma})\cap D_{\conti,b;T}(R_{\conti,b})\cap D_{\conti,\sigma;T}(R_{\conti,\sigma})\cap D_{\ellip;T}(R_\ellip).
\end{equation*}

Now we are ready to state our main result. The proof is postponed to \cref{section_proof}.


\begin{theo}\label{theo_main}
Fix data $\data$ satisfying \cref{assum}. Suppose that we are given a weak solution $\weaksol$ of the SFDE \eqref{eq_SFDE} associated with $\data$. For each $T\in(0,\infty)$ and $\pi\in\Pi_T$, let $X^\pi$ be the Euler--Maruyama scheme defined on $(\Omega^\pi,\cF^\pi,\bP^\pi)$ and given by \eqref{eq_EM} with initial distribution $\mu_0$. Then, the following hold:
\begin{itemize}
\item[(i)]
For any $T\in(0,\infty)$, $\poly[X^\pi]\to X_T$ weakly on $\cC_T$ as $|\pi|\downarrow0$ along $\pi\in\Pi_T$. In particular, uniqueness in law holds for the SFDE \eqref{eq_SFDE} associated with $\data$.
\item[(ii)]
For any $T\in(0,\infty)$, $\pi\in\Pi_T$, $\Delta\in(0,1]$ and $\vec{R}=(R_{\growth,b},R_{\growth,\sigma},R_{\conti,b},R_{\conti,\sigma},R_{\ellip})\in[1,\infty)^5$, it holds that
\begin{equation}\label{eq_error}
\begin{split}
	&d_\LP\big(\Law_\bP(X_T),\Law_{\bP^\pi}(\poly[X^\pi])\big)\\
	&\leq\bP\left(\inf_{t\in[0,T]}\dist\!\left(X(t),\bR^n\setminus D_T(\vec{R})\right)\leq\Delta\right)\\
	&\hspace{0.5cm}+2^{12}n\left(T+\frac{1}{T}\right)K_{\ellip;T}R_\ellip\max\left\{\cE,K_{\conti,b;T}R_{\conti,b}\cE^{\alpha_b},\frac{\left(K_{\conti,\sigma;T}R_{\conti,\sigma}\cE^{\alpha_\sigma}\right)^2}{\Delta}\log\frac{1}{\Delta}\right\},
\end{split}
\end{equation}
where $\cE=\cE(|\pi|,\Delta,R_{\growth,b},R_{\growth,\sigma})\in(0,\infty)$ is defined by
\begin{equation}\label{eq_cE}
	\cE:=\max\left\{\Delta,K_{\growth,b;T}R_{\growth,b}|\pi|,K_{\growth,\sigma;T}R_{\growth,\sigma}\sqrt{|\pi|\log\frac{T}{|\pi|}}\right\}.
\end{equation}
\item[(iii)]
Let $T\in(0,\infty)$ be fixed. Assume that there exist $\widehat{C}\in(0,\infty)$ and $\vec{\beta}=(\beta_0,\beta_{\growth,b},\beta_{\growth,\sigma},\beta_{\conti,b},\beta_{\conti,\sigma},\beta_\ellip)\in(0,\infty)^6$ such that
\begin{equation}\label{eq_rare-event}
\begin{split}
	\bP\left(\inf_{t\in[0,T]}\dist\!\left(X(t),\bR^n\setminus D_T(\vec{R})\right)\leq\Delta\right)
	&\leq\widehat{C}\max\left\{\Delta^{\beta_0},R_{\growth,b}^{-\beta_{\growth,b}},R_{\growth,\sigma}^{-\beta_{\growth,\sigma}},R_{\conti,b}^{-\beta_{\conti,b}},R_{\conti,\sigma}^{-\beta_{\conti,\sigma}},R_\ellip^{-\beta_\ellip}\right\}
\end{split}
\end{equation}
for any $\Delta\in(0,1]$ and $\vec{R}=(R_{\growth,b},R_{\growth,\sigma},R_{\conti,b},R_{\conti,\sigma},R_\ellip)\in[1,\infty)^5$. Define
\begin{equation}\label{eq_gamma*}
\begin{split}
	\gamma_*&:=\min\left\{\frac{\alpha_b}{1+\beta_\ellip^{-1}+\beta_{\conti,b}^{-1}+\beta_{\growth,b}^{-1}\alpha_b},\frac{1}{2}\cdot\frac{\alpha_b}{1+\beta_\ellip^{-1}+\beta_{\conti,b}^{-1}+\beta_{\growth,\sigma}^{-1}\alpha_b},\right.\\
	&\hspace{1.5cm}\left.\frac{2\alpha_\sigma}{1+\beta_\ellip^{-1}+2\beta_{\conti,\sigma}^{-1}+2\beta_{\growth,b}^{-1}\alpha_\sigma+\beta_*^{-1}},\frac{\alpha_\sigma}{1+\beta_\ellip^{-1}+2\beta_{\conti,\sigma}^{-1}+2\beta_{\growth,\sigma}^{-1}\alpha_\sigma+\beta_*^{-1}}\right\},
\end{split}
\end{equation}
where
\begin{equation}\label{eq_beta*}
	\beta_*:=\min\left\{\beta_0,\frac{\alpha_b}{1+\beta_\ellip^{-1}+\beta_{\conti,b}^{-1}},\frac{2\alpha_\sigma-1}{1+\beta_\ellip^{-1}+2\beta_{\conti,\sigma}^{-1}}\right\}.
\end{equation}
Then, for any $\gamma\in(0,\gamma_*)$, there exists a constant $\cK_{T,\gamma}\in[1,\infty)$, which depends only on $T$, $K_{\growth,b;T}$, $K_{\growth,\sigma;T}$, $K_{\conti,b;T}$, $K_{\conti,\sigma;T}$, $K_{\ellip;T}$, $\alpha_b$, $\alpha_\sigma$ and $\gamma_*-\gamma$, such that
\begin{equation*}
	d_\LP\big(\Law_\bP(X_T),\Law_{\bP^\pi}(\poly[X^\pi])\big)\leq\big(\widehat{C}+n\cK_{T,\gamma}\big)|\pi|^\gamma
\end{equation*}
for any $\pi\in\Pi_T$.
\end{itemize}
\end{theo}


\begin{rem}\label{rem_gamma*}
The numbers $\beta_*=\beta_*(\alpha_b,\alpha_\sigma,\vec{\beta})$ and $\gamma_*=\gamma_*(\alpha_b,\alpha_\sigma,\vec{\beta})$ defined in \eqref{eq_beta*} and \eqref{eq_gamma*}, respectively, are positive and non-decreasing with respect to each variable of $\vec{\beta}=(\beta_0,\beta_{\growth,b},\beta_{\growth,\sigma},\beta_{\conti,b},\beta_{\conti,\sigma},\beta_\ellip)\in(0,\infty)^6$. Moreover,
\begin{equation*}
	\sup_{\vec{\beta}\in(0,\infty)^6}\beta_*(\alpha_b,\alpha_\sigma,\vec{\beta})=\alpha_b\wedge\left(2\alpha_\sigma-1\right)\ \ \text{and}\ \ \sup_{\vec{\beta}\in(0,\infty)^6}\gamma_*(\alpha_b,\alpha_\sigma,\vec{\beta})=\frac{\alpha_b}{2}\wedge\left(\alpha_\sigma-\frac{1}{2}\right).
\end{equation*}
\end{rem}

As an immediate consequence of the above result, we obtain the next corollary for the SFDE \eqref{eq_SFDE} defined on $\bR^n$ or $(0,\infty)^n$ having ``polynomial type'' growth conditions. Here, in order to state the results in a unified manner, we introduce a progressively measurable map $\Psi:\Lambda[D]\to[0,\infty)$, which is defined depending on whether $D=\bR^n$ or $D=(0,\infty)^n$:
\begin{equation*}
	\Psi(t,x):=
	\begin{dcases}
	\|x_t\|_\infty\ \ &\text{if $D=\bR^n$},\\
	\|x_t\|_\infty+\|x^{-1}_t\|_\infty\ \ &\text{if $D=(0,\infty)^n$},
	\end{dcases}
\end{equation*}
for $(t,x)\in\Lambda[D]$, where $\xi^{-1}:=(\xi_1^{-1},\dots,\xi_n^{-1})^\top$ for each $\xi=(\xi_1,\dots,\xi_n)^\top\in(0,\infty)^n$, and $x_t^{-1}:=x(t\wedge\cdot)^{-1}$ for each $(t,x)\in\Lambda[(0,\infty)^n]$.


\begin{cor}\label{cor_rate-polynomial}
Consider data $\data$ with $D=\bR^n$ or $D=(0,\infty)^n$. Assume that there exist constants $\check{C}\in[1,\infty)$, $p\in[0,\infty)$, $\alpha_b\in(0,1]$ and $\alpha_\sigma\in(\frac{1}{2},1]$ such that the following hold:
\begin{itemize}
\item[\rm{(G')}]
\rm{(The polynomial growth condition.)}
\begin{equation*}
	|\varphi(t,x)|\leq\check{C}\big(1+\Psi(t,x)^p\big)
\end{equation*}
for any $(t,x)\in\Lambda[\supp\mu_0,D]$ and $\varphi\in\{b,\sigma\}$.
\item[\rm{(C')}]
\rm{(The local H\"{o}lder continuity condition with polynomial coefficients.)}
\begin{equation*}
	|\varphi(t,x)-\varphi(t,y)|\leq\check{C}\big(1+\Psi(t,x)^p+\Psi(t,y)^p\big)\|x_t-y_t\|^{\alpha_\varphi}_\infty
\end{equation*}
and
\begin{equation*}
	|\varphi(t,x)-\varphi(s,x)|\leq\check{C}\big(1+\Psi(t,x)^p\big)\Big\{\varpi(t,x;t-s)+(t-s)^{1/2}\Big\}^{\alpha_\varphi}
\end{equation*}
for any $(t,x),(t,y)\in\Lambda[\supp\mu_0,D]$, $s\in[0,t]$ and $\varphi\in\{b,\sigma\}$.
\item[\rm{(E')}]
\rm{(The local ellipticity condition with polynomial coefficients.)}
\begin{equation*}
	\langle\sigma(t,x)\sigma(t,x)^\top\xi,\xi\rangle\geq\frac{|\xi|^2}{\check{C}^2\big(1+\Psi(t,x)^p\big)^2}
\end{equation*}
for any $(t,x)\in\Lambda[\supp\mu_0,D]$ and $\xi\in\bR^n$.
\end{itemize}
Then, the data $\data$ satisfy \cref{assum}. Suppose that we are given a weak solution $\weaksol$ to the SFDE \eqref{eq_SFDE} associated with $\data$. Let $T\in(0,\infty)$ be fixed. For each $\pi\in\Pi_T$, let $X^\pi$ be the Euler--Maruyama scheme defined on $(\Omega^\pi,\cF^\pi,\bP^\pi)$ and given by \eqref{eq_EM} with initial distribution $\mu_0$.
\begin{itemize}
\item[(i)]
Assume that
\begin{equation}\label{eq_moment}
	\bE\big[\Psi(T,X)^q\big]<\infty\ \ \text{for any $q\in[1,\infty)$}.
\end{equation}
Then, for any $\gamma\in(0,\frac{\alpha_b}{2}\wedge(\alpha_\sigma-\frac{1}{2}))$, there exists a constant $C_\gamma\in[1,\infty)$ such that
\begin{equation*}
	d_\LP\big(\Law_\bP(X_T),\Law_{\bP^\pi}(\poly[X^\pi])\big)\leq C_\gamma|\pi|^\gamma
\end{equation*}
for any $\pi\in\Pi_T$.
\item[(ii)]
Suppose that $D=\bR^n$ and $p=0$, that is, the progressively measurable maps $b:\Lambda[\supp\mu_0,\bR^n]\to\bR^n$ and $\sigma:\Lambda[\supp\mu_0,\bR^n]\to\bR^{n\times d}$ are bounded, globally H\"{o}lder continuous (with exponent $\alpha_b\in(0,1]$ and $\alpha_\sigma\in(\frac{1}{2},1]$, respectively) and uniformly elliptic.
\begin{itemize}
\item[(ii-A)]
If $\alpha_b<2\alpha_\sigma-1$ or the diffusion coefficient $\sigma$ is constant, then there exists a constant $C\in[1,\infty)$ such that
\begin{equation}\label{eq_error-log1}
	d_\LP\big(\Law_\bP(X_T),\Law_{\bP^\pi}(\poly[X^\pi])\big)\leq C|\pi|^{\alpha_b/2}\left(1+\log\frac{T}{|\pi|}\right)^{\alpha_b/2}
\end{equation}
for any $\pi\in\Pi_T$.
\item[(ii-B)]
If $\alpha_b\geq2\alpha_\sigma-1$ and the diffusion coefficient $\sigma$ is not constant, then there exists a constant $C\in[1,\infty)$ such that
\begin{equation}\label{eq_error-log2}
	d_\LP\big(\Law_\bP(X_T),\Law_{\bP^\pi}(\poly[X^\pi])\big)\leq C|\pi|^{\alpha_\sigma-1/2}\left(1+\log\frac{T}{|\pi|}\right)^{\alpha_\sigma+1/2}
\end{equation}
for any $\pi\in\Pi_T$.
\end{itemize}
\end{itemize}
\end{cor}


\begin{proof}
\underline{Proof of (i).} We prove assertion (i) only for the case $D=(0,\infty)^n$. The case of $D=\bR^n$ can be proved by a similar (and simpler) manner. To show assertion (i) for $D=(0,\infty)^n$, without loss of generality, we may assume that $p>0$. By the assumptions in this corollary, we see that \cref{assum} holds with H\"{o}lder exponents $\alpha_b$ and $\alpha_\sigma$, the sets
\begin{align*}
	&D_{\growth,b;T}(R)=D_{\growth,\sigma;T}(R)=D_{\conti,b;T}(R)=D_{\conti,\sigma;T}(R)=D_{\ellip;T}(R)=\left(R^{-1/p},2R^{1/p}\right)^n,\ R\in[1,\infty),
\end{align*}
and the constants
\begin{equation*}
	K_{\growth,b;T}=K_{\growth,\sigma;T}=K_{\conti,b;T}=K_{\conti,\sigma;T}=K_{\ellip;T}=\check{C}\big(1+2\cdot3^pn^{p/2}\big),
\end{equation*}
for each $T\in(0,\infty)$. Suppose that we are given a weak solution $\weaksol$ to the SFDE \eqref{eq_SFDE} associated with $\data$ with $D=(0,\infty)^n$. Let $T\in(0,\infty)$ be fixed. We apply assertion (iii) in \cref{theo_main} to the current setting. To do so, we check the condition \eqref{eq_rare-event}. Observe that, for any $\Delta\in(0,1]$ and $\vec{R}=(R_{\growth,b},R_{\growth,\sigma},R_{\conti,b},R_{\conti,\sigma},R_\ellip)\in[1,\infty)^5$, using the notation $R_{\min}=\min\{R_{\growth,b},R_{\growth,\sigma},R_{\conti,b},R_{\conti,\sigma},R_\ellip\}$ and noting that $\Delta\leq1\leq R_{\min}^{1/p}$,
\begin{align*}
	&\bP\left(\inf_{t\in[0,T]}\dist\!\left(X(t),\bR^n\setminus D_T(\vec{R})\right)\leq\Delta\right)\\
	&=\bP\left(X(t)\notin\left(R_{\min}^{-1/p}+\Delta,2R_{\min}^{1/p}-\Delta\right)^n\ \ \text{for some $t\in[0,T]$}\right)\\
	&\leq\bP\left(\|X_T\|_\infty\geq2R_{\min}^{1/p}-\Delta\right)+\bP\left(\|X^{-1}_T\|_\infty\geq\left(R_{\min}^{-1/p}+\Delta\right)^{-1}\right)\\
	&\leq\bP\left(\|X_T\|_\infty\geq R_{\min}^{1/p}\right)+\bP\left(\|X^{-1}_T\|_\infty\geq\left(R_{\min}^{-1/p}+\Delta\right)^{-1}\right).
\end{align*}
By Markov's inequality and the assumption \eqref{eq_moment}, we obtain
\begin{align*}
	\bP\left(\inf_{t\in[0,T]}\dist\!\left(X(t),\bR^n\setminus D_T(\vec{R})\right)\leq\Delta\right)
	&\leq R_{\min}^{-q/p}\bE\big[\|X_T\|_\infty^q\big]+\left(R_{\min}^{-1/p}+\Delta\right)^q\bE\big[\|X^{-1}_T\|_\infty^q\big]\\
	&\leq C_q\max\left\{\Delta^q,R_{\growth,b}^{-q/p},R_{\growth,\sigma}^{-q/p},R_{\conti,b}^{-q/p},R_{\conti,\sigma}^{-q/p},R_\ellip^{-q/p}\right\}
\end{align*}
for any $q\in[1,\infty)$, where $C_q\in(0,\infty)$ is a constant which does not depend on $\Delta$ or $\vec{R}$. The above estimate shows that \eqref{eq_rare-event} holds for any $\vec{\beta}=(\beta_0,\beta_{\growth,b},\beta_{\growth,\sigma},\beta_{\conti,b},\beta_{\conti,\sigma},\beta_\ellip)\in(0,\infty)^6$. Therefore, noting \cref{rem_gamma*}, by \cref{theo_main} (iii), we see that assertion (i) of \cref{cor_rate-polynomial} holds.

\underline{Proof of (ii).} Let $D=\bR^n$ and $p=0$. We apply assertion (ii) of \cref{theo_main} to this setting. Notice that \cref{assum} holds with exponents $\alpha_b$ and $\alpha_\sigma$, the sets
\begin{equation*}
	D_{\growth,b;T}(R)=D_{\growth,\sigma;T}(R)=D_{\conti,b;T}(R)=D_{\conti,\sigma;T}(R)=D_{\ellip;T}(R)=\bR^n,\ \ R\in[1,\infty),
\end{equation*}
and the constants
\begin{equation*}
	K_{\growth,b;T}=K_{\growth,\sigma;T}=K_{\conti,b;T}=K_{\conti,\sigma;T}=K_{\ellip;T}=3\check{C},
\end{equation*}
for each $T\in(0,\infty)$. Moreover, if the diffusion coefficient $\sigma$ is constant, then we can take the constant $K_{\conti,\sigma;T}$ as an arbitrarily small positive number. Let $T\in(0,\infty)$ and $\pi\in\Pi_T$ be fixed. We apply the estimate \eqref{eq_error} with
\begin{equation*}
	\Delta=\sqrt{\frac{|\pi|}{T}\left(1+\log\frac{T}{|\pi|}\right)}\ \ \text{and}\ \ R_{\growth,b}=R_{\growth,\sigma}=R_{\conti,b}=R_{\conti,\sigma}=R_\ellip=1.
\end{equation*}
Observe that $D_T(\vec{R})=\bR^n$, and hence the probability appearing in the right-hand side of \eqref{eq_error} is zero. Furthermore,
\begin{equation*}
	\cE\leq C_T|\pi|^{1/2}\left(1+\log\frac{T}{|\pi|}\right)^{1/2},\ \ \log\frac{1}{\Delta}\leq1+\log\frac{T}{|\pi|},
\end{equation*}
and
\begin{align*}
	&d_\LP\big(\Law_\bP(X_T),\Law_{\bP^\pi}(\poly[X^\pi])\big)\\
	&\leq nC_T\max\left\{|\pi|^{\alpha_b/2}\left(1+\log\frac{T}{|\pi|}\right)^{\alpha_b/2},K_{\conti,\sigma;T}^2|\pi|^{\alpha_\sigma-1/2}\left(1+\log\frac{T}{|\pi|}\right)^{\alpha_\sigma+1/2}\right\},
\end{align*}
where $C_T\in[1,\infty)$ is a constant which depends only on $\check{C}$ and $T$ and varies from line to line. As we mentioned above, $K_{\conti,\sigma;T}$ can be taken as an arbitrary positive number if $\sigma$ is constant, and $K_{\conti,\sigma;T}=3\check{C}$ otherwise. In the former case, taking the limit $K_{\conti,\sigma;T}\downarrow0$ in the above estimate, we get \eqref{eq_error-log1}. In the latter case, noting that $1+\log x\leq\frac{1}{\nu}x^\nu$ for any $x\in[1,\infty)$ and $\nu\in(0,1]$, we get \eqref{eq_error-log1} if $\alpha_b<2\alpha_\sigma-1$ and \eqref{eq_error-log2} if $\alpha_b\geq2\alpha_\sigma-1$. This completes the proof.
\end{proof}


\section{Remarks on the main result and special classes of SFDEs}\label{section_remark}

Let us make some remarks on the main results.

\subsection{On the generality of the setting}

\cref{theo_main} shows uniqueness in law and functional type weak approximation results for SFDEs in a highly general setting, providing a quantitative weak error estimate in terms of the L\'{e}vy--Prokhorov metric. Our results allow the coefficients to be path-dependent. Interestingly, thanks to the path-dependence, we can apply our results not only to ``purely non-Markovian'' systems such as SDEs with delay but also to some non-standard (Markovian) SDEs such as reflected SDEs and stochastic oscillator models after appropriate transformations. We discuss these special classes of SFDEs in \cref{subsec_remark-specialSFDE} in detail. Furthermore, it is worth mentioning that \cref{assum} does not impose either boundedness (or linear growth), global Lipschitz (or global H\"{o}lder) continuity or uniform ellipticity condition on the coefficients, which are often assumed in the existing works on weak approximation of SDEs and SFDEs. Instead, \cref{assum} incorporates the ``localities'' of the boundedness, continuity and ellipticity conditions which are characterized by the families of subsets $D_{\growth,b;T}(R)$, $D_{\growth,\sigma;T}(R)$, $D_{\conti,b;T}(R)$, $D_{\conti,\sigma;T}(R)$ and $D_{\ellip;T}(R)$ of the domain $D$. Thanks to these kinds of localities (and the path-dependences), we can apply our results to stochastic systems arising in a wide range of scientific fields such as physics, chemistry, economics, mathematical finance and so on, which are beyond the framework of existing works on Euler--Maruyama approximations; see \cref{section_example} where we study ten concrete examples arising in science and provide a precise weak error estimate for each model.

To the best of our knowledge, even the uniqueness in law result itself shown in \cref{theo_main} (i) is new in this generality. A comparable result can be found in the paper \cite{KuSc20} by Kulik and Scheutzow, where they show (besides other important results on ergodicity) that weak existence and uniqueness in law hold for SFDEs defined on $D=\bR^n$ under the assumptions that the drift coefficient $b$ satisfies the one-sided (finite-range) $\alpha_b$-H\"{o}lder ($\alpha_b\in(0,1]$) continuity condition and the one-sided linear-growth condition and that the diffusion coefficient $\sigma$ satisfies the (finite-range) $\alpha_\sigma$-H\"{o}lder ($\alpha_\sigma\in(\frac{1}{2},1]$) continuity condition, the linear growth condition and the uniform ellipticity condition. Compared with the result in \cite{KuSc20}, our result on uniqueness in law for SFDEs includes the case where the domain $D$ is a proper subset of $\bR^n$ and allows for locally bounded, locally H\"older continuous and locally uniformly elliptic coefficients. Also, this paper is the first to provide functional type weak approximations for SFDEs in this generality. 

\subsection{On the statements of \cref{theo_main} and \cref{cor_rate-polynomial}}

Notice that the weak error estimate in \cref{theo_main} (ii) has flexibility for the parameters $\Delta\in(0,1]$ and $\vec{R}=(R_{\growth,b},R_{\growth,\sigma},R_{\conti,b},R_{\conti,\sigma},R_{\ellip})\in[1,\infty)^5$ which are regarded as benchmark parameters and related to the locality of the assumption discussed above. The right-hand side of \eqref{eq_error} consists of the following two error terms:
\begin{equation*}
	\mathrm{Error}_1=\bP\left(\inf_{t\in[0,T]}\dist\!\left(X(t),\bR^n\setminus D_T(\vec{R})\right)\leq\Delta\right)
\end{equation*}
and
\begin{equation*}
	\mathrm{Error}_2=(\mathrm{constant})\times K_{\ellip;T}R_\ellip\max\left\{\cE,K_{\conti,b;T}R_{\conti,b}\cE^{\alpha_b},\frac{\left(K_{\conti,\sigma;T}R_{\conti,\sigma}\cE^{\alpha_\sigma}\right)^2}{\Delta}\log\frac{1}{\Delta}\right\}
\end{equation*}
with $\cE=\cE(|\pi|,\Delta,R_{\growth,b},R_{\growth,\sigma})$ defined by \eqref{eq_cE}. The first term $\mathrm{Error}_1$ is the probability term, which depends on the (unique) law of the given weak solution but does not depend on the partition $\pi$. This term can be seen as a probability of a ``rare event'' where the given weak solution exits a large sub-domain of $D$ until the time $T$. The second term $\mathrm{Error}_2$ is the discretization error term, which depends on the partition $\pi$ but does not depend on the law of the weak solution. The two terms $\mathrm{Error}_1$ and $\mathrm{Error}_2$ are related to each other via the free parameters $\Delta$ and $\vec{R}$. It can be easily shown (see \cref{subsec_proof-main}) that, on the one hand, the probability term ($\mathrm{Error}_1$) tends to zero as $\Delta\downarrow0$ and $R_{\growth,b},R_{\growth,\sigma},R_{\conti,b},R_{\conti,\sigma},R_{\ellip}\uparrow\infty$. On the other hand, thanks to the assumption $\alpha_\sigma>\frac{1}{2}$, the discretization error term ($\mathrm{Error}_2$) tends to zero as $|\pi|\downarrow0$ and $\Delta\downarrow0$ (in this order) for each fixed $\vec{R}$, but diverges as $R_{\growth,b},R_{\growth,\sigma},R_{\conti,b},R_{\conti,\sigma},R_{\ellip}\uparrow\infty$ for each fixed $\pi$ and $\Delta$. Therefore, in order to get a convergence speed (with respect to $|\pi|$) of $\lim_{|\pi|\downarrow0}d_\LP(\Law_\bP(X_T),\Law_{\bP^\pi}(\poly[X^\pi]))=0$ as fast as possible, it is important to choose appropriate free parameters $\Delta$ and $\vec{R}$ according to $|\pi|$. This is possible if we a priori know the convergence rate of the probability of the rare event, that is the first term $\mathrm{Error}_1$, with respect to $\Delta$ and $\vec{R}$. Indeed, assertion (iii) of \cref{theo_main}, which follows from assertion (ii), provides a convergence rate (in the form of the power of $|\pi|$) assuming the a priori estimate \eqref{eq_rare-event} for the probability of the rare event. This result shows that, in order to get a more precise weak convergence rate for each model, it suffices to investigate the a priori estimate of the form \eqref{eq_rare-event} and determine the parameters $\vec{\beta}=(\beta_0,\beta_{\growth,b},\beta_{\growth,\sigma},\beta_{\conti,b},\beta_{\conti,\sigma},\beta_\ellip)\in(0,\infty)^6$. After that, computing the numbers $\beta_*=\beta_*(\alpha_b,\alpha_\sigma,\vec{\beta})$ and $\gamma_*=\gamma_*(\alpha_b,\alpha_\sigma,\vec{\beta})$ by the formulas \eqref{eq_beta*} and \eqref{eq_gamma*}, we will obtain the rate of (almost) $\gamma_*$ for the convergence order. We again stress that the probability of the rare event (the term $\mathrm{Error}_1$) depends only on the (unique) law of the given weak solution and hence determined by the structure of the system. Actually, as we demonstrate in \cref{section_example}, this term can be easily estimated in many examples.

\cref{cor_rate-polynomial} (i) immediately follows from \cref{theo_main} (iii) and provides us with a useful sufficient condition \eqref{eq_moment} to obtain the rate of (almost) $\frac{\alpha_b}{2}\wedge(\alpha_\sigma-\frac{1}{2})$ for the convergence order for SFDEs defined on $D=\bR^n$ or $D=(0,\infty)^n$ under the ``polynomial type'' conditions (G'), (C') and (E'). One of the simplest sufficient conditions ensuring the moment condition \eqref{eq_moment} is the one-sided linear growth condition for the coefficients; see \cref{appendix_LG}. As we mentioned in \cref{rem_gamma*}, the number $\frac{\alpha_b}{2}\wedge(\alpha_\sigma-\frac{1}{2})$ is the supremum of $\gamma_*$ over all $\vec{\beta}$.

Furthermore, under the assumptions that $D=\bR^n$ and that the coefficients are bounded, globally H\"{o}lder continuous and uniformly elliptic, \cref{cor_rate-polynomial} (ii) shows a more accurate convergence order than the one obtained in \cref{cor_rate-polynomial} (i), incorporating the logarithmic terms. In particular, in the case of \cref{cor_rate-polynomial} (ii-A) with $\alpha_b=1$ and the diffusion coefficient $\sigma$ being constant, we reach the order $|\pi|^{1/2}(1+\log\frac{T}{|\pi|})^{1/2}$, which turns out to be sharp as we discuss below.

\subsection{On the sharpness of the results}\label{subsubsec_sharp}

The significance of our result lies not only in the generality of the assumption as discussed above but also in the sharpness of the weak error estimate. More precisely, our results are sharp in both (I) topological and (II) quantitative senses, that is:
\begin{itemize}
\item[(I)]
The weak convergence results and the estimates in terms of the L\'{e}vy--Prokhorov metric in \cref{theo_main} and \cref{cor_rate-polynomial} (i) cannot be improved to the topology of the convergence in terms of the Wasserstein metric in general.
\item[(II)]
The weak convergence order with respect to the L\'{e}vy--Prokhorov metric obtained in \cref{cor_rate-polynomial} (ii-A) is optimal for the case where $\alpha_b=1$ and the diffusion coefficient $\sigma$ is constant.
\end{itemize}

In order to explain statement (I) (the sharpness in the topological sense), let us first recall the definition of the ($L^1$-)Wasserstein metric $\cW_1$ associated with the base space $(\cC_T,\|\cdot\|_\infty)$. It is a metric on the space
\begin{equation*}
	\cP_1(\cC_T):=\left\{\mu\in\cP(\cC_T)\relmiddle|\int_{\cC_T}\|x\|_\infty\,\mu(\diff x)<\infty\right\}
\end{equation*}
defined by
\begin{equation*}
	\cW_1(\mu,\nu):=\inf_{\fm\in\sC(\mu,\nu)}\int_{\cC_T\times\cC_T}\|x-y\|_\infty\,\fm(\diff x\diff y),\ \ \mu,\nu\in\cP_1(\cC_T).
\end{equation*}
The Kantorovich--Rubinstein theorem (cf.\ \cite[Theorem 11.8.2]{Du02}) states that $\cW_1$ coincides with the metric $d_\Xi$ introduced in \cref{subsec_LP} associated with the class $\Xi=\Lip$ of Lipschitz continuous test functions, that is,
\begin{equation*}
	\cW_1(\mu,\nu)=d_\Lip(\mu,\nu):=\sup_{f\in\Lip}\left|\int_{\cC_T}f(x)\,\mu(\diff x)-\int_{\cC_T}f(x)\,\nu(\diff x)\right|\ \ \text{for any $\mu,\nu\in\cP_1(\cC_T)$},
\end{equation*}
where
\begin{equation*}
	\Lip:=\left\{f:\cC_T\to\bR\relmiddle|\sup_{\substack{x,y\in\cC_T\\x\neq y}}\frac{|f(x)-f(y)|}{\|x-y\|_\infty}\leq1\right\}.
\end{equation*}
It is known that, for $\mu,\mu_m\in\cP_1(\cC_T)$, $m\in\bN$, $\lim_{m\to\infty}\cW_1(\mu,\mu_m)=0$ if and only if $(\mu_m)_{m\in\bN}$ weakly converges to $\mu$ and is uniformly integrable (cf.\ \cite[Theorem 5.5]{CaDe18}). The latter is also equivalent to
\begin{equation*}
	\lim_{m\to\infty}d_\LP(\mu,\mu_m)=0\ \ \text{and}\ \ \lim_{r\to\infty}\sup_{m\in\bN}\int_{\cC_T}\|x\|_\infty\1_{\{\|x\|_\infty>r\}}\,\mu_m(\diff x)=0.
\end{equation*}
It is shown in \cite{HuJeKl11,HuJeKl15} that, for any non-degenerate, one-dimensional and Markovian SDEs defined on $D=\bR$, if either the drift coefficient $b$ or the diffusion coefficient $\sigma$ is super-linearly growing, then any moments of the corresponding (standard) Euler--Maruyama scheme with respect to the uniform partition $\pi_{T,m}=(\frac{k}{m}T)^m_{k=0}\in\Pi_T$ ($m\in\bN$) must diverge as $m\to\infty$ in finite time, that is,
\begin{equation}\label{eq_moment-divergence}
	\lim_{m\to\infty}\bE_{\bP^{\pi_{T,m}}}\big[|X^{\pi_{T,m}}(T)|^p\big]=\infty\ \ \text{for any $p\in(0,\infty)$ and $T\in(0,\infty)$}.
\end{equation}
In particular, for such a Markovian SDE, (if the weak solution $\weaksol$ satisfies $\bE[\|X\|_\infty]<\infty$) we have
\begin{equation}\label{eq_Wasserstein-divergence}
	\lim_{m\to\infty}\cW_1\big(\Law_\bP(X_T),\Law_{\bP^{\pi_{T,m}}}(\fp^{\pi_{T,m}}[X^{\pi_{T,m}}])\big)=\infty\ \ \text{for any $T\in(0,\infty)$}.
\end{equation}
Our results on weak convergence and error estimates in terms of the L\'{e}vy--Prokhorov metric are valid even for (multi-dimensional and path-dependent) SFDEs having super-linearly growing coefficients $b$ and $\sigma$, but cannot be improved to the Wasserstein metric in general in the above sense. The following provides a concrete example showing the sharpness in the topological sense.


\begin{exam}\label{exam_Wasserstein-divergence}
Consider the following one-dimensional Markovian SDE:
\begin{equation*}
	\diff X(t)=\big\{X(t)-X(t)^3\big\}\,\diff t+\diff W(t),\ \ t\in[0,\infty),\ \ X(0)=\xi\in\bR.
\end{equation*}
This SDE can be seen as SFDE \eqref{eq_SFDE} with data $(\bR,\delta_\xi,x(t)-x(t)^3,1)$. Clearly, the conditions (G'), (C') (with $\alpha_b=\alpha_\sigma=1$) and (E') in \cref{cor_rate-polynomial} hold. Furthermore, by \cref{appendix_lemm_LG} (i), we see that $\bE[\|X_T\|_\infty^p]<\infty$ for any $p\in[2,\infty)$ and $T\in(0,\infty)$, and hence the moment condition \eqref{eq_moment} holds as well. Hence, by \cref{cor_rate-polynomial} (i), for any $T\in(0,\infty)$ and $\gamma\in(0,\frac{1}{2})$, there exists a constant $C_\gamma\in(0,\infty)$ such that
\begin{equation*}
	d_\LP\big(\Law_\bP(X_T),\Law_{\bP^\pi}(\poly[X^\pi])\big)\leq C_\gamma|\pi|^\gamma
\end{equation*}
for any $\pi\in\Pi_T$, where $X^\pi$ is the corresponding (standard) Euler--Maruyama scheme with initial condition $X^\pi(t_0)=\xi$. However, by \cite[Theorem 2.1]{HuJeKl15}, we have the divergences \eqref{eq_moment-divergence} and \eqref{eq_Wasserstein-divergence} of moments of the Euler--Maruyama scheme and Wasserstein error in finite time.
\end{exam}

Concerning statement (II) (the sharpness in the quantitative sense), notice that \cref{cor_rate-polynomial} (ii-A) in particular implies that, in the case of $D=\bR^n$, bounded Lipschitz (but still path-dependent) drift coefficient $b$ and constant diffusion coefficient $\sigma$, the L\'{e}vy--Prokhorov metric $d_\LP(\Law_\bP(X_T),\Law_{\bP^{\pi_{T,m}}}(\fp^{\pi_{T,m}}[X^{\pi_{T,m}}])\big)$ converges to zero as $m\to\infty$ for any $T\in(0,\infty)$ along the uniform partition $\pi_{T,m}=(\frac{k}{m}T)^m_{k=0}\in\Pi_T$ ($m\in\bN$) with order of at least $\sqrt{\frac{1}{m}\log m}$. This result then implies the same convergence order with respect to the metric $d_\BL$ induced by the class $\BL$ of bounded and Lipschitz continuous (with respect to $\|\cdot\|_\infty$) functions from $\cC_T$ to $\bR$; see \cref{subsec_LP} above. The following proposition shows that these convergence orders with respect to $d_\LP$ and $d_\BL$ are optimal in the case of the standard Brownian motion $X=W$ (which corresponds to SFDE \eqref{eq_SFDE} with data $\data=(\bR^n,\delta_0,0,I_{n\times n})$). In the following proposition, we denote by $\bW_T\in\cP(\cC_T)$ the Wiener measure on $\cC_T$ and by $\pi_{T,m}:=(\frac{k}{m}T)^m_{k=0}\in\Pi_T$ the uniform partition of $[0,T]$.


\begin{prop}\label{prop_lower-bound}
For any $T\in(0,\infty)$ and $m\in\bN$ with $m\geq3\vee(T\log m)$, it holds that\footnote{The numbers $\pi$ appearing in the right-hand sides of \eqref{eq_Wiener-BL} and \eqref{eq_Wiener-LP} represent the circle ratio, not the partition of $[0,T]$. Be careful not to confuse them. These are clear from the context, and the circle ratio does not appear explicitly anywhere else in this article.}
\begin{equation}\label{eq_Wiener-BL}
	\inf_{\mu\in\cP((\bR^n)^{m+1})}d_\BL\big(\bW_T,\mu\circ(\fp^{\pi_{T,m}})^{-1}\big)\geq\frac{1}{5}\left(1-\exp\left(-\frac{n}{\sqrt{2\pi}}\right)\right)\sqrt{\frac{T}{m}\log m},
\end{equation}
and in particular,
\begin{equation}\label{eq_Wiener-LP}
	\inf_{\mu\in\cP((\bR^n)^{m+1})}d_\LP\big(\bW_T,\mu\circ(\fp^{\pi_{T,m}})^{-1}\big)\geq\frac{1}{10}\left(1-\exp\left(-\frac{n}{\sqrt{2\pi}}\right)\right)\sqrt{\frac{T}{m}\log m}.
\end{equation}
\end{prop}


\begin{proof}
Estimate \eqref{eq_Wiener-LP} follows from \eqref{eq_Wiener-BL} and the second inequality in \eqref{eq_LP-BL}. We show estimate \eqref{eq_Wiener-BL}. Fix $T\in(0,\infty)$ and $m\in\bN$ with $m\geq3\vee(T\log m)$. Define $f_{T,m}:\cC_T\to\bR$ by
\begin{align*}
	f_{T,m}(x)
	&:=\frac{1}{5}\min\left\{\sqrt{\frac{T}{m}\log m},\ \max_{k\in\{0,\dots,m-1\}}\left|x\left(\frac{k}{m}T\right)+x\left(\frac{k+1}{m}T\right)-2x\left(\frac{k+\frac{1}{2}}{m}T\right)\right|\right\},
	\ x\in\cC_T.
\end{align*}
Observe that, for any $x,y\in\cC_T$,
\begin{equation*}
	|f_{T,m}(x)-f_{T,m}(y)|\leq\frac{4}{5}\|x-y\|_\infty\ \ \text{and}\ \ 0\leq f_{T,m}(x)\leq\frac{1}{5}\sqrt{\frac{T}{m}\log m}\leq\frac{1}{5}.
\end{equation*}
In particular, the function $f_{T,m}$ is in the class $\BL$ defined by \eqref{eq_BL} with $(S,d_S)=(\cC_T,\|\cdot-\cdot\|_\infty)$. Moreover, notice that
\begin{equation*}
	f_{T,m}(\fp^{\pi_{T,m}}[z])=0\ \ \text{for any $z=(z(t_k))^m_{k=0}\in(\bR^n)^{m+1}$},
\end{equation*}
and
\begin{equation*}
	f_{T,m}(x)\geq\frac{1}{5}\sqrt{\frac{T}{m}\log m}\1_{A_{T,m}}(x)\ \ \text{for any $x\in\cC_T$},
\end{equation*}
where $A_{T,m}\in\cB(\cC_T)$ is defined by
\begin{align*}
	A_{T,m}
	&:=\left\{x\in\cC_T\relmiddle|\max_{k\in\{0,\dots,m-1\}}\left|x\left(\frac{k}{m}T\right)+x\left(\frac{k+1}{m}T\right)-2x\left(\frac{k+\frac{1}{2}}{m}T\right)\right|>\sqrt{\frac{T}{m}\log m}\right\}.
\end{align*}
Therefore, for any $\mu\in\cP((\bR^n)^{m+1})$, we have
\begin{align*}
	d_\BL\big(\bW_T,\mu\circ(\fp^{\pi_{T,m}})^{-1}\big)&\geq\left|\int_{\cC_T}f_{T,m}(x)\,\bW_T(\diff x)-\int_{\cC_T}f_{T,m}(x)\,\big(\mu\circ(\fp^{\pi_{T,m}})^{-1}\big)(\diff x)\right|
	=\int_{\cC_T}f_{T,m}(x)\,\bW_T(\diff x)\\
	&\geq\frac{1}{5}\sqrt{\frac{T}{m}\log m}\bW_T(A_{T,m})
	=\frac{1}{5}\sqrt{\frac{T}{m}\log m}\left(1-\bW_T\left(A_{T,m}^\complement\right)\right).
\end{align*}
Hence, in order to prove \eqref{eq_Wiener-BL}, it suffices to show that
\begin{equation}\label{eq_Wiener-estimate}
	\bW_T\left(A_{T,m}^\complement\right)\leq\exp\left(-\frac{n}{\sqrt{2\pi}}\right).
\end{equation}
Observe that, under the Wiener measure $\bW_T$, the random variables
\begin{align*}
	\cC_T\ni x\mapsto&\,x\left(\frac{k}{m}T\right)+x\left(\frac{k+1}{m}T\right)-2x\left(\frac{k+\frac{1}{2}}{m}T\right)\\
	&=x\left(\frac{k+1}{m}T\right)-x\left(\frac{k+\frac{1}{2}}{m}T\right)-\left(x\left(\frac{k+\frac{1}{2}}{m}T\right)-x\left(\frac{k}{m}T\right)\right)\in\bR^n,\ k\in\{0,\dots,m-1\},
\end{align*}
are independent and identically distributed according to the $n$-dimensional normal distribution $N(0,\frac{T}{m}I_{n\times n})$ with mean zero and covariance matrix $\frac{T}{m}I_{n\times n}$. Hence,
\begin{align*}
	\bW_T\left(A_{T,m}^\complement\right)
	&=\bW_T\left(\bigcap^{m-1}_{k=0}\left\{x\in\cC_T\relmiddle|\left|x\left(\frac{k}{m}T\right)+x\left(\frac{k+1}{m}T\right)-2x\left(\frac{k+\frac{1}{2}}{m}T\right)\right|\leq\sqrt{\frac{T}{m}\log m}\right\}\right)\\
	&=N\left(0,\frac{T}{m}I_{n\times n}\right)\left(\left\{z\in\bR^n\relmiddle||z|\leq\sqrt{\frac{T}{m}\log m}\right\}\right)^m\\
	&\leq\left(1-\delta_m\right)^{nm},
\end{align*}
where
\begin{equation*}
	\delta_m:=\frac{2}{\sqrt{2\pi}}\int^\infty_{\sqrt{\log m}}\exp\left(-\frac{\theta^2}{2}\right)\,\diff\theta.
\end{equation*}
Noting that $1-\delta\leq\exp(-\delta)$ for any $\delta\in[0,1]$ and that
\begin{equation*}
	\int^\infty_\xi\exp\left(-\frac{\theta^2}{2}\right)\,\diff\theta\geq\left(\xi+\frac{1}{\xi}\right)^{-1}\exp\left(-\frac{\xi^2}{2}\right)\geq\frac{1}{2}\exp\left(-\xi^2\right)
\end{equation*}
for any $\xi\geq1$, together with $\sqrt{\log m}\geq\sqrt{\log3}\geq1$, we see that
\begin{align*}
	\bW_T\left(A_{T,m}^\complement\right)&\leq\exp\left(-nm\delta_m\right)\\
	&\leq\exp\left(-nm\cdot\frac{2}{\sqrt{2\pi}}\cdot\frac{1}{2}\exp\left(-\big(\sqrt{\log m}\big)^2\right)\right)\\
	&=\exp\left(-\frac{n}{\sqrt{2\pi}}\right).
\end{align*}
Hence, we obtain the desired estimate \eqref{eq_Wiener-estimate}. This completes the proof.
\end{proof}


\begin{rem}
\begin{itemize}
\item
\cref{prop_lower-bound} provides a lower bound for the error between the law of the $n$-dimensional standard Brownian motion $W$ on $[0,T]$ (that is, the Wiener measure $\bW_T$) and the law of the linear interpolation of a random vector in $(\bR^n)^{m+1}$ with respect to the partition $\pi_{T,m}=(\frac{k}{m}T)^m_{k=0}$ of $[0,T]$ with the $m$-equal length $|\pi_{T,m}|=\frac{T}{m}$, and the estimate holds uniformly in the laws of the underlying random vectors. In other words, this provides a lower bound, which does not depend on the choice of the underlying random vectors, for the weak convergence order in the Donsker-type functional central limit theorem; the weak convergence order as $m\to\infty$ cannot be faster than $\sqrt{\frac{1}{m}\log m}$ in terms of either the metric $d_\LP$ or $d_\BL$.
\item
A similar lower bound as in \cref{prop_lower-bound} can be found in \cite[Theorem 1.16]{CsHo93}, which states that, for ``any'' sequence $(Y_k)^\infty_{k=1}$ of independent and identically distributed random variables satisfying $\bE[Y_1]=0$, $\bE[Y_1^2]=1$, and $\bE[\exp(a|Y_1|)]<\infty$ for some positive constant $a$, the sharp order of the weak convergence of the laws of the processes $S^m(t)=\frac{1}{\sqrt{m}}\sum_{k=1}^{\lfloor mt\rfloor}Y_k$, $t\in[0,1]$, $m\in\bN$, to the Wiener measure with respect to the L\'{e}vy--Prokhorov metric on the space of $\bR$-valued right-continuous functions on $[0,1]$ is $\sqrt{\frac{1}{m}}\log m$, which is slightly slower than $\sqrt{\frac{1}{m}\log m}$. We would like to stress that our results do not contradict to (a more precise statement of) \cite[Theorem 1.16]{CsHo93} by the following two reasons. First, the frameworks in the present paper and in \cite[Theorem 1.16]{CsHo93} are different; the present paper considers the quantitative weak approximation by the polylinear extension of the discrete time random vectors in the space of continuous functions equipped with the supremum norm $\|\cdot\|_\infty$, while \cite[Theorem 1.16]{CsHo93} considers quantitative weak approximation by the discontinuous extension of the random walk in the space of right continuous functions equipped with the Skorokhod metric. Second, more importantly, the proof of \cite[Theorem 1.16]{CsHo93} (or more precisely \cite[Equation (1.1.37)]{CsHo93}) does not provide a meaningful lower bound (in the sense that the constant $C_{15}$ in \cite[Equation (1.1.32)]{CsHo93} should be zero) when the underlying random variables $Y_k$ are distributed according to the standard normal distribution, which is nothing but the case corresponding to our framework; recall that each random variable $Z^\pi_k$, $k\in\{0,\dots,m-1\}$, appearing in the definition \eqref{eq_EM} of the Euler--Maruyama scheme is assumed to be distributed according to $N(0,(t_{k+1}-t_k)I_{d\times d})$.
\end{itemize}
\end{rem}

Let us make a remark that statement (I) (the sharpness in the topological sense) above does not mean that the Wasserstein metric always diverges under our general assumption. Also, statement (II) (the sharpness in the quantitative sense) above does not mean that the polynomial convergence rate in \cref{theo_main} (iii) or the convergence order in \cref{cor_rate-polynomial} (ii-B) is sharp. From the observations in this section, the following questions naturally arise:
\begin{itemize}
\item
Is it possible to weaken the assumptions in \cite{HuJeKl11,HuJeKl15} showing the divergence of the Wasserstein metric?
\item
Are the convergence orders shown in \cref{theo_main} and \cref{cor_rate-polynomial} optimal in the general setting (including the H\"{o}lder coefficients case)?
\item
Is it possible to show similar results as in \cref{theo_main} and \cref{cor_rate-polynomial} for the case of Euler--Maruyama scheme \eqref{eq_EM} with $Z^\pi=(Z^\pi_k)^{m-1}_{k=0}$ replaced by more general (not necessarily Gaussian) random vectors?
\end{itemize}
We do not have positive answers to these questions at this moment, and we leave them to the future research.


\subsection{Special classes of SFDEs}\label{subsec_remark-specialSFDE}

As we mentioned above, the path-dependency of the coefficients makes it possible to apply our results to many kinds of Markovian and non-Markovian systems. Here, we demonstrate that our framework can be applied to the following four classes of stochastic models (after appropriate transformations for the last two):
\begin{itemize}
\item
\cref{subsubsec_Markov}:
Markovian stochastic differential equations
\item
\cref{subsubsec_delay}:
Stochastic delay differential equations
\item
\cref{subsubsec_reflected}:
Reflected stochastic differential equations
\item
\cref{subsubsec_oscillator}:
Stochastic oscillator models / Stochastic integro-differential equations
\end{itemize}


\subsubsection{Markovian stochastic differential equations}\label{subsubsec_Markov}

We begin with the following Markovian SDE defined on a non-empty, open and convex subset $D$ of $\bR^n$:
\begin{equation}\label{eq_Markov}
	\diff X(t)=\bar{b}(X(t))\,\diff t+\bar{\sigma}(X(t))\,\diff W(t),\ \ t\in[0,\infty),
\end{equation}
where $\bar{b}:D\to\bR^n$ and $\bar{\sigma}:D\to\bR^{n\times d}$ are Borel measurable maps. Let an initial distribution $\mu_0\in\cP(\bR^n)$ with $\mu_0(D)=1$ be given. The above SDE can be seen as SFDE \eqref{eq_SFDE} with data $(D,\mu_0,b,\sigma)$, where $b:\Lambda[\supp\mu_0,D]\to\bR^n$ and $\sigma:\Lambda[\supp\mu_0,D]\to\bR^{n\times d}$ are given by
\begin{equation}\label{eq_coefficient-Markov}
	\varphi(t,x):=\bar{\varphi}(x(t)),\ \ (t,x)\in\Lambda[\supp\mu_0,D],\ \ \varphi\in\{b,\sigma\}.
\end{equation}
Since the map $\Lambda[\supp\mu_0,D]\ni(t,x)\mapsto x(t)\in D$ is progressively measurable and $\bar{b}:D\to\bR^n$ and $\bar{\sigma}:D\to\bR^{n\times d}$ are Borel measurable, the coefficients $b$ and $\sigma$ defined by \eqref{eq_coefficient-Markov} are progressively measurable.

For each $T\in(0,\infty)$ and $\pi\in\Pi_T$, the corresponding Euler--Maruyama scheme \eqref{eq_EM} with $b$ and $\sigma$ defined by \eqref{eq_coefficient-Markov} and extended to be $0$ outside $\Lambda[\supp\mu_0,D]$ can be written as
\begin{equation}\label{eq_EM-Markov}
	\begin{dcases}
	X^\pi(t_0)=\xi^\pi,\\
	X^\pi(t_{k+1})=X^\pi(t_k)+\bar{b}\big(X^\pi(t_k)\big)(t_{k+1}-t_k)+\bar{\sigma}\big(X^\pi(t_k)\big)Z^\pi_k,\ \ k\in\{0,\dots,m-1\},
	\end{dcases}
\end{equation}
where we set $\bar{b}(\bar{x})=0$ and $\bar{\sigma}(\bar{x})=0$ for each $\bar{x}\in\bR^n\setminus D$; see \cref{rem_domain}. The above is the standard Euler--Maruyama scheme for the Markovian SDE \eqref{eq_Markov}.

Now let us check condition (C) in \cref{assum} for the map $b$ defined by \eqref{eq_coefficient-Markov} in view of the drift coefficient $\bar{b}$ of the original Markovian SDE \eqref{eq_Markov}; the case of $\varphi=\sigma$ is similar, and the correspondences for conditions (G) and (E) are straightforward. Suppose that there exist constants $L\in(0,\infty)$ and $\alpha\in(0,1]$ and an increasing family $\{U(R)\}_{R\in[1,\infty)}$ of open and convex subsets of $D$ with $\bigcup_{R\in[1,\infty)}U(R)=D$ such that
\begin{equation*}
	|\bar{b}(\bar{x})-\bar{b}(\bar{y})|\leq LR|\bar{x}-\bar{y}|^\alpha
\end{equation*}
for any $\bar{x},\bar{y}\in U(R)$ and $R\in[1,\infty)$. Then, the corresponding map $b$ defined by \eqref{eq_coefficient-Markov} satisfies condition (C) in \cref{assum} with
\begin{equation*}
	\alpha_b=\alpha,\ \ K_{\conti,b;T}=L,\ \ D_{\conti,b;T}(R)=U(R),\ \ R\in[1,\infty),\ \ T\in(0,\infty).
\end{equation*}
Indeed, \eqref{eq_C1} is clear, and \eqref{eq_C2} follows from the following calculation: for any $(t,x)\in\Lambda[\supp\mu_0,U(R)]$ with $R\in[1,\infty)$,
\begin{align*}
	|b(t,x)-b(s,x)|&=|\bar{b}(x(t))-\bar{b}(x(s))|\leq LR|x(t)-x(s)|^\alpha\leq LR\varpi(t,x;t-s)^\alpha.
\end{align*}
Notice that the term $\varpi(t,x;t-s)$ naturally arises in \eqref{eq_C2} even in the (time-homogeneous) Markovian setting. The case of Markovian SDE \eqref{eq_Markov} with time-dependent coefficients can be treated similarly.


\subsubsection{Stochastic delay differential equations}\label{subsubsec_delay}

Next, consider the following stochastic delay differential equation (SDDE) defined on a non-empty, open and convex subset $D$ of $\bR^n$:
\begin{equation}\label{eq_delay}
	\begin{dcases}
	\diff X(t)=\bar{b}\big(X(t),X(t-\tau)\big)\,\diff t+\bar{\sigma}\big(X(t),X(t-\tau)\big)\,\diff W(t),\ \ t\in[0,\infty),\\
	X(t)=\xi(t),\ \ t\in[-\tau,0],
	\end{dcases}
\end{equation}
where $\tau\in(0,\infty)$ is a fixed constant, $\xi:[-\tau,0]\to D$ is a given continuous function, and $\bar{b}:D\times D\to\bR^n$ and $\bar{\sigma}:D\times D\to\bR^{n\times d}$ are Borel measurable maps.
The above SDDE can be seen as SFDE \eqref{eq_SFDE} with data $(D,\delta_{\xi(0)},b,\sigma)$, where $b:\Lambda[\{\xi(0)\},D]\to\bR^n$ and $\sigma:\Lambda[\{\xi(0)\},D]\to\bR^{n\times d}$ are given by
\begin{equation}\label{eq_coefficient-delay}
\begin{split}
	&\varphi(t,x):=\bar{\varphi}(x(t),\xi(t-\tau))\1_{[0,\tau)}(t)+\bar{\varphi}(x(t),x(t-\tau))\1_{[\tau,\infty)}(t),\\
	&\hspace{5cm}(t,x)\in\Lambda[\{\xi(0)\},D],\ \ \varphi\in\{b,\sigma\}.
\end{split}
\end{equation}
Since the map $\Lambda[\{\xi(0)\},D]\ni(t,x)\mapsto(x(t),\xi(t-\tau)\1_{[0,\tau)}(t)+x(t-\tau)\1_{[\tau,\infty)}(t))\in D\times D$ is progressively measurable and $\bar{b}:D\times D\to\bR^n$ and $\bar{\sigma}:D\times D\to\bR^{n\times d}$ are Borel measurable, the coefficients $b$ and $\sigma$ defined by \eqref{eq_coefficient-delay} are progressively measurable.

Consider the corresponding Euler--Maruyama scheme \eqref{eq_EM} with $b$ and $\sigma$ defined by \eqref{eq_coefficient-delay} and extended to be $0$ outside $\Lambda[\{\xi(0)\},D]$. For simplicity, we consider the special partition $\pi=(t_k)^m_{k=0}\in\Pi_T$ of $[0,T]$ with $\tau<T<m\tau$ given by $t_k=\frac{k\tau}{N}$, $k\in \{0,\ldots,m-1\}$ and $t_{m}=T$, where $N=N_{m}\in\{1,\dots,m-1\}$ is such that $(m-1)\tau/T<N\leq m\tau/T$. Also, we set $t_k:=\frac{k\tau}{N}$ for $k\in\{-N,\dots,-1\}$. Noting that $t_{-N}=-\tau$, $t_0=0$ and $t_N=\tau$, the corresponding Euler--Maruyama scheme $X^\pi=\EM$ with initial distribution $\mu_0=\delta_{\xi(0)}$ is defined inductively as follows:
\begin{equation*}
	\begin{dcases}
	X^\pi(t_0)=\xi(0),\\
	X^\pi(t_{k+1})=X^\pi(t_k)+\bar{b}\big(X^\pi(t_k),\xi(t_{k-N})\big)\,(t_{k+1}-t_k)+\bar{\sigma}\big(X^\pi(t_k),\xi(t_{k-N})\big) Z^{\pi}_{k}\ 
	\text{for $k\in\{0,\dots,N-1\}$},\\
	X^\pi(t_{k+1})=X^\pi(t_k)+\bar{b}\big(X^\pi(t_k),X^\pi(t_{k-N})\big)\,(t_{k+1}-t_k)+\bar{\sigma}\big(X^\pi(t_k),X^\pi(t_{k-N})\big) Z^{\pi}_{k}\ 
	\text{for $k \in \{N,\ldots,m-1\}$}.
	\end{dcases}
\end{equation*}
Here, as before, we set $\bar{b}(\bar{x}_1,\bar{x}_2)=0$ and $\bar{\sigma}(\bar{x}_1,\bar{x}_2)=0$ for each $(\bar{x}_1,\bar{x}_2)\in(\bR^n\times\bR^n)\setminus(D\times D)$.

Let us check condition (C) in Assumption \ref{assum} for the map $b$ defined by \eqref{eq_coefficient-delay} in view of the drift coefficient $\bar{b}$ of the original SDDE \eqref{eq_delay}; the case of $\varphi=\sigma$ is similar, and the correspondences for conditions (G) and (E) can be discussed similarly as well, and hence we omit them. Assume that the initial function $\xi:[-\tau,0]\to D$ is $\frac{1}{2}$-H\"{o}lder continuous, that is,
\begin{equation*}
	[\xi]_{1/2}:=\sup_{-\tau\leq s<t\leq0}\frac{|\xi(t)-\xi(s)|}{(t-s)^{1/2}}<\infty.
\end{equation*}
Furthermore, assume that there exist constants $L\in(0,\infty)$ and $\alpha\in(0,1]$ and an increasing family $\{U(R)\}_{R\in[1,\infty)}$ of open and convex subsets of $D$ with $\bigcup_{R\in[1,\infty)}U(R)=D$ such that
\begin{equation*}
	|\bar{b}(\bar{x}_{1},\bar{x}_{2})-\bar{b}(\bar{y}_{1},\bar{y}_{2})|
	\leq
	LR
	\big\{|\bar{x}_{1}-\bar{y}_{1}|^\alpha+|\bar{x}_{2}-\bar{y}_{2}|^\alpha\big\}
\end{equation*}
for any $\bar{x}_{1},\bar{x}_{2},\bar{y}_{1},\bar{y}_{2}\in U(R)$ and $R\in[1,\infty)$. Since $\xi:[-\tau,0]\to D$ is continuous and $\{U(R)\}_{R\in[1,\infty)}$ is an increasing open covering of $D$, there exists $R_0\in[1,\infty)$ such that $\xi(t)\in U(R)$ for any $t\in[-\tau,0]$ and $R\in[R_0,\infty)$. Also, recall that the initial distribution of the corresponding SFDE is $\mu_0=\delta_{\xi(0)}$. In this setting, the corresponding map $b$ defined by \eqref{eq_coefficient-delay} satisfies condition (C) in \cref{assum} with
\begin{equation}\label{eq_choice-delay}
	\alpha_b=\alpha,\ \ K_{\conti,b;T}=(2+[\xi]_{1/2}^\alpha)L,\ \ D_{\conti,b;T}(R)=
	\begin{dcases}
	\emptyset\ \ &\text{if $R\in[1,R_0)$},\\
	U(R)\ \ &\text{if $R\in[R_0,\infty)$},
	\end{dcases}
	\ \ \text{for $T\in(0,\infty)$}.
\end{equation}
Indeed, for any $(t,x),(t,y)\in\Lambda[\{\xi(0)\},U(R)]$ with $R\in[R_0,\infty)$,
\begin{align*}
	|b(t,x)-b(t,y)|
	&=|\bar{b}(x(t),\xi(t-\tau))-\bar{b}(y(t),\xi(t-\tau))|\1_{[0,\tau)}(t)\\
	&\hspace{2cm}+|\bar{b}(x(t),x(t-\tau))-\bar{b}(y(t),y(t-\tau))|\1_{[\tau,\infty)}(t)\\
	&\leq LR|x(t)-y(t)|^\alpha\1_{[0,\tau)}(t)+LR\big\{|x(t)-y(t)|^\alpha+|x(t-\tau)-y(t-\tau)|^\alpha\big\}\1_{[\tau,\infty)}(t)\\
	&\leq2LR\|x_t-y_t\|_\infty^\alpha,
\end{align*}
and hence the first condition \eqref{eq_C1} holds. Moreover, the second condition \eqref{eq_C2} can be checked by the following observations: Let $(t,x)\in\Lambda[\{\xi(0)\},U(R)]$ with $R\in[R_0,\infty)$ and $s\in[0,t]$. Notice that $x(0)=\xi(0)$ by the definition of the set $\Lambda[\{\xi(0)\},U(R)]$. Also, recall that $\xi:[-\tau,0]\to D$ is $\frac{1}{2}$-H\"{o}lder continuous. We consider the three cases of $0\leq s\leq t\leq\tau$, $0 \leq s \leq \tau <t$ and $\tau <s\leq t$ separately. First, if $0\leq s\leq t\leq\tau$, we have
\begin{align*}
	|b(t,x)-b(s,x)|
	&=\Big|\bar{b}(x(t),\xi(t-\tau))-\bar{b}(x(s),\xi(s-\tau))\Big|
	\\&\leq
	LR\Big\{\big|x(t)-x(s)\big|^\alpha+\big|\xi(t-\tau)-\xi(s-\tau)\big|^\alpha\Big\}
	\\&\leq
	LR\Big\{\varpi(t,x;t-s)^\alpha+[\xi]_{1/2}^\alpha(t-s)^{\alpha/2}\Big\}
	\\&\leq
	(1+[\xi]_{1/2}^\alpha)LR
	\Big\{\varpi(t,x;t-s)+(t-s)^{1/2}\Big\}^\alpha.
\end{align*}
Second, if $0 \leq s \leq \tau <t$, thanks to the fact that $x(0)=\xi(0)$, we have
\begin{align*}
	|b(t,x)-b(s,x)|
	&=\Big|\bar{b}(x(t),x(t-\tau))-\bar{b}(x(t),x(0))+\bar{b}(x(t),\xi(0))-\bar{b}(x(s),\xi(s-\tau))\Big|
	\\&\leq
	LR\Big\{
		\big|x(t-\tau)-x(0)\big|^\alpha
		+
		\big|x(t)-x(s)\big|^\alpha
		+
		\big|\xi(0)-\xi(s-\tau)\big|^\alpha
	\Big\}
	\\&\leq
	LR\Big\{
		2\varpi(t,x;t-s)^\alpha
		+
		[\xi]_{1/2}^\alpha(t-s)^{\alpha/2}
	\Big\}
	\\&\leq
	(2+[\xi]_{1/2}^\alpha)LR
	\Big\{\varpi(t,x;t-s)+(t-s)^{1/2}\Big\}^\alpha.
\end{align*}
Third, if $\tau <s\leq t$, we have
\begin{align*}
	|b(t,x)-b(s,x)|
	&=\Big|\bar{b}(x(t),x(t-\tau))-\bar{b}(x(s),x(s-\tau))\Big|
	\\&\leq
	LR\Big\{\big|x(t)-x(s)\big|^\alpha+\big|x(t-\tau)-x(s-\tau)\big|^\alpha\Big\}
	\\&\leq
	2LR\varpi(t,x;t-s)^\alpha.
\end{align*}
The above observations show that the map $b$ satisfies \eqref{eq_C2}, and hence condition (C) in \cref{assum} holds with the setting of \eqref{eq_choice-delay}.


\subsubsection{Reflected stochastic differential equations}\label{subsubsec_reflected}

Next, we investigate reflected SDEs and show that our framework is applicable to this non-standard (Markovian) SDE after an appropriate transformation using the so-called Skorokhod map.
Here, for the sake of simplicity of presentation, we focus only on one-dimensional reflected SDEs defined on the half real line $[0,\infty)$ having reflection at $\{0\}$. Specifically, we consider the following reflected SDE:
\begin{equation}\label{eq_reflected}
	\diff \check{X}(t)
	=
	\bar{b}(\check{X}(t))\,
	\diff t
	+
	\bar{\sigma}(\check{X}(t))\,
	\diff W(t)
	+
	\diff \Phi_{\check{X}}(t)
	,\ \ t\in[0,\infty),
\end{equation}
where $\bar{b}:[0,\infty)\to\bR$ and $\bar{\sigma}:[0,\infty)\to\bR$ are Borel measurable maps. Fix $\mu_0\in\cP(\bR)$ with $\supp\mu_0 \subset [0,\infty)$. We call a tuple $(\check{X},\Phi_{\check{X}},W,\Omega,\cF,\bF,\bP)$ a weak solution of the reflected SDE \eqref{eq_reflected} with initial distribution $\mu_0$ if $(\Omega,\cF,\bP)$ is a complete probability space, $\bF$ is a filtration satisfying the usual conditions, $W$ is a one-dimensional Brownian motion relative to $\bF$, $\check{X}$ is a one-dimensional continuous $\bF$-adapted process with $\bP\circ\check{X}(0)^{-1}=\mu_0$ such that
\begin{equation*}
	\bP\left(\text{$\check{X}(t)\geq0$ for any $t\in[0,\infty)$}\right)=1
\end{equation*}
and
\begin{equation*}
	\bP\left(\text{$\int^T_0\big\{|\bar{b}(\check{X}(s))|+|\bar{\sigma}(\check{X}(s))|^2\big\}\,\diff s<\infty$ for any $T\in(0,\infty)$}\right)=1,
\end{equation*}
$\Phi_{\check{X}}$ is a one-dimensional, non-decreasing and $\bF$-adapted process with $\Phi_{\check{X}}(0)=0$, and $\bP$-a.s., they satisfy
\begin{equation*}
	\check{X}(t)=\check{X}(0)+\int^t_0\bar{b}(\check{X}(s))\,\diff s+\int^t_0\bar{\sigma}(\check{X}(s))\,\diff W(s)+\Phi_{\check{X}}(t)
\end{equation*}
and
\begin{equation*}
	\Phi_{\check{X}}(t)=\int^t_0\1_{\{0\}}(\check{X}(s))\,\diff\Phi_{\check{X}}(s)
\end{equation*}
for all $t\in[0,\infty)$. The non-decreasing process $\Phi_{\check{X}}$ plays the role of the reflection of $\check{X}$ at the boundary $\{0\}$ of $[0,\infty)$.

The reflected SDE \eqref{eq_reflected} can be understood as a solution of the Skorokhod equation on $[0,\infty)$.
Here, a pair $(\psi,\phi)$ is called a solution of the Skorokhod equation on $[0,\infty)$ associated with a given input function $x\in\cC_+:=\{x\in C([0,\infty);\bR)\mid x(0)\geq0\}$ if $(\psi,\phi)\in\cC_+\times\cC_+$, $\psi(t)\geq0$ for all $t\in[0,\infty)$, $\phi$ is non-decreasing with $\phi(0)=0$, and the following hold:
\begin{align*}
\psi(t)=x(t)+\phi(t)
\ \ \text{and} \ \
\phi(t)=\int_{0}^{t} \1_{\{0\}}(\psi(s))\,\diff \phi(s),
\ \ t \in [0,\infty).
\end{align*}
It is known that, for any $x\in\cC_+$, there exists a unique solution $(\psi,\phi)$ of the Skorokhod equation on $[0,\infty)$ associated with $x$ (see for example \cite[Chapter III, Lemma 4.2]{IkWa81} and also \cite{Sa87} for more general domain).
The map $\Gamma:x \mapsto \psi$ is called the Skorokhod map, which can be explicitly written by
\begin{equation}\label{eq_reflected-Skorokhod}
	(\Gamma x)(t)=x(t)-\min_{s \in [0,t]}(x(s) \wedge 0),\ \ t\in[0,\infty),\ x\in\cC_+,
\end{equation}
see \cite[Chapter III, Lemma 4.2]{IkWa81} for more details. Observe that
\begin{equation}\label{eq_reflected-Skorokhod-Lip}
	\Gamma0=0,\ \ |(\Gamma x)(t)-(\Gamma y)(t)|\leq2\|x_t-y_t\|_\infty\ \ \text{and}\ \ |(\Gamma x)(t)-(\Gamma x)(s)|\leq2\varpi(t,x;t-s)
\end{equation}
for any $x,y\in\cC_+$ and $0\leq s\leq t<\infty$. Note also that the map $\Lambda[[0,\infty),\bR]\ni(t,x)\mapsto(\Gamma x)(t)\in[0,\infty)$ is progressively measurable.

Using the Skorokhod map $\Gamma$, we can convert the reflected SDE \eqref{eq_reflected} to the following path-dependent SDE:
\begin{equation}\label{eq_reflected-SFDE}
	\diff X(t)=\bar{b}\big((\Gamma X)(t)\big)\,\diff t+\bar{\sigma}\big((\Gamma X)(t)\big)\,\diff W(t),\ \ t \in [0,\infty),
\end{equation}
which can be seen as SFDE \eqref{eq_SFDE} with data $(\bR,\mu_0,b,\sigma)$, where the progressively measurable coefficients $b,\sigma:\Lambda[\supp\mu_0,\bR]\to\bR$ are given by
\begin{equation}\label{eq_coefficient-reflected}
	\varphi(t,x):=\bar{\varphi}\big((\Gamma x)(t)\big),\ \ (t,x)\in\Lambda[\supp\mu_0,\bR],\ \ \varphi\in\{b,\sigma\}.
\end{equation}
The solutions of \eqref{eq_reflected} and \eqref{eq_reflected-SFDE} are related to each other via the following one-to-one correspondence:
\begin{equation}\label{eq_reflected-correspondence1}
	(\check{X},\Phi_{\check{X}})=(\Gamma X,\Gamma X-X)
\end{equation}
and
\begin{equation}\label{eq_reflected-correspondence2}
	X=\check{X}(0)+\int^\cdot_0\bar{b}(\check{X}(s))\,\diff s+\int^\cdot_0\bar{\sigma}(\check{X}(s))\,\diff W(s).
\end{equation}
Indeed, if $\weaksol$ is a weak solution of the SFDE \eqref{eq_reflected-SFDE} with initial distribution $\mu_0$, then the pair $(\check{X},\Phi_{\check{X}})$ defined by \eqref{eq_reflected-correspondence1} is nothing but the unique solution (for each $\omega\in\Omega$) of the Skorokhod equation on $[0,\infty)$ associated with the input
\begin{align*}
	X&=X(0)+\int^\cdot_0\bar{b}\big((\Gamma X)(s)\big)\,\diff s+\int^\cdot_0\bar{\sigma}\big((\Gamma X)(s)\big)\,\diff W(s)\\
	&=\check{X}(0)+\int^\cdot_0\bar{b}(\check{X}(s))\,\diff s+\int^\cdot_0\bar{\sigma}(\check{X}(s))\,\diff W(s),
\end{align*}
which in turn implies that the tuple $(\check{X},\Phi_{\check{X}},W,\Omega,\cF,\bF,\bP)$ is a weak solution of the reflected SDE \eqref{eq_reflected} with initial distribution $\mu_0$ and that the relation \eqref{eq_reflected-correspondence2} holds. Conversely, if $(\check{X},\Phi_{\check{X}},W,\Omega,\cF,\bF,\bP)$ is a weak solution of the reflected SDE \eqref{eq_reflected} with initial distribution $\mu_0$, then the pair $(\check{X},\Phi_{\check{X}})$ is the unique solution (for each $\omega\in\Omega$) of the Skorokhod equation on $[0,\infty)$ associated with the input $X$ defined by \eqref{eq_reflected-correspondence2}, and hence the relation \eqref{eq_reflected-correspondence1} holds. Inserting the relation \eqref{eq_reflected-correspondence1} to \eqref{eq_reflected-correspondence2}, we see that $\weaksol$ is a weak solution of the SFDE \eqref{eq_reflected-SFDE} with initial distribution $\mu_0$.

Although the reflected SDE \eqref{eq_reflected} itself is beyond the framework of the present paper due to the additional reflection term $\Phi_{\check{X}}$, the SFDE \eqref{eq_reflected-SFDE} fits into our framework. Then, we focus on the latter and consider its weak approximation by means of the Euler--Maruyama scheme \eqref{eq_EM}. Noting the explicit expression \eqref{eq_reflected-Skorokhod} of the Skorokhod map $\Gamma$, for each $T \in (0,\infty)$ and $\pi=(t_k)_{k=0}^{m} \in\Pi_{T}$, the Euler--Maruyama scheme \eqref{eq_EM} corresponding to the SFDE \eqref{eq_reflected-SFDE} is written by
\begin{equation}\label{eq_EM-reflected}
\begin{dcases}
X^\pi(t_0)=\xi^{\pi},\\
X^\pi(t_{k+1})
=X^\pi(t_k)+\bar{b}\big(X^\pi(t_k)-\min\{X^\pi(t_0),\ldots,X^{\pi}(t_{k}),0\}\big)\,(t_{k+1}-t_k)\\
\hspace{2.0cm}+\bar{\sigma}\big(X^\pi(t_k)-\min\{X^\pi(t_0),\ldots,X^{\pi}(t_{k}),0\}\big)Z_{k}^{\pi},\ \ k\in\{0,\dots,m-1\}.
\end{dcases}
\end{equation}
Thanks to the Lipschitz continuity of the mapping \eqref{eq_reflected-correspondence1} from $X$ to $(\check{X},\Phi_{\check{X}})$, \cref{appendix_lemm_LP} in \cref{appendix_LP} ensures that a quantitative estimate for the weak convergence $\poly[X^\pi]\to X_T$ on $\cC_T$ is carried over to that of $(\Gamma\poly[X^\pi],\Gamma\poly[X^\pi]-\poly[X^\pi])\to(\check{X}_T,(\Phi_{\check{X}})_T)$ on $\cC_T\times\cC_T$ with the same rate of convergence with respect to the L\'evy--Prokhorov metric. This leads to a quantitative weak approximation for the original reflected SDE \eqref{eq_reflected}. The idea to convert the reflected SDE \eqref{eq_reflected} to the SFDE \eqref{eq_reflected-SFDE} originates from \cite{AnOr76} and is used in \cite{AiKiKu18} for the purpose of Euler--Maruyama approximation under the Lipschitz continuity condition for the coefficients.

Let us check condition (C) in Assumption \ref{assum} for the map $b$ defined by \eqref{eq_coefficient-reflected} in view of the coefficient $\bar{b}$ of the original reflected SDE \eqref{eq_reflected}; the case of $\sigma$ is similar, and the correspondences for conditions (G) and (E) can be discussed similarly as well, and hence we omit them.
Suppose that there exist a non-decreasing function $V:[0,\infty)\to[1,\infty)$ and a constant $\alpha\in(0,1]$ such that
\begin{equation*}
	|\bar{b}(\bar{x})-\bar{b}(\bar{y})|
	\leq
	V(\bar{x}+\bar{y})|\bar{x}-\bar{y}|^\alpha
	\ \ \text{for any $\bar{x},\bar{y}\in [0,\infty)$.}
\end{equation*}
Thanks to \eqref{eq_reflected-Skorokhod-Lip}, we have
\begin{align*}
	|b(t,x)-b(t,y)|
	&=\Big|\bar{b}\big((\Gamma x)(t)\big)-\bar{b}\big((\Gamma y)(t)\big)\Big|\\
	&\leq 
	V((\Gamma x)(t)+(\Gamma y)(t))
	\Big|(\Gamma x)(t)-(\Gamma y)(t)\Big|^\alpha\\
	&\leq
	2^\alpha
	V(2\|x_{t}\|_{\infty}+2\|y_{t}\|_{\infty})
	\|x_{t}-y_{t}\|_{\infty}^\alpha
\end{align*}
for any $(t,x),(t,y)\in\Lambda[\supp\mu_0,\bR]$. Similarly,
\begin{align*}
	|b(t,x)-b(s,x)|
	&=\Big|\bar{b}\big((\Gamma x)(t)\big)-\bar{b}\big((\Gamma x)(s)\big)\Big|\\
	&\leq 
	V((\Gamma x)(t)+(\Gamma x)(s))
	\Big|(\Gamma x)(t)-(\Gamma x)(s)\Big|^\alpha\\
	&\leq
	2^\alpha
	V(4\|x_{t}\|_{\infty})
	\varpi(t,x;t-s)^\alpha
\end{align*}
for any $(t,x)\in\Lambda[\supp\mu_0,\bR]$ and $s\in[0,t]$. Hence, the map $b$ satisfies condition (C) in Assumption \ref{assum} with
\begin{equation*}
	\alpha_b=\alpha,\ \ K_{\conti,b;T}=2^\alpha,\ \
	D_{\conti,b;T}
	=
	\begin{dcases}
		\left\{\bar{x}\in\bR\relmiddle||\bar{x}|<V^{-1}(R)/4\right\} \ &\text{for $R\in[1,V(\infty))$,}\\
		\bR\ &\text{for $R\in[V(\infty),\infty)$,}
	\end{dcases}
\end{equation*}
for $T\in(0,\infty)$, where $V(\infty):=\lim_{\theta \to \infty}V(\theta)$, and $V^{-1}(R):=\inf\{\theta \geq 0\,|\,V(\theta)>R\}$ for $R \in [1,V(\infty))$.


\subsubsection{Stochastic oscillator models / Stochastic integro-differential equations}\label{subsubsec_oscillator}

We consider the stochastic oscillator model formally described by
\begin{align*}
	\ddot{x}(t)+f(x(t),\dot{x}(t))=g(x(t),\dot{x}(t))\dot{W}(t),\ \ (x(0),\dot{x}(0))=(\xi,\dot{\xi}) \in \bR^{n} \times \bR^{n},
\end{align*}
where $\dot{W}$ is a $d$-dimensional white noise and $f:\bR^{n}\times \bR^{n}\to\bR^n$ and $g:\bR^{n}\times \bR^{n}\to\bR^{n\times d}$ are Borel measurable maps. The above oscillator model is rewritten as the following $2n$-dimensional Markovian SDE for the pair $(x,\dot{x})$:
\begin{align}\label{eq_oscillator}
\begin{dcases}
	\diff x(t)=\dot{x}(t)\,\diff t,\\
	\diff\dot{x}(t)= -f(x(t),\dot{x}(t))\,\diff t+g(x(t),\dot{x}(t))\,\diff W(t),\\
	(x(0),\dot{x}(0))=(\xi,\dot{\xi}).
\end{dcases}
\ \ t\in [0,\infty),
\end{align}
One of the main difficulties of the system \eqref{eq_oscillator} lies in the degeneracy of the noise. Due to this difficulty, the above Markovian SDE does not satisfy the condition (E) in \cref{assum} even if the coefficient $g$ is uniformly elliptic. However, we can still apply our main results to this setting after transforming the degenerate Markovian SDE \eqref{eq_oscillator} to a ``non-degenerate'' and non-Markovian SFDE. Namely, noting that the first component $x$ of the solution of \eqref{eq_oscillator} is explicitly written by $x(t)=\xi+\int^t_0\dot{x}(s)\,\diff s$, we can rewrite the system \eqref{eq_oscillator} as a sole equation for the second component $X=\dot{x}$, which is of form of the following stochastic integro-differential equation:
\begin{equation}\label{eq_oscillator2}
\begin{dcases}
	\diff X(t)= -f\left(\xi+\int_{0}^{t}X(s)\,\diff s,X(t)\right)\,\diff t+g\left(\xi+\int_{0}^{t}X(s)\,\diff s,X(t)\right)\,\diff W(t),\\
	\hspace{8cm}t\in [0,\infty),\\
	X(0)=\dot{\xi}.
\end{dcases}
\end{equation}
Clearly, if $(x,\dot{x})$ solves the SDE \eqref{eq_oscillator}, then $X:=\dot{x}$ solves the stochastic integro-differential equation \eqref{eq_oscillator2}. Conversely, if $X$ solves the stochastic integro-differential equation \eqref{eq_oscillator2}, then the pair $(x,\dot{x}):=(\xi+\int^\cdot_0X(s)\,\diff s,X)$ solves the SDE \eqref{eq_oscillator}.

The stochastic integro-differential equation \eqref{eq_oscillator2} can be seen as SFDE \eqref{eq_SFDE} with data $\data=(\bR^n,\delta_{\dot{\xi}},b,\sigma)$, where the coefficients $b:\Lambda[\{\dot{\xi}\},\bR^n]\to\bR^n$ and $\sigma:\Lambda[\{\dot{\xi}\},\bR^n]\to\bR^{n\times d}$ are given by
\begin{equation}\label{eq_coefficient-oscillator}
\begin{split}
	&b(t,\dot{x}):=-f\left(\xi+\int^t_0\dot{x}(s)\,\diff s,\dot{x}(t)\right),\ \ \sigma(t,\dot{x}):=g\left(\xi+\int^t_0\dot{x}(s)\,\diff s,\dot{x}(t)\right),\\
	&\hspace{5cm}(t,\dot{x})\in\Lambda[\{\dot{\xi}\},\bR^n].
\end{split}
\end{equation}
Since the map $\Lambda[\{\dot{\xi}\},\bR^n]\ni(t,\dot{x})\mapsto(\xi+\int^t_0\dot{x}(s)\,\diff s,\dot{x}(t))\in\bR^n\times\bR^n$ is progressively measurable and $f:\bR^n\times\bR^n\to\bR^n$ and $g:\bR^n\times\bR^n\to\bR^{n\times d}$ are Borel measurable, the coefficients $b$ and $\sigma$ defined by \eqref{eq_coefficient-oscillator} are progressively measurable.

The Euler--Maruyama scheme \eqref{eq_EM} corresponding to the stochastic integro-differential equation \eqref{eq_oscillator2} is of the following form:
\begin{equation}\label{eq_EM-oscillator}
\begin{split}
	X^\pi(t_{k+1})&=X^\pi(t_k)-f\left(\xi+\sum^{k-1}_{j=0}\frac{t_{j+1}-t_j}{2}\left(X^\pi(t_j)+X^\pi(t_{j+1})\right),X^\pi(t_k)\right)(t_{k+1}-t_k)\\
	&\hspace{0.5cm}+g\left(\xi+\sum^{k-1}_{j=0}\frac{t_{j+1}-t_j}{2}\left(X^\pi(t_j)+X^\pi(t_{j+1})\right),X^\pi(t_k)\right)Z^\pi_k,
	\ k\in\{0,\dots,m-1\},
\end{split}
\end{equation}
with initial condition $X^\pi(t_0)=\dot{\xi}$. By virtue of the Lipschitz continuity of the mapping $\dot{x}\mapsto(\xi+\int^{\cdot\wedge T}_0\dot{x}(s)\,\diff s,\dot{x})$ from $\cC_T$ to $\cC_T\times\cC_T$, a quantitative estimate for the weak convergence $\poly[X^\pi]\to X_T$ on $\cC_T$ is translated into that of $(\xi+\int^{\cdot\wedge T}_0\poly[X^\pi](s)\,\diff s,\poly[X^\pi])\to(\xi+\int^{\cdot\wedge T}_0X(s)\,\diff s,X_T)$ on $\cC_T\times\cC_T$ via \cref{appendix_lemm_LP} in \cref{appendix_LP}. This yields a quantitative weak approximation with the same convergence rate for the original stochastic oscillator model \eqref{eq_oscillator}.

The important thing is that, although the equation \eqref{eq_oscillator2} is a path-dependent (and hence non-Markovian) SFDE, it is non-degenerate in the sense of condition (E) in \cref{assum} under an appropriate ellipticity condition for $g$. Indeed, assuming that
\begin{equation*}
	\big\langle g(x_1,x_2)g(x_1,x_2)^\top\eta,\eta\big\rangle\geq\frac{|\eta|^2}{V(|x_1|+|x_2|)^2}\ \ \text{for any $(x_1,x_2)\in\bR^n\times\bR^n$ and $\eta\in\bR^n$}
\end{equation*}
for some non-decreasing function $V:[0,\infty)\to[1,\infty)$, the map $\sigma$ defined by \eqref{eq_coefficient-oscillator} satisfies condition (E) in \cref{assum} with
\begin{equation*}
	K_{\ellip;T}=1,\ \ D_{\ellip;T}(R)=
	\begin{dcases}
	\left\{\bar{x}\in\bR^n\relmiddle||\xi|+(T+1)|\bar{x}|<V^{-1}(R)\right\}\ \ &\text{for $R\in[1,V(\infty))$},\\
	\bR^n\ \ &\text{for $R\in[V(\infty),\infty)$},
	\end{dcases}
\end{equation*}
for $T\in(0,\infty)$, where $V(\infty):=\lim_{\theta\to\infty}V(\theta)$, and $V^{-1}(R):=\inf\{\theta\geq0\,|\,V(\theta)>R\}$ for $R\in[1,V(\infty))$. The correspondence for condition (G) in \cref{assum} is similar. Now let us check condition (C) in \cref{assum} for the map $b$ defined by \eqref{eq_coefficient-oscillator} in view of the coefficient $f$ of the original stochastic oscillator model \eqref{eq_oscillator}. Suppose that
\begin{equation*}
	|f(x_1,x_2)-f(x_1',x_2')|\leq V\big(|x_1|+|x_2|+|x_1'|+|x_2'|\big)\big\{|x_1-x_1'|+|x_2-x_2'|\big\}^\alpha
\end{equation*}
for any $x_1,x_2,x_1',x_2'\in\bR^n$ for some constant $\alpha\in(0,1]$ and non-decreasing function $V:[0,\infty)\to[1,\infty)$. Observe that, for any $(t,\dot{x}),(t,\dot{y})\in\Lambda[\{\dot{\xi}\},\bR^n]$,
\begin{align*}
	&|b(t,\dot{x})-b(t,\dot{y})|\\
	&\leq V\left(2|\xi|+\int^t_0|\dot{x}(s)|\,\diff s+|\dot{x}(t)|+\int^t_0|\dot{y}(s)|\,\diff s+|\dot{y}(t)|\right)
	\left\{\int^t_0|\dot{x}(s)-\dot{y}(s)|\,\diff s+|\dot{x}(t)-\dot{y}(t)|\right\}^\alpha\\
	&\leq V\Big(2|\xi|+(t+1)(\|\dot{x}_t\|_\infty+\|\dot{y}_t\|_\infty)\Big)(t+1)\|\dot{x}_t-\dot{y}_t\|_\infty^\alpha.
\end{align*}
Also, for any $(t,\dot{x})\in\Lambda[\{\dot{\xi}\},\bR^n]$ and $s\in[0,t]$, we have
\begin{align*}
	|b(t,\dot{x})-b(s,\dot{x})|
	&\leq V\big(2|\xi|+2(t+1)\|\dot{x}_t\|_\infty\big)\left\{\int^t_s|\dot{x}(r)|\,\diff r+|\dot{x}(t)-\dot{x}(s)|\right\}^\alpha\\
	&\leq V\Big(2|\xi|+2(t+1)\|\dot{x}_t\|_\infty\Big)(\sqrt{t}\|\dot{x}_t\|_\infty+1)^\alpha\big\{(t-s)^{1/2}+\varpi(t,\dot{x};t-s)\big\}^\alpha.
\end{align*}
Hence, the map $b$ satisfies condition (C) in \cref{assum} with
\begin{equation*}
	\alpha_b=\alpha,\ \ K_{\conti,b;T}=T+1,\ \ D_{\conti,b;T}(R)=\left\{\bar{x}\in\bR^n\relmiddle||\bar{x}|<\widetilde{V}_{T,|\xi|}^{-1}(R)\right\}\ \ \text{for $R\in[1,\infty)$},
\end{equation*}
for $T\in(0,\infty)$, where $\widetilde{V}_{T,|\xi|}(\theta):=V(2|\xi|+2(T+1)\theta)(\sqrt{T}\theta+1)^\alpha$ for $\theta\in[0,\infty)$.


\section{Proof of the main result: A generalized coupling approach}\label{section_proof}

This section is devoted to the proof of \cref{theo_main}. First, let us sketch the idea of the proof.

Our goal is to estimate, for each $T\in(0,\infty)$ and $\pi=(t_k)^m_{k=0}\in\Pi_T$, the weak approximation error $d_\LP(\Law_\bP(X_T),\Law_{\bP^\pi}(\poly[X^\pi]))$ between a (given) weak solution $\weaksol$ of the original SFDE \eqref{eq_SFDE} and (the linear interpolation of) the Euler--Maruyama scheme $X^\pi=\EM$ with the same initial distribution as $X$. To do so, by the duality formula \eqref{eq_LP-coupling} for the L\'evy--Prokhorov metric, it is important to construct a suitable coupling between $\Law_\bP(X_T)\in\cP(\cC_T)$ and $\Law_{\bP^\pi}(\poly[X^\pi])\in\cP(\cC_T)$, that is, a pair of $\cC_T\times\cC_T$-valued random element $(\widetilde{X}_T,\poly[\widetilde{X}^\pi])$ defined on a probability space $(\widetilde{\Omega},\widetilde{\cF},\widetilde{\bP})$ with $\Law_{\widetilde{\bP}}(\widetilde{X}_T)=\Law_\bP(X_T)$ and $\Law_{\widetilde{\bP}}(\poly[\widetilde{X}^\pi])=\Law_{\bP^\pi}(\poly[X^\pi])$ such that the pathwise error $\|\widetilde{X}_T-\poly[\widetilde{X}^\pi]\|_\infty$ becomes as small as possible with respect to the topology of the convergence in $\widetilde{\bP}$-probability. When the coefficients $b$ and $\sigma$ satisfy good regularity conditions such as the global Lipschitz continuity condition, considering the synchronous coupling of $\Law_\bP(X_T)$ and $\Law_{\bP^\pi}(\poly[X^\pi])$, that is, the pair of the (strong) solution $\widetilde{X}_T=X_T$ of \eqref{eq_SFDE} and the Euler--Maruyama scheme $\widetilde{X}^\pi$ defined on $(\Omega,\cF,\bP)$ and given by \eqref{eq_EM} with $(\xi^\pi,Z^\pi_0,\dots,Z^\pi_{m-1})$ replaced by $(X(0),W(t_1)-W(t_0),\dots,W(t_m)-W(t_{m-1}))$, we can estimate the error $\|\widetilde{X}_T-\poly[\widetilde{X}^\pi]\|_\infty$ (in the $L^p(\bP)$-sense) by a standard argument based on the stability estimate. However, under our general setting of \cref{assum}, such a standard argument does not work well, and a direct construction of the (true) coupling $(\widetilde{X}_T,\poly[\widetilde{X}^\pi])$ and the estimate of $\|\widetilde{X}_T-\poly[\widetilde{X}^\pi]\|_\infty$ are quite difficult tasks.

The main idea for the proof of our main result (\cref{theo_main}) is to divide the weak approximation error $d_\LP(\Law_\bP(X_T),\Law_{\bP^\pi}(\poly[X^\pi]))$ into the following two terms:
\begin{align*}
	&d_\LP\big(\Law_\bP(X_T),\Law_{\bP^\pi}(\poly[X^\pi])\big)\\
	&\leq d_\LP\big(\Law_\bP(X_T),\Law_\bP(\poly[\hX^\pi])\big)+d_\LP\big(\Law_\bP(\poly[\hX^\pi]),\Law_{\bP^\pi}(\poly[X^\pi])\big).
\end{align*}
Here, the auxiliary term $\hX^\pi$, which will be introduced in \cref{subsec_controlledEM} below, is a ``controlled version'' of the (true) Euler--Maruyama scheme $X^\pi$. In this paper, we call $\hX^\pi$ a \emph{controlled Euler--Maruyama scheme}. We will construct $\hX^\pi$ on the same probability space $(\Omega,\cF,\bP)$ as the (given) weak solution $\weaksol$ of the original SFDE \eqref{eq_SFDE} to meet the following requirements:
\begin{itemize}
\item[(i)]
The pathwise error $\|X_T-\poly[\hX^\pi]\|_\infty$ becomes as small as possible with respect to the topology of the convergence in $\bP$-probability, which provides us an estimate of $d_\LP(\Law_\bP(X_T),\Law_\bP(\poly[\hX^\pi]))$.
\item[(ii)]
The error between $\Law_\bP(\poly[\hX^\pi])$ and $\Law_{\bP^\pi}(\poly[X^\pi])$ becomes as small as possible in the sense of the total variation, which provides us an estimate of $d_\LP(\Law_\bP(\poly[\hX^\pi]),\Law_{\bP^\pi}(\poly[X^\pi]))$.
\end{itemize}
The pair $(X_T,\poly[\hX]^\pi)$ of random elements on $(\Omega,\cF,\bP)$ can be seen as a \emph{generalized coupling} between the prescribed probability measures $\Law_\bP(X_T)$ and $\Law_{\bP^\pi}(\poly[X^\pi])$ on $\cC_T$ in the sense that, although it is not a true coupling since the second marginal $\Law_\bP(\poly[\hX^\pi])$ is not equal to $\Law_{\bP^\pi}(\poly[X^\pi])$, they are close in some sense. Using a stochastic control technique, we construct the auxiliary term $\hX^\pi$ which ``controls'' the corrective error by means of (i) and (ii) above. This kind of argument is called a \emph{generalized coupling approach} (also known as the Control-and-Reimburse strategy) and has been applied to the study of ergodicity of infinite-dimensional Markov models in \cite{BuKuSc20,KuSc20} among others. To the best of our knowledge, the present paper is the first to apply the idea of generalized couplings to the theory of Euler--Maruyama approximations.


\subsection{Controlled Euler--Maruyama scheme}\label{subsec_controlledEM}

Now we introduce the controlled Euler--Maruyama scheme. In the rest of this section, we use the following additional notations: for each $T\in(0,\infty)$ and $\pi=(t_k)^m_{k=0}\in\Pi_T$, define
\begin{equation*}
	\pi(t):=\sum^{m-1}_{k=0}t_k\1_{[t_k,t_{k+1})}(t)+t_m\1_{[t_m,\infty)}(t),\ \ t\in[0,\infty).
\end{equation*}
Also, with slight abuse of notation, we define the polygonal function $\poly[x]\in\cC_T$ for each $x\in\cC$ by the same way as in \eqref{eq_poly}, which represents the linear interpolation of the points $(x(t_k))^m_{k=0}$ .

Let data $\data$ satisfying \cref{assum} be given. As before, we extend the domain $\Lambda[\supp\mu_0,D]$ of the coefficients $b$ and $\sigma$ to the whole space $\Lambda$ by setting $b(t,x)=0$ and $\sigma(t,x)=0$ for each $(t,x)\in\Lambda\setminus\Lambda[\supp\mu_0,D]$. Suppose that we are given a weak solution $\weaksol$ of SFDE \eqref{eq_SFDE} associated with $\data$. For each $T\in(0,\infty)$, $\pi\in\Pi_T$, $\lambda\in(0,\infty)$, $\Delta\in(0,1]$ and $\vec{R}=(R_{\growth,b},R_{\growth,\sigma},R_{\conti,b},R_{\conti,\sigma},R_\ellip)\in[1,\infty)^5$, consider the following SFDE:
\begin{equation}\label{eq_controlledEM}
	\begin{dcases}
	\diff \hX^\pi(t)=b\big(\pi(t),\poly[\hX^\pi]\big)\,\diff t+\sigma\big(\pi(t),\poly[\hX^\pi]\big)\,\diff W(t)\\
	\hspace{4cm}+\frac{\lambda}{\Delta}\big(X(t)-\hX^\pi(t)\big)\1_{[0,T\wedge\tau^\pi\wedge\zeta)}(t)\,\diff t,\ \ t\in[0,\infty),\\
	\hX^\pi(0)=X(0),\ \ \tau^\pi=\inf\left\{t\geq0\relmiddle|\big|X(t)-\hX^\pi(t)\big|\geq\Delta\right\},
	\end{dcases}
\end{equation}
where
\begin{equation*}
	\zeta:=\inf\left\{t\geq0\relmiddle|\dist\!\left(X(t),\bR^n\setminus D_T(\vec{R})\right)\leq\Delta\right\}.
\end{equation*}
Notice that, for each $t\in[t_k,t_{k+1})$ with $k\in\{0,\dots,m-1\}$, we have $b(\pi(t),\poly[\hX^\pi])=b(t_k,\poly[\hX^\pi]_{t_k})$, $\sigma(\pi(t),\poly[\hX^\pi])=\sigma(t_k,\poly[\hX^\pi]_{t_k})$, and
\begin{align*}
	\poly[\hX^\pi]_{t_k}(\cdot)&=\sum^{k-1}_{j=0}\left\{\frac{t_{j+1}-\cdot}{t_{j+1}-t_j}\hX^\pi(t_j)+\frac{\cdot-t_j}{t_{j+1}-t_j}\hX^\pi(t_{j+1})\right\}\1_{[t_j,t_{j+1})}(\cdot)
	+\hX^\pi(t_k)\1_{[t_k,\infty)}(\cdot),
\end{align*}
which depends only on $(\hX^\pi(t_j))^k_{j=0}$. Also, noting that $t_m=T$, for each $t\in[T,\infty)$, we have $b(\pi(t),\poly[\hX^\pi])=b(T,\poly[\hX^\pi])$, $\sigma(\pi(t),\poly[\hX^\pi])=\sigma(T,\poly[\hX^\pi])$, and
\begin{equation*}
	\poly[\hX^\pi](\cdot)=\sum^{m-1}_{j=0}\left\{\frac{t_{j+1}-\cdot}{t_{j+1}-t_j}\hX^\pi(t_j)+\frac{\cdot-t_j}{t_{j+1}-t_j}\hX^\pi(t_{j+1})\right\}\1_{[t_j,t_{j+1})}(\cdot)+\hX^\pi(T)\1_{[T,\infty)}(\cdot),
\end{equation*}
which depends only on $(\hX^\pi(t_j))^m_{j=0}$. By \cref{appendix_lemm_controlledEM} in \cref{appendix_controlledEM} (see also \cref{appendix_rem_controlledEM}), there exists a unique (up to $\bP$-indistinguishability) $\bR^n$-valued continuous and $\bF$-adapted process $\hX^\pi=(\hX^\pi(t))_{t\in[0,\infty)}$ on $(\Omega,\cF,\bP)$ satisfying \eqref{eq_controlledEM}. We call $\hX^\pi$ the \emph{controlled Euler--Maruyama scheme} associated with $\weaksol$ and $(T,\pi,\lambda,\Delta,\vec{R})$.


\begin{rem}
\begin{itemize}
\item
Notice that \eqref{eq_controlledEM} is an SFDE defined on the filtered probability space and $d$-dimensional Brownian motion appearing in a given weak solution $\weaksol$ of the original SFDE, and it involves the stopping time $\tau^\pi$ which depends on the solution $\hX^\pi$ itself. Furthermore, we do not assume that the coefficient $b$ or $\sigma$ is globally Lipschitz continuous. Hence, the existence and uniqueness of the solution $\hX^\pi$ is a non-trivial issue. However, thanks to the appearance of the time-discretization in $b(\pi(t),\poly[\hX^\pi])$ and $\sigma(\pi(t),\poly[\hX^\pi])$, we can construct the solution $\hX^\pi$ by the step-by-step argument. For more details, see \cref{appendix_controlledEM}.
\item
The controlled Euler--Maruyama scheme $\hX^\pi$ itself cannot be used for the approximation purpose since it involves the true solution $X$. We will use it just as an auxiliary process.
\item
The additional drift term $\frac{\lambda}{\Delta}(X(t)-\hX^\pi(t))\1_{[0,T\wedge\tau^\pi\wedge\zeta)}(t)$ in \eqref{eq_controlledEM} is regarded as a control process, which plays a role of ``dissipation'' of the system and is valid until the stopping time $T\wedge\tau^\pi\wedge\zeta$. Namely, if $\hX^\pi(t)=X(t)+\kappa v$ for some $\kappa\in(0,\Delta)$ and $v\in\bR^n$ with $|v|=1$ at time $t\in[0,T\wedge\tau^\pi\wedge\zeta)$, then the control process attempts to move $\hX^\pi$ to the direction $-v$ with size $\lambda\times\frac{\kappa}{\Delta}$ in the infinitesimal time duration $[t,t+\diff t]$. The control parameters $\lambda$, $\Delta$ and $\vec{R}=(R_{\growth,b},R_{\growth,\sigma},R_{\conti,b},R_{\conti,\sigma},R_\ellip)$ have the following interpretations:
\begin{itemize}
\item
The parameter $\lambda$ can be seen as an intensity of the dissipation;
\item
The parameter $\Delta$ can be seen as a benchmark of the (pathwise) error between $X$ and $\hX^\pi$;
\item
The vector-parameter $\vec{R}$ can be seen as a benchmark of the ``locality'' of the growth, continuity and ellipticity conditions specified in \cref{assum}.
\end{itemize}
Intuitively speaking, the larger the intensity parameter $\lambda$ is, the closer the controlled Euler--Maruyama scheme $\hX^\pi$ is to the true (weak) solution $X$ in the pathwise sense. On the other hand, the smaller the intensity parameter $\lambda$ is, the closer $\hX^\pi$ is to the original Euler--Maruyama scheme $X^\pi$ in the sense of probability laws. In order to understand the latter idea, consider the case of $\lambda=0$ (at least formally). In this case, the system \eqref{eq_controlledEM} is nothing but a continuous time analogue of the Euler--Maruyama scheme \eqref{eq_EM} defined on the probability basis $(W,\Omega,\cF,\bF,\bP)$.
\end{itemize}
\end{rem}

The important thing is to seek for ``optimal control parameters'' $\lambda$, $\Delta$ and $\vec{R}$ which minimize the corrective error term
\begin{equation*}
	d_\LP\big(\Law_\bP(X_T),\Law_\bP(\poly[\hX^\pi])\big)+d_\LP\big(\Law_\bP(\poly[\hX^\pi]),\Law_{\bP^\pi}(\poly[X^\pi])\big).
\end{equation*}
To do so, the first task is to formulate an appropriate ``control problem'' by estimating the above two terms separately for each fixed control parameters. An estimate of the first term $d_\LP\big(\Law_\bP(X_T),\Law_\bP(\poly[\hX^\pi])\big)$ is given in \cref{subsec_X-hXpi} below, where the growth condition (G) and the local H\"{o}lder continuity condition (C) in \cref{assum} are crucial. The second term $d_\LP\big(\Law_\bP(\poly[\hX^\pi]),\Law_{\bP^\pi}(\poly[X^\pi])\big)$ is estimated in \cref{subsec_hXpi-Xpi} below, where we utilize the local ellipticity condition (E) in \cref{assum} to construct a new probability measure $\widehat{\bP}^\pi\sim\bP$ on $(\Omega,\cF)$ under which the controlled Euler--Maruyama scheme $\hX^\pi$ is distributed according to the law of the true Euler--Maruyama scheme $X^\pi$. Based on these two estimates, in \cref{subsec_proof-main} below, we tackle the corresponding ``control problem'' and complete the proof of \cref{theo_main}.

Notice that, by the definitions of the stopping times $\tau^\pi$ and $\zeta$,
\begin{equation*}
	B_{X(t)}(\Delta)\subset D_T(\vec{R}) \ \ \text{for any $t\in[0,\zeta)$, and}\ \ \hX^\pi(t)\in B_{X(t)}(\Delta)\ \ \text{for any $t\in[0,\tau^\pi)$},
\end{equation*}
where $B_{X(t)}(\Delta)$ denotes the open ball in $\bR^n$ with center $X(t)$ and radius $\Delta$. From the above observation, together with the facts that $\hX^\pi(0)=X(0)\in\supp\mu_0$ $\bP$-a.s.\ and that $D_T(\vec{R})$ is convex, we see that the following hold $\bP$-a.s.:
\begin{equation}\label{eq_stopping}
\begin{split}
	&(t,X),(\pi(t),\poly[X])\in\Lambda\big[\supp\mu_0,D_T(\vec{R})\big]\ \ \text{for any $t\in[0,T\wedge\zeta)$, and}\\
	&\big(t,\hX^\pi\big),\big(\pi(t),\poly[\hX^\pi]\big)\in\Lambda\big[\supp\mu_0,D_T(\vec{R})\big]\ \ \text{for any $t\in[0,T\wedge\tau^\pi\wedge\zeta)$.}
\end{split}
\end{equation}


\subsection{Error estimate between the weak solution and controlled Euler--Maruyama scheme}\label{subsec_X-hXpi}

First, we show a probabilistic estimate for the pathwise error $\|X_T-\poly[\hX^\pi]\|_\infty$ in terms of the control parameters. To do so, we need the following standard lemma:


\begin{lemm}\label{lemm_polylinear-operator}
Let $T\in(0,\infty)$ and $\pi=(t_k)^m_{k=0}\in\Pi_T$ be fixed. Then, for any $x,y\in\cC$ and $k\in\{0,\dots,m\}$,
\begin{equation}\label{eq_polylinear-operator1}
	\|x_{t_k}-\poly[x]_{t_k}\|_\infty\leq\varpi(t_k,x;|\pi|)
\end{equation}
and
\begin{equation}\label{eq_polylinear-operator2}
	\|\poly[x]_{t_k}-\poly[y]_{t_k}\|_\infty\leq\|x_{t_k}-y_{t_k}\|_\infty.
\end{equation}
\end{lemm}


\begin{proof}
Notice that, for any $z\in\cC$,
\begin{equation*}
	\poly[z]_{t_k}(\cdot)=\sum^{k-1}_{j=0}\left\{\frac{t_{j+1}-\cdot}{t_{j+1}-t_j}z(t_j)+\frac{\cdot-t_j}{t_{j+1}-t_j}z(t_{j+1})\right\}\1_{[t_j,t_{j+1})}(\cdot)+z(t_k)\1_{[t_k,\infty)}(\cdot).
\end{equation*}
On the one hand,
\begin{align*}
	\|x_{t_k}-\poly[x]_{t_k}\|_\infty
	&=\max_{j\in\{0,\dots,k-1\}}\sup_{s\in[t_j,t_{j+1})}\left|\frac{t_{j+1}-s}{t_{j+1}-t_j}\big(x(s)-x(t_j)\big)+\frac{s-t_j}{t_{j+1}-t_j}\big(x(s)-x(t_{j+1})\big)\right|\\
	&\leq\max_{j\in\{0,\dots,k-1\}}\sup_{s\in[t_j,t_{j+1})}\left\{\frac{t_{j+1}-s}{t_{j+1}-t_j}\big|x(s)-x(t_j)\big|+\frac{s-t_j}{t_{j+1}-t_j}\big|x(s)-x(t_{j+1})\big|\right\}\\
	&\leq\max_{j\in\{0,\dots,k-1\}}\varpi(t_{j+1},x;t_{j+1}-t_j)\\
	&\leq\varpi(t_k,x;|\pi|).
\end{align*}
Hence, \eqref{eq_polylinear-operator1} holds. On the other hand,
\begin{align*}
	\|\poly[x]_{t_k}-\poly[y]_{t_k}\|_\infty
	&=\max_{j\in\{0,\dots,k-1\}}\sup_{s\in[t_j,t_{j+1})}\left|\frac{t_{j+1}-s}{t_{j+1}-t_j}\big(x(t_j)-y(t_j)\big)+\frac{s-t_j}{t_{j+1}-t_j}\big(x(t_{j+1})-y(t_{j+1})\big)\right|\\
	&\hspace{5cm}\vee|x(t_k)-y(t_k)|\\
	&\leq\max_{j\in\{0,\dots,k-1\}}\sup_{s\in[t_j,t_{j+1})}\left\{\frac{t_{j+1}-s}{t_{j+1}-t_j}\big|x(t_j)-y(t_j)\big|+\frac{s-t_j}{t_{j+1}-t_j}\big|x(t_{j+1})-y(t_{j+1})\big|\right\}\\
	&\hspace{5cm}\vee|x(t_k)-y(t_k)|\\
	&\leq\max_{j\in\{0,\dots,k\}}|x(t_j)-y(t_j)|\\
	&\leq\|x_{t_k}-y_{t_k}\|_\infty.
\end{align*}
Hence, \eqref{eq_polylinear-operator2} holds.
\end{proof}

Using the above standard lemma and fundamental results in stochastic calculus shown in \cref{appendix_prob}, we provide the following estimate between the weak solution and controlled Euler--Maruyama scheme. Here, the growth condition (G) and the local H\"{o}lder continuity condition (C) in \cref{assum} are essential.


\begin{prop}\label{prop_X-hXpi}
Fix data $\data$ satisfying \cref{assum}. Suppose that we are given a weak solution $\weaksol$ of the SFDE \eqref{eq_SFDE} associated with $\data$. Fix $T\in(0,\infty)$, $\pi\in\Pi_T$, $\lambda\in(0,\infty)$, $\Delta\in(0,1]$ and $\vec{R}=(R_{\growth,b},R_{\growth,\sigma},R_{\conti,b},R_{\conti,\sigma},R_\ellip)\in[1,\infty)^5$, and assume that the parameters satisfy the constraint
\begin{equation}\label{eq_constraint}
	\lambda\geq2^{10}n\max\left\{K_{\conti,b;T}R_{\conti,b}\cE^{\alpha_b},\frac{\left(K_{\conti,\sigma;T}R_{\conti,\sigma}\cE^{\alpha_\sigma}\right)^2}{\Delta}\left(\log\frac{1}{\Delta}+\log\left(1\vee\frac{2T\lambda}{\Delta}\right)\right)\right\},
\end{equation}
where $\cE=\cE(|\pi|,\Delta,R_{\growth,b},R_{\growth,\sigma})\in(0,\infty)$ is defined by \eqref{eq_cE}. Let $\hX^\pi$ be the controlled Euler--Maruyama scheme, given by \eqref{eq_controlledEM}, associated with $\weaksol$ and $(T,\pi,\lambda,\Delta,\vec{R})$. Then, it holds that
\begin{equation}\label{eq_X-hXpi-prob}
\begin{split}
	\bP\Big(\|X_T-\poly[\hX^\pi]\|_\infty\geq6\sqrt{n}\cE\Big)
	&\leq\bP\left(\inf_{t\in[0,T]}\dist\!\left(X(t),\bR^n\setminus D_T(\vec{R})\right)\leq\Delta\right)+4n\frac{|\pi|}{T}+219\Delta.
\end{split}
\end{equation}
In particular, it holds that
\begin{equation}\label{eq_X-hXpi}
\begin{split}
	d_\LP\big(\Law_\bP(X_T),\Law_\bP(\poly[\hX^\pi])\big)
	&\leq\bP\left(\inf_{t\in[0,T]}\dist\!\left(X(t),\bR^n\setminus D_T(\vec{R})\right)\leq\Delta\right)+\left(\frac{4n}{T}+219+6\sqrt{n}\right)\cE.
\end{split}
\end{equation}
\end{prop}


\begin{proof}
The estimate \eqref{eq_X-hXpi} follows from \eqref{eq_X-hXpi-prob}. Indeed, by the fundamental inequality \eqref{eq_LP-KF} for the L\'evy--Prokhorov metric, we have
\begin{equation*}
	d_\LP\big(\Law_\bP(X_T),\Law_\bP(\poly[\hX^\pi])\big)\leq\inf\left\{\ep>0\relmiddle|\bP\big(\|X_T-\poly[\hX^\pi]\|_\infty>\ep\big)<\ep\right\}.
\end{equation*}
Then, the estimate \eqref{eq_X-hXpi-prob} and the definition \eqref{eq_cE} of $\cE$ imply that the number
\begin{equation*}
	\ep=\bP\left(\inf_{t\in[0,T]}\dist\!\left(X(t),\bR^n\setminus D_T(\vec{R})\right)\leq\Delta\right)+\left(\frac{4n}{T}+219+6\sqrt{n}\right)\cE
\end{equation*}
belongs to the set over which the infimum is taken, proving \eqref{eq_X-hXpi}. In the following, we show that the estimate \eqref{eq_X-hXpi-prob} holds under the constraint \eqref{eq_constraint} on the parameters. Without loss of generality, we may assume that $\log\frac{T}{|\pi|}\geq1$ and $\log\frac{1}{\Delta}\geq1$; otherwise $4n\frac{|\pi|}{T}+219\Delta>1$, and hence the estimate \eqref{eq_X-hXpi-prob} becomes trivial.

Define a stopping time $\eta^\pi$ by
\begin{equation*}
	\eta^\pi:=\inf\left\{t\geq0\relmiddle|\varpi\big(t,X;|\pi|\big)\geq K_{\growth,b;T}R_{\growth,b}|\pi|+4K_{\growth,\sigma;T}R_{\growth,\sigma}\sqrt{n|\pi|\log\frac{T}{|\pi|}}\right\}.
\end{equation*}
Then, by the definitions of the stopping times $\tau^\pi$ and $\eta^\pi$,
\begin{align*}
	\Big\{T<\eta^\pi\wedge\tau^\pi\Big\}
	&=\left\{\varpi\big(T,X;|\pi|\big)<K_{\growth,b;T}R_{\growth,b}|\pi|+4K_{\growth,\sigma;T}R_{\growth,\sigma}\sqrt{n|\pi|\log\frac{T}{|\pi|}}\right\}
	\cap\Big\{\|X_T-\hX^\pi_T\|_\infty<\Delta\Big\}.
\end{align*}
Furthermore, by \cref{lemm_polylinear-operator}, we have
\begin{align*}
	\|X_T-\poly[\hX^\pi]\|_\infty&\leq\|X_T-\poly[X]\|_\infty+\|\poly[X]-\poly[\hX^\pi]\|_\infty\\
	&\leq\varpi\big(T,X;|\pi|\big)+\|X_T-\hX^\pi_T\|_\infty.
\end{align*}
Therefore, noting the definition \eqref{eq_cE} of $\cE=\cE(|\pi|,\Delta,R_{\growth,b},R_{\growth,\sigma})\in(0,\infty)$, we obtain 
\begin{align*}
	\Big\{T<\eta^\pi\wedge\tau^\pi\Big\}
	&\subset\left\{\|X_T-\poly[\hX^\pi]\|_\infty<K_{\growth,b;T}R_{\growth,b}|\pi|+4K_{\growth,\sigma;T}R_{\growth,\sigma}\sqrt{n|\pi|\log\frac{T}{|\pi|}}+\Delta\right\}\\
	&\subset\Big\{\|X_T-\poly[\hX^\pi]\|_\infty<6\sqrt{n}\cE\Big\}.
\end{align*}
Hence,
\begin{align*}
	\Big\{\|X_T-\poly[\hX^\pi]\|_\infty\geq6\sqrt{n}\cE\Big\}&\subset\Big\{\eta^\pi\wedge\tau^\pi\leq T\Big\}
	\subset\Big\{\zeta\leq T\Big\}\cup\Big\{\eta^\pi\leq T\wedge\zeta\Big\}\cup\Big\{\tau^\pi\leq T\wedge\zeta\wedge\eta^\pi\Big\},
\end{align*}
which implies that
\begin{equation}\label{eq_prob}
	\bP\Big(\|X_T-\poly[\hX^\pi]\|_\infty\geq6\sqrt{n}\cE\Big)\leq P_1+P_2+P_3,
\end{equation}
where
\begin{equation*}
	P_1:=\bP\Big(\zeta\leq T\Big),\ \ P_2:=\bP\Big(\eta^\pi\leq T\wedge\zeta\Big),\ \ \text{and}\ \ P_3:=\bP\Big(\tau^\pi\leq T\wedge\zeta\wedge\eta^\pi\Big).
\end{equation*}
We estimate the three terms in the right-hand side of \eqref{eq_prob}.

\underline{Step 1. Estimate of $P_1$.} By the definition of the stopping time $\zeta$, we have
\begin{equation}\label{eq_prob-1}
	P_1=\bP\Big(\zeta\leq T\Big)=\bP\left(\inf_{t\in[0,T]}\dist\!\left(X(t),\bR^n\setminus D_T(\vec{R})\right)\leq\Delta\right).
\end{equation}

\underline{Step 2. Estimate of $P_2$.} By the definition of the stopping time $\eta^\pi$, we have
\begin{align}
	P_2&=\bP\Big(\eta^\pi\leq T\wedge\zeta\Big)
	\label{eq_eta1}
	=\bP\left(\varpi\big(T\wedge\zeta,X;|\pi|\big)\geq K_{\growth,b;T}R_{\growth,b}|\pi|+4K_{\growth,\sigma;T}R_{\growth,\sigma}\sqrt{n|\pi|\log\frac{T}{|\pi|}}\right).
\end{align}
Observe that
\begin{equation}\label{eq_eta2}
	\varpi\big(T\wedge\zeta,X;|\pi|\big)\leq\varpi(T,V;|\pi|)+\varpi(T,M;|\pi|),
\end{equation}
where
\begin{equation*}
	V(t):=\int^t_0b(s,X)\1_{[0,T\wedge\zeta)}(s)\,\diff s,\ \ M(t):=\int^t_0\sigma(s,X)\1_{[0,T\wedge\zeta)}(s)\,\diff W(s),\ \ t\in[0,\infty).
\end{equation*}
By \eqref{eq_stopping}, we see that $(s,X)\in\Lambda[\supp\mu_0,D_{\growth,b;T}(R_{\growth,b})]\cap\Lambda[\supp\mu_0,D_{\growth,\sigma;T}(R_{\growth,\sigma})]$ for any $s\in[0,T\wedge\zeta)$ $\bP$-a.s. Thus, by condition (G) in \cref{assum}, we have
\begin{equation*}
	|b(s,X)|\leq K_{\growth,b;T}R_{\growth,b}\ \ \text{and}\ \ |\sigma(s,X)|\leq K_{\growth,\sigma;T}R_{\growth,\sigma}\ \ \text{for any $s\in[0,T\wedge\zeta)$ $\bP$-a.s.},
\end{equation*}
which implies that
\begin{equation}\label{eq_eta3}
	\varpi(T,V;|\pi|)\leq K_{\growth,b;T}R_{\growth,b}|\pi|\ \ \text{and}\ \ \max_{i\in\{1,\dots,n\}}\varpi(T,\langle M^i\rangle;|\pi|)\leq(K_{\growth,\sigma;T}R_{\growth,\sigma})^2|\pi|\ \ \text{$\bP$-a.s.}
\end{equation}
Hence, by \eqref{eq_eta1}, \eqref{eq_eta2} and the first estimate in \eqref{eq_eta3}, we have
\begin{equation*}
	P_2\leq\bP\left(\varpi(T,M;|\pi|)\geq4K_{\growth,\sigma;T}R_{\growth,\sigma}\sqrt{n|\pi|\log\frac{T}{|\pi|}}\right).
\end{equation*}
Then, noting the second estimate in \eqref{eq_eta3}, applying \cref{appendix_lemm_mod-of-conti} to the $n$-dimensional continuous local $\bP$-martingale $M$ and constants $\delta=|\pi|$, $\kappa=(K_{\growth,\sigma;T}R_{\growth,\sigma})^2|\pi|$ and $\theta=\sqrt{\log\frac{T}{|\pi|}}\geq1$, we obtain
\begin{equation}\label{eq_prob-2}
	P_2\leq4n\frac{|\pi|}{T}.
\end{equation}

\underline{Step 3. Estimate of $P_3$.} By the definition of the stopping time $\tau^\pi$, we have
\begin{equation}\label{eq_tau1}
	P_3=\bP\Big(\tau^\pi\leq T\wedge\zeta\wedge\eta^\pi\Big)=\bP\left(\sup_{t\in[0,T\wedge\tau^\pi\wedge\zeta\wedge\eta^\pi]}\big|X(t)-\hX^\pi(t)\big|\geq\Delta\right).
\end{equation}
We will apply \cref{appendix_lemm_Kulik--Scheutzou} to the nonnegative It\^{o} process $Z(t)=|X(t)-\hX^\pi(t)|^2$ and nonnegative random variable $\varsigma=T\wedge\tau^\pi\wedge\zeta\wedge\eta^\pi$, together with constants $\kappa$, $A$, $B$ and $\theta$ to be determined in the following observation. By using It\^{o}'s formula, we have
\begin{equation}\label{eq_Ito}
	\big|X(t)-\hX^\pi(t)\big|^2=\int^t_0v^\pi(s)\,\diff s+N^\pi(t),\ \ t\in[0,\infty),
\end{equation}
where, for each $t\in[0,\infty)$,
\begin{align*}
	v^\pi(t)&:=-\frac{2\lambda}{\Delta}\big|X(t)-\hX^\pi(t)\big|^2\1_{[0,T\wedge\tau^\pi\wedge\zeta)}(t)\\
	&\hspace{0.5cm}+2\big\langle X(t)-\hX^\pi(t),b(t,X)-b\big(\pi(t),\poly[\hX^\pi]\big)\big\rangle+\big|\sigma(t,X)-\sigma\big(\pi(t),\poly[\hX^\pi]\big)\big|^2
\end{align*}
and
\begin{equation*}
	N^\pi(t):=2\int^t_0\Big\langle X(s)-\hX^\pi(s),\Big(\sigma(s,X)-\sigma\big(\pi(s),\poly[\hX^\pi]\big)\Big)\,\diff W(s)\Big\rangle.
\end{equation*}
For each $\varphi\in\{b,\sigma\}$ and $t\in[0,\infty)$, define
\begin{equation*}
	v^\pi_\varphi(t):=\big|\varphi(t,X)-\varphi\big(\pi(t),\poly[\hX^\pi]\big)\big|.
\end{equation*}
By the definition of the stopping time $\tau^\pi$, we have
\begin{equation}\label{eq_v0}
	v^\pi(t)\leq-\frac{2\lambda}{\Delta}\big|X(t)-\hX^\pi(t)\big|^2+2\Delta v^\pi_b(t)+v^\pi_\sigma(t)^2,\ \ t\in[0,T\wedge\tau^\pi\wedge\zeta).
\end{equation}
Also, notice that $N^\pi$ is a one-dimensional continuous local martingale on $(\Omega,\cF,\bF,\bP)$ with the quadratic variation satisfying
\begin{equation}\label{eq_N0}
	\frac{\diff\langle N^\pi\rangle(t)}{\diff t}\leq4\big|X(t)-\hX^\pi(t)\big|^2\,\big|\sigma(t,X)-\sigma\big(\pi(t),\poly[\hX^\pi]\big)\big|^2\leq4\Delta^2v^\pi_\sigma(t)^2,\ \ t\in[0,T\wedge\tau^\pi).
\end{equation}

Now we estimate the terms $v^\pi_b(t)$ and $v^\pi_\sigma(t)$ for $t\in[0,T\wedge\tau^\pi\wedge\zeta\wedge\eta^\pi)$. Notice that, by \eqref{eq_stopping},
\begin{align*}
	&(t,X),\,\big(\pi(t),\poly[X]\big),\,\big(\pi(t),\poly[\hX^\pi]\big)\in\Lambda\big[\supp\mu_0,D_{\conti,\varphi;T}(R_{\conti,\varphi})\big]
	\ \text{for any $t\in[0,T\wedge\tau^\pi\wedge\zeta)$ $\bP$-a.s.},\ \ \varphi\in\{b,\sigma\}.
\end{align*}
Hence, thanks to condition (C) in \cref{assum}, we have
\begin{equation}\label{eq_v123}
\begin{split}
	v^\pi_\varphi(t)&\leq\big|\varphi(t,X)-\varphi\big(\pi(t),X\big)\big|+\big|\varphi\big(\pi(t),X\big)-\varphi\big(\pi(t),\poly[X]\big)\big|
	+\big|\varphi\big(\pi(t),\poly[X]\big)-\varphi\big(\pi(t),\poly[\hX^\pi]\big)\big|\\
	&\leq K_{\conti,\varphi;T}R_{\conti,\varphi}\big\{V^\pi_1(t)^{\alpha_\varphi}+V^\pi_2(t)^{\alpha_\varphi}+V^\pi_3(t)^{\alpha_\varphi}\big\}\ 
	\text{for any $t\in[0,T\wedge\tau^\pi\wedge\zeta\wedge\eta^\pi)$ $\bP$-a.s.},\ \ \varphi\in\{b,\sigma\},
\end{split}
\end{equation}
where
\begin{align*}
	V^\pi_1(t)&:=\Big\{\varpi\big(t,X;t-\pi(t)\big)+\big(t-\pi(t)\big)^{1/2}\Big\}\1_{[0,T\wedge\eta^\pi)}(t),\\
	V^\pi_2(t)&:=\big\|X_{\pi(t)}-\poly[X]_{\pi(t)}\big\|_\infty\1_{[0,\eta^\pi)}(t),\\
	V^\pi_3(t)&:=\big\|\poly[X]_{\pi(t)}-\poly[\hX^\pi]_{\pi(t)}\big\|_\infty\1_{[0,\tau^\pi)}(t).
\end{align*}
As for $V^\pi_1$, we have
\begin{align}
	\nonumber
	V^\pi_1(t)&\leq\Big\{\varpi\big(t,X;|\pi|\big)+\sqrt{|\pi|}\Big\}\1_{[0,\eta^\pi)}(t)\\
	\nonumber
	&\leq K_{\growth,b;T}R_{\growth,b}|\pi|+4K_{\growth,\sigma;T}R_{\growth,\sigma}\sqrt{n|\pi|\log\frac{T}{|\pi|}}+\sqrt{|\pi|}\\
	\label{eq_v1}
	&\leq6\sqrt{n}\cE,
\end{align}
where we used the fact that $0\leq t-\pi(t)\leq|\pi|$ for any $t\in[0,T]$ in the first inequality, the definition of $\eta^\pi$ in the second inequality, and the definition \eqref{eq_cE} of $\cE=\cE(|\pi|,\Delta,R_{\growth,b},R_{\growth,\sigma})$ together with the assumptions that $\log\frac{T}{|\pi|}\geq1$ (imposed in this proof) and $K_{\growth,\sigma;T}R_{\growth,\sigma}\geq1$ in the third inequality. Similarly, as for $V^\pi_2$, by using \eqref{eq_polylinear-operator1} in \cref{lemm_polylinear-operator} and noting the definition of the stopping time $\eta^\pi$, we have
\begin{align}
	\nonumber
	V^\pi_2(t)&\leq\varpi(\pi(t),X;|\pi|)\1_{[0,\eta^\pi)}(t)\\
	\nonumber
	&\leq K_{\growth,b;T}R_{\growth,b}|\pi|+4K_{\growth,\sigma;T}R_{\growth,\sigma}\sqrt{n|\pi|\log\frac{T}{|\pi|}}\\
	\label{eq_v2}
	&\leq5\sqrt{n}\cE.
\end{align}
As for $V^\pi_3$, by using \eqref{eq_polylinear-operator2} in \cref{lemm_polylinear-operator} and noting the definition of the stopping time $\tau^\pi$, we have
\begin{equation}\label{eq_v3}
	V^\pi_3(t)\leq\|X_{\pi(t)}-\hX^\pi_{\pi(t)}\|_\infty\1_{[0,\tau^\pi)}(t)\leq\Delta\leq\cE.
\end{equation}
By \eqref{eq_v123}, \eqref{eq_v1}, \eqref{eq_v2} and \eqref{eq_v3}, we get
\begin{equation}\label{eq_v-phi}
	v^\pi_\varphi(t)\leq12\sqrt{n}K_{\conti,\varphi;T}R_{\conti,\varphi}\cE^{\alpha_\varphi}\ \ \text{for any $t\in[0,T\wedge\tau^\pi\wedge\zeta\wedge\eta^\pi)$ $\bP$-a.s.},\ \ \varphi\in\{b,\sigma\}.
\end{equation}

By \eqref{eq_v0} and \eqref{eq_v-phi}, we obtain
\begin{equation}\label{eq_v}
\begin{split}
	&v^\pi(t)\leq-\frac{2\lambda}{\Delta}\big|X(t)-\hX^\pi(t)\big|^2+24\sqrt{n}K_{\conti,b;T}R_{\conti,b}\Delta\cE^{\alpha_b}+144n\big(K_{\conti,\sigma;T}R_{\conti,\sigma}\cE^{\alpha_\sigma}\big)^2\\
	&\hspace{3cm}\text{for any $t\in[0,T\wedge\tau^\pi\wedge\zeta\wedge\eta^\pi)$ $\bP$-a.s.}
\end{split}
\end{equation}
Also, by \eqref{eq_N0} and \eqref{eq_v-phi} with $\varphi=\sigma$, we obtain
\begin{equation}\label{eq_N}
	\frac{\diff\langle N^\pi\rangle(t)}{\diff t}\leq576n\left(K_{\conti,\sigma;T}R_{\conti,\sigma}\Delta\cE^{\alpha_\sigma}\right)^2\ \ \text{for any $t\in[0,T\wedge\tau^\pi\wedge\zeta\wedge\eta^\pi)$ $\bP$-a.s.}
\end{equation}
Now we apply \cref{appendix_lemm_Kulik--Scheutzou} to the nonnegative It\^{o} process
\begin{equation}\label{eq_appendix-Z}
	Z(t)=|X(t)-\hX^\pi(t)|^2,
\end{equation}
nonnegative random variable
\begin{equation}\label{eq_appendix-varsigma}
	\varsigma=T\wedge\tau^\pi\wedge\zeta\wedge\eta^\pi,
\end{equation}
and constants
\begin{equation}\label{eq_appendix-muAB}
\begin{split}
	&\kappa=\frac{2\lambda}{\Delta},\ \ A=24\sqrt{n}K_{\conti,b;T}R_{\conti,b}\Delta\cE^{\alpha_b}+144n\big(K_{\conti,\sigma;T}R_{\conti,\sigma}\cE^{\alpha_\sigma}\big)^2,\\
	&B=576n\left(K_{\conti,\sigma;T}R_{\conti,\sigma}\Delta\cE^{\alpha_\sigma}\right)^2,
\end{split}
\end{equation}
and
\begin{equation}\label{eq_appendix-theta}
	\theta=\sqrt{\log\frac{1}{\Delta}}.
\end{equation}
By \eqref{eq_Ito}, \eqref{eq_v} and \eqref{eq_N}, we see that $Z$, $\varsigma$, $\kappa$, $A$ and $B$ specified in \eqref{eq_appendix-Z}, \eqref{eq_appendix-varsigma} and \eqref{eq_appendix-muAB} satisfy
\begin{equation*}
	\diff Z(t)=v^\pi(t)\diff t+\diff N^\pi(t),\ \ t\in[0,\infty),\ \ Z(0)=0,
\end{equation*}
and
\begin{equation*}
	v^\pi(t)\leq-\kappa Z(t)+A,\ \ \frac{\diff\langle N^\pi\rangle(t)}{\diff t}\leq B,\ \ \text{for any $t\in[0,\varsigma)$ $\bP$-a.s.}
\end{equation*}
Furthermore, under the constraint \eqref{eq_constraint} on the parameters and the assumption that $\log\frac{1}{\Delta}\geq1$ imposed in this proof, we see that the constants $\kappa,A,B$ and $\theta$ specified in \eqref{eq_appendix-muAB} and \eqref{eq_appendix-theta} satisfy
\begin{equation*}
	\frac{A}{\kappa}+\sqrt{\frac{2B\big(\theta^2+\log(1\vee(T\kappa))\big)}{\kappa}}\leq\Delta^2.
\end{equation*}
Hence, by \eqref{eq_tau1} and \cref{appendix_lemm_Kulik--Scheutzou}, we have
\begin{align}
	\nonumber
		P_3&=\bP\left(\sup_{t\in[0,\varsigma]}Z(t)\geq\Delta^2\right)\\
	\nonumber
	&\leq\bP\left(\sup_{t\in[0,\varsigma]}Z(t)\geq\frac{A}{\kappa}+\sqrt{\frac{2B\big(\theta^2+\log(1\vee(T\kappa))\big)}{\kappa}}\right)\\
	\nonumber
	&\leq219\exp\left(-\theta^2\right)\\
	\label{eq_prob-3}
	&=219\Delta.
\end{align}
By \eqref{eq_prob}, \eqref{eq_prob-1}, \eqref{eq_prob-2} and \eqref{eq_prob-3}, we get the estimate \eqref{eq_X-hXpi-prob}. This completes the proof.
\end{proof}


\subsection{Error estimate between the controlled and true Euler–Maruyama schemes}\label{subsec_hXpi-Xpi}

Next, we provide an estimate for the error-in-law between the controlled Euler--Maruyama scheme $\hX^\pi$ and the true Euler--Maruyama scheme $X^\pi$. Before that, we recall some standard facts on several metrics on the space of probability measures. The total variation distance $d_\TV(\mu,\nu)$ between two probability measures $\mu,\nu\in\cP(S)$ on a measurable space $(S,\cS)$ is defined by
\begin{equation*}
	d_\TV(\mu,\nu):=\sup_{A\in\cS}|\mu(A)-\nu(A)|.
\end{equation*}
Also, for any $\mu,\nu\in\cP(S)$ such that $\mu\ll\nu$, the so-called Kullback--Leibler divergence of $\mu$ from $\nu$ is defined by
\begin{equation*}
	D_\KL(\mu\|\nu):=\int_S\log\frac{\diff\mu}{\diff\nu}\,\diff\mu.
\end{equation*}
The following are fundamental facts:
\begin{itemize}
\item
Let $(S,d_S)$ be a complete separable metric space and $\cS=\cB(S)$ be the corresponding Borel $\sigma$-algebra. Observe that the number $\ep=\sup_{A\in\cS}|\mu(A)-\nu(A)|$ belongs to the set over which the infimum in the definition \eqref{eq_LPdef} of the L\'{e}vy--Prokhorov metric is taken, and hence we have $d_\LP(\mu,\nu)\leq d_\TV(\mu,\nu)$ for any $\mu,\nu\in\cP(S)$.
\item
It is clear from the definition that
\begin{equation*}
	d_{\TV,T}(\mu\circ f^{-1},\nu\circ f^{-1})\leq d_{\TV,S}(\mu,\nu)
\end{equation*}
for any $\mu,\nu\in\cP(S)$ and any measurable map $f$ from a measurable space $(S,\cS)$ to another measurable space $(T,\cT)$, where $d_{\TV,S}$ and $d_{\TV,T}$ denote the total variation distances on $(S,\cS)$ and $(T,\cT)$, respectively.
\item
The well-known Pinsker's inequality (cf.\ \cite[Lemma 2.5 (i)]{Ts08}) shows that
\begin{equation*}
	d_\TV(\mu,\nu)\leq\sqrt{\frac{1}{2}D_\KL(\mu\|\nu)}
\end{equation*}
for any $\mu,\nu\in\cP(S)$ such that $\mu\ll\nu$.
\end{itemize}

Now we provide an estimate for the error between the controlled and true Euler--Maruyama schemes. Here, the local ellipticity condition (E) in \cref{assum} is crucial.


\begin{prop}\label{prop_hXpi-Xpi}
Fix data $\data$ satisfying \cref{assum}. Suppose that we are given a weak solution $\weaksol$ of the SFDE \eqref{eq_SFDE} associated with $\data$. Fix $T\in(0,\infty)$, $\pi\in\Pi_T$, $\lambda\in(0,\infty)$, $\Delta\in(0,1]$ and $\vec{R}=(R_{\growth,b},R_{\growth,\sigma},R_{\conti,b},R_{\conti,\sigma},R_\ellip)\in[1,\infty)^5$. Let $X^\pi$ be the Euler--Maruyama scheme defined on $(\Omega^\pi,\cF^\pi,\bP^\pi)$ and given by \eqref{eq_EM} with initial distribution $\mu_0$. Let $\hX^\pi$ be the controlled Euler--Maruyama scheme, given by \eqref{eq_controlledEM}, associated with $\weaksol$ and $(T,\pi,\lambda,\Delta,\vec{R})$. Then, there exists a probability measure $\widehat{\bP}^\pi\sim\bP$ on $(\Omega,\cF_T)$ such that
\begin{equation}\label{eq_hatP-Law}
	\Law_{\widehat{\bP}^\pi}(\poly[\hX^\pi])=\Law_{\bP^\pi}(\poly[X^\pi])
\end{equation}
and
\begin{equation}\label{eq_hatP-KL}
	D_\KL(\widehat{\bP}^\pi\|\bP)\leq\frac{T}{2}\big(K_{\ellip;T}R_\ellip\lambda\big)^2.
\end{equation}
In particular, it holds that
\begin{equation}\label{eq_hXpi-Xpi}
	d_\LP\big(\Law_\bP(\poly[\hX^\pi]),\Law_{\bP^\pi}(\poly[X^\pi])\big)\leq\frac{\sqrt{T}}{2}K_{\ellip;T}R_\ellip\lambda.
\end{equation}
\end{prop}


\begin{proof}
The inequality \eqref{eq_hXpi-Xpi} follows from \eqref{eq_hatP-Law} and \eqref{eq_hatP-KL}. Indeed,
\begin{align*}
	d_\LP\big(\Law_\bP(\poly[\hX^\pi]),\Law_{\bP^\pi}(\poly[X^\pi])\big)&=d_\LP\big(\Law_\bP(\poly[\hX^\pi]),\Law_{\widehat{\bP}^\pi}(\poly[\hX^\pi])\big)\\
	&\leq d_\TV\big(\Law_\bP(\poly[\hX^\pi]),\Law_{\widehat{\bP}^\pi}(\poly[\hX^\pi])\big)\\
	&\leq d_\TV\big(\bP,\widehat{\bP}^\pi\big)\\
	&\leq\sqrt{\frac{1}{2}D_\KL\big(\widehat{\bP}^\pi\|\bP\big)}\\
	&\leq\frac{\sqrt{T}}{2}K_{\ellip;T}R_\ellip\lambda,
\end{align*}
where the first equality follows from \eqref{eq_hatP-Law}, and the last inequality follows from \eqref{eq_hatP-KL}; the second, third and fourth inequalities follow from the fundamental properties of metrics on the space of probability measures summarized above. In the following, we construct a probability measure $\widehat{\bP}^\pi\sim\bP$ on $(\Omega,\cF_T)$ such that \eqref{eq_hatP-Law} and \eqref{eq_hatP-KL} hold.

By \eqref{eq_stopping}, we see that $(\pi(t),\poly[\hX^\pi])\in\Lambda[\supp\mu_0,D_{\ellip;T}(R_\ellip)]$ for any $t\in[0,T\wedge\tau^\pi\wedge\zeta)$ outside a $\bP$-null set $N\in\cF$. Thus, by condition (E) in \cref{assum}, we have
\begin{equation*}
	\big\langle\sigma\big(\pi(t),\poly[\hX^\pi]\big)\sigma\big(\pi(t),\poly[\hX^\pi]\big)^\top\xi,\xi\big\rangle\geq\frac{|\xi|^2}{K_{\ellip;T}^2R_\ellip^2}
\end{equation*}
for any $t\in[0,T\wedge\tau^\pi\wedge\zeta)$ and $\xi\in\bR^n$ on $\Omega\setminus N$. In particular, on $\Omega\setminus N$, for any $t\in[0,T\wedge\tau^\pi\wedge\zeta)$, we can define the pseudo-inverse $\sigma\big(\pi(t),\poly[\hX^\pi]\big)^\dagger\in\bR^{d\times n}$ of the matrix $\sigma\big(\pi(t),\poly[\hX^\pi]\big)\in\bR^{n\times d}$, that is,
\begin{equation*}
	\sigma\big(\pi(t),\poly[\hX^\pi]\big)^\dagger:=\sigma\big(\pi(t),\poly[\hX^\pi]\big)^\top\Big(\sigma\big(\pi(t),\poly[\hX^\pi]\big)\sigma\big(\pi(t),\poly[\hX^\pi]\big)^\top\Big)^{-1}.
\end{equation*}
Notice that
\begin{equation}\label{eq_pseudo-inverse1}
	\sigma\big(\pi(t),\poly[\hX^\pi]\big)\sigma\big(\pi(t),\poly[\hX^\pi]\big)^\dagger=I_{n\times n},\ \ \text{for $t\in[0,T\wedge\tau^\pi\wedge\zeta)$ on $\Omega\setminus N$},
\end{equation}
and
\begin{equation}\label{eq_pseudo-inverse2}
	\|\sigma\big(\pi(t),\poly[\hX^\pi]\big)^\dagger\|_\op\leq K_{\ellip;T}R_\ellip\ \ \text{for $t\in[0,T\wedge\tau^\pi\wedge\zeta)$ on $\Omega\setminus N$},
\end{equation}
where $\|\cdot\|_\op$ denotes the operator norm. Define an $\bR^d$-valued $\bF$-progressively measurable process $u^\pi=(u^\pi(t))_{t\in[0,T]}$ on $(\Omega,\cF_T)$ by
\begin{equation*}
	u^\pi(t):=\frac{\lambda}{\Delta}\sigma\big(\pi(t),\poly[\hX^\pi]\big)^\dagger\big(X(t)-\hX^\pi(t)\big)\1_{[0,T\wedge\tau^\pi\wedge\zeta)}(t)\1_{\Omega\setminus N},\ \ t\in[0,T].
\end{equation*}
Then, by \eqref{eq_pseudo-inverse2} and the definition of the stopping time $\tau^\pi$, we have
\begin{align}
	\nonumber
	\int^T_0\big|u^\pi(t)\big|^2\,\diff t&\leq\frac{\lambda^2}{\Delta^2}\int^{T\wedge\tau^\pi\wedge\zeta}_0\big\|\sigma\big(\pi(t),\poly[\hX^\pi]\big)^\dagger\big\|^2_\op\,\big|X(t)-\hX^\pi(t)\big|^2\,\diff t\\
	\nonumber
	&\leq\frac{\big(K_{\ellip;T}R_\ellip\lambda\big)^2}{\Delta^2}\int^{T\wedge\tau^\pi}_0\big|X(t)-\hX^\pi(t)\big|^2\,\diff t\\
	\label{eq_u^pi}
	&\leq T\big(K_{\ellip;T}R_\ellip\lambda\big)^2\ \ \text{$\bP$-a.s.}
\end{align}
Now we define a new measure $\widehat{\bP}^\pi\sim\bP$ on $(\Omega,\cF_T)$ by
\begin{equation*}
	\frac{\diff\widehat{\bP}^\pi}{\diff\bP}:=\exp\Big(-\int^T_0\big\langle u^\pi(t),\diff W(t)\big\rangle-\frac{1}{2}\int^T_0\big|u^\pi(t)\big|^2\,\diff t\Big).
\end{equation*}
By \eqref{eq_u^pi}, Novikov's condition is satisfied, and hence $\widehat{\bP}^\pi$ is a probability measure on $(\Omega,\cF_T)$. We show that $\widehat{\bP}^\pi$ satisfies \eqref{eq_hatP-Law} and \eqref{eq_hatP-KL}.

By Girsanov's theorem, the process
\begin{equation*}
	\widehat{W}^\pi(t):=W(t)+\int^t_0u^\pi(s)\,\diff s,\ \ t\in[0,T],
\end{equation*}
is a $d$-dimensional Brownian motion on $(\Omega,\cF_T,\widehat{\bP}^\pi)$ relative to $\bF$. Furthermore, by \eqref{eq_pseudo-inverse1} and the definition of $u^\pi$, we see that $\hX^\pi$ solves the following SFDE $\widehat{\bP}^\pi$-a.s.:
\begin{equation*}
	\begin{dcases}
	\diff \hX^\pi(t)=b\big(\pi(t),\poly[\hX^\pi]\big)\,\diff t+\sigma\big(\pi(t),\poly[\hX^\pi]\big)\,\diff\widehat{W}^\pi(t),\ \ t\in[0,T],\\
	\hX^\pi(0)=X(0).
	\end{dcases}
\end{equation*}
In particular, denoting $\pi=(t_k)^m_{k=0}$,
\begin{equation*}
	\begin{dcases}
	\hX^\pi(t_{k+1})=\hX^\pi(t_k)+b\big(t_k,\poly[\hX^\pi]\big)(t_{k+1}-t_k)+\sigma\big(t_k,\poly[\hX^\pi]\big)\widehat{Z}^\pi_k,\\
	\hspace{7cm}k\in\{0,\dots,m-1\},\\
	\hX^\pi(t_0)=X(0),
	\end{dcases}
\end{equation*}
where $\widehat{Z}^\pi_k:=\widehat{W}^\pi(t_{k+1})-\widehat{W}^\pi(t_k)$ for each $k\in\{0,\dots,m-1\}$.
Therefore, $(\hX^\pi(t_k))^m_{k=0}$ satisfies \eqref{eq_EM} on the probability space $(\Omega,\cF_T,\widehat{\bP}^\pi)$ with $(\xi^\pi,Z^\pi_0,\dots,Z^\pi_{m-1})$ replaced by $(X(0),\widehat{Z}^\pi_0,\dots,\widehat{Z}^\pi_{m-1})$. Thus, we have $(\hX^\pi(t_k))^m_{k=0}=\Psi^\pi(X(0),\widehat{Z}^\pi_0,\dots,\widehat{Z}^\pi_{m-1})$, where $\Psi^\pi:\bR^n\times(\bR^d)^m\to(\bR^n)^{m+1}$ is the measurable map appearing in the argument just after \eqref{eq_EM}. Notice that $\widehat{\bP}^\pi\circ X(0)^{-1}=\bP\circ X(0)^{-1}=\mu_0$, $\widehat{\bP}^\pi\circ(\widehat{Z}^\pi_k)^{-1}=N(0,(t_{k+1}-t_k)I_{d\times d})$ for each $k\in\{0,\dots,m-1\}$, and the random variables $X(0),\widehat{Z}^\pi_0,\dots,\widehat{Z}^\pi_{m-1}$ are independent under $\widehat{\bP}^\pi$. On the other hand, recall that $X^\pi=(X^\pi(t_k))^m_{k=0}=\Psi^\pi(\xi^\pi,Z^\pi_0,\dots,Z^\pi_{m-1})$, $\bP^\pi\circ(\xi^\pi)^{-1}=\mu_0$ by the assumption, $\bP^\pi\circ(Z^\pi_k)^{-1}=N(0,(t_{k+1}-t_k)I_{d\times d})$ for each $k\in\{0,\dots,m-1\}$, and the random variables $\xi^\pi,Z^\pi_0,\dots,Z^\pi_{m-1}$ are independent under $\bP^\pi$. Therefore,
\begin{align*}
	\Law_{\widehat{\bP}^\pi}(\poly[\hX^\pi])&=\left(\mu_0\otimes\bigotimes^{m-1}_{k=0}N\left(0,(t_{k+1}-t_k)I_{d\times d}\right)\right)\circ(\Psi^\pi)^{-1}\circ(\poly)^{-1}\\
	&=\Law_{\bP^\pi}(\poly[X^\pi]),
\end{align*}
and hence the equality \eqref{eq_hatP-Law} holds. Furthermore, denoting by $\bE_{\widehat{\bP}^\pi}[\cdot]$ the expectation operator on $(\Omega,\cF_T,\widehat{\bP}^\pi)$, the estimate \eqref{eq_u^pi} and the fact that $\widehat{\bP}^\pi\sim\bP$ yield
\begin{align*}
	D_\KL(\widehat{\bP}^\pi\|\bP)&=\bE_{\widehat{\bP}^\pi}\Big[-\int^T_0\big\langle u^\pi(t),\diff W(t)\big\rangle-\frac{1}{2}\int^T_0\big|u^\pi(t)\big|^2\,\diff t\Big]\\
	&=\bE_{\widehat{\bP}^\pi}\Big[-\int^T_0\big\langle u^\pi(t),\diff\widehat{W}^\pi(t)\big\rangle+\frac{1}{2}\int^T_0\big|u^\pi(t)\big|^2\,\diff t\Big]\\
	&=\frac{1}{2}\bE_{\widehat{\bP}^\pi}\Big[\int^T_0\big|u^\pi(t)\big|^2\,\diff t\Big]\\
	&\leq\frac{T}{2}\big(K_{\ellip;T}R_\ellip\lambda\big)^2.
\end{align*}
Thus, we obtain the desired estimate \eqref{eq_hatP-KL}. This completes the proof.
\end{proof}


\subsection{Optimal control and completion of the proof of the main result}\label{subsec_proof-main}

Fix data $\data$ satisfying \cref{assum}. Suppose that we are given a weak solution $\weaksol$ of the SFDE \eqref{eq_SFDE} associated with $\data$. For each $T\in(0,\infty)$ and $\pi\in\Pi_T$, let $X^\pi$ be the Euler--Maruyama scheme defined on $(\Omega^\pi,\cF^\pi,\bP^\pi)$ and given by \eqref{eq_EM} with initial distribution $\mu_0$.  For each $T\in(0,\infty)$, $\pi\in\Pi_T$, $\lambda\in(0,\infty)$, $\Delta\in(0,1]$ and $\vec{R}=(R_{\growth,b},R_{\growth,\sigma},R_{\conti,b},R_{\conti,\sigma},R_\ellip)\in[1,\infty)^5$, let $\hX^\pi$ be the controlled Euler--Maruyama scheme, given by \eqref{eq_controlledEM}, associated with $\weaksol$ and $(T,\pi,\lambda,\Delta,\vec{R})$. By the triangle inequality, we have
\begin{align*}
	&d_\LP\big(\Law_\bP(X_T),\Law_{\bP^\pi}(\poly[X^\pi])\big)\\
	&\leq d_\LP\big(\Law_\bP(X_T),\Law_\bP(\poly[\hX^\pi])\big)+d_\LP\big(\Law_\bP(\poly[\hX^\pi]),\Law_{\bP^\pi}(\poly[X^\pi])\big).
\end{align*}
Thus, \cref{prop_X-hXpi} and \cref{prop_hXpi-Xpi} yield
\begin{equation}\label{eq_error'}
\begin{split}
	&d_\LP\big(\Law_\bP(X_T),\Law_{\bP^\pi}(\poly[X^\pi])\big)\\
	&\leq\bP\left(\inf_{t\in[0,T]}\dist\!\left(X(t),\bR^n\setminus D_T(\vec{R})\right)\leq\Delta\right)+\left(\frac{4n}{T}+219+6\sqrt{n}\right)\cE+\frac{\sqrt{T}}{2}K_{\ellip;T}R_\ellip\lambda,
\end{split}
\end{equation}
whenever the parameters satisfy the constraint \eqref{eq_constraint}. Notice that the left-hand side of \eqref{eq_error'} does not depend on the choice of the control parameter $\lambda$, $\Delta$, or $\vec{R}=(R_{\growth,b},R_{\growth,\sigma},R_{\conti,b},R_{\conti,\sigma},R_\ellip)\in[1,\infty)^5$. In order to show the convergence of $d_\LP(\Law_\bP(X_T),\Law_{\bP^\pi}(\poly[X^\pi]))$ as $|\pi|\downarrow0$ with convergence speed as fast as possible, we focus on the minimization problem for the right-hand side of \eqref{eq_error'} in terms of the control parameters satisfying the constraint \eqref{eq_constraint}. This is the main idea of the proof of \cref{theo_main}. In the following proof, we first show assertion (ii) by choosing suitable intensity parameter $\lambda$ which makes the right-hand side of \eqref{eq_error'} as small as possible (in view of the convergence order) for fixed $T$, $\pi$, $\Delta$ and $\vec{R}$. Then, assertion (i) follows from (ii) by taking the limits $|\pi|\downarrow0$, $\Delta\downarrow0$ and $R_{\growth,b},R_{\growth,\sigma},R_{\conti,b},R_{\conti,\sigma},R_\ellip\uparrow\infty$ in this order. Lastly, based on the estimate derived in assertion (ii), together with the additional assumption \eqref{eq_rare-event} on the decay rate of the probability of the ``rare event'', we solve an optimal control problem with respect to $\Delta$ and $\vec{R}$ and obtain the (nearly) optimal convergence rate in terms of $|\pi|$, showing assertion (iii).


\begin{rem}\label{rem_proof-main}
Under the constraint \eqref{eq_constraint}, the last term in the right-hand side of \eqref{eq_error'} is greater than
\begin{equation*}
	2^9n\sqrt{T}K_{\ellip;T}R_\ellip\max\left\{K_{\conti,b;T}R_{\conti,b}\cE^{\alpha_b},\frac{\left(K_{\conti,\sigma;T}R_{\conti,\sigma}\cE^{\alpha_\sigma}\right)^2}{\Delta}\log\frac{1}{\Delta}\right\},
\end{equation*}
which is a benchmark for the error estimate of $d_\LP(\Law_\bP(X_T),\Law_{\bP^\pi}(\poly[X^\pi]))$; notice that, by the definition \eqref{eq_cE} of $\cE=\cE(|\pi|,\Delta,R_{\growth,b},R_{\growth,\sigma})\in(0,\infty)$, the benchmark-term above converges to zero as $|\pi|\downarrow0$ and $\Delta\downarrow0$ if and only if $\alpha_\sigma>\frac{1}{2}$, which is the standing assumption imposed in \cref{assum}.
\end{rem}

With the above observations in mind, we provide a proof of our main result.


\begin{proof}[Proof of \cref{theo_main}]
We first show assertion (ii). Assertions (i) and (iii) follow from (ii).

\underline{Proof of (ii).} Fix $T\in(0,\infty)$, $\pi\in\Pi_T$, $\Delta\in(0,1]$ and $\vec{R}=(R_{\growth,b},R_{\growth,\sigma},R_{\conti,b},R_{\conti,\sigma},R_\ellip)\in[1,\infty)^5$. We prove the estimate \eqref{eq_error}. Without loss of generality, we may assume that each term in the right-hand side of \eqref{eq_error} is less than $1$; otherwise the estimate becomes trivial. In particular, noting that $K_{\ellip;T}R_{\ellip}\geq1$, we may assume that
\begin{equation}\label{eq_wlog}
	2^{12}nT\max\left\{K_{\conti,b;T}R_{\conti,b}\cE^{\alpha_b},\frac{\left(K_{\conti,\sigma;T}R_{\conti,\sigma}\cE^{\alpha_\sigma}\right)^2}{\Delta}\log\frac{1}{\Delta}\right\}\leq1.
\end{equation}
As discussed above, the estimate \eqref{eq_error'} holds whenever the parameters satisfy the constraint \eqref{eq_constraint}. With the observations in \cref{rem_proof-main} in mind, we take
\begin{equation*}
	\lambda=2^{11}n\max\left\{K_{\conti,b;T}R_{\conti,b}\cE^{\alpha_b},\frac{\left(K_{\conti,\sigma;T}R_{\conti,\sigma}\cE^{\alpha_\sigma}\right)^2}{\Delta}\log\frac{1}{\Delta}\right\}.
\end{equation*}
Notice that the assumption \eqref{eq_wlog} implies that $2T\lambda\leq1$, and hence $\log(1\vee\frac{2T\lambda}{\Delta})\leq\log\frac{1}{\Delta}$. Therefore, the above $\lambda$ satisfies the constraint \eqref{eq_constraint}, and the estimate \eqref{eq_error'} yields
\begin{align*}
	&d_\LP\big(\Law_\bP(X_T),\Law_{\bP^\pi}(\poly[X^\pi])\big)\\
	&\leq\bP\left(\inf_{t\in[0,T]}\dist\!\left(X(t),\bR^n\setminus D_T(\vec{R})\right)\leq\Delta\right)\\
	&\hspace{0.5cm}+\left(\frac{4n}{T}+219+6\sqrt{n}+2^{10}n\sqrt{T}\right)
	K_{\ellip;T}R_\ellip\max\left\{\cE,K_{\conti,b;T}R_{\conti,b}\cE^{\alpha_b},\frac{\left(K_{\conti,\sigma;T}R_{\conti,\sigma}\cE^{\alpha_\sigma}\right)^2}{\Delta}\log\frac{1}{\Delta}\right\}.
\end{align*}
Since $\frac{4n}{T}+219+6\sqrt{n}+2^{10}n\sqrt{T}\leq2^{12}n(T+\frac{1}{T})$, we get the desired estimate \eqref{eq_error}.

\underline{Proof of (i).} Let $T\in(0,\infty)$ be given. We take the limits
\begin{equation*}
	|\pi|\downarrow0,\ \Delta\downarrow0\ \text{and}\ R_{\growth,b},R_{\growth,\sigma},R_{\conti,b},R_{\conti,\sigma},R_\ellip\uparrow\infty
\end{equation*}
in \eqref{eq_error} in this order. First, observe that the number $\cE=\cE(|\pi|,\Delta,R_{\growth,b},R_{\growth,\sigma})\in(0,\infty)$ defined by \eqref{eq_cE} tends to $\Delta$ as $|\pi|\downarrow0$, and hence the estimate \eqref{eq_error} yields
\begin{align*}
	&\limsup_{|\pi|\downarrow0}d_\LP\big(\Law_\bP(X_T),\Law_{\bP^\pi}(\poly[X^\pi])\big)\\
	&\leq\bP\left(\inf_{t\in[0,T]}\dist\!\left(X(t),\bR^n\setminus D_T(\vec{R})\right)\leq\Delta\right)\\
	&\hspace{0.5cm}+2^{12}n\left(T+\frac{1}{T}\right)K_{\ellip;T}R_\ellip\max\left\{\Delta,K_{\conti,b;T}R_{\conti,b}\Delta^{\alpha_b},K_{\conti,\sigma;T}^2R_{\conti,\sigma}^2\Delta^{2\alpha_\sigma-1}\log\frac{1}{\Delta}\right\}.
\end{align*}
Then, noting that $2\alpha_\sigma-1>0$, by taking the limit $\Delta\downarrow0$, we have
\begin{equation*}
	\limsup_{|\pi|\downarrow0}d_\LP\big(\Law_\bP(X_T),\Law_{\bP^\pi}(\poly[X^\pi])\big)\leq\bP\left(\inf_{t\in[0,T]}\dist\!\left(X(t),\bR^n\setminus D_T(\vec{R})\right)=0\right).
\end{equation*}
Lastly, we take the limits $R_{\growth,b},R_{\growth,\sigma},R_{\conti,b},R_{\conti,\sigma},R_\ellip\to\infty$ in the right-hand side of the above inequality. Notice that the set
\begin{equation*}
	D_T(\vec{R})=D_{\growth,b;T}(R_{\growth,b})\cap D_{\growth,\sigma;T}(R_{\growth,\sigma})\cap D_{\conti,b;T}(R_{\conti,b})\cap D_{\conti,\sigma;T}(R_{\conti,\sigma})\cap D_{\ellip;T}(R_\ellip)
\end{equation*}
is open in $\bR^n$ for each $\vec{R}=(R_{\growth,b},R_{\growth,\sigma},R_{\conti,b},R_{\conti,\sigma},R_\ellip)\in[1,\infty)^5$, increasing in each of the variables, and $\bigcup_{\vec{R}\in[1,\infty)^5}D_T(\vec{R})=D$. Notice also that, for $\bP$-a.e.\ $\omega\in\Omega$, the path $t\mapsto X(\omega,t)$ is continuous and hence the set $\{X(\omega,t)\,|\,t\in[0,T]\}$ is compact in $\bR^n$. From these observations, we see that
\begin{align*}
	&\lim_{R_{\growth,b},R_{\growth,\sigma},R_{\conti,b},R_{\conti,\sigma},R_\ellip\to\infty}\bP\left(\inf_{t\in[0,T]}\dist\!\left(X(t),\bR^n\setminus D_T(\vec{R})\right)=0\right)\\
	&=\bP\left(\inf_{t\in[0,T]}\dist\!\left(X(t),\bR^n\setminus D_T(\vec{R})\right)=0\ \text{for any $\vec{R}\in[1,\infty)^5$}\right)\\
	&=\bP\left(\left\{\omega\in\Omega\relmiddle|\big\{X(\omega,t)\,\big|\,t\in[0,T]\big\}\subset D_T(\vec{R})\ \text{for some $\vec{R}\in[1,\infty)^5$}\right\}^\complement\right)\\
	&=\bP\left(\left\{\omega\in\Omega\relmiddle|\big\{X(\omega,t)\,\big|\,t\in[0,T]\big\}\subset D\right\}^\complement\right)\\
	&=\bP\Big(X(t)\notin D\ \text{for some $t\in[0,T]$}\Big),
\end{align*}
where it does not matter in which order the limits are taken in the first line. By the definition of the weak solution of the SFDE \eqref{eq_SFDE} associated with $\data$, we have $\bP(\text{$X(t)\in D$ for any $t\in[0,\infty)$})=1$, and hence the last term above is zero. Consequently, we get
\begin{equation*}
	\lim_{|\pi|\downarrow0}d_\LP\big(\Law_\bP(X_T),\Law_{\bP^\pi}(\poly[X^\pi])\big)=0,
\end{equation*}
and hence $\poly[X^\pi]\to X_T$ weakly on $\cC_T$ as $|\pi|\downarrow0$. This result indicates that, for any weak solution $\weaksol$ of the SFDE \eqref{eq_SFDE} associated with the data $\data$, the law of $X_T$ on $\cC_T$ is characterized as the weak limit of the law of $\poly[X^\pi]$ on $\cC_T$. Since $T\in(0,\infty)$ is arbitrary, we see that uniqueness in law holds for the SFDE \eqref{eq_SFDE} associated with $\data$.

\underline{Proof of (iii).} Since $\alpha_\sigma>\frac{1}{2}$, we can take a number $\ep\in(0,2\alpha_\sigma-1)$. Then, noting that $\log x\leq\frac{1}{\ep}x^\ep$ for any $x\in[1,\infty)$, we obtain the following simpler (but slightly weaker) version of assertion (ii): For any $T\in(0,\infty)$, $\pi\in\Pi_T$, $\Delta\in(0,1]$ and $\vec{R}=(R_{\growth,b},R_{\growth,\sigma},R_{\conti,b},R_{\conti,\sigma},R_{\ellip})\in[1,\infty)^5$, it holds that
\begin{equation}\label{eq_error-ep}
\begin{split}
	&d_\LP\big(\Law_\bP(X_T),\Law_{\bP^\pi}(\poly[X^\pi])\big)\\
	&\leq\bP\left(\inf_{t\in[0,T]}\dist\!\left(X(t),\bR^n\setminus D_T(\vec{R})\right)\leq\Delta\right)+n\widehat{\cK}_{T,\ep}R_{\ellip}\max\left\{R_{\conti,b}\cE_\ep^{\alpha_b},\frac{R_{\conti,\sigma}^2\cE_\ep^{2\alpha_\sigma}}{\Delta^{1+\ep}}\right\},
\end{split}
\end{equation}
where the constant $\cE_\ep=\cE_\ep(|\pi|,\Delta,R_{\growth,b},R_{\growth,\sigma})\in(0,\infty)$ is defined by
\begin{equation*}
	\cE_\ep:=\max\left\{\Delta,R_{\growth,b}|\pi|,R_{\growth,\sigma}|\pi|^{(1-\ep)/2}\right\},
\end{equation*}
and the constant $\widehat{\cK}_{T,\ep}\in[1,\infty)$, which depends only on $T$, $K_{\growth,b;T}$, $K_{\growth,\sigma;T}$, $K_{\conti,b;T}$, $K_{\conti,\sigma;T}$, $K_{\ellip;T}$, $\alpha_b$, $\alpha_\sigma$ and $\ep$, is defined by
\begin{equation*}
	\widehat{\cK}_{T,\ep}:=2^{12}\left(T+\frac{1}{T}\right)K_{\ellip;T}\max\left\{\widetilde{\cK}_{T,\ep},K_{\conti,b;T}\widetilde{\cK}_{T,\ep}^{\alpha_b},\frac{1}{\ep}K_{\conti,\sigma;T}^2\widetilde{\cK}_{T,\ep}^{2\alpha_\sigma}\right\}
\end{equation*}
with
\begin{equation*}
	\widetilde{\cK}_{T,\ep}:=\max\left\{K_{\growth,b;T},\sqrt{\frac{T^\ep}{\ep}}K_{\growth,\sigma;T}\right\}.
\end{equation*}
Here, we used the facts that the left-hand side of \eqref{eq_error-ep} is less than or equal to $1$ and that $\cE_\ep\wedge1\leq R_{\conti,b}\cE_\ep^{\alpha_b}$; the latter is due to $\alpha_b\leq1$ and $R_{\conti,b}\geq1$. Combining \eqref{eq_error-ep} and the assumption \eqref{eq_rare-event} on the decay rate of the probability of the ``rare event'', we will show that, for each $\ep\in(0,2\alpha_\sigma-1)$, $T\in(0,\infty)$ and $\pi\in\Pi_T$,
\begin{equation}\label{eq_error-optimal-ep}
	d_\LP\big(\Law_\bP(X_T),\Law_{\bP^\pi}(\poly[X^\pi])\big)\leq\big(\widehat{C}+n\widehat{\cK}_{T,\ep}\big)|\pi|^{\gamma_\ep},
\end{equation}
where the exponent $\gamma_\ep=\gamma_\ep(\alpha_b,\alpha_\sigma,\vec{\beta})\in(0,\infty)$ is determined in \eqref{eq_optimization} below. After that, we will take the limit $\ep\downarrow0$ and obtain the statement of (iii).

Let $\ep\in(0,2\alpha_\sigma-1)$, $T\in(0,\infty)$ and $\pi\in\Pi_T$ be fixed. In order to show \eqref{eq_error-optimal-ep}, without loss of generality, we may assume that $|\pi|\leq1$; otherwise the estimate becomes trivial. We apply the estimate \eqref{eq_error-ep} and the assumption \eqref{eq_rare-event} to the parameters $\Delta\in(0,1]$ and $\vec{R}\in[1,\infty)^5$ of the following forms:
\begin{align*}
	&\Delta=|\pi|^\delta,\\
	&R_{\growth,b}=|\pi|^{-r_{\growth,b}},\ R_{\growth,\sigma}=|\pi|^{-r_{\growth,\sigma}},\ R_{\conti,b}=|\pi|^{-r_{\conti,b}},\ R_{\conti,\sigma}=|\pi|^{-r_{\conti,\sigma}},\ R_\ellip=|\pi|^{-r_\ellip},
\end{align*}
for some constants $\delta,r_{\growth,b},r_{\growth,\sigma},r_{\conti,b},r_{\conti,\sigma},r_\ellip\in[0,\infty)$. Then, we obtain
\begin{equation*}
	d_\LP\big(\Law_\bP(X_T),\Law_{\bP^\pi}(\poly[X^\pi])\big)\leq\big(\widehat{C}+n\widehat{\cK}_{T,\ep}\big)|\pi|^{\Gamma_\ep(\delta,r_{\growth,b},r_{\growth,\sigma},r_{\conti,b},r_{\conti,\sigma},r_\ellip)},
\end{equation*}
where
\begin{align*}
	&\Gamma_\ep(\delta,r_{\growth,b},r_{\growth,\sigma},r_{\conti,b},r_{\conti,\sigma},r_\ellip)\\
	&:=\min\Big\{\beta_0\delta,\beta_{\growth,b}r_{\growth,b},\beta_{\growth,\sigma}r_{\growth,\sigma},\beta_{\conti,b}r_{\conti,b},\beta_{\conti,\sigma}r_{\conti,\sigma},\beta_\ellip r_\ellip,\\
	&\hspace{2cm}-r_\ellip-r_{\conti,b}+\alpha_b\min\Big\{\delta,1-r_{\growth,b},\frac{1-\ep}{2}-r_{\growth,\sigma}\Big\},\\
	&\hspace{2cm}-r_\ellip-2r_{\conti,\sigma}+2\alpha_\sigma \min\Big\{\delta,1-r_{\growth,b},\frac{1-\ep}{2}-r_{\growth,\sigma}\Big\}-(1+\ep)\delta\Big\}.
\end{align*}
Since the parameters $\delta,r_{\growth,b},r_{\growth,\sigma},r_{\conti,b},r_{\conti,\sigma},r_\ellip\in[0,\infty)$ are arbitrary, we see that \eqref{eq_error-optimal-ep} holds with exponent $\gamma_\ep=\gamma_\ep(\alpha_b,\alpha_\sigma,\vec{\beta})$ given by
\begin{equation}\label{eq_optimization}
	\gamma_\ep:=\sup_{\delta,r_{\growth,b},r_{\growth,\sigma},r_{\conti,b},r_{\conti,\sigma},r_\ellip\in[0,\infty)}\Gamma_\ep(\delta,r_{\growth,b},r_{\growth,\sigma},r_{\conti,b},r_{\conti,\sigma},r_\ellip).
\end{equation}
Notice that, extending the domain of the function $\Gamma_\ep$ to $\bR^6$, we see that $\Gamma_\ep:\bR^6\to\bR$ is continuous and $\Gamma_\ep\leq0$ on $\bR^6\setminus(0,1)^6$. Furthermore, thanks to $2\alpha_\sigma-1-\ep>0$, we can easily show that $\Gamma_\ep$ takes positive values on a compact subset of $(0,1)^6$. Therefore, the supremum in \eqref{eq_optimization} over $[0,\infty)^6$ is a maximum over $\bR^6$, and there exists an optimizer in $(0,1)^6$ which attains the maximal value $\gamma_\ep>0$.

In the following, we solve the optimization problem \eqref{eq_optimization} and provide a more precise expression of the maximal value $\gamma_\ep$. Observe that
\begin{equation}\label{eq_optimization'}
	\gamma_\ep=\max_{\delta,r_{\growth,b},r_{\growth,\sigma}\in\bR}\min\Big\{\beta_0\delta,\beta_{\growth,b}r_{\growth,b},\beta_{\growth,\sigma}r_{\growth,\sigma},\gamma^{(1)}_\ep(\delta,r_{\growth,b},r_{\growth,\sigma})\Big\},
\end{equation}
where
\begin{align*}
	\gamma^{(1)}_\ep(\delta,r_{\growth,b},r_{\growth,\sigma})&:=\max_{r_\ellip\in\bR}\min\left\{\beta_\ellip r_\ellip,\gamma^{(2)}_\ep(\delta,r_{\growth,b},r_{\growth,\sigma},r_\ellip),\gamma^{(3)}_\ep(\delta,r_{\growth,b},r_{\growth,\sigma},r_\ellip)\right\},\\
	\gamma^{(2)}_\ep(\delta,r_{\growth,b},r_{\growth,\sigma},r_\ellip)
	&:=\max_{r_{\conti,b}\in\bR}\min\left\{\beta_{\conti,b}r_{\conti,b},-r_\ellip-r_{\conti,b}+\alpha_b\min\Big\{\delta,1-r_{\growth,b},\frac{1-\ep}{2}-r_{\growth,\sigma}\Big\}\right\},\\
	\gamma^{(3)}_\ep(\delta,r_{\growth,b},r_{\growth,\sigma},r_\ellip)
	&:=\max_{r_{\conti,\sigma}\in\bR}\min\left\{\beta_{\conti,\sigma}r_{\conti,\sigma},\mathalpha{-}r_\ellip\mathalpha{-}2r_{\conti,\sigma}\mathalpha{+}2\alpha_\sigma\min\Big\{\delta,1\mathalpha{-}r_{\growth,b},\frac{1\mathalpha{-}\ep}{2}\mathalpha{-}r_{\growth,\sigma}\Big\}\mathalpha{-}(1\mathalpha{+}\ep)\delta\right\}.
\end{align*}
By elementary calculus, we see that
\begin{align*}
	\gamma^{(2)}_\ep(\delta,r_{\growth,b},r_{\growth,\sigma},r_\ellip)
	&=\frac{\beta_{\conti,b}}{\beta_{\conti,b}+1}\Big(-r_\ellip+\alpha_b\min\Big\{\delta,1-r_{\growth,b},\frac{1-\ep}{2}-r_{\growth,\sigma}\Big\}\Big),\\
	\gamma^{(3)}_\ep(\delta,r_{\growth,b},r_{\growth,\sigma},r_\ellip)
	&=\frac{\beta_{\conti,\sigma}}{\beta_{\conti,\sigma}+2}\Big(-r_\ellip+2\alpha_\sigma\min\Big\{\delta,1-r_{\growth,b},\frac{1-\ep}{2}-r_{\growth,\sigma}\Big\}-(1+\ep)\delta\Big).
\end{align*}
Hence,
\begin{align*}
	\gamma^{(1)}_\ep(\delta,r_{\growth,b},r_{\growth,\sigma})
	&=\max_{r_\ellip\in\bR}\min\left\{\beta_\ellip r_\ellip,\frac{\beta_{\conti,b}}{\beta_{\conti,b}+1}\Big(-r_\ellip+\alpha_b\min\Big\{\delta,1-r_{\growth,b},\frac{1-\ep}{2}-r_{\growth,\sigma}\Big\}\Big),\right.\\
	&\left.\hspace{2cm}\frac{\beta_{\conti,\sigma}}{\beta_{\conti,\sigma}+2}\Big(-r_\ellip+2\alpha_\sigma\min\Big\{\delta,1-r_{\growth,b},\frac{1-\ep}{2}-r_{\growth,\sigma}\Big\}-(1+\ep)\delta\Big)\right\}.
\end{align*}
In order to compute the above term, we use the following elementary fact; for any $\kappa_0,\kappa_1,\kappa_2>0$ and $\theta_1,\theta_2\in\bR$,
\begin{equation}\label{eq_elementary-optimization}
	\max_{x\in\bR}\min\Big\{\kappa_0x,-\kappa_1x+\theta_1,-\kappa_2x+\theta_2\Big\}=\min\left\{\frac{\kappa_0}{\kappa_0+\kappa_1}\theta_1,\frac{\kappa_0}{\kappa_0+\kappa_2}\theta_2\right\}.
\end{equation}
From the above fact, we get
\begin{align*}
	\gamma^{(1)}_\ep(\delta,r_{\growth,b},r_{\growth,\sigma})
	&=\min\left\{\frac{1}{1+\beta_\ellip^{-1}+\beta_{\conti,b}^{-1}}\alpha_b\min\Big\{\delta,1-r_{\growth,b},\frac{1-\ep}{2}-r_{\growth,\sigma}\Big\},\right.\\
	&\left.\hspace{2cm}\frac{1}{1+\beta_\ellip^{-1}+2\beta_{\conti,\sigma}^{-1}}\Big(2\alpha_\sigma\min\Big\{\delta,1-r_{\growth,b},\frac{1-\ep}{2}-r_{\growth,\sigma}\Big\}-(1+\ep)\delta\Big)\right\}.
\end{align*}
Inserting the above expression to \eqref{eq_optimization'}, we obtain
\begin{equation}\label{eq_optimization''}
	\gamma_\ep=\max_{\delta\in\bR}\min\Big\{\beta_\ep\delta,\gamma^{(4)}_\ep(\delta),\gamma^{(5)}_\ep(\delta)\Big\},
\end{equation}
where
\begin{equation}\label{eq_beta-ep}
	\beta_\ep:=\min\left\{\beta_0,\frac{\alpha_b}{1+\beta_\ellip^{-1}+\beta_{\conti,b}^{-1}},\frac{2\alpha_\sigma-1-\ep}{1+\beta_\ellip^{-1}+2\beta_{\conti,\sigma}^{-1}}\right\}>0
\end{equation}
and
\begin{align*}
	&\gamma^{(4)}_\ep(\delta):=\max_{r_{\growth,b}\in\bR}\min\left\{\beta_{\growth,b}r_{\growth,b},\frac{\alpha_b(1-r_{\growth,b})}{1+\beta_\ellip^{-1}+\beta_{\conti,b}^{-1}},\frac{2\alpha_\sigma(1-r_{\growth,b})-(1+\ep)\delta}{1+\beta_\ellip^{-1}+2\beta_{\conti,\sigma}^{-1}}\right\},\\
	&\gamma^{(5)}_\ep(\delta):=\max_{r_{\growth,\sigma}\in\bR}\min\left\{\beta_{\growth,\sigma}r_{\growth,\sigma},\frac{\alpha_b\Big(\frac{1-\ep}{2}-r_{\growth,\sigma}\Big)}{1+\beta_\ellip^{-1}+\beta_{\conti,b}^{-1}},\frac{2\alpha_\sigma\Big(\frac{1-\ep}{2}-r_{\growth,\sigma}\Big)-(1+\ep)\delta}{1+\beta_\ellip^{-1}+2\beta_{\conti,\sigma}^{-1}}\right\}.
\end{align*}
Again by using the elementary calculus \eqref{eq_elementary-optimization}, we have
\begin{align*}
	&\gamma^{(4)}_\ep(\delta)=\min\left\{\frac{\alpha_b}{1+\beta_\ellip^{-1}+\beta_{\conti,b}^{-1}+\beta_{\growth,b}^{-1}\alpha_b},\frac{2\alpha_\sigma-(1+\ep)\delta}{1+\beta_\ellip^{-1}+2\beta_{\conti,\sigma}^{-1}+2\beta_{\growth,b}^{-1}\alpha_\sigma}\right\},\\
	&\gamma^{(5)}_\ep(\delta)=\min\left\{\frac{1-\ep}{2}\cdot\frac{\alpha_b}{1+\beta_\ellip^{-1}+\beta_{\conti,b}^{-1}+\beta_{\growth,\sigma}^{-1}\alpha_b},\frac{(1-\ep)\alpha_\sigma-(1+\ep)\delta}{1+\beta_\ellip^{-1}+2\beta_{\conti,\sigma}^{-1}+2\beta_{\growth,\sigma}^{-1}\alpha_\sigma}\right\}.
\end{align*}
Inserting the above expressions to \eqref{eq_optimization''}, we have
\begin{equation}\label{eq_optimization'''}
	\gamma_\ep=\min\left\{\frac{\alpha_b}{1+\beta_\ellip^{-1}+\beta_{\conti,b}^{-1}+\beta_{\growth,b}^{-1}\alpha_b},\frac{1-\ep}{2}\cdot\frac{\alpha_b}{1+\beta_\ellip^{-1}+\beta_{\conti,b}^{-1}+\beta_{\growth,\sigma}^{-1}\alpha_b},\gamma^{(6)}_\ep\right\},
\end{equation}
where
\begin{equation*}
	\gamma^{(6)}_\ep:=\max_{\delta\in\bR}\min\left\{\beta_\ep\delta,\frac{2\alpha_\sigma-(1+\ep)\delta}{1+\beta_\ellip^{-1}+2\beta_{\conti,\sigma}^{-1}+2\beta_{\growth,b}^{-1}\alpha_\sigma},\frac{(1-\ep)\alpha_\sigma-(1+\ep)\delta}{1+\beta_\ellip^{-1}+2\beta_{\conti,\sigma}^{-1}+2\beta_{\growth,\sigma}^{-1}\alpha_\sigma}\right\}.
\end{equation*}
Using the elementary calculus \eqref{eq_elementary-optimization}, we have
\begin{align*}
	\gamma^{(6)}_\ep&=\min\left\{\frac{2\alpha_\sigma}{1+\beta_\ellip^{-1}+2\beta_{\conti,\sigma}^{-1}+2\beta_{\growth,b}^{-1}\alpha_\sigma+(1+\ep)\beta_\ep^{-1}},
	\frac{(1-\ep)\alpha_\sigma}{1+\beta_\ellip^{-1}+2\beta_{\conti,\sigma}^{-1}+2\beta_{\growth,\sigma}^{-1}\alpha_\sigma+(1+\ep)\beta_\ep^{-1}}\right\}.
\end{align*}
Inserting the above expression to \eqref{eq_optimization'''}, we obtain
\begin{equation}\label{eq_gamma-ep}
\begin{split}
	\gamma_\ep&=\min\left\{\frac{\alpha_b}{1+\beta_\ellip^{-1}+\beta_{\conti,b}^{-1}+\beta_{\growth,b}^{-1}\alpha_b},\frac{1-\ep}{2}\cdot\frac{\alpha_b}{1+\beta_\ellip^{-1}+\beta_{\conti,b}^{-1}+\beta_{\growth,\sigma}^{-1}\alpha_b},\right.\\
	&\left.\hspace{1.5cm}\frac{2\alpha_\sigma}{1+\beta_\ellip^{-1}+2\beta_{\conti,\sigma}^{-1}+2\beta_{\growth,b}^{-1}\alpha_\sigma+(1+\ep)\beta_\ep^{-1}},
	\frac{(1-\ep)\alpha_\sigma}{1+\beta_\ellip^{-1}+2\beta_{\conti,\sigma}^{-1}+2\beta_{\growth,\sigma}^{-1}\alpha_\sigma+(1+\ep)\beta_\ep^{-1}}\right\}.
\end{split}
\end{equation}

To summarize the above arguments, the estimate \eqref{eq_error-optimal-ep} holds for any $\ep\in(0,2\alpha_\sigma-1)$, $T\in(0,\infty)$ and $\pi\in\Pi_T$, where the exponent $\gamma_\ep>0$ is defined as the maximal value of the optimization problem \eqref{eq_optimization} and given by the expression \eqref{eq_gamma-ep}, and the number $\beta_\ep>0$ is given by \eqref{eq_beta-ep}. Also, notice that $\beta_\ep$ and $\gamma_\ep$ are monotonically decreasing with respect to $\ep\in(0,2\alpha_\sigma-1)$ (that is, monotonically increasing as $\ep\downarrow0$), and
\begin{equation*}
	\sup_{\ep\in(0,2\alpha_\sigma-1)}\beta_\ep=\lim_{\ep\downarrow0}\beta_\ep=\beta_*,\ \ \sup_{\ep\in(0,2\alpha_\sigma-1)}\gamma_\ep=\lim_{\ep\downarrow0}\gamma_\ep=\gamma_*,
\end{equation*}
where $\beta_*$ and $\gamma_*$ are given by \eqref{eq_beta*} and \eqref{eq_gamma*}, respectively. Moreover, the convergences above are uniform with respect to $\vec{\beta}=(\beta_0,\beta_{\growth,b},\beta_{\growth,\sigma},\beta_{\conti,b},\beta_{\conti,\sigma},\beta_\ellip)$. Therefore, for any $\gamma\in(0,\gamma_*)$, we can take $\ep\in(0,2\alpha_\sigma-1)$, which depends only on $\alpha_b$, $\alpha_\sigma$ and $\gamma_*-\gamma$, such that $\gamma<\gamma_\ep\leq\gamma_*$. Hence, defining $\cK_{T,\gamma}:=\widehat{\cK}_{T,\ep}$ for such an $\ep$, we obtain the conclusion of (iii). This completes the proof.
\end{proof}


\section{Applications to examples}\label{section_example}

In this section, we apply our general results (\cref{theo_main} and \cref{cor_rate-polynomial}) to some concrete examples of Markovian and non-Markovian models appearing in physics, chemistry, mathematical finance, economics, population biology, and so on, and provide precise weak error estimates for them.
To be specific, we pick up the following ten examples which have different features and difficulties from each other:

\begin{itemize}
\item
\cref{subsec_Lorenz}: Stochastic Lorenz model
\item
\cref{subsec_Brusselator}: Stochastic Brusselator in the well-stirred case
\item
\cref{subsec_AitSahalia}: A\"it--Sahalia-type interest rate model with delay
\item
\cref{subsec_neoclassical}: Stochastic delay differential neoclassical growth model
\item
\cref{subsec_ROU}: Reflected Ornstein--Uhlenbeck process
\item
\cref{subsec_Duffing}: Stochastic Duffing--van der Pol oscillator
\item
\cref{subsec_WF}: Wright--Fisher diffusion with seed bank
\item
\cref{subsec_polydiff}: Multi-dimensional polynomial diffusion
\item
\cref{subsec_3/2}: Volatility process in the $3/2$-stochastic volatility model
\item
\cref{subsec_Dyson}: Dyson's Brownian motion
\end{itemize}
For the first six examples, we can use \cref{cor_rate-polynomial} (i) directly since, as we will see below, all of them satisfy the required moment condition \eqref{eq_moment}. The last four examples are more complicated; by means of a careful moment estimate related to the ``rare event'' in \eqref{eq_rare-event}, we apply \cref{theo_main} (iii) and obtain a weak convergence rate depending on the parameters for each model. In all examples in this section, it is the first to obtain convergence rates for functional type weak approximations by means of the linear interpolation of the (standard) Euler--Maruyama scheme.


\subsection{Stochastic Lorenz model}\label{subsec_Lorenz}

The Lorenz model is one of the most famous nonlinear models of chaos and has been extensively studied in a wide range of areas since the original study by Lorenz \cite{Lo63} on a simplified mathematical model of atmospheric convection.
Here, as a special case of the model considered in \cite{Ke96}, we consider the stochastic Lorenz model with additive noise, which is described by the following three dimensional coupled nonlinear system:
\begin{equation}\label{eq_Lorenz}
\begin{split}
	&\diff X_1(t)=\kappa_1\Big\{X_2(t)-X_1(t)\Big\}\,\diff t+\theta_1\,\diff W_1(t),\\
	&\diff X_2(t)=\Big\{\kappa_2X_1(t)-X_2(t)-X_1(t)X_3(t)\Big\}\,\diff t+\theta_2\,\diff W_2(t),\\
	&\diff X_3(t)=\Big\{X_1(t)X_2(t)-\kappa_3X_3(t)\Big\}\,\diff t+\theta_3\,\diff W_3(t),\ \ t\in[0,\infty),
\end{split}
\end{equation}
where $\kappa_1,\kappa_2,\kappa_3,\theta_1,\theta_2,\theta_3\in\bR$ are given parameters. The above system can be seen as a Markovian SDE \eqref{eq_Markov} with $d=n=3$, $D=\bR^3$ and coefficients $\bar{b}:\bR^3\to\bR^3$ and $\bar{\sigma}:\bR^3\to\bR^{3\times 3}$ defined by
\begin{equation*}
	\bar{b}(\bar{x})=
	\begin{pmatrix}
	\kappa_1(\bar{x}_2-\bar{x}_1)\\
	\kappa_2\bar{x}_1-\bar{x}_2-\bar{x}_1\bar{x}_3\\
	\bar{x}_1\bar{x}_2-\kappa_3\bar{x}_3
	\end{pmatrix},
	\ \ 
	\bar{\sigma}(\bar{x})=
	\begin{pmatrix}
	\theta_1&0&0\\
	0&\theta_2&0\\
	0&0&\theta_3
	\end{pmatrix},
\end{equation*}
for $\bar{x}=(\bar{x}_1,\bar{x}_2,\bar{x}_3)^\top\in\bR^3$. It is known that the system \eqref{eq_Lorenz} admits a unique strong solution $X=(X_1,X_2,X_3)^\top$ for any initial condition $X(0)=\xi\in\bR^3$ (see \cite[Theorem 4.4]{Ke96}).

In view of the numerical study for the stochastic Lorenz model \eqref{eq_Lorenz}, there is a difficulty that the drift coefficient $\bar{b}$ is super-linearly growing. This indicates that the standard Euler--Maruyama scheme corresponding to \eqref{eq_Lorenz} might diverge in the $L^p$ sense for any $p\in(0,\infty)$ (see \cite{HuJeKl11} and \cref{subsubsec_sharp}).
In order to overcome this difficulty, Hutzenthaler and Jentzen \cite{HuJe15} study an appropriately modified Euler--Maruyama scheme called increment-tamed Euler--Maruyama scheme and show its converge to the solution $X=(X_1,X_2,X_{3})^{\top}$ of the Markovian SDE \eqref{eq_Lorenz} at each fixed time $T\in(0,\infty)$ in the $L^{p}$ sense for any $p\in(0,\infty)$.
However, they do not obtain any convergence rate for it.
Compared with \cite{HuJe15}, we consider functional type weak convergence of the linear interpolation of the standard Euler--Maruyama scheme \eqref{eq_EM-Markov} and give a convergence rate in terms of the L\'evy--Prokhorov metric.

By the same arguments as in \cref{subsubsec_Markov}, we see that the system \eqref{eq_Lorenz} satisfies conditions (G'), (C') (with $\alpha_b=\alpha_\sigma=1$) and (E') in \cref{cor_rate-polynomial} under the non-degeneracy condition $\theta_1\theta_2\theta_3\neq0$. Furthermore, by \cite[Corollary 4.5 (ii)]{Ke96}, the solution $X$ with fixed initial condition $X(0)=\xi\in\bR^3$ satisfies $\bE\big[\|X_T\|_\infty^p\big]<\infty$ for any $T\in(0,\infty)$ and $p\in[1,\infty)$ (which follows also from \cref{appendix_lemm_LG} (i)). Therefore, the moment condition \eqref{eq_moment} in \cref{cor_rate-polynomial} (i) holds, and we immediately obtain the following result.


\begin{theo}
Fix $\kappa_1,\kappa_2,\kappa_3,\theta_1,\theta_2,\theta_3\in\bR$, and assume that $\theta_1\theta_2\theta_3\neq0$. Let $\weaksol$ be the weak solution of the SDE \eqref{eq_Lorenz} on $D=\bR^3$ with initial condition $X(0)=\xi\in\bR^3$.
For each $T\in(0,\infty)$ and $\pi\in\Pi_T$, let $X^\pi$ be the Euler--Maruyama scheme defined on $(\Omega^\pi,\cF^\pi,\bP^\pi)$ and given by \eqref{eq_EM-Markov} with initial condition $X^\pi(t_0)=\xi$.
Then, for any $\gamma\in(0,\frac{1}{2})$, there exists a constant $C_\gamma\in(0,\infty)$ such that
\begin{equation*}
	d_\LP\big(\Law_\bP(X_T),\Law_{\bP^\pi}(\poly[X^\pi])\big)\leq C_\gamma|\pi|^\gamma
\end{equation*}
for any $\pi\in\Pi_{T}$.
\end{theo}


\subsection{Stochastic Brusselator in the well-stirred case}\label{subsec_Brusselator}

The Brusselator introduced by Prigogine and Lefever \cite{PrLe68} is one of theoretical models to describe the evolution of the concentration of the reactants in chemistry. In the well-stirred case, Dawson \cite{Da80} proposed a stochastic Brusselator (see also \cite{Sc86}), which is described by the following two dimensional coupled nonlinear system with multiplicative noise:
\begin{equation}\label{eq_Brusselator}
\begin{split}
	&\diff X_1(t)=\Big\{\kappa_1-(\kappa_2+1)X_1(t)+X_2(t)X_1(t)^2\Big\}\,\diff t+\theta_1(X_1(t))\,\diff W_1(t),\\
	&\diff X_2(t)=\Big\{\kappa_2X_1(t)-X_2(t)X_1(t)^2\Big\}\,\diff t+\theta_2(X_2(t))\,\diff W_2(t),\ \ t\in[0,\infty),
\end{split}
\end{equation}
where $\kappa_1,\kappa_2\in(0,\infty)$ are fixed constants, and $\theta_1,\theta_2:[0,\infty)\to\bR$ are globally Lipschitz continuous functions with $\theta_1(0)=\theta_2(0)=0$ and $\sup_{x\in[0,\infty)}|\theta_2(x)|<\infty$. The above system can be seen as a Markovian SDE \eqref{eq_Markov} with $d=n=2$, $D=(0,\infty)^2$ and coefficients $\bar{b}:D\to\bR^2$ and $\bar{\sigma}:D\to\bR^{2\times 2}$ defined by
\begin{equation*}
	\bar{b}(\bar{x})=
	\begin{pmatrix}
	\kappa_1-(\kappa_2+1)\bar{x}_1+\bar{x}_2\bar{x}_1^2\\
	\kappa_2\bar{x}_1-\bar{x}_2\bar{x}_1^2
	\end{pmatrix},
	\ \ 
	\bar{\sigma}(\bar{x})=
	\begin{pmatrix}
	\theta_1(\bar{x}_1)&0\\
	0&\theta_2(\bar{x}_2)
	\end{pmatrix},
\end{equation*}
for $\bar{x}=(\bar{x}_1,\bar{x}_2)^\top\in D$. By \cite[Theorem 2.1 a)]{Sc86}, for each initial condition $X(0)=\xi\in[0,\infty)^2$, there exists a unique strong solution $X=(X_1,X_2)^\top$ to the SDE \eqref{eq_Brusselator} such that, a.s., $X_1(t)\geq0$ and $X_2(t)\geq0$ for any $t\in[0,\infty)$. Furthermore, \cref{lemm_Brusselator-moment} below shows that the solution never hits the boundary of $D=(0,\infty)^2$ if $\xi\in D$.

As in the case of the stochastic Lorenz model \eqref{eq_Lorenz}, due to the super-linear terms in the drift coefficient, the standard Euler--Maruyama scheme of the above SDE might diverge in the $L^p$ sense for any $p\in(0,\infty)$ (see \cite{HuJeKl11} and \cref{subsubsec_sharp}).
Hutzenthaler and Jentzen \cite{HuJe15} then consider the increment-tamed Euler--Maruyama scheme for \eqref{eq_Brusselator} and show its convergence to the true solution $X$ at each fixed time $T\in(0,\infty)$ in the $L^{p}$ sense for any $p\in(0,\infty)$.
However, they do not obtain any convergence rate for it. Compared with \cite{HuJe15}, we consider functional type weak convergence of the linear interpolation of the standard Euler--Maruyama scheme \eqref{eq_EM-Markov}.

We will apply \cref{cor_rate-polynomial} (i) to this setting. To do so, we need to check the moment condition \eqref{eq_moment}.
Notice that the coefficients of the SDE \eqref{eq_Brusselator} satisfy neither \eqref{appendix_eq_LG} nor \eqref{appendix_eq_LG-negative} in \cref{appendix_lemm_LG}. Nevertheless, the following lemma ensures that the desired moment estimate \eqref{eq_moment} holds in this setting.


\begin{lemm}\label{lemm_Brusselator-moment}
Let $\kappa_1,\kappa_2\in(0,\infty)$ be fixed constants, and let $\theta_1,\theta_2:[0,\infty)\to\bR$ be two globally Lipschitz continuous functions with $\theta_1(0)=\theta_2(0)=0$ and $\sup_{x\in[0,\infty)}|\theta_2(x)|<\infty$. Then, there exists a unique strong solution $X=(X_1,X_2)^\top$ of the SDE \eqref{eq_Brusselator} on $D=(0,\infty)^2$ with initial condition $X(0)=\xi\in D$. Furthermore, for any $T\in(0,\infty)$ and $p\in[1,\infty)$, it holds that
\begin{equation*}
	\bE\big[\|X_T\|_\infty^p+\|X^{-1}_T\|_\infty^p\big]<\infty,
\end{equation*}
where $X^{-1}(\cdot):=(X_1(\cdot)^{-1},X_2(\cdot)^{-1})^\top$.
\end{lemm}


\begin{proof}
By \cite[Theorem 2.1 a)]{Sc86}, there exists a unique strong solution $X$ to the SDE \eqref{eq_Brusselator} such that, a.s., $X_1(t)\geq0$ and $X_2(t)\geq0$ for any $t\in[0,\infty)$. Let $T\in(0,\infty)$ and $p\in[1,\infty)$ be fixed. In this proof, we denote by $C$ a positive constant depending only on $p,T,\kappa_1,\kappa_2$ and the Lipschitz constants of $\theta_1$ and $\theta_2$, which varies from line to line.

We first show that $\bE[\|X_T\|^p_\infty]<\infty$. To do so, it suffices to show that $\bE[\sup_{t\in[0,T]}(X_1(t)+X_2(t))^p]<\infty$. Observe that
\begin{align*}
	X_1(t)+X_2(t)
	&=\xi_1+\xi_2+\int^t_0\big\{\kappa_1-X_1(s)\big\}\,\diff s+\int^t_0\theta_1(X_1(s))\,\diff W_1(s)+\int^t_0\theta_2(X_2(s))\,\diff W_2(s)
\end{align*}
for any $t\in[0,T]$ a.s. For each $N\in\bN$, define $\tau_N:=\inf\{t\geq0\,|\,X_1(t)+X_2(t)\geq N\}$ and $a_N(t):=\bE[\sup_{s\in[0,t\wedge\tau_N]}(X_1(s)+X_2(s))^p]$ for $t\in[0,T]$. Noting that $\theta_1,\theta_2:[0,\infty)\to\bR$ are globally Lipschitz continuous, by using the Burkholder--Davis--Gundy inequality, we get
\begin{equation*}
	a_N(t)\leq C\left\{1+|\xi|^p+\int^t_0a_N(s)\,\diff s\right\}
\end{equation*}
for any $t\in[0,T]$. By Gronwall's inequality, we get $a_N(T)\leq C\{1+|\xi|^p\}$. Letting $N\to\infty$, Fatou's lemma yields $\bE[\sup_{t\in[0,T]}(X_1(t)+X_2(t))^p]\leq C(1+|\xi|^p)$, and hence $\bE[\|X_T\|^p_\infty]<\infty$.

Next, we show that $\bP(\text{$X_1(t)>0$ for any $t\in[0,T]$})=1$ and that
\begin{equation}\label{eq_Brusselator-estimate1}
	\bE\left[\sup_{t\in[0,T]}X_1(t)^{-p}\right]\leq C\big(1+\xi_1^{-p}\big).
\end{equation}
For each $N\in\bN$, define $\varsigma_{1,N}:=\inf\{t\geq0\,|\,X_1(t)\leq 1/N\}$. Applying It\^o's formula to $X_1(t\wedge\varsigma_{1,N})^{-p}$, we obtain
\begin{align*}
	X_1(t\wedge\varsigma_{1,N})^{-p}
	&=\xi_1^{-p}+p\int^{t\wedge\varsigma_{1,N}}_0F_p(X_1(s),X_2(s))\,\diff s-p\int^{t\wedge\varsigma_{1,N}}_0X_1(s)^{-p-1}\theta_1(X_1(s))\,\diff W_1(s)
\end{align*}
for any $t\in[0,T]$ a.s., where $F_p:(0,\infty)\times[0,\infty)\to\bR$ is defined by
\begin{equation*}
	F_p(x_1,x_2):=-\kappa_1x_1^{-p-1}+(\kappa_2+1)x_1^{-p}-x_2x_1^{-p+1}+\frac{p+1}{2}x_1^{-p-2}|\theta_1(x_1)|^2
\end{equation*}
for $(x_1,x_2)\in(0,\infty)\times[0,\infty)$. Since $\theta_1:[0,\infty)\to\bR$ is globally Lipschitz continuous and satisfies $\theta_1(0)=0$, we have $|\theta_1(x_1)|\leq Cx_1$ for any $x_1\in(0,\infty)$. From this and the positivity of the constant $\kappa_1$, we see that the function $F_p$ is bounded from above on $(0,\infty)\times[0,\infty)$. Thus, by the Burkholder--Davis--Gundy inequality and Young's inequality, we obtain
\begin{align*}
	\bE\left[\sup_{s\in[0,t\wedge\varsigma_{1,N}]}X_1(s)^{-p}\right]
	&\leq\xi_1^{-p}+C+C\bE\left[\left(\int^{t\wedge\varsigma_{1,N}}_0X_1(s)^{-2p-2}|\theta_1(X_1(s))|^2\,\diff s\right)^{1/2}\right]\\
	&\leq\xi_1^{-p}+C+C\bE\left[\left(\int^{t\wedge\varsigma_{1,N}}_0X_1(s)^{-2p}\,\diff s\right)^{1/2}\right]\\
	&\leq\xi_1^{-p}+C+C\bE\left[\left(\int^{t\wedge\varsigma_{1,N}}_0X_1(s)^{-p}\,\diff s\right)^{1/2}\left(\sup_{s\in[0,t\wedge\varsigma_{1,N}]}X_1(s)^{-p}\right)^{1/2}\right]\\
	&\leq\xi_1^{-p}+C+C\int^t_0\bE\left[\sup_{r\in[0,s\wedge\varsigma_{1,N}]}X_1(r)^{-p}\right]\,\diff s+\frac{1}{2}\bE\left[\sup_{s\in[0,t\wedge\varsigma_{1,N}]}X_1(s)^{-p}\right],
\end{align*}
and hence
\begin{equation*}
	\bE\left[\sup_{s\in[0,t\wedge\varsigma_{1,N}]}X_1(s)^{-p}\right]\leq C\left\{1+\xi_1^{-p}+\int^t_0\bE\left[\sup_{r\in[0,s\wedge\varsigma_{1,N}]}X_1(r)^{-p}\right]\,\diff s\right\}
\end{equation*}
for any $t\in[0,T]$. By Gronwall's inequality, we get
\begin{equation*}
	\bE\left[\sup_{s\in[0,T\wedge\varsigma_{1,N}]}X_1(s)^{-p}\right]\leq C\big(1+\xi_1^{-p}\big).
\end{equation*}
This implies that $\bP(\varsigma_{1,N}>T)\to1$ as $N\to\infty$,
and hence, $\bP$-a.s., $X_1(t)>0$ for any $t\in[0,T]$. Also, by letting $N\to\infty$ in the above estimate and using Fatou's lemma, we obtain \eqref{eq_Brusselator-estimate1}.

Lastly, we show that $\bP(\text{$X_2(t)>0$ for any $t\in[0,T]$})=1$ and that
\begin{equation}\label{eq_Brusselator-estimate2}
	\bE\left[\sup_{s\in[0,T]}X_2(s)^{-p}\right]\leq C\big(1+|\xi^{-1}|^p+|\xi|^{p+2}\big).
\end{equation}
Define $\varsigma_{2,N}:=\inf\{t\geq0\,|\,X_2(t)\leq1/N\}$ for each $N\in\bN$. Applying It\^o's formula to $X_2(t\wedge\varsigma_{2,N})^{-p}$, we have
\begin{align*}
	X_2(t\wedge\varsigma_{2,N})^{-p}
	&=\xi_2^{-p}+p\int^{t\wedge\varsigma_{2,N}}_0G_p(X_1(s),X_2(s))\,\diff s-p\int^{t\wedge\varsigma_{2,N}}_0X_2(s)^{-p-1}\theta_2(X_2(s))\,\diff W_2(s)
\end{align*}
for any $t\in[0,T]$ a.s., where $G_p:(0,\infty)\times(0,\infty)\to\bR$ is defined by
\begin{equation*}
	G_p(x_1,x_2):=-\kappa_2x_1x_2^{-p-1}+x_2^{-p}x_1^2+\frac{p+1}{2}x_2^{-p-2}|\theta_2(x_2)|^2
\end{equation*}
for $(x_1,x_2)\in(0,\infty)\times(0,\infty)$. Since $\kappa_2>0$ and $|\theta_2(x_2)|\leq C x_2$ for any $x_2\in(0,\infty)$, by using Young's inequality, we see that $G_p(x_1,x_2)\leq C(x_1^{p+2}+x_1^{-p})$ for any $(x_1,x_2)\in(0,\infty)\times(0,\infty)$. In the above arguments, we have shown that
\begin{equation*}
	\bE\left[\sup_{t\in[0,T]}X_1(t)^{p+2}\right]\leq C(1+|\xi|^{p+2})\ \ \text{and}\ \ \bE\left[\sup_{t\in[0,T]}X_1(t)^{-p}\right]\leq C(1+\xi_1^{-p}).
\end{equation*}
Therefore, by the same argument using the Burkholder--Davis--Gundy inequality as above, we get
\begin{equation*}
	\bE\left[\sup_{s\in[0,T\wedge\varsigma_{2,N}]}X_2(s)^{-p}\right]\leq C\big(1+|\xi^{-1}|^p+|\xi|^{p+2}\big).
\end{equation*}
This implies that $\bP(\varsigma_{2,N}>T)\to1$ as $N\to\infty$,
and hence, $\bP$-a.s., $X_2(t)>0$ for any $t\in[0,T]$. Also, by letting $N\to\infty$ in the above estimate and using Fatou's lemma, we obtain \eqref{eq_Brusselator-estimate2}. This completes the proof.
\end{proof}

The following result is an immediate consequence of \cref{lemm_Brusselator-moment} and \cref{cor_rate-polynomial} (i). Here, we require that the functions $\theta_1,\theta_2:[0,\infty)\to\bR$ are non-degenerate on $(0,\infty)$ in a suitable sense as specified below.


\begin{theo}\label{theo_Brusselator}
Let $\kappa_1,\kappa_2\in(0,\infty)$ be fixed constants, and let $\theta_1,\theta_2:[0,\infty)\to\bR$ be two globally Lipschitz continuous functions with $\theta_1(0)=\theta_2(0)=0$ and $\sup_{x\in[0,\infty)}|\theta_2(x)|<\infty$. Assume that there exist constants $C,r\in(0,\infty)$ such that
\begin{equation*}
	|\theta_1(x)|\wedge|\theta_2(x)|\geq\frac{1}{C(1+x^r+x^{-r})}
\end{equation*}
for any $x\in(0,\infty)$. Let $\weaksol$ be the weak solution of the SDE \eqref{eq_Brusselator} on $D=(0,\infty)^2$ with initial condition $X(0)=\xi\in D$.
For each $T\in(0,\infty)$ and $\pi\in\Pi_T$, let $X^\pi$ be the Euler--Maruyama scheme defined on $(\Omega^\pi,\cF^\pi,\bP^\pi)$ and given by \eqref{eq_EM-Markov} with initial condition $X^\pi(t_0)=\xi$.
Then, for any $\gamma\in(0,\frac{1}{2})$, there exists a constant $C_\gamma\in(0,\infty)$ such that
\begin{equation*}
	d_\LP\big(\Law_\bP(X_T),\Law_{\bP^\pi}(\poly[X^\pi])\big)\leq C_\gamma|\pi|^\gamma
\end{equation*}
for any $\pi\in\Pi_{T}$.
\end{theo}


\subsection{A\"it--Sahalia-type interest rate model with delay}\label{subsec_AitSahalia}

Consider the following A\"it--Sahalia-type interest rate model with delay:
\begin{equation}\label{eq_AitSahalia}
	\begin{dcases}
	\diff X(t)=\big\{\kappa_{-1}X(t)^{-1}-\kappa_0+\kappa_1X(t)-\kappa_2X(t)^{\rho}\big\}\,\diff t+\eta(X(t-\tau))X(t)^\theta\,\diff W(t),\ t\in[0,\infty),\\
	X(t)=\xi(t),\ \ t\in[-\tau,0],
	\end{dcases}
\end{equation}
where $\xi:[-\tau,0]\to(0,\infty)$ is a given continuous function, $\kappa_{-1},\kappa_{0},\kappa_{1},\kappa_{2} \in (0,\infty)$ and $\rho,\theta\in(1,\infty)$ are parameters, $\eta:(0,\infty)\to(0,\infty)$ is a measurable function, and $\tau\in(0,\infty)$ is a constant which represents the length of delay. The term $\eta(X(t-\tau))$ represents the past-level-dependent volatility function. Observe that the equation \eqref{eq_AitSahalia} has strongly nonlinear terms $X(t)^{-1}$, $X(t)^\rho$ and $X(t)^\theta$ in both drift and diffusion coefficients as well as the delay term $\eta(X(t-\tau))$ in the diffusion coefficient. The equation \eqref{eq_AitSahalia} can be seen as an SDDE \eqref{eq_delay} on $D=(0,\infty)$ with $d=n=1$ and coefficients $\bar{b},\bar{\sigma}:D \times D \to \bR$ defined by
\begin{equation*}
	\bar{b}(x,y)=
		\kappa_{-1}x^{-1}-\kappa_0+\kappa_1x-\kappa_2x^{\rho},
		\ \ 
		\bar{\sigma}(x,y)=
		\eta(y)x^{\theta},\ \ x,y\in D.
\end{equation*}
Hence, the equation \eqref{eq_AitSahalia} fits into our framework; see \cref{subsubsec_delay}.

The standard A\"it--Sahalia model (that is, \eqref{eq_AitSahalia} with $\eta$ being constant) has been applied in the field of mathematical finance \cite{Ai96} as an interest rate model.
Coffie and Mao \cite{CoMa21} study the A\"it--Sahalia model with delay \eqref{eq_AitSahalia} and prove that, under the assumptions that $\eta$ is bounded and measurable and that $1+\rho>2\theta$, the equation has a unique strong solution $X$ such that, a.s., $X(t)\in D$ for any $t\in[0,\infty)$ (see \cite[Theorem 2.3]{CoMa21}).

Due to the strong nonlinear terms as well as the delay term, existing results on the standard Euler--Maruyama scheme cannot be applied to \eqref{eq_AitSahalia}. In order to overcome this difficulty, in the standard A\"it--Sahalia model without delay, Szpruch, Mao, Higham and Pan \cite{SzMaHiPa10} and Neuenkirch and Szpruch \cite{NeSz14} consider backward(-type) Euler--Maruyama schemes, and Chassagneux, Jacquier and Mihaylov \cite{ChJaMi16} consider a truncated Euler--Maruyama scheme. In the case with delay, Coffie and Mao \cite{CoMa21} consider a truncated Euler--Maruyama scheme.

We will apply \cref{cor_rate-polynomial} (i) to this setting.
To do so, we need to check the moment condition \eqref{eq_moment}.


\begin{lemm}\label{lemm_AitSahalia-moment}
Let $\kappa_{-1},\kappa_0,\kappa_1,\kappa_2\in(0,\infty)$, $\rho,\theta\in(1,\infty)$ and $\tau\in(0,\infty)$ be fixed constants, let $\eta:(0,\infty)\to(0,\infty)$ be a bounded measurable function, and let $\xi:[-\tau,0]\to(0,\infty)$ be a given continuous function. Assume that $1+\rho>2\theta$. Then, there exists a unique strong solution $X$ of the SDDE \eqref{eq_AitSahalia} on $D=(0,\infty)$. Furthermore, for any $T\in(0,\infty)$ and $p\in[1,\infty)$, it holds that
\begin{equation*}
	\bE\big[\|X_T\|_\infty^p+\|X^{-1}_T\|_\infty^p\big]<\infty.
\end{equation*}
\end{lemm}


\begin{proof}
By \cite[Theorem 2.3]{CoMa21}, there exists a unique strong solution $X$ of the SDDE \eqref{eq_AitSahalia} on $D=(0,\infty)$. Furthermore, \cite[Lemmas 2.4 and 2.5]{CoMa21} show that
\begin{equation}\label{eq_AitSahalia-estimate}
\begin{split}
	&\bE\left[\sup_{t\in[0,T]}X(t)^p\right]+\sup_{t\in[0,\infty)}\bE\big[X(t)^p\big]+\sup_{t\in[0,\infty)}\bE\big[X(t)^{-p}\big]<\infty\ \text{for any $T\in(0,\infty)$ and $p\in[1,\infty)$}.
\end{split}
\end{equation}
Thus, it remains to show that
\begin{equation}\label{eq_AitSahalia-estimate'}
	\bE\left[\sup_{t\in[0,T]}X(t)^{-p}\right]<\infty
\end{equation}
for any $T\in(0,\infty)$ and $p\in[1,\infty)$.

Let $T\in(0,\infty)$ and $p\in[1,\infty)$ be fixed. Then, by It\^{o}'s formula, we get
\begin{equation}\label{eq_AitSahalia-Ito}
	X(t)^{-p}=\xi(0)^{-p}+p\int^t_0F_p(X(s),X(s-\tau))\,\diff s-p\int^t_0\eta(X(s-\tau))X(s)^{-p-1+\theta}\,\diff W_s
\end{equation}
for $t\in[0,T]$, where
\begin{equation*}
	F_p(x,y):=-\kappa_{-1}x^{-p-2}+\kappa_0x^{-p-1}-\kappa_1x^{-p}+\kappa_2x^{-p-1+\rho}+\frac{p+1}{2}\eta(y)^2x^{-p-2+2\theta}
\end{equation*}
for $(x,y)\in(0,\infty)\times(0,\infty)$. Applying the Burkholder--Davis--Gundy inequality to \eqref{eq_AitSahalia-Ito} and noting that $\eta$ is bounded, we obtain \eqref{eq_AitSahalia-estimate'} from the estimate \eqref{eq_AitSahalia-estimate}. This completes the proof.
\end{proof}

By \cref{lemm_AitSahalia-moment} and \cref{cor_rate-polynomial} (i), we immediately obtain the following result.


\begin{theo}\label{theo_AitSahalia}
Let $\kappa_{-1},\kappa_0,\kappa_1,\kappa_2\in(0,\infty)$, $\rho,\theta\in(1,\infty)$ and $\tau\in(0,\infty)$ be fixed constants, let $\eta:(0,\infty)\to(0,\infty)$ be a bounded measurable function, and let $\xi:[-\tau,0]\to(0,\infty)$ be a given continuous function. Assume that
$1+\rho>2\theta$. Furthermore, assume that there exist constants $C,r\in(0,\infty)$ and $\alpha_\eta\in(\frac{1}{2},1]$ such that
\begin{equation*}
	|\xi(t)-\xi(s)|\leq C(t-s)^{1/2}
\end{equation*}
for any $-\tau\leq s\leq t\leq0$, and
\begin{equation*}
	|\eta(x)-\eta(y)|\leq C\left(1+x^r+x^{-r}+y^r+y^{-r}\right)|x-y|^{\alpha_\eta},\ \ \eta(x)\geq\frac{1}{C\left(1+x^r+x^{-r}\right)},
\end{equation*}
for any $x,y\in(0,\infty)$. Let $\weaksol$ be the weak solution of the SDDE \eqref{eq_AitSahalia} on $D=(0,\infty)$.
For each $T\in(0,\infty)$ and $\pi\in\Pi_T$, let $X^\pi$ be the Euler--Maruyama scheme defined on $(\Omega^\pi,\cF^\pi,\bP^\pi)$ and given by \eqref{eq_EM} with initial condition $X^\pi(t_0)=\xi(0)$.
Then, for any $\gamma\in(0,\alpha_\eta-\frac{1}{2})$, there exists a constant $C_\gamma\in(0,\infty)$ such that
\begin{equation*}
	d_\LP\big(\Law_\bP(X_T),\Law_{\bP^\pi}(\poly[X^\pi])\big)\leq C_\gamma|\pi|^\gamma
\end{equation*}
for any $\pi\in\Pi_{T}$.
\end{theo}


\subsection{Stochastic delay differential neoclassical growth model}\label{subsec_neoclassical}

The examination of economic growth models is one of the most frequently discussed issues in mathematical economics. Here, we consider the following $n$-connected stochastic delay differential neoclassical growth model:
\begin{equation}\label{eq_neoclassical}
	\begin{dcases}
	\diff X_i(t)=\left\{-a_iX_i(t)+\sum_{j\neq i}b_{i,j}X_j(t)+c_iX_i(t-\tau_i)^{\gamma_i}e^{-\delta_iX_i(t-\tau_i)}\right\}\,\diff t
	+\theta_iX_i(t)\,\diff W_i(t),\ \ t\in[0,\infty),\\
	X_i(t)=\xi_i(t),\ \ t\in[-\tau_i,0],\ \ i\in\{1,\dots,n\},
	\end{dcases}
\end{equation}
where $a_i,b_{i,j},c_i,\gamma_i,\delta_i,\tau_i,\theta_i\in(0,\infty)$ are given parameters, and $\xi_i:[-\tau_i,0]\to(0,\infty)$ is a given continuous function for each $i\in\{1,\dots,n\}$. The stochastic delay differential neoclassical growth model \eqref{eq_neoclassical} is studied in \cite{WaCh19} for the case of $n=1$, in \cite{Sh19} for the case of $n=2$ and in \cite{AlKh22} for the case of $n=3$. In the model \eqref{eq_neoclassical}, $X_i(t)$ stands for the capital per labor at time $t$ in the patch $i$, $a_i$ is the sum of labor growth rate and capital depreciation rate multiplied by average saving rate, $b_{i,j}$ is the dispersal coefficient of the capital from patch $j$ to patch $i$, $\theta_i$ denotes noise intensity, and the term $c_iX_i(t-\tau_i)^{\gamma_i}e^{-\delta_iX_i(t-\tau_i)}$ describes the delayed reproduction function for patch $i$.
For more detailed economic interpretations, see for example \cite{AlKh22,Sh19,WaCh19}.
When $\gamma_i=1$, the model \eqref{eq_neoclassical} reduces to the famous Nicholson's blowflies model \cite{YiLi19}.

The stochastic delay differential neoclassical growth model \eqref{eq_neoclassical} can be seen as an SFDE \eqref{eq_SFDE} with data $(D,\mu_0,b,\sigma)=((0,\infty)^n,\delta_{\xi(0)},b,\sigma)$, where $b:\Lambda[\{\xi(0)\},(0,\infty)^n]\to\bR^n$ and $\sigma:\Lambda[\{\xi(0)\},(0,\infty)^n]\to\bR^{n\times n}$ are given by
\begin{align*}
	&b_i(t,x)=-a_ix_i(t)+\sum_{j\neq i}b_{i,j}x_j(t)+\bar{b}_i(\xi_i(t-\tau_i))\1_{[0,\tau_i)}(t)+\bar{b}_i(x_i(t-\tau_i))\1_{[\tau_i,\infty)}(t),\ \ i\in\{1,\dots,n\},\\
	&\sigma_{i,j}(t,x)=\1_{i=j}\theta_ix_i(t),\ \ i,j\in\{1,\dots,n\},
\end{align*}
for $(t,x)\in\Lambda[\{\xi(0)\},(0,\infty)^n]$, where $\bar{b}_i:(0,\infty)\to(0,\infty)$ is defined by
\begin{equation*}
	\bar{b}_{i}(\bar{x})=c_i \bar{x}^{\gamma_i}e^{-\delta_i \bar{x}}
,\ \ \bar{x} \in(0,\infty).
\end{equation*}
To the best of our knowledge, there is no existing work on numerical approximations for the stochastic delay differential neoclassical growth model \eqref{eq_neoclassical}.

Notice that \eqref{eq_neoclassical} has multiple delay terms. Nevertheless, similar arguments as in \cref{subsubsec_delay} are applicable to this setting. In particular, we can easily show that conditions (G'), (C') and (E') in \cref{cor_rate-polynomial} hold with $D=(0,\infty)^n$. We will apply \cref{cor_rate-polynomial} (i) to this setting.
To do so, we need to check the moment condition \eqref{eq_moment}.


\begin{lemm}\label{lemm_neoclassical-moment}
Let $n\in\bN$ and $a_i,b_{i,j},c_i,\gamma_i,\delta_i,\tau_i,\theta_i\in(0,\infty)$ with $i,j\in\{1,\dots,n\}$ be fixed, and let $\xi_i:[-\tau_i,0]\to(0,\infty)$ be a continuous function for each $i\in\{1,\dots,n\}$. Then, there exists a unique strong solution $X=(X_1,\dots,X_n)^\top$ of the SDDE \eqref{eq_neoclassical} on $D=(0,\infty)^n$. Furthermore, for any $p\in[2,\infty)$, there exists a constant $C_p\in(0,\infty)$, which depends only on $a_i,b_{i,j},c_i,\gamma_i,\delta_i,\theta_i$ with $i,j\in\{1,\dots,n\}$ as well as $p$, such that
\begin{equation}\label{eq_neoclassical-estimate}
	\bE\big[\|X_T\|_\infty^p\big]\leq C_pe^{C_pT}\big(1+|\xi(0)|^p\big)\ \ \text{and}\ \ \bE\big[\|X^{-1}_T\|_\infty^p\big]\leq C_pe^{C_pT}\big(1+|\xi(0)^{-1}|^p\big).
\end{equation}
for any $T\in(0,\infty)$.
\end{lemm}


\begin{proof}
Since the coefficients of the SDDE \eqref{eq_neoclassical} are locally Lipschitz continuous on $D=(0,\infty)^n$, by a standard truncation argument, we see that there exists a pathwise unique maximal local solution $X=(X_1,\dots,X_n)^\top$; see \cite[Theorem 3.1]{Ma07}.

For each $N\in\bN$, define the stopping time
\begin{equation*}
	\tau_N:=\inf\left\{t\geq0\relmiddle|X(t)\notin\left(\frac{1}{N},N\right)^n\right\}
\end{equation*}
and the progressively measurable set
\begin{equation*}
	\Lambda_N:=\Lambda\left[\{\xi(0)\},\left(\frac{1}{N},N\right)^n\right].
\end{equation*}
Notice that $(X_{\tau_N},W,\Omega,\cF,\bF,\bP)$ is a weak solution of the SFDE with data $((0,\infty)^n,\delta_{\xi(0)},\1_{\Lambda_N}b,\1_{\Lambda_N}\sigma)$. Noting that the function $y\mapsto y^\gamma e^{-\delta y}$ is bounded on $(0,\infty)$ for any constants $\gamma,\delta>0$, straightforward calculations show that the coefficients $\1_{\Lambda_N}b$ and $\1_{\Lambda_N}\sigma$ satisfy the condition in \cref{appendix_lemm_LG} (i). Also, thanks to $b_{i,j},c_i>0$ for each $i,j\in\{1,\dots,n\}$, we see that the condition in \cref{appendix_lemm_LG} (ii) holds as well. Furthermore, the corresponding constant $\widehat{C}$ appearing in \cref{appendix_lemm_LG} depends only on $a_i,b_{i,j},c_i,\gamma_i,\delta_i,\theta_i$ with $i,j\in\{1,\dots,n\}$ but not on $N$. Hence, \cref{appendix_lemm_LG} (i) and (ii) yield
\begin{equation}\label{eq_neoclassical-estimate'}
	\bE\big[\|X_{T\wedge\tau_N}\|_\infty^p\big]\leq C_pe^{C_pT}\big(1+|\xi(0)|^p\big)\ \ \text{and}\ \ \bE\big[\|X^{-1}_{T\wedge\tau_N}\|_\infty^p\big]\leq C_pe^{C_pT}\big(1+|\xi(0)^{-1}|^p\big)
\end{equation}
for any $T\in(0,\infty)$, $p\in[2,\infty)$ and $N\in\bN$. Here, the positive constant $C_p$ depends only on $a_i,b_{i,j},c_i,\gamma_i,\delta_i,\theta_i$ with $i,j\in\{1,\dots,n\}$ and $p$. The estimate \eqref{eq_neoclassical-estimate'} shows that $\lim_{N\to\infty}\tau_N=\infty$ a.s. Moreover, letting $N\to\infty$ in \eqref{eq_neoclassical-estimate'} and using Fatou's lemma, we get the desired estimate \eqref{eq_neoclassical-estimate}. This completes the proof.
\end{proof}

By \cref{lemm_neoclassical-moment} and \cref{cor_rate-polynomial} (i), we immediately obtain the following result.


\begin{theo}\label{theo_neoclassical}
Let $n\in\bN$ and $a_i,b_{i,j},c_i,\gamma_i,\delta_i,\tau_i,\theta_i\in(0,\infty)$ with $i,j\in\{1,\dots,n\}$ be fixed. For each $i\in\{1,\dots,n\}$, let $\xi_i:[-\tau_i,0]\to(0,\infty)$ be a continuous function such that
\begin{equation*}
	|\xi_i(t)-\xi_i(s)|\leq C(t-s)^{1/2}
\end{equation*}
for any $-\tau_i\leq s\leq t\leq 0$ for some constant $C\in(0,\infty)$. Let $\weaksol$ be the weak solution of the SDDE \eqref{eq_neoclassical} on $D=(0,\infty)^n$.
For each $T\in(0,\infty)$ and $\pi\in\Pi_T$, let $X^\pi$ be the Euler--Maruyama scheme defined on $(\Omega^\pi,\cF^\pi,\bP^\pi)$ and given by \eqref{eq_EM} with initial condition $X^\pi(t_0)=\xi(0)$.
Then, for any $\gamma\in(0,\frac{1}{2})$, there exists a constant $C_\gamma\in(0,\infty)$ such that
\begin{equation*}
	d_\LP\big(\Law_\bP(X_T),\Law_{\bP^\pi}(\poly[X^\pi])\big)\leq C_\gamma|\pi|^\gamma
\end{equation*}
for any $\pi\in\Pi_{T}$.
\end{theo}


\subsection{Reflected Ornstein--Uhlenbeck process}\label{subsec_ROU}

The reflected Ornstein--Uhlenbeck process is a mean-reversion model having a reflection and plays a crucial role in the field of queueing models with reneging or balking.
Here, as in \cite{RiSa87,WaGl03}, we consider the reflected Ornstein--Uhlenbeck process on $[0,\infty)$ which is defined as the solution of the following reflected SDE:
\begin{equation}\label{eq_ROU}
\diff \check{X}(t)
=
(\kappa_{1}-\kappa_{2}\check{X}(t))\,\diff t
+
\theta\,\diff W(t)
+
\diff \Phi_{\check{X}}(t)
,\ \ t\in[0,\infty),\ \ \check{X}(0)=\xi \in [0,\infty),
\end{equation}
where $\kappa_1\in \bR$ and $\kappa_2, \theta \in (0,\infty)$ are given parameters. By \cite[Theorem 3.1]{Zh94}, there exists a unique strong solution to the reflected SDE \eqref{eq_ROU}.
In the context of numerical approximations of reflected SDEs, L\'epingle \cite{Le95} studies the rate of convergence in the $L^{2}$-sup norm on $[0,T]$ by means of the Euler--Maruyama-type scheme (interpolated by the Brownian path).

As mentioned in Section \ref{subsubsec_reflected}, the reflected SDE \eqref{eq_ROU} can be transformed to an SFDE involving the Skorokhod map $\Gamma$ given by \eqref{eq_reflected-Skorokhod}. The corresponding SFDE is of the following form:
\begin{align}\label{eq_ROU2}
\diff X(t)
=
\left\{\kappa_{1}-\kappa_{2}\Big(X(t)-\min_{s \in [0,t]}(X(s)\wedge 0)\Big)\right\}\,\diff t
+
\theta\,\diff W(t),\ \ t \in [0,\infty),\ \ X(0)=\xi.
\end{align}
The above can be seen as an SFDE \eqref{eq_SFDE} with data $(\bR,\delta_{\xi},b,\sigma)$, where the coefficients $b,\sigma:\Lambda[\{\xi\},\bR]\to\bR$ are given by
\begin{equation*}
	b(t,x)=\kappa_1-\kappa_2\left(x(t)-\min_{s\in[0,t]}(x(s)\wedge0)\right),\ \ \sigma(t,x)=\theta,
\end{equation*}
for $(t,x)\in\Lambda[\{\xi\},\bR]$. Clearly, the data of the SFDE \eqref{eq_ROU2} satisfies conditions (G'), (C') and (E') in \cref{cor_rate-polynomial}. Furthermore, by using \cref{appendix_lemm_LG} (i), we see that the solution $X$ satisfies $\bE[\|X_T\|_\infty^p]<\infty$ for any $T\in(0,\infty)$ and $p\in[2,\infty)$. Therefore, the moment condition \eqref{eq_moment} in \cref{cor_rate-polynomial} (i) holds, and we immediately obtain the following result.


\begin{theo}
Let $\kappa_1\in \bR,\kappa_2, \theta \in (0,\infty)$ be fixed parameters, and let $\weaksol$ be the weak solution of the SFDE \eqref{eq_ROU2} with initial condition $X(0)=\xi \in [0,\infty)$.
For each $T\in(0,\infty)$ and $\pi\in\Pi_T$, let $X^\pi$ be the Euler--Maruyama scheme defined on $(\Omega^\pi,\cF^\pi,\bP^\pi)$ and given by \eqref{eq_EM-reflected} with initial condition $X^{\pi}(t_0)=\xi$.
Then, for any $\gamma\in(0,\frac{1}{2})$, there exists a constant $C_\gamma\in(0,\infty)$ such that
\begin{equation*}
	d_\LP\big(\Law_\bP(X_T),\Law_{\bP^\pi}(\poly[X^\pi])\big)\leq C_\gamma|\pi|^\gamma
\end{equation*}
for any $\pi\in\Pi_{T}$.
\end{theo}


\subsection{Stochastic Duffing--van der Pol oscillator}\label{subsec_Duffing}
The Duffing--van der Pol oscillator is important in the theory of stability and bifurcations of nonlinear dynamical systems (see for example \cite{HoRa80}).
As in \cite{ScHo96}, we consider the system incorporating an affine-linear noise term.
To be specific, let $\kappa_1,\kappa_2,\kappa_3,\theta_1,\theta_2,\theta_3\in\bR$ be fixed parameters with $\kappa_3\geq0$, and consider the stochastic Duffing--van der Pol oscillator formally described by
\begin{equation*}
	\ddot{x}-\kappa_1x+\kappa_2\dot{x}+\kappa_3\dot{x}x^2+x^3=\theta_1x\dot{W}_1+\theta_2\dot{x}\dot{W}_2+\theta_3\dot{W}_3,\ \ (x(0),\dot{x}(0))=(\xi,\dot{\xi}),
\end{equation*}
where $\xi,\dot{\xi}\in\bR$ are initial conditions, and $\dot{W}_1,\dot{W}_2,\dot{W}_3$ are three independent one-dimensional white noises. The above is rewritten as the $2$-dimensional Markovian SDE for $(x,\dot{x})$:
\begin{equation}\label{eq_Duffing-SDE}
	\begin{dcases}
	\diff x(t)=\dot{x}(t)\,\diff t,\ \ t\in[0,\infty)\\
	\diff\dot{x}(t)=\left\{\kappa_1x(t)-\kappa_2\dot{x}(t)-\kappa_3\dot{x}(t)x(t)^2-x(t)^3\right\}\,\diff t\\
	\hspace{2cm}+\theta_1x(t)\,\diff W_1(t)+\theta_2\dot{x}(t)\,\diff W_2(t)+\theta_3\,\diff W_3(t),\ \ t\in[0,\infty),\\
	(x(0),\dot{x}(0))=(\xi,\dot{\xi}).
	\end{dcases}
\end{equation}
In view of numerical analysis, the SDE \eqref{eq_Duffing-SDE} has at least two difficulties, namely, the super-linearity of the drift coefficient and degeneracy of the diffusion coefficient. On the one hand, due to the super-linearity of the drift coefficient, the standard Euler--Maruyama scheme of the above SDE might diverge in the $L^p$ sense for any $p\in(0,\infty)$ (see \cite{HuJeKl11}). Hutzenthaler and Jentzen \cite{HuJe15} then consider an appropriately modified Euler--Maruyama scheme for \eqref{eq_Duffing-SDE} and show its convergence to the true solution $(x,\dot{x})^{\top}$ at each fixed time $T\in(0,\infty)$ in the $L^{p}$ sense for any $p\in(0,\infty)$. However, they do not obtain any convergence rate for it. On the other hand, due to the degeneracy of the diffusion coefficient, our main results cannot be applied to the SDE \eqref{eq_Duffing-SDE} directly.

In order to overcome the second difficulty mentioned above, as in \cref{subsubsec_oscillator}, we consider the equivalent stochastic integro-differential equation \eqref{eq_oscillator2} for $X=\dot{x}$ which is of the following form:
\begin{equation}\label{eq_Duffing-SFDE}
	\begin{dcases}
	\diff X(t)=\left\{\kappa_1\left(\xi+\int^t_0X(s)\,\diff s\right)-\kappa_2X(t)-\kappa_3X(t)\left(\xi+\int^t_0X(s)\,\diff s\right)^2
	-\left(\xi+\int^t_0X(s)\,\diff s\right)^3\right\}\,\diff t\\
	\hspace{1.5cm}+\theta_1\left(\xi+\int^t_0X(s)\,\diff s\right)\,\diff W_1(t)+\theta_2X(t)\,\diff W_2(t)+\theta_3\diff W_3(t),\ \ t\in[0,\infty),\\
	X(0)=\dot{\xi}.
	\end{dcases}
\end{equation}
This transformed equation can be seen as an SFDE \eqref{eq_SFDE} with data $(D,\mu_0,b,\sigma)=(\bR,\delta_{\dot{\xi}},b,\sigma)$, where the coefficients $b:\Lambda[\{\dot{\xi}\},\bR]\to\bR$ and $\sigma:\Lambda[\{\dot{\xi}\},\bR]\to\bR^{1\times3}$ are given by
\begin{equation*}
	b(t,\dot{x})=\kappa_1\left(\xi+\int^t_0\dot{x}(s)\,\diff s\right)-\kappa_2\dot{x}(t)-\kappa_3\dot{x}(t)\left(\xi+\int^t_0\dot{x}(s)\,\diff s\right)^2-\left(\xi+\int^t_0\dot{x}(s)\,\diff s\right)^3
\end{equation*}
and
\begin{equation*}
	\sigma(t,\dot{x})=
	\begin{pmatrix}\displaystyle
	\theta_1\left(\xi+\int^t_0\dot{x}(s)\,\diff s\right)&\theta_2\dot{x}(t)&\theta_3
	\end{pmatrix}
\end{equation*}
for $(t,\dot{x})\in\Lambda[\{\dot{\xi}\},\bR]$. Observe that the diffusion coefficient $\sigma$ satisfies condition (E') in \cref{cor_rate-polynomial} if $\theta_3\neq0$. Also, the arguments in \cref{subsubsec_oscillator} show that conditions (G') and (C') in \cref{cor_rate-polynomial} hold as well. We will apply \cref{cor_rate-polynomial} (i) to the stochastic integro-differential equation \eqref{eq_Duffing-SFDE} and show the functional type weak convergence of the Euler--Maruyama scheme \eqref{eq_EM-oscillator} to the law of the solution $X$. To do so, we need to check the moment condition \eqref{eq_moment}.
The following moment estimate is based on \cite[Corollary 3.17 and Section 4.3]{HuJe15}.


\begin{lemm}\label{lemm_Duffing-moment}
Let $\kappa_1,\kappa_2,\kappa_3,\theta_1,\theta_2,\theta_3\in\bR$ with $\kappa_3\geq0$ be fixed parameters. Then, for any $\xi,\dot{\xi}\in\bR$, there exists a unique strong solution $X$ to the stochastic integro-differential equation \eqref{eq_Duffing-SFDE} with initial condition $X(0)=\dot{\xi}$ on $D=\bR$. Furthermore, for any $T\in(0,\infty)$ and $p\in[1,\infty)$, it holds that
\begin{equation*}
	\bE\big[\|X_T\|^p_\infty\big]<\infty.
\end{equation*}
\end{lemm}


\begin{proof}
As discussed in \cref{subsubsec_oscillator}, the Markovian SDE \eqref{eq_Duffing-SDE} and the stochastic integro-differential equation \eqref{eq_Duffing-SFDE} are in one-to-one correspondence via $(x,\dot{x})=(\xi+\int^\cdot_0X(s)\,\diff s,X)$. By the arguments in \cite[Section 4.3]{HuJe15}, we see that there exists a pathwise unique global solution $(x,\dot{x})$ of the SDE \eqref{eq_Duffing-SDE}, and hence the same is true for the stochastic integro-differential equation \eqref{eq_Duffing-SFDE}. Furthermore, \cite[Section 4.3]{HuJe15} shows that
\begin{equation}\label{eq_Duffing-moment-fixedT}
	\sup_{t\in[0,T]}\bE\big[|x(t)|^p+|\dot{x}(t)|^p\big]<\infty\ \ \text{for any $T\in(0,\infty)$ and $p\in[1,\infty)$}.
\end{equation}
Let $T\in(0,\infty)$ and $p\in[1,\infty)$ be fixed. Observe that, by the second equation in \eqref{eq_Duffing-SDE},
\begin{align*}
	\sup_{t\in[0,T]}|\dot{x}(t)|^p
	&\leq3^{p-1}|\dot{\xi}|^p+3^{p-1}\left(\int^T_0\Big|\kappa_1x(s)-\kappa_2\dot{x}(s)-\kappa_3\dot{x}(s)x(s)^2-x(s)^3\Big|\,\diff s\right)^p\\
	&\hspace{1cm}+3^{p-1}\sup_{t\in[0,T]}\left|\int^t_0\theta_1x(s)\,\diff W_1(s)+\int^t_0\theta_2\dot{x}(s)\,\diff W_2(s)+\int^t_0\theta_3\,\diff W_3(s)\right|^p.
\end{align*}
Then, using the Burkholder--Davis--Gundy inequality and \eqref{eq_Duffing-moment-fixedT}, we obtain $\bE[\sup_{t\in[0,T]}|\dot{x}(t)|^p]<\infty$, and hence $\bE[\|X_T\|_\infty^p]<\infty$. This completes the proof.
\end{proof}

By \cref{lemm_Duffing-moment} and \cref{cor_rate-polynomial} (i), we immediately obtain the following result.


\begin{theo}\label{theo_Duffing}
Let $\kappa_1,\kappa_2,\kappa_3,\theta_1,\theta_2,\theta_3\in\bR$ with $\kappa_3\geq0$ be fixed parameters, and assume that $\theta_3\neq0$. Fix $\xi,\dot{\xi}\in\bR$, and let $\weaksol$ be the weak solution of the stochastic integro-differential equation \eqref{eq_Duffing-SFDE} with initial condition $X(0)=\dot{\xi}$ on $D=\bR$.
For each $T\in(0,\infty)$ and $\pi\in\Pi_T$, let $X^\pi$ be the Euler--Maruyama scheme defined on $(\Omega^\pi,\cF^\pi,\bP^\pi)$ and given by \eqref{eq_EM-oscillator} with initial condition $X^\pi(t_0)=\dot{\xi}$.
Then, for any $\gamma\in(0,\frac{1}{2})$, there exist constants $C_\gamma\in(0,\infty)$ such that
\begin{equation*}
	d_\LP\big(\Law_\bP(X_T),\Law_{\bP^\pi}(\poly[X^\pi])\big)\leq C_\gamma|\pi|^\gamma
\end{equation*}
for any $\pi\in\Pi_{T}$.
\end{theo}


\subsection{Wright--Fisher diffusion with seed bank}\label{subsec_WF}

The Wright--Fisher diffusion is a classical probabilistic object in population genetics, which describes the scaling limit of the fraction of the neutral allele in a large haploid population. 
Recently, Blath et al.\ \cite{BlCaKuWi16} derived the Wright--Fisher diffusion with seed bank, which incorporates the so-called strong seed-bank effect to the classical Wright--Fisher model.
Here, as in \cite{BlBuCaWi19}, we consider the Wright--Fisher diffusion with seed bank including mutation, which is defined as the $[0,1]^2$-valued strong Markov process $(X(t),Y(t))^\top$ solving the following SDE:
\begin{equation}\label{eq_WF}
	\begin{dcases}
	\diff X(t)=\big\{\!-\kappa_1X(t)+\kappa_2(1-X(t))+\kappa_3(Y(t)-X(t))\big\}\,\diff t+\theta\sqrt{X(t)(1-X(t))}\,\diff W(t),\\
	\diff Y(t)=\big\{\!-\kappa_1'Y(t)+\kappa_2'(1-Y(t))+\kappa_3'(X(t)-Y(t))\big\}\,\diff t,\ \ t\in[0,\infty),\\
	(X(0),Y(0))^\top=(\xi,\eta)^\top\in[0,1]^2,\\	
	\end{dcases}
\end{equation}
where $\kappa_1,\kappa_2,\kappa_3,\kappa_1',\kappa_2',\kappa_3'\in[0,\infty)$ and $\theta\in(0,\infty)$ are given parameters.
Notice that, when $\kappa_3=0$, the first line of the above system is reduced to the one-dimensional SDE for the classical Wright--Fisher diffusion. By \cite[Theorem 3.2]{ShSh80}, there exists a unique strong solution $(X,Y)^\top$ to the SDE \eqref{eq_WF} such that, a.s., $(X(t),Y(t))^\top\in[0,1]^2$ for any $t\in[0,\infty)$. Furthermore, \cite[Theorem 3.1]{BlBuCaWi19} shows that the first component $X$ of the solution started from $\xi\in(0,1)$ will never hit $0$ or $1$ if and only if $2(\kappa_1\wedge\kappa_2)\geq\theta^2$.

Euler--Maruyama-type approximations for the classical Wright--Fisher diffusion (that is, \eqref{eq_WF} with $\kappa_3=0$) have been studied in several papers; Neuenkirch and Szpruch \cite{NeSz14} provide a rate of convergence in the $L^{p}$ sense by means of the Lamperti-backward Euler--Maruyama scheme for $p \in [2,\frac{4(\kappa_{1}\wedge \kappa_{2})}{3\theta^{2}}]$ under the assumption that $2(\kappa_{1} \wedge \kappa_{2}) \geq 3\theta^{2}$, and Mickel and Neuenkirch \cite{MiNe24} provide a rate of convergence in the $L^{1}$ sense by means of the truncated type Euler--Maruyama scheme under the assumption that $\kappa_{1} \wedge \kappa_{2}>0$ (in this case, the Wright--Fisher diffusion can hit $0$ and $1$). In the general case of the SDE \eqref{eq_WF} with $\kappa_3\neq0$, the pathwise uniqueness result \cite[Theorem 3.2]{ShSh80} and \cite[Theorem D]{KaNa88} imply that the (standard) Euler--Maruyama scheme (interpolated by the Brownian path) converges in the $L^{2}$-sup sense, but these results do not provide any convergence rates.

Notice that the system \eqref{eq_WF} is a degenerate $2$-dimensional Markovian SDE driven by $1$-dimensional Brownian motion $W$. The degeneracy prevents us from applying our main result to the system \eqref{eq_WF}. However, using the trick in \cite{BlBuCaWi19}, which is similar to the arguments discussed in \cref{subsubsec_oscillator} for the case of stochastic oscillator models, we can convert the system \eqref{eq_WF} to an SFDE which fits into our framework. Indeed, since the second line of \eqref{eq_WF} is a linear differential equation, the variation of constants formula shows that the second component $Y$ of the solution is represented in terms of the first component $X$:
\begin{equation}\label{eq_WF-Y}
	Y(t)=\eta e^{-\kappa't}+\int^t_0e^{-\kappa'(t-s)}\big(\kappa_2'+\kappa_3'X(s)\big)\,\diff s,\ \ t\in[0,\infty),
\end{equation}
where $\kappa':=\kappa_1'+\kappa_2'+\kappa_3'$. Inserting this expression to the first equation in \eqref{eq_WF}, we get the following stochastic integro-differential equation for $X$:
\begin{equation}\label{eq_WF-SFDE}
\begin{split}
	\diff X(t)&=\kappa_3\left\{\eta e^{-\kappa't}+\int^t_0e^{-\kappa'(t-s)}\big(\kappa_2'+\kappa_3'X(s)\big)\,\diff s\right\}\,\diff t\\
	&\hspace{0.5cm}+\big\{-(\kappa_1+\kappa_3)X(t)+\kappa_2(1-X(t))\big\}\,\diff t+\theta\sqrt{X(t)(1-X(t))}\,\diff W(t),\ 
	t\in[0,\infty),
\end{split}
\end{equation}
with initial condition $X(0)=\xi$. Notice that $(X,Y)$ solves the system \eqref{eq_WF} if and only if $X$ solves \eqref{eq_WF-SFDE} and $Y$ is given by \eqref{eq_WF-Y}. The delay term in \eqref{eq_WF-SFDE} reveals the underlying age structure of the model; see \cite{BlBuCaWi19}. Observe that the equation \eqref{eq_WF-SFDE} with initial condition $X(0)=\xi\in(0,1)$ can be seen as an SFDE \eqref{eq_SFDE} with data $(D,\mu_0,b,\sigma)=((0,1),\delta_{\xi},b,\sigma)$, where the coefficients $b,\sigma:\Lambda[\{\xi\},(0,1)]\to\bR$ are given by
\begin{equation}\label{eq_coefficient-WF}
\begin{split}
	&b(t,x)=\kappa_3\left\{\eta e^{-\kappa't}+\int^t_0e^{-\kappa'(t-s)}\big(\kappa'_2+\kappa_3'x(s)\big)\,\diff s\right\}-(\kappa_1+\kappa_3)x(t)+\kappa_2(1-x(t)),\\
	&\sigma(t,x)=\theta\sqrt{x(t)(1-x(t))},
\end{split}
\end{equation}
for $(t,x)\in\Lambda[\{\xi\},(0,1)]$. The coefficients $b$ and $\sigma$ defined by \eqref{eq_coefficient-WF} are progressively measurable and satisfy \cref{assum} with the choices of the exponents $\alpha_{b}=\alpha_{\sigma}=1$, the sets
\begin{equation}\label{eq_WF-subset}
\begin{split}
	&D_{\growth,b;T}(R)=D_{\growth,\sigma;T}(R)=D_{\conti,b;T}(R)=D=(0,1),\\
	&D_{\conti,\sigma;T}(R)=D_{\ellip;T}(R)=\left(\frac{1}{R},1-\frac{1}{R}\right),\ \ R\in[1,\infty),
\end{split}
\end{equation}
and some constants $K_{\growth,b;T},K_{\growth,\sigma;T},K_{\conti,b;T},K_{\conti,\sigma;T},K_{\ellip;T}$ depending only on $\kappa_1,\kappa_2,\kappa_3$ and $\theta$.

As mentioned above, if the parameters satisfy $2(\kappa_1\wedge\kappa_2)\geq\theta^2$ and $\xi\in D=(0,1)$, then the SDE \eqref{eq_WF}, and hence also the stochastic integro-differential equation \eqref{eq_WF-SFDE}, has a unique solution $X$ such that, a.s., $X(t)\in D$ for any $t\in[0,\infty)$; see \cite[Theorem 3.1]{BlBuCaWi19}. Hence, under the conditions that $2(\kappa_1\wedge\kappa_2)\geq\theta^2$ and $\xi\in D$, \cref{theo_main} (i) implies that
\begin{equation}\label{eq_WF-weakconv}
	\poly[X^\pi]\to X_T\ \ \text{weakly on $\cC_T$ as $|\pi|\downarrow0$ along $\pi\in\Pi_T$ for any $T\in(0,\infty)$},
\end{equation}
where $X^\pi=\EM$ is the Euler--Maruyama scheme given by \eqref{eq_EM} with initial condition $X^\pi(t_0)=\xi$. In order to get a weak convergence order with respect to the L\'{e}vy--Prokhorov metric, we use \cref{theo_main} (iii). To do so, we need to determine the parameters $\vec{\beta}=(\beta_0,\beta_{\growth,b},\beta_{\growth,\sigma},\beta_{\conti,b},\beta_{\conti,\sigma},\beta_\ellip)\in(0,\infty)^6$ appearing in the assumption \eqref{eq_rare-event} on the probability of the ``rare event''. The following moment estimates play crucial roles for such a purpose.


\begin{lemm}\label{lemm_WF-moment}
Let $\kappa_1,\kappa_2,\kappa_3,\kappa_1',\kappa_2',\kappa_3'\in[0,\infty)$ and $\theta\in(0,\infty)$ be given parameters with $2(\kappa_1\wedge\kappa_2)\geq\theta^2$. Fix $\eta\in[0,1]$, and let $\weaksol$ be the weak solution of the stochastic integro-differential equation \eqref{eq_WF-SFDE} on $D=(0,1)$ with initial condition $X(0)=\xi\in D$.
\begin{itemize}
\item[(i)]
Assume that $2\kappa_2>\theta^2$. Then, for any $p\in(0,\frac{2\kappa_2}{\theta^2}-1)$, there exists a constant $C_p\in(0,\infty)$, which depends only on $\kappa_1,\kappa_2,\kappa_3,\theta$ and $p$, such that, for any $T\in(0,\infty)$,
\begin{equation*}
	\bE\left[\sup_{t\in[0,T]}X(t)^{-p}+\int^T_0X(t)^{-p-1}\,\diff t\right]\leq C_pe^{C_p T}\xi^{-p}.
\end{equation*}
\item[(ii)]
Assume that $2\kappa_1>\theta^2$. Then, for any $p\in(0,\frac{2\kappa_1}{\theta^2}-1)$, there exists a constant $C_p\in(0,\infty)$, which depends only on $\kappa_1,\kappa_2,\kappa_3,\theta$ and $p$, such that, for any $T\in(0,\infty)$,
\begin{equation*}
	\bE\left[\sup_{t\in[0,T]}\big(1-X(t)\big)^{-p}+\int^T_0\big(1-X(t)\big)^{-p-1}\,\diff t\right]\leq C_pe^{C_p T}\big(1-\xi\big)^{-p}.
\end{equation*}
\end{itemize}
\end{lemm}


\begin{proof}
Notice that the process $1-X$ solves the stochastic integro-differential equation \eqref{eq_WF-SFDE} with the parameters $(\kappa_1,\kappa_2,\kappa_3,\kappa'_1,\kappa'_2,\kappa_3',\theta,\xi,\eta)$ replaced by
\begin{equation*}
	(\kappa_2,\kappa_1,\kappa_3,\kappa'_2,\kappa'_1,\kappa_3',\theta,1-\xi,1-\eta).
\end{equation*}
Thus, assertions (i) and (ii) are parallel, and we only need to show (i).

Assume that $2\kappa_2>\theta^2$, and let $p\in(0,\frac{2\kappa_2}{\theta^2}-1)$ be fixed. In this proof, we denote by $C_p$ a positive constant which depends only on $\kappa_1,\kappa_2,\kappa_3,\theta$ and $p$ and varies from line to line. By It\^o's formula, it holds that
\begin{equation}\label{eq_WF-Ito}
\begin{split}
	&X(t)^{-p}+p\left\{\kappa_2-\frac{(p+1)\theta^2}{2}\right\}\int^t_0X(s)^{-p-1}\,\diff s+p\kappa_3\int^t_0X(s)^{-p-1}Y(s)\,\diff s\\
	&=\xi^{-p}+p\left\{\kappa_1+\kappa_2+\kappa_3-\frac{(p+1)\theta^2}{2}\right\}\int^t_0X(s)^{-p}\,\diff s-p\theta\int^t_0X(s)^{-p-1/2}\sqrt{1-X(s)}\,\diff W(s)
\end{split}
\end{equation}
for any $t\in[0,\infty)$ $\bP$-a.s., where $Y$ is defined by \eqref{eq_WF-Y}. Notice that the three terms in the left-hand side of \eqref{eq_WF-Ito} are nonnegative.
For any $N \in \bN$, define a stopping time $\tau_N$ by $\tau_N:=\inf\{t\geq0\,|\,X(t)\leq 1/N\}$. Since $\bP(\text{$X(t)\in D$ for any $t\in[0,\infty)$})=1$, we have $\tau_N\to\infty$ as $N\to\infty$, $\bP$-a.s. Let $T\in(0,\infty)$, and take an arbitrary $T_1\in[0,T]$. On the one hand, letting $t=T_1\wedge\tau_N$ and taking expectations in \eqref{eq_WF-Ito}, we see that
\begin{align*}
	&p\left\{\kappa_2-\frac{(p+1)\theta^2}{2}\right\}\bE\left[\int^{T_1\wedge\tau_N}_0X(s)^{-p-1}\,\diff s\right]\\
	&\leq\xi^{-p}+p\left\{\kappa_1+\kappa_2+\kappa_3-\frac{(p+1)\theta^2}{2}\right\}\bE\left[\int^{T_1\wedge\tau_N}_0X(s)^{-p}\,\diff s\right].
\end{align*}
Since $\kappa_2-\frac{(p+1)\theta^2}{2}>0$, we get
\begin{equation}\label{eq_WF-estimate1}
	\bE\left[\int^{T_1\wedge\tau_N}_0X(s)^{-p-1}\,\diff s\right]\leq C_p\left\{\xi^{-p}+\bE\left[\int^{T_1\wedge\tau_N}_0X(s)^{-p}\,\diff s\right]\right\}.
\end{equation}
On the other hand, taking the supremum with respect to $t\in[0,T_1\wedge\tau_N]$ and then taking expectations in \eqref{eq_WF-Ito}, by the Burkholder--Davis--Gundy inequality, we get
\begin{align*}
	&\bE\left[\sup_{t\in[0,T_1]}X(t\wedge\tau_N)^{-p}\right]\\&\leq C_p\left\{\xi^{-p}+\bE\left[\int^{T_1\wedge\tau_N}_0X(s)^{-p}\,\diff s\right]
	+\bE\left[\left(\int^{T_1\wedge\tau_N}_0X(s)^{-2p-1}\big(1-X(s)\big)\,\diff s\right)^{1/2}\right]\right\}.
\end{align*}
From this, together with the estimate 
\begin{align*}
	&C_p\bE\left[\left(\int^{T_1\wedge\tau_N}_0X(s)^{-2p-1}\big(1-X(s)\big)\,\diff s\right)^{1/2}\right]\\
	&\leq C_p\bE\left[\left(\sup_{s\in[0,T_1]}X(s\wedge\tau_N)^{-p}\right)^{1/2}\left(\int^{T_1\wedge\tau_N}_0X(s)^{-p-1}\,\diff s\right)^{1/2}\right]\\
	&\leq\frac{1}{2}\bE\left[\sup_{s\in[0,T_1]}X(s\wedge\tau_N)^{-p}\right]+C_p\bE\left[\int^{T_1\wedge\tau_N}_0X(s)^{-p-1}\,\diff s\right],
\end{align*}
we obtain
\begin{equation}\label{eq_WF-estimate2}
\begin{split}
	\bE\left[\sup_{t\in[0,T_1]}X(t\wedge\tau_N)^{-p}\right]
	&\leq C_p\left\{\xi^{-p}+\bE\left[\int^{T_1\wedge\tau_N}_0X(s)^{-p}\,\diff s\right]+\bE\left[\int^{T_1\wedge\tau_N}_0X(s)^{-p-1}\,\diff s\right]\right\}.
\end{split}
\end{equation}
By \eqref{eq_WF-estimate1} and \eqref{eq_WF-estimate2}, we get
\begin{align*}
	\bE\left[\sup_{t\in[0,T_1]}X(t\wedge\tau_N)^{-p}\right]&\leq C_p\left\{\xi^{-p}+\bE\left[\int^{T_1\wedge\tau_N}_0X(s)^{-p}\,\diff s\right]\right\}\\
	&\leq C_p\left\{\xi^{-p}+\int^{T_1}_0\bE\left[\sup_{t\in[0,s]}X(t\wedge\tau_N)^{-p}\right]\,\diff s\right\}.
\end{align*}
Noting that $T_1\in[0,T]$ is arbitrary, Gronwall's inequality yields that
\begin{equation*}
	\bE\left[\sup_{t\in[0,T]}X(t\wedge\tau_N)^{-p}\right]\leq C_pe^{C_p T}\xi^{-p}.
\end{equation*}
From this estimate and \eqref{eq_WF-estimate1} with $T_1=T$, letting $N\to\infty$ and using Fatou's lemma, we obtain the desired estimate in assertion (i). This completes the proof.
\end{proof}

Combining the above estimates with \cref{theo_main} (iii), we get the following convergence order for the weak approximation \eqref{eq_WF-weakconv} under the slightly stronger condition $2(\kappa_1\wedge\kappa_2)>\theta^2$.


\begin{theo}\label{theo_WF}
Let $\kappa_1,\kappa_2,\kappa_3,\kappa_1',\kappa_2',\kappa_3'\in[0,\infty)$ and $\theta\in(0,\infty)$ be given parameters with $2(\kappa_1\wedge\kappa_2)>\theta^2$. Fix $\eta\in[0,1]$, and let $\weaksol$ be the weak solution of the stochastic integro-differential equation \eqref{eq_WF-SFDE} on $D=(0,1)$ with initial condition $X(0)=\xi\in D$.
For each $T\in(0,\infty)$ and $\pi\in\Pi_T$, let $X^\pi$ be the Euler--Maruyama scheme defined on $(\Omega^\pi,\cF^\pi,\bP^\pi)$ and given by \eqref{eq_EM} with initial condition $X^\pi(t_0)=\xi$.
Then, for any
\begin{equation*}
	\gamma\in\left(0,\frac{1}{4} \cdot \frac{2(\kappa_1\wedge\kappa_2)-\theta^{2}}{(\kappa_1\wedge\kappa_2)+\theta^2}\right),
\end{equation*}
there exists a constant $C_\gamma \in(0,\infty)$ such that
\begin{equation*}
	d_\LP\big(\Law_\bP(X_T),\Law_{\bP^\pi}(\poly[X^\pi])\big)
	\leq C_\gamma|\pi|^\gamma
\end{equation*}
for any $\pi\in\Pi_{T}$.
\end{theo}


\begin{proof}
Recall that the coefficients of the stochastic integro-differential equation \eqref{eq_WF-SFDE} satisfy \cref{assum} with $\alpha_{b}=\alpha_{\sigma}=1$, the sets $D_{\growth,b;T},D_{\growth,\sigma;T},D_{\conti,b;T},D_{\conti,\sigma;T},D_{\ellip;T}\subset D=(0,1)$ given by \eqref{eq_WF-subset}, and some constants $K_{\growth,b;T},K_{\growth,\sigma;T},K_{\conti,b;T},K_{\conti,\sigma;T},K_{\ellip;T}$ depending only on $\kappa_1,\kappa_2,\kappa_3$ and $\theta$. In order to apply \cref{theo_main} (iii), we check the condition \eqref{eq_rare-event}. Observe that, for any $\Delta\in(0,1]$ and $\vec{R}=(R_{\growth,b},R_{\growth,\sigma},R_{\conti,b},R_{\conti,\sigma},R_\ellip)\in[1,\infty)^5$,
\begin{align*}
	&\bP\left(\inf_{t\in[0,T]}\dist\!\left(X(t),\bR\setminus D_T(\vec{R})\right)\leq \Delta\right)\\
	&=\bP\left(\text{$X(t)\leq\frac{1}{R_{\conti,\sigma}\wedge R_\ellip}+\Delta$ or $X(t)\geq1-\frac{1}{R_{\conti,\sigma}\wedge R_\ellip}-\Delta$ for some $t\in[0,T]$}\right)\\
	&\leq\bP\left(\sup_{t \in [0,T]}X(t)^{-1} \geq \left(\frac{1}{R_{\conti,\sigma}\wedge R_\ellip}+\Delta\right)^{-1}\right)
	+\bP\left(\sup_{t \in [0,T]}\big(1-X(t)\big)^{-1} \geq \left(\frac{1}{R_{\conti,\sigma}\wedge R_\ellip}+\Delta\right)^{-1}\right),
\end{align*}
and thus, by Markov's inequality and \cref{lemm_WF-moment},
\begin{align*}
	&\bP\left(\inf_{t\in[0,T]}\dist\!\left(X(t),\bR\setminus D_T(\vec{R})\right)\leq \Delta\right)\\
	&\leq\left(\frac{1}{R_{\conti,\sigma}\wedge R_\ellip}+\Delta\right)^p\left\{\bE\left[\sup_{t \in [0,T]}X(t)^{-p}\right]+\bE\left[\sup_{t \in [0,T]}\big(1-X(t)\big)^{-p}\right]\right\}\\
	&\leq C_p\max\left\{\Delta^p,R_{\conti,\sigma}^{-p},R_\ellip^{-p}\right\},
\end{align*}
for any $p\in(0,\frac{2(\kappa_1\wedge\kappa_2)}{\theta^2}-1)$. Here, $C_p\in(0,\infty)$ is a constant which does not depend on $\Delta$ or $\vec{R}$. The above estimate shows that \eqref{eq_rare-event} holds for any $\vec{\beta}=(\beta_0,\beta_{\growth,b},\beta_{\growth,\sigma},\beta_{\conti,b},\beta_{\conti,\sigma},\beta_\ellip)$ with
\begin{equation*}
	\beta_{\growth,b},\beta_{\growth,\sigma},\beta_{\conti,b}\in(0,\infty),\ \beta_0,\beta_{\conti,\sigma},\beta_\ellip\in\left(0,\frac{2(\kappa_1\wedge\kappa_2)}{\theta^2}-1\right).
\end{equation*}
Concerning the constants $\beta_*$ and $\gamma_*$ defined by \eqref{eq_beta*} and \eqref{eq_gamma*} , we have
\begin{equation*}
\begin{split}
\beta_*&=
\min\left\{
\beta_0,
\frac{1}{1+\beta_\ellip^{-1}+\beta_{\conti,b}^{-1}},
\frac{1}{1+\beta_\ellip^{-1}+2\beta_{\conti,\sigma}^{-1}}
\right\}
\to
\frac{1}{2}\cdot \frac{2(\kappa_1\wedge\kappa_2)-\theta^2}{(\kappa_1\wedge\kappa_2)+\theta^2}
\end{split}
\end{equation*}
and
\begin{equation*}
\begin{split}
\gamma_*&=
\min\left\{
\frac{1}{1+\beta_\ellip^{-1}+\beta_{\conti,b}^{-1}+\beta_{\growth,b}^{-1}},
\frac{1}{2}\cdot\frac{1}{1+\beta_\ellip^{-1}+\beta_{\conti,b}^{-1}+\beta_{\growth,\sigma}^{-1}},\right.\\
&\hspace{1.5cm}\left.
\frac{2}{1+\beta_\ellip^{-1}+2\beta_{\conti,\sigma}^{-1}+2\beta_{\growth,b}^{-1}+\beta_*^{-1}},
\frac{1}{1+\beta_\ellip^{-1}+2\beta_{\conti,\sigma}^{-1}+2\beta_{\growth,\sigma}^{-1}+\beta_*^{-1}}
\right\}\\
&\to
\frac{1}{4} \cdot \frac{2(\kappa_1\wedge\kappa_2)-\theta^{2}}{(\kappa_1\wedge\kappa_2)+\theta^2}
\end{split}
\end{equation*}
as $\beta_{\growth,b},\beta_{\growth,\sigma},\beta_{\conti,b}\to\infty$ and $\beta_0,\beta_{\conti,\sigma},\beta_\ellip\to\frac{2(\kappa_1\wedge\kappa_2)}{\theta^2}-1$. Therefore, noting \cref{rem_gamma*}, by \cref{theo_main} (iii), we get the desired estimate. This completes the proof.
\end{proof}


\subsection{Multi-dimensional polynomial diffusion}\label{subsec_polydiff}

The standard Wright--Fisher diffusion (that is, \eqref{eq_WF-SFDE} with $\kappa_3=0$) is a kind of one-dimensional polynomial diffusions. Now we consider the following $n$-dimensional polynomial diffusion with $n\geq2$:
\begin{equation}\label{eq_polydiff}
	\diff X(t)=-\kappa X(t)\,\diff t+\theta\,\sqrt{1-|X(t)|^{2}}I_{n\times n}\,\diff W(t),\ \ t\in[0,\infty),\ \ X(0)=\xi,
\end{equation}
where $\kappa,\theta\in(0,\infty)$ are given parameters, and $W$ denotes an $n$-dimensional Brownian motion. The solution $X$ can be seen as a multi-dimensional extension of the Jacobi process. The above SDE was first investigated in \cite{Sw02} and then applied in the field of mathematical finance in \cite{FiLa16}.
It is known that the SDE \eqref{eq_polydiff} has a unique strong solution $X$ on the closed unit ball $\overline{D}=\{x\in\bR^n\,|\,|x|\leq1\}$ when the parameters satisfy $\kappa\geq(\sqrt{2}-1)\theta^2$ for any initial condition $\xi\in\overline{D}$ (see \cite[Theorem 4.6]{LaPu17}).
Moreover, it is also known that, when $\xi\in D=\{x\in\bR^n\,|\,|x|<1\}$, $\bP(\text{$X(t)\in D$ for any $t\in[0,\infty)$})=1$ if and only if $\kappa\geq\theta^2$ (see \cite[Proposition 2]{Sw02} and \cite[Proposition 2.2]{LaPu17}).

The pathwise uniqueness result \cite[Theorem 4.6]{LaPu17} and \cite[Theorem D]{KaNa88} imply that the (standard) Euler--Maruyama scheme (interpolated by the Brownian path) for the SDE \eqref{eq_polydiff} converges in the $L^{2}$-sup sense when the parameters satisfy $\kappa\geq(\sqrt{2}-1)\theta^2$, but these results do not provide any convergence rates.
Numerical approximations for the polynomial diffusion \eqref{eq_polydiff} are studied in \cite{NaTaYu23}, where the authors provide a convergence rate in the $L^{2}$ sense for the semi-implicit Euler--Maruyama scheme under the condition $\kappa>3\theta^{2}$.

Let $\kappa\geq\theta^2$ and $\xi\in D=\{x\in\bR^n\,|\,|x|<1\}$. As discussed in \cref{subsubsec_Markov}, the $n$-dimensional Markovian SDE \eqref{eq_polydiff} on $D$ fits into the framework of the present paper. More precisely, the data satisfies \cref{assum} with the choices of the exponents $\alpha_b=\alpha_\sigma=1$, the sets
\begin{equation}\label{eq_polydiff-subset}
\begin{split}
	&D_{\growth,b;T}(R)=D_{\growth,\sigma;T}(R)=D_{\conti,b;T}(R)=D,\\
	&D_{\conti,\sigma;T}(R)=D_{\ellip;T}(R)=\left\{x\in D\relmiddle||x|<\left(1-\frac{1}{R^2}\right)^{1/2}\right\},\ \ R\in[1,\infty),
\end{split}
\end{equation}
and some constants $K_{\growth,b;T},K_{\growth,\sigma;T},K_{\conti,b;T},K_{\conti,\sigma;T},K_{\ellip;T}$ depending only on $n$, $\kappa$ and $\theta$.
Hence, \cref{theo_main} (i) shows that
\begin{equation}\label{eq_polydiff-weakconv}
	\poly[X^\pi]\to X_T\ \ \text{weakly on $\cC_T$ as $|\pi|\downarrow0$ along $\pi\in\Pi_T$ for any $T\in(0,\infty)$},
\end{equation}
where $X^\pi=\EM$ is the standard Euler--Maruyama scheme given by \eqref{eq_EM-Markov} with initial condition $X^\pi(t_0)=\xi$. In order to get a weak convergence order with respect to the L\'{e}vy--Prokhorov metric, we use \cref{theo_main} (iii). As before, a moment estimate corresponding to the ``rare event'' appearing in \eqref{eq_rare-event} plays a crucial role in determining the parameters $\vec{\beta}=(\beta_0,\beta_{\growth,b},\beta_{\growth,\sigma},\beta_{\conti,b},\beta_{\conti,\sigma},\beta_\ellip)\in(0,\infty)^6$. Here, we require a slightly stronger condition $\kappa>\theta^2$.


\begin{lemm}\label{lemm_polydiff-moment}
Let $\kappa,\theta\in(0,\infty)$ satisfy $\kappa>\theta^2$, and let $\weaksol$ be the weak solution of the SDE \eqref{eq_polydiff} on $D=\{x\in \bR^{n}\,|\,|x|<1\}$ with initial condition $X(0)=\xi\in D$. Then, for any $p\in(0,\frac{\kappa}{\theta^2}-1)$, there exists a constant $C_p\in(0,\infty)$, which depends only on $n,\kappa,\theta$ and $p$ such that, for any $T\in(0,\infty)$,
\begin{equation*}
	\bE\left[\sup_{t\in[0,T]}\big(1-|X(t)|^2\big)^{-p}+\int^T_0\big(1-|X(t)|^2\big)^{-p-1}\,\diff t\right]\leq C_pe^{C_pT}\big(1-|\xi|^2\big)^{-p}.
\end{equation*}
\end{lemm}


\begin{proof}
Define $\widetilde{X}:=|X|^2$. As in the proof of \cite[Proposition 2.4]{NaTaYu23}, we see that $\widetilde{X}$ solves the SDE:
\begin{equation*}
\begin{dcases}
	\diff\widetilde{X}(t)=\big\{\!-2\kappa\widetilde{X}(t)+n\theta^2(1-\widetilde{X}(t))\big\}\,\diff t+2\theta\sqrt{\widetilde{X}(t)(1-\widetilde{X}(t))}\,\diff\widetilde{W}(t),\ \ t\in[0,\infty),\\
	\widetilde{X}(0)=|\xi|^2,
\end{dcases}
\end{equation*}
for some one-dimensional Brownian motion $\widetilde{W}$. The above is a (classical) Wright--Fisher diffusion \eqref{eq_WF-SFDE} with $\kappa_3=0$ and parameters $(\kappa_1,\kappa_2,\theta)$ replaced by $(2\kappa,n\theta^2,2\theta)$. Therefore, by \cref{lemm_WF-moment} (ii), we get the desired estimate.\footnote{Although $\widetilde{X}(0)=0$ when $\xi=0$, the proof of \cref{lemm_WF-moment} (ii) applies unchanged: it uses It\^o's formula for $(1-\widetilde{X})^{-p}$ and stopping times that localize only near the boundary $1$, so no strict positivity of $\widetilde{X}$ is required.}
\end{proof}

By using \cref{lemm_polydiff-moment} and \cref{theo_main} (iii), we can obtain the following convergence order for the weak approximation \eqref{eq_polydiff-weakconv} under the slightly stronger condition $\kappa>\theta^2$.


\begin{theo}\label{theo_polydiff}
Let $\kappa,\theta\in(0,\infty)$ satisfy $\kappa>\theta^2$, and let $\weaksol$ be the weak solution of the multi-dimensional polynomial diffusion \eqref{eq_polydiff} on $D=\{x\in \bR^{n}\,|\,|x|<1\}$ with initial condition $X(0)=\xi\in D$.
For each $T\in(0,\infty)$ and $\pi\in\Pi_T$, let $X^\pi$ be the Euler--Maruyama scheme defined on $(\Omega^\pi,\cF^\pi,\bP^\pi)$ and given by \eqref{eq_EM-Markov} with initial condition $X^\pi(t_0)=\xi$.
Then, for any
\begin{equation*}
	\gamma\in\left(0,\frac{\kappa-\theta^2}{2\kappa+\theta^2}\right),
\end{equation*}
there exists a constant $C_\gamma\in(0,\infty)$ such that
\begin{equation*}
	d_\LP\big(\Law_\bP(X_T),\Law_{\bP^\pi}(\poly[X^\pi])\big)\leq C_\gamma|\pi|^\gamma
\end{equation*}
for any $\pi\in\Pi_{T}$.
\end{theo}


\begin{proof}
Noting \eqref{eq_polydiff-subset}, for any $\Delta\in(0,1]$ and $\vec{R}=(R_{\growth,b},R_{\growth,\sigma},R_{\conti,b},R_{\conti,\sigma},R_\ellip)\in[1,\infty)^5$, we have
\begin{align*}
	\bP\left(\inf_{t\in[0,T]}\dist\!\left(X(t),\bR^n\setminus D_T(\vec{R})\right)\leq \Delta\right)
	&=\bP\left(\text{$|X(t)|\geq\left(1-\frac{1}{R^2_{\conti,\sigma}\wedge R^2_\ellip}\right)^{1/2}-\Delta$ for some $t\in[0,T]$}\right)\\
	&\leq\bP\left(\text{$1-|X(t)|^2\leq\frac{1}{R^2_{\conti,\sigma}\wedge R^2_\ellip}+3\Delta$ for some $t\in[0,T]$}\right)\\
	&=\bP\left(\sup_{t \in [0,T]}\big(1-|X(t)|^2\big)^{-1}\geq\left(\frac{1}{R^2_{\conti,\sigma}\wedge R^2_\ellip}+3\Delta\right)^{-1}\right),
\end{align*}
and thus, by Markov's inequality and \cref{lemm_polydiff-moment},
\begin{align*}
	\bP\left(\inf_{t\in[0,T]}\dist\!\left(X(t),\bR^n\setminus D_T(\vec{R}) \right)\leq \Delta\right)
	&\leq\left(\frac{1}{R^2_{\conti,\sigma}\wedge R^2_\ellip}+3\Delta\right)^p\bE\left[\sup_{t\in[0,T]}\big(1-|X(t)|^2\big)^{-p}\right]\\
	&\leq C_p\max\left\{\Delta^p,R_{\conti,\sigma}^{-2p},R_\ellip^{-2p}\right\},
\end{align*}
for any $p\in(0,\frac{\kappa}{\theta^2}-1)$. Here, $C_p\in(0,\infty)$ is a constant which does not depend on $\Delta$ or $\vec{R}$. The above estimate shows that \eqref{eq_rare-event} holds for any $\vec{\beta}=(\beta_0,\beta_{\growth,b},\beta_{\growth,\sigma},\beta_{\conti,b},\beta_{\conti,\sigma},\beta_\ellip)$ with
\begin{equation*}
	\beta_0\in\left(0,\frac{\kappa}{\theta^2}-1\right),\
	\beta_{\growth,b},\beta_{\growth,\sigma},\beta_{\conti,b}\in(0,\infty),\
	\beta_{\conti,\sigma},\beta_\ellip\in\left(0,\frac{2\kappa}{\theta^2}-2\right).
\end{equation*}
Concerning the constants $\beta_*$ and $\gamma_*$ defined by \eqref{eq_beta*} and \eqref{eq_gamma*}, we have
\begin{equation*}
\begin{split}
\beta_*
&=
\min\left\{
\beta_0,
\frac{1}{1+\beta_\ellip^{-1}+\beta_{\conti,b}^{-1}},
\frac{1}{1+\beta_\ellip^{-1}+2\beta_{\conti,\sigma}^{-1}}
\right\}
\to
\frac{2(\kappa-\theta^{2})}{2\kappa+\theta^{2}}
\end{split}
\end{equation*}
and
\begin{equation*}
\begin{split}
\gamma_*
&=
\min\left\{
\frac{1}{1+\beta_\ellip^{-1}+\beta_{\conti,b}^{-1}+\beta_{\growth,b}^{-1}},
\frac{1}{2}\cdot\frac{1}{1+\beta_\ellip^{-1}+\beta_{\conti,b}^{-1}+\beta_{\growth,\sigma}^{-1}},\right.\\
&\hspace{1.5cm}\left.
\frac{2}{1+\beta_\ellip^{-1}+2\beta_{\conti,\sigma}^{-1}+2\beta_{\growth,b}^{-1}+\beta_*^{-1}},
\frac{1}{1+\beta_\ellip^{-1}+2\beta_{\conti,\sigma}^{-1}+2\beta_{\growth,\sigma}^{-1}+\beta_*^{-1}}
\right\}
\\&\to
\frac{\kappa-\theta^{2}}{2\kappa+\theta^{2}}
\end{split}
\end{equation*}
as $\beta_0\to\frac{\kappa}{\theta^2}-1$, $\beta_{\growth,b},\beta_{\growth,\sigma},\beta_{\conti,b}\to\infty$ and $\beta_{\conti,\sigma},\beta_\ellip\to\frac{2\kappa}{\theta^2}-2$. Therefore, noting \cref{rem_gamma*}, by \cref{theo_main} (iii), we get the desired estimate. This completes the proof.
\end{proof}


\subsection{Volatility process in the $3/2$-stochastic volatility model}\label{subsec_3/2}

We consider the following one-dimensional Markovian SDE on $D=(0,\infty)$:
\begin{equation}\label{eq_3/2}
	\diff X(t)=\kappa_{1}X(t)(\kappa_{2}-X(t))\,\diff t+\theta X(t)^{3/2}\,\diff W(t),\ \ t\in[0,\infty),\ \ X(0)=\xi\in D,
\end{equation}
where $\kappa_1,\kappa_2\in\bR$ and $\theta\in(0,\infty)$ are given parameters. The above SDE describes the dynamics of the volatility process in the so-called $3/2$-stochastic volatility model \cite{He97}.
By Feller's test (cf.\ \cite[Chapter 5, Theorem 5.29]{KaSh91}), we can show that the SDE \eqref{eq_3/2} admits a unique strong solution $X$ satisfying $\bP(\text{$X(t)\in D$ for any $t\in[0,\infty)$})=1$ if and only if the parameters satisfy $2\kappa_1+\theta^2\geq0$. In the literature on numerical approximations of the SDE \eqref{eq_3/2}, Neuenkirch and Szpruch \cite[Proposition 3.2]{NeSz14} provide a convergence rate in $L^{p}$ by means of the Lamperti-backward Euler--Maruyama scheme for $p \in [1,(\kappa_{1}+\theta^{2})/(3\theta^{2}))$ under the assumption that $\kappa_{1}-2\theta^{2}>0$, and Sabanis \cite[Theorem 2 and Appendix]{Sa16} provide a convergence rate in $L^{p}$ by means of the tamed Euler--Maruyama scheme for $p \in [2,(2\kappa_{1}+\theta^{2})/(3\theta^{2})]$ under the assumption that $2\kappa_{1}-5\theta^{2}>0$.

Let $2\kappa_1+\theta^2\geq0$. As discussed in \cref{subsubsec_Markov}, the one-dimensional Markovian SDE \eqref{eq_3/2} on $D=(0,\infty)$ fits into the framework of the present paper. More precisely, the data satisfies \cref{assum} with the choices of the exponents $\alpha_b=\alpha_\sigma=1$, the sets
\begin{equation}\label{eq_3/2-subset}
\begin{split}
	D_{\growth,b;T}(R)&=(0,2R^{1/2}),\ D_{\growth,\sigma;T}(R)=(0,2R^{2/3}),\\
	D_{\conti,b;T}(R)&=(0,2R),\ D_{\conti,\sigma;T}(R)=(0,2R^{2}),\ D_{\ellip;T}(R)=\left(\frac{1}{R^{2/3}},\infty\right),\ \ R\in[1,\infty),
\end{split}
\end{equation}
and some constants $K_{\growth,b;T},K_{\growth,\sigma;T},K_{\conti,b;T},K_{\conti,\sigma;T},K_{\ellip;T}$ depending only on $\kappa_1$, $\kappa_2$ and $\theta$; the multiplication by $2$ in \eqref{eq_3/2-subset} is just for a technical reason to make the computation in the proof of \cref{theo_3/2} below simple.

Since $\bP(\text{$X(t)\in D$ for any $t\in[0,\infty)$})=1$, \cref{theo_main} (i) shows that
\begin{equation}\label{eq_3/2-weakconv}
	\poly[X^\pi]\to X_T\ \ \text{weakly on $\cC_T$ as $|\pi|\downarrow0$ along $\pi\in\Pi_T$ for any $T\in(0,\infty)$},
\end{equation}
where $X^\pi=\EM$ is the standard Euler--Maruyama scheme given by \eqref{eq_EM-Markov} with initial condition $X^\pi(t_0)=\xi$. In order to get a weak convergence order with respect to the L\'{e}vy--Prokhorov metric, we use \cref{theo_main} (iii). To do so, we first investigate the following moment estimates corresponding to the ``rare event'' appearing in \eqref{eq_rare-event}. Here, we require a slightly stronger condition $2\kappa_1+\theta^2>0$.


\begin{lemm}\label{lemm_3/2-moment}
Let $\kappa_1,\kappa_2\in\bR$ and $\theta\in(0,\infty)$ satisfy $2\kappa_1+\theta^2>0$, and let $\weaksol$ be the weak solution of the SDE \eqref{eq_3/2} on $D=(0,\infty)$ with initial condition $X(0)=\xi\in D$. Let $T\in(0,\infty)$ be fixed. Then the following hold:
\begin{itemize}
\item[(i)]
For any $p\in(0,\frac{2\kappa_1}{\theta^2}+1)$, there exists a constant $C_p\in(0,\infty)$, which depends only on $\kappa_1,\kappa_2,\theta$ and $p$, such that
\begin{equation*}
	\bE\left[\sup_{t\in[0,T]}X(t)^p+\int^T_0X(t)^{p+1}\,\diff t\right]\leq C_pe^{C_pT}\xi^p.
\end{equation*}
\item[(ii)]
For any $p\in[2,\infty)$, there exists a constant $C_p\in(0,\infty)$, which depends only on $\kappa_1,\kappa_2,\theta$ and $p$, such that
\begin{equation*}
	\bE\big[\|X^{-1}_T\|_\infty^p\big]\leq C_pe^{C_pT}\big(1+\xi^{-p}\big).
\end{equation*}
\end{itemize}
\end{lemm}


\begin{proof}
Assertion (ii) follows from \cref{appendix_lemm_LG} (ii). We show assertion (i).

Let $T\in(0,\infty)$ and $p\in(0,\frac{2\kappa_1}{\theta^2}+1)$ be fixed. We denote by $C_p$ a positive constant which depends only on $\kappa_1,\kappa_2,\theta$ and $p$ and varies from line to line. By It\^o's formula, we have
\begin{equation}\label{eq_3/2-Ito}
\begin{split}
	&X(t)^p+p\left\{\kappa_1-\frac{(p-1)\theta^2}{2}\right\}\int^t_0X(s)^{p+1}\,\diff s
	=\xi^p+p\kappa_1\kappa_2\int^t_0X(s)^p\,\diff s+p\theta\int^t_0X(s)^{p+1/2}\,\diff W(s)
\end{split}
\end{equation}
for any $t\in[0,\infty)$ $\bP$-a.s.
For each $N \in \bN$, define a stopping time $\tau_N$ by $\tau_N:=\inf\{t\geq0\,|\,X(t)\geq N\}$.
Notice that $\tau_N\to\infty$ as $N\to\infty$ $\bP$-a.s. Take an arbitrary $T_1\in[0,T]$. On the one hand, letting $t=T_1\wedge\tau_N$ and taking expectations in \eqref{eq_3/2-Ito}, we see that
\begin{equation*}
	p\left\{\kappa_1-\frac{(p-1)\theta^2}{2}\right\}\bE\left[\int^{T_1\wedge\tau_N}_0X(s)^{p+1}\,\diff s\right]\leq \xi^p+p\kappa_1\kappa_2\bE\left[\int^{T_1\wedge\tau_N}_0X(s)^p\,\diff s\right].
\end{equation*}
Since $\kappa_1-\frac{(p-1)\theta^2}{2}>0$, we get
\begin{equation}\label{eq_3/2-estimate1}
	\bE\left[\int^{T_1\wedge\tau_N}_0X(s)^{p+1}\,\diff s\right]\leq C_p\left\{\xi^p+\bE\left[\int^{T_1\wedge\tau_N}_0X(s)^p\,\diff s\right]\right\}.
\end{equation}
On the other hand, considering the supremum with respect to $t\in[0,T_1\wedge\tau_N]$ and taking expectations in \eqref{eq_3/2-Ito}, by the Burkholder--Davis--Gundy inequality, we see that
\begin{align*}
	\bE\left[\sup_{t\in[0,T_1]}X(t\wedge\tau_N)^p\right]
	\leq C_p\left\{\xi^p+\bE\left[\int^{T_1\wedge\tau_N}_0X(s)^p\,\diff s\right]+\bE\left[\left(\int^{T_1\wedge\tau_N}_0X(s)^{2p+1}\,\diff s\right)^{1/2}\right]\right\}.
\end{align*}
From this, together with the estimate 
\begin{align*}
	C_p\bE\left[\left(\int^{T_1\wedge\tau_N}_0X(s)^{2p+1}\,\diff s\right)^{1/2}\right]
	&\leq C_p\bE\left[\sup_{s\in[0,T_1]}X(s\wedge\tau_N)^{p/2}\left(\int^{T_1\wedge\tau_N}_0X(s)^{p+1}\,\diff s\right)^{1/2}\right]\\
	&\leq\frac{1}{2}\bE\left[\sup_{s\in[0,T_1]}X(s\wedge\tau_N)^p\right]+C_p\bE\left[\int^{T_1\wedge\tau_N}_0X(s)^{p+1}\,\diff s\right],
\end{align*}
we obtain
\begin{equation}\label{eq_3/2-estimate2}
\begin{split}
	\bE\left[\sup_{t\in[0,T_1]}X(t\wedge\tau_N)^p\right]
	&\leq C_p\left\{\xi^p+\bE\left[\int^{T_1\wedge\tau_N}_0X(s)^p\,\diff s\right]+\bE\left[\int^{T_1\wedge\tau_N}_0X(s)^{p+1}\,\diff s\right]\right\}.
\end{split}
\end{equation}
By \eqref{eq_3/2-estimate1} and \eqref{eq_3/2-estimate2}, we get
\begin{align*}
	\bE\left[\sup_{t\in[0,T_1]}X(t\wedge\tau_N)^p\right]&\leq C_p\left\{\xi^p+\bE\left[\int^{T_1\wedge\tau_N}_0X(s)^p\,\diff s\right]\right\}
	\leq C_p\left\{\xi^p+\int^{T_1}_0\bE\left[\sup_{t\in[0,s]}X(t\wedge\tau_N)^p\right]\,\diff s\right\}.
\end{align*}
Noting that $T_1\in[0,T]$ is arbitrary, Gronwall's inequality yields
\begin{equation*}
	\bE\left[\sup_{t\in[0,T]}X(t\wedge\tau_N)^p\right]\leq C_pe^{C_p T}\xi^p.
\end{equation*}
From this estimate and \eqref{eq_3/2-estimate1} with $T_1=T$, letting $N\to\infty$ and using Fatou's lemma, we obtain the desired estimate in assertion (i).
\end{proof}

By \cref{lemm_3/2-moment} and \cref{theo_main} (iii), we can obtain the following convergence order for the weak approximation \eqref{eq_3/2-weakconv} under the slightly stronger condition $2\kappa_1+\theta^2>0$.


\begin{theo}\label{theo_3/2}
Let $\kappa_1,\kappa_2\in\bR$ and $\theta\in(0,\infty)$ satisfy $2\kappa_1+\theta^2>0$, and let $\weaksol$ be the weak solution of the SDE \eqref{eq_3/2} on $D=(0,\infty)$ with initial condition $X(0)=\xi\in D$.
For each $T\in(0,\infty)$ and $\pi\in\Pi_T$, let $X^\pi$ be the Euler--Maruyama scheme defined on $(\Omega^\pi,\cF^\pi,\bP^\pi)$ and given by \eqref{eq_EM-Markov} with initial condition $X^\pi(t_0)=\xi$.
Then, for any
\begin{equation*}
	\gamma\in\left(0,\frac{2\kappa_1+\theta^2}{4\kappa_1+7\theta^2}\right),
\end{equation*}
there exists a constant $C_\gamma\in(0,\infty)$ such that
\begin{equation*}
	d_\LP\big(\Law_\bP(X_T),\Law_{\bP^\pi}(\poly[X^\pi])\big)\leq C_\gamma|\pi|^\gamma
\end{equation*}
for any $\pi\in\Pi_{T}$.
\end{theo}


\begin{proof}
Recall the choices of the subsets in \eqref{eq_3/2-subset}. For any $\Delta\in(0,1]$ and $\vec{R}=(R_{\growth,b},R_{\growth,\sigma},R_{\conti,b},R_{\conti,\sigma},R_\ellip)\in[1,\infty)^5$, noting that $\Delta\leq1\leq\min\left\{R^{1/2}_{\growth,b},R^{2/3}_{\growth,\sigma},R_{\conti,b},R^2_{\conti,\sigma}\right\}$, we have
\begin{align*}
	&\bP\left(\inf_{t\in[0,T]}\dist\!\left(X(t),\bR\setminus D_T(\vec{R}) \right)\leq \Delta\right)\\
	&=\bP\left(\text{$X(t)\geq2\min\left\{R^{1/2}_{\growth,b},R^{2/3}_{\growth,\sigma},R_{\conti,b},R^2_{\conti,\sigma}\right\}-\Delta$}\vphantom{\frac{1}{R^{2/3}_\ellip}}\ 
	\text{or $X(t)\leq\frac{1}{R^{2/3}_\ellip}+\Delta$ for some $t\in[0,T]$}\right)\\
	&\leq\bP\left(\sup_{t\in[0,T]}X(t)\geq\min\left\{R^{1/2}_{\growth,b},R^{2/3}_{\growth,\sigma},R_{\conti,b},R^2_{\conti,\sigma}\right\}\right)
	+\bP\left(\sup_{t\in[0,T]}X(t)^{-1}\geq\left(\frac{1}{R^{2/3}_\ellip}+\Delta\right)^{-1}\right),
\end{align*}
and thus, by Markov's inequality and \cref{lemm_3/2-moment},
\begin{align*}
	&\bP\left(\inf_{t\in[0,T]}\dist\!\left(X(t),\bR\setminus D_T(\vec{R}) \right)\leq \Delta\right)\\
	&\leq\left(\min\left\{R^{1/2}_{\growth,b},R^{2/3}_{\growth,\sigma},R_{\conti,b},R^2_{\conti,\sigma}\right\}\right)^{-p}\bE\left[\sup_{t\in[0,T]}X(t)^p\right]
	+\left(\frac{1}{R^{2/3}_\ellip}+\Delta\right)^q\bE\left[\sup_{t\in[0,T]}X(t)^{-q}\right]\\
	&\leq C_{p,q}\max\left\{\Delta^q,R_{\growth,b}^{-p/2},R_{\growth,\sigma}^{-2p/3},R_{\conti,b}^{-p},R_{\conti,\sigma}^{-2p},R_\ellip^{-2q/3}\right\},
\end{align*}
for any $p\in(0,\frac{2\kappa_1}{\theta^2}+1)$ and $q\in(0,\infty)$. Here, $C_{p,q}\in(0,\infty)$ is a constant which does not depend on $\Delta$ or $\vec{R}$. The above estimate shows that \eqref{eq_rare-event} holds for any $\vec{\beta}=(\beta_0,\beta_{\growth,b},\beta_{\growth,\sigma},\beta_{\conti,b},\beta_{\conti,\sigma},\beta_\ellip)$ with
\begin{align*}
	&\beta_0,\beta_\ellip\in(0,\infty),\ \beta_{\growth,b}\in\left(0,\frac{\kappa_1}{\theta^2}+\frac{1}{2}\right),\ \beta_{\growth,\sigma}\in\left(0,\frac{4\kappa_1}{3\theta^2}+\frac{2}{3}\right),\\
	&\beta_{\conti,b}\in\left(0,\frac{2\kappa_1}{\theta^2}+1\right),\ \beta_{\conti,\sigma}\in\left(0,\frac{4\kappa_1}{\theta^2}+2\right).
\end{align*}
Concerning the constants $\beta_*$ and $\gamma_*$ defined by \eqref{eq_beta*} and \eqref{eq_gamma*}, we have
\begin{equation*}
\begin{split}
\beta_*&=
\min\left\{
\beta_0,
\frac{1}{1+\beta_\ellip^{-1}+\beta_{\conti,b}^{-1}},
\frac{1}{1+\beta_\ellip^{-1}+2\beta_{\conti,\sigma}^{-1}}
\right\}
\to
\frac{2\kappa_{1}+\theta^{2}}{2(\kappa_{1}+\theta^{2})},
\end{split}
\end{equation*}
and
\begin{equation*}
\begin{split}
\gamma_*
&=
\min\left\{
\frac{1}{1+\beta_\ellip^{-1}+\beta_{\conti,b}^{-1}+\beta_{\growth,b}^{-1}},
\frac{1}{2}\cdot\frac{1}{1+\beta_\ellip^{-1}+\beta_{\conti,b}^{-1}+\beta_{\growth,\sigma}^{-1}},\right.\\
&\hspace{1.5cm}\left.
\frac{2}{1+\beta_\ellip^{-1}+2\beta_{\conti,\sigma}^{-1}+2\beta_{\growth,b}^{-1}+\beta_*^{-1}},
\frac{1}{1+\beta_\ellip^{-1}+2\beta_{\conti,\sigma}^{-1}+2\beta_{\growth,\sigma}^{-1}+\beta_*^{-1}}
\right\}
\\&\to
\frac{2\kappa_{1}+\theta^{2}}{4\kappa_{1}+7\theta^{2}}
\end{split}
\end{equation*}
as $\beta_0,\beta_\ellip\to\infty$, $\beta_{\growth,b}\to\frac{\kappa_1}{\theta^2}+\frac{1}{2}$, $\beta_{\growth,\sigma}\to\frac{4\kappa_1}{3\theta^2}+\frac{2}{3}$, $\beta_{\conti,b}\to\frac{2\kappa_1}{\theta^2}+1$ and $\beta_{\conti,\sigma}\to\frac{4\kappa_1}{\theta^2}+2$. Therefore, by \cref{theo_main} (iii), we get the desired estimate. This completes the proof.
\end{proof}


\subsection{Dyson's Brownian motion}\label{subsec_Dyson}
Let $d=n\geq2$, and consider the following SDE defined on $D=\{x\in \bR^{n}|\,x_{1}>x_{2}>\cdots>x_{n}\}$:
\begin{equation}\label{eq_Dyson}
	\diff X_{i}(t)=\sum_{j \neq i}\frac{\kappa}{X_{i}(t)-X_{j}(t)}\,\diff t+\diff W_{i}(t),\ \ t\in[0,\infty),\ \ i\in\{1,\dots,n\},\ \ X(0)=\xi\in D,
\end{equation}
where $\kappa$ is a fixed positive constant. The SDE \eqref{eq_Dyson} has a unique strong solution $X=(X_1,\dots,X_n)^\top$ satisfying $\bP(\text{$X(t)\in D$ for any $t\in[0,\infty)$})=1$ if and only if $\kappa\geq1/2$. Indeed, the ``if part'' (that is, the sufficiency of $\kappa\geq1/2$) is shown in \cite[Theorem 3.1 and Proposition 4.1]{CeLe97} and \cite[Lemma 1]{RoSh93}, and the ``only if part'' (that is, the necessity of $\kappa\geq1/2$) can be shown by comparing the dynamics of the process $X_{i,i+1}:=\frac{1}{\sqrt{2}}(X_i-X_{i+1})$ (see \eqref{eq_Dyson-ij} below) with the Bessel process for each $i\in\{1,\dots,n-1\}$. The solution $X$ of the SDE \eqref{eq_Dyson} is called Dyson's Brownian motion, which arises in mathematical physics as a non-colliding particle system \cite{Dy62} and in random matrix theory as dynamics of the eigenvalues of some matrix-valued Brownian motions \cite{AGZ10,Dy62}.

Numerical approximations for Dyson's Brownian motion are studied in \cite{DNT26,NT20,NT24+}, where the authors provide convergence rates in the $L^{p}$-sup sense for some $p \geq 2$ (which depends on $\kappa$) by means of the backward/truncated Euler--Maruyama schemes assuming that $\kappa$ is sufficiently large; for example, \cite[Theorem 4.2]{DNT26} requires that $\kappa>9/2$.

Let $\kappa\geq1/2$. As discussed in \cref{subsubsec_Markov}, the Markovian SDE \eqref{eq_Dyson} on $D=\{x\in \bR^{n}|\,x_{1}>x_{2}>\cdots>x_{n}\}$ fits into the framework of the present paper. More precisely, the data satisfies \cref{assum} with the choices of the exponents $\alpha_b=\alpha_\sigma=1$, the sets
\begin{equation}\label{eq_Dyson-subset}
\begin{split}
	D_{\growth,b;T}(R)=\left\{x\in D\relmiddle|x_i-x_j>\frac{1}{R}\ \text{for any $1\leq i<j\leq n$}\right\},\\
	D_{\conti,b;T}(R)=\left\{x\in D\relmiddle|x_i-x_j>\frac{1}{R^{1/2}}\ \text{for any $1\leq i<j\leq n$}\right\},\\
	D_{\growth,\sigma;T}(R)=D_{\conti,\sigma;T}(R)=D_{\ellip;T}(R)=D,\ \ R\in[1,\infty),
\end{split}
\end{equation}
and some constants $K_{\growth,b;T},K_{\growth,\sigma;T},K_{\conti,b;T},K_{\conti,\sigma;T},K_{\ellip;T}$ depending only on $\kappa$ and $n$.

Since $\bP(\text{$X(t)\in D$ for any $t\in[0,\infty)$})=1$, \cref{theo_main} (i) shows that
\begin{equation}\label{eq_Dyson-weakconv}
	\poly[X^\pi]\to X_T\ \ \text{weakly on $\cC_T$ as $|\pi|\downarrow0$ along $\pi\in\Pi_T$ for any $T\in(0,\infty)$},
\end{equation}
where $X^\pi=\EM$ is the standard Euler--Maruyama scheme given by \eqref{eq_EM-Markov} with initial condition $X^\pi(t_0)=\xi$. In order to get a weak convergence order with respect to the L\'{e}vy--Prokhorov metric, we use \cref{theo_main} (iii). To do so, we first investigate moment estimates corresponding to the ``rare event'' appearing in \eqref{eq_rare-event}. In the literature, it has been shown that
\begin{equation*}
	\max_{1\leq i<j\leq n}\sup_{t\in[0,T]}\bE\left[\big(X_i(t)-X_j(t)\big)^{-p}\right]<\infty
\end{equation*}
for any $p \in (0,2\kappa-1)$ under the condition $\kappa>1/2$, and
\begin{equation*}
	\max_{1\leq i<j\leq n}\bE\left[\sup_{t\in[0,T]}\big(X_i(t)-X_j(t)\big)^{-p}\right]<\infty
\end{equation*}
for any $p \in (0,2\kappa-3)$ under the condition $\kappa>3/2$ (see \cite[Theorem 2.7]{DNT26}). Notice that there is a gap between the above two estimates in terms of the parameters $\kappa$ and $p$.
The following lemma fills the gap and improves the above results.


\begin{lemm}\label{lemm_Dyson-moment}
Assume that $\kappa>1/2$, and let $\weaksol$ be the weak solution of the SDE \eqref{eq_Dyson} on $D=\{x\in \bR^{n}|\,x_{1}>x_{2}>\cdots>x_{n}\}$ with initial condition $X(0)=\xi\in D$. Then, for any $p\in(0,2\kappa-1)$, there exists a constant $C_p\in(0,\infty)$, which depends only on $\kappa,n$ and $p$, such that
\begin{align*}
	&\max_{1\leq i<j\leq n}\bE\left[\sup_{t\in[0,\infty)}\big(X_i(t)-X_j(t)\big)^{-p}+\int^\infty_0\big(X_i(t)-X_j(t)\big)^{-p-2}\,\diff t\right]
	\leq C_p\max_{1\leq i<j\leq n}\big(\xi_i-\xi_j\big)^{-p}.
\end{align*}
\end{lemm}


\begin{proof}
For each $1\leq i<j\leq n$, we set $X_{i,j}:=\frac{1}{\sqrt{2}}(X_{i}-X_{j})$ and $W_{i,j}:=\frac{1}{\sqrt{2}}(W_{i}-W_{j})$.
Notice that $W_{i,j}$ is a one-dimensional standard Brownian motion and that, $\bP$-a.s., $X_{i,j}(t)>0$ for any $t\in[0,\infty)$. Moreover, $X_{i,j}$ evolves as
\begin{equation}\label{eq_Dyson-ij}
\begin{split}
	\diff X_{i,j}(t)&=\diff W_{i,j}(t)+\kappa X_{i,j}(t)^{-1}\,\diff t-\frac{\kappa}{2}\left\{\sum_{k\leq i-1}X_{k,i}(t)^{-1}+\sum_{\ell\geq j+1}X_{j,\ell}(t)^{-1}\right\}\,\diff t\\
	&\hspace{0.5cm}+\frac{\kappa}{2}\left\{\sum_{k\leq j-1,k\neq i}X_{k,j}(t)^{-1}+\sum_{\ell\geq i+1,\ell\neq j}X_{i,\ell}(t)^{-1}\right\}\,\diff t,\ \ t\in[0,\infty),
\end{split}
\end{equation}
with initial condition $X_{i,j}(0)=\xi_{i,j}:=\frac{1}{\sqrt{2}}(\xi_i-\xi_j)>0$. In the above expression, the summation with respect to an empty set of indexes is understood to be zero. Let $p\in(0,2\kappa-1)$ be fixed. We show that
\begin{equation}\label{eq_Dyson-estimate}
	\max_{1\leq i<j\leq n}\bE\left[\sup_{t\in[0,\infty)}X_{i,j}(t)^{-p}+\int^\infty_0X_{i,j}(t)^{-p-2}\,\diff t\right]\leq C_p\max_{1\leq i<j\leq n}\xi_{i,j}^{-p}
\end{equation}
for some constant $C_p\in(0,\infty)$ which depends only on $\kappa,n$ and $p$. To do so, we first show that there exist some positive constants $a_1,\dots,a_{n-1}$, which depend only on $\kappa,n$ and $p$, such that
\begin{equation}\label{eq_Dyson-sum}
\begin{split}
&\sum_{1\leq i<j\leq n}a_{j-i}X_{i,j}(t)^{-p}
+
p\sum_{1\leq i<j\leq n}
\int_{0}^{t}
X_{i,j}(s)^{-p-2}
\,\diff s
\\&\leq
\sum_{1\leq i<j\leq n}a_{j-i}\xi_{i,j}^{-p}
-
p \sum_{1\leq i<j\leq n}a_{j-i}
\int_{0}^{t}
X_{i,j}(s)^{-p-1}
\,\diff W_{i,j}(s)
\end{split}
\end{equation}
for any $t\in[0,\infty)$ $\bP$-a.s.

By It\^o's formula, for each $1\leq i<j\leq n$, we have
\begin{align*}
X_{i,j}(t)^{-p}
&=
\xi_{i,j}^{-p}
-
p\int_{0}^{t}
X_{i,j}(s)^{-p-1}
\,\diff W_{i,j}(s)
-p\left(
\kappa-\frac{p+1}{2}
\right)
\int_{0}^{t}
X_{i,j}(s)^{-p-2}
\,\diff s
\\&\quad
+\frac{p\kappa}{2}
\int_{0}^{t}
X_{i,j}(s)^{-p-1}
\left\{
\sum_{k \leq i-1} X_{k,i}(s)^{-1}
+
\sum_{\ell \geq j+1} X_{j,\ell}(s)^{-1}
\right\}
\,\diff s
\\&\quad
-\frac{p\kappa}{2}
\int_{0}^{t}
X_{i,j}(s)^{-p-1}
\left\{
\sum_{k\leq j-1,k\neq i} X_{k,j}(s)^{-1}
+
\sum_{\ell\geq i+1,\ell\neq j} X_{i,\ell}(s)^{-1}
\right\}
\,\diff s,
\end{align*}
and hence
\begin{align*}
	&X_{i,j}(t)^{-p}+p\left(\kappa-\frac{p+1}{2}\right)\int^t_0X_{i,j}(s)^{-p-2}\,\diff s\\
	&\leq\xi_{i,j}^{-p}-p\int^t_0X_{i,j}(s)^{-p-1}\,\diff W_{i,j}(s)
	+\frac{p\kappa}{2}\int^t_0X_{i,j}(s)^{-p-1}\left\{\sum_{k\leq i-1}X_{k,i}(s)^{-1}+\sum_{\ell\geq j+1}X_{j,\ell}(s)^{-1}\right\}\,\diff s.
\end{align*}
Let $a_1,\dots,a_{n-1}$ be positive constants, which will be determined later. Multiplying $a_{j-i}$ to both sides of the above inequality and then summing up over $1\leq i<j\leq n$, we get
\begin{equation}\label{eq_Dyson-estimate1}
\begin{split}
	&\sum_{1\leq i<j\leq n}a_{j-i}X_{i,j}(t)^{-p}+p\left(\kappa-\frac{p+1}{2}\right)\sum_{1\leq i<j\leq n}a_{j-i}\int^t_0X_{i,j}(s)^{-p-2}\,\diff s\\
	&\leq\sum_{1\leq i<j\leq n}a_{j-i}\xi_{i,j}^{-p}-p\sum_{1\leq i<j\leq n}a_{j-i}\int^t_0X_{i,j}(s)^{-p-1}\,\diff W_{i,j}(s)\\
	&\hspace{0.5cm}+\frac{p\kappa}{2}\int^t_0\sum_{1\leq i<j\leq n}a_{j-i}X_{i,j}(s)^{-p-1}\left\{\sum_{k\leq i-1}X_{k,i}(s)^{-1}+\sum_{\ell\geq j+1}X_{j,\ell}(s)^{-1}\right\}\,\diff s.
\end{split}
\end{equation}
We now consider the integrand of the last term in the right hand side of \eqref{eq_Dyson-estimate1}. Observe that
\begin{align*}
&\sum_{1\leq i<j\leq n}a_{j-i}
X_{i,j}(s)^{-p-1}
\left\{
\sum_{k \leq i-1} X_{k,i}(s)^{-1}
+
\sum_{\ell \geq j+1} X_{j,\ell}(s)^{-1}
\right\}
\\&=
\sum_{1\leq i<j<k\leq n}
\Big\{
a_{j-i}X_{i,j}(s)^{-p-1}X_{j,k}(s)^{-1}
+
a_{k-j}X_{i,j}(s)^{-1}X_{j,k}(s)^{-p-1}
\Big\}.
\end{align*}
We fix $1 \leq i<j<k\leq n$. On the one hand, if
\begin{equation*}
	X_{i,j}(s)\geq(a_{j-i}+a_{k-j}+1)X_{j,k}(s)\ \text{or}\ X_{j,k}(s) \geq (a_{j-i}+a_{k-j}+1)X_{i,j}(s),
\end{equation*}
then it holds that
\begin{align*}
a_{j-i}X_{i,j}(s)^{-p-1}X_{j,k}(s)^{-1}
+
a_{k-j}X_{i,j}(s)^{-1}X_{j,k}(s)^{-p-1}
\leq
X_{i,j}(s)^{-p-2}+X_{j,k}(s)^{-p-2}.
\end{align*}
On the other hand, if
\begin{equation*}
	X_{i,j}(s)<(a_{j-i}+a_{k-j}+1)X_{j,k}(s)\ \text{and}\ X_{j,k}(s)<(a_{j-i}+a_{k-j}+1)X_{i,j}(s),
\end{equation*}
then we have
\begin{align*}
X_{i,k}(s)=X_{i,j}(s)+X_{j,k}(s)<(a_{j-i}+a_{k-j}+2)(X_{i,j}(s) \wedge X_{j,k}(s)),
\end{align*}
and hence it holds that
\begin{align*}
	&a_{j-i}X_{i,j}(s)^{-p-1}X_{j,k}(s)^{-1}+a_{k-j}X_{i,j}(s)^{-1}X_{j,k}(s)^{-p-1}
	<\big(a_{j-i}+a_{k-j}+2\big)^{p+3}X_{i,k}(s)^{-p-2}.
\end{align*}
Therefore, we obtain
\begin{align*}
	&\sum_{1\leq i<j\leq n}a_{j-i}X_{i,j}(s)^{-p-1}\left\{\sum_{k \leq i-1}X_{k,i}(s)^{-1}+\sum_{\ell \geq j+1} X_{j,\ell}(s)^{-1}\right\}\\
	&\leq\sum_{1\leq i<j<k\leq n}\Big\{X_{i,j}(s)^{-p-2}+X_{j,k}(s)^{-p-2}+\big(a_{j-i}+a_{k-j}+2\big)^{p+3}X_{i,k}(s)^{-p-2}\Big\}\\
	&=\sum_{1\leq i<j\leq n}\left\{n-j+i-1+\sum_{1\leq \ell\leq j-i-1}\big(a_\ell+a_{j-i-\ell}+2\big)^{p+3}\right\}X_{i,j}(s)^{-p-2}.
\end{align*}
From the above estimate, together with \eqref{eq_Dyson-estimate1}, we obtain
\begin{align*}
&\sum_{1\leq i<j\leq n}a_{j-i}X_{i,j}(t)^{-p}
+
p\sum_{1\leq i<j\leq n}
\widehat{a}_{j-i}
\int_{0}^{t}
X_{i,j}(s)^{-p-2}
\,\diff s
\\&\leq
\sum_{1\leq i<j\leq n}a_{j-i}\xi_{i,j}^{-p}
-
p \sum_{1\leq i<j\leq n}a_{j-i}
\int_{0}^{t}
X_{i,j}(s)^{-p-1}
\,\diff W_{i,j}(s),
\end{align*}
where $\widehat{a}_1,\dots,\widehat{a}_{n-1}$ are defined by
\begin{equation*}
	\widehat{a}_m:=a_m\left(\kappa-\frac{p+1}{2}\right)-\frac{\kappa}{2}\left\{n-m-1+\sum_{1\leq \ell\leq m-1}\big(a_\ell+a_{m-\ell}+2\big)^{p+3}\right\}
\end{equation*}
for $m\in\{1,\dots,n-1\}$. Notice that the summation term in the expression of $\widehat{a}_m$ is equal to zero for $m=1$ and independent of $a_m,\dots,a_{n-1}$ for $2\leq m\leq n-1$. Thus, noting that $\kappa-\frac{p+1}{2}>0$, we can set positive constants $a_1,\dots,a_{n-1}$ inductively such that $\widehat{a}_m=1$ for any $m\in\{1,\dots,n-1\}$. Then, we get the estimate \eqref{eq_Dyson-sum}. Notice that the constants $a_1,\dots,a_{n-1}$ depend only on $\kappa,n$ and $p$.

For $N \in \bN$, define a stopping time $\tau_{N}$ by
\begin{equation*}
	\tau_N:=\inf\left\{t\geq0\relmiddle| X_{i,j}(t)\leq1/N\, \text{for some}\,1\leq i<j\leq n \right\}.
\end{equation*}
Notice that $\tau_{N}\to\infty$ as $N\to \infty$ $\bP$-a.s. By \eqref{eq_Dyson-sum}, we have
\begin{equation}\label{eq_Dyson-estimate2}
	\max_{1\leq i<j\leq n}\bE\left[\int_{0}^{\tau_{N}}
X_{i,j}(s)^{-p-2}\,\diff s\right]\leq C_p\max_{1\leq i<j\leq n}\xi_{i,j}^{-p}.
\end{equation}
Here and in the rest of this proof, $C_p\in(0,\infty)$ denotes a constant which depends only on $\kappa,n$ and $p$ and varies from line to line. Taking supremum with respect to $t\in[0,\tau_N]$ and then taking the expectations in \eqref{eq_Dyson-sum}, we have
\begin{align*}
	&a_{j'-i'}\bE\left[\sup_{t\in[0,\tau_N]}X_{i',j'}(t)^{-p}\right]\\
	&\leq\sum_{1\leq i<j\leq n}a_{j-i}\xi_{i,j}^{-p}+p\bE\left[\sup_{t\in[0,\tau_N]}\left|\sum_{1\leq i<j\leq n}a_{j-i}\int^t_0X_{i,j}(s)^{-p-1}\,\diff W_{i,j}(s)\right|\right]\\
	&\leq\sum_{1\leq i<j\leq n}a_{j-i}\xi_{i,j}^{-p}+p\sum_{1\leq i<j\leq n}a_{j-i}\bE\left[\sup_{t\in[0,\tau_N]}\left|\int^t_0X_{i,j}(s)^{-p-1}\,\diff W_{i,j}(s)\right|\right]
\end{align*}
for any $1\leq i'<j'\leq n$, and thus,
\begin{equation}\label{eq_Dyson-estimate3}
\begin{split}
	\max_{1\leq i<j\leq n}\bE\left[\sup_{t\in[0,\tau_N]}X_{i,j}(t)^{-p}\right]
	&\leq C_p\max_{1\leq i<j\leq n}\xi_{i,j}^{-p}+C_p\max_{1\leq i<j\leq n}\bE\left[\sup_{t\in[0,\tau_N]}\left|\int^t_0X_{i,j}(s)^{-p-1}\,\diff W_{i,j}(s)\right|\right].
\end{split}
\end{equation}
Concerning the second term above, by using the Burkholder--Davis--Gundy inequality and Young's inequality, we get
\begin{align}
	\nonumber
	&C_p\max_{1\leq i<j\leq n}\bE\left[\sup_{t\in[0,\tau_N]}\left|\int^t_0X_{i,j}(s)^{-p-1}\,\diff W_{i,j}(s)\right|\right]\\
	\nonumber
	&\leq C_p\max_{1\leq i<j\leq n}\bE\left[\left(\int^{\tau_N}_0X_{i,j}(s)^{-2p-2}\,\diff s\right)^{1/2}\right]\\
	\nonumber
	&\leq C_p\max_{1\leq i<j\leq n}\bE\left[\sup_{t\in[0,\tau_N]}X_{i,j}(t)^{-p/2}\left(\int^{\tau_N}_0X_{i,j}(s)^{-p-2}\,\diff s\right)^{1/2}\right]\\
	\label{eq_Dyson-estimate4}
	&\leq\frac{1}{2}\max_{1\leq i<j\leq n}\bE\left[\sup_{t\in[0,\tau_N]}X_{i,j}(t)^{-p}\right]+C_p\max_{1\leq i<j\leq n}\bE\left[\int^{\tau_N}_0X_{i,j}(s)^{-p-2}\,\diff s\right]
\end{align}
By \eqref{eq_Dyson-estimate2}, \eqref{eq_Dyson-estimate3} and \eqref{eq_Dyson-estimate4}, we obtain
\begin{equation*}
	\max_{1\leq i<j\leq n}\bE\left[\sup_{t\in[0,\tau_N]}X_{i,j}(t)^{-p}+\int^{\tau_N}_0X_{i,j}(s)^{-p-2}\,\diff s\right]\leq C_p\max_{1\leq i<j\leq n}\xi_{i,j}^{-p}.
\end{equation*}
By taking the limit $N \to \infty$ and using Fatou's lemma, we obtain the desired estimate \eqref{eq_Dyson-estimate}. This completes the proof.
\end{proof}

By \cref{lemm_Dyson-moment} and \cref{theo_main} (iii), we can obtain the following convergence order for the weak approximation \eqref{eq_Dyson-weakconv} under the slightly stronger condition $\kappa>1/2$.


\begin{theo}\label{theo_Dyson}
Assume that $\kappa>1/2$, and let $\weaksol$ be the weak solution of the SDE \eqref{eq_Dyson} on $D=\{x\in \bR^{n}|\,x_{1}>x_{2}>\cdots>x_{n}\}$ with initial condition $X(0)=\xi\in D$.
For each $T\in(0,\infty)$ and $\pi\in\Pi_T$, let $X^\pi$ be the Euler--Maruyama scheme defined on $(\Omega^\pi,\cF^\pi,\bP^\pi)$ and given by \eqref{eq_EM-Markov} with initial condition $X^\pi(t_0)=\xi$.
Then, for any
\begin{equation*}
	\gamma\in\left(0,\frac{2\kappa-1}{2(2\kappa+1)}\right),
\end{equation*}
there exists a constant $C_\gamma\in(0,\infty)$ such that
\begin{equation*}
	d_\LP\big(\Law_\bP(X_T),\Law_{\bP^\pi}(\poly[X^\pi])\big)\leq C_\gamma|\pi|^\gamma
\end{equation*}
for any $\pi\in\Pi_{T}$.
\end{theo}


\begin{proof}
Noting \eqref{eq_Dyson-subset}, for any $\Delta\in(0,1]$ and $\vec{R}=(R_{\growth,b},R_{\growth,\sigma},R_{\conti,b},R_{\conti,\sigma},R_\ellip)\in[1,\infty)^5$, we have
\begin{align*}
	&\bP\left(\inf_{t\in[0,T]}\dist\!\left(X(t),\bR^n\setminus D_T(\vec{R}) \right)\leq \Delta\right)\\
	&\leq\bP\left(X_i(t)-X_j(t)\leq\frac{1}{R_{\growth,b}\wedge R^{1/2}_{\conti,b}}+\sqrt{2}\Delta\ \ \text{for some $1\leq i<j\leq n$ and $t\in[0,T]$}\right)\\
	&\leq\sum_{1\leq i<j\leq n}\bP\left(\sup_{t\in[0,T]}\big(X_i(t)-X_j(t)\big)^{-1}\geq\left(\frac{1}{R_{\growth,b}\wedge R^{1/2}_{\conti,b}}+\sqrt{2}\Delta\right)^{-1}\right),
\end{align*}
and thus, by Markov's inequality and \cref{lemm_Dyson-moment},
\begin{align*}
	\bP\left(\inf_{t\in[0,T]}\dist\!\left(X(t),\bR^n\setminus D_T(\vec{R}) \right)\leq \Delta\right)
	&\leq\left(\frac{1}{R_{\growth,b}\wedge R^{1/2}_{\conti,b}}+\sqrt{2}\Delta\right)^p\sum_{1\leq i<j\leq n}\bE\left[\sup_{t\in[0,T]}\big(X_i(t)-X_j(t)\big)^{-p}\right]\\
	&\leq C_p\max\left\{\Delta^p,R_{\growth,b}^{-p},R_{\conti,b}^{-p/2}\right\},
\end{align*}
for any $p\in(0,2\kappa-1)$. Here, $C_p\in(0,\infty)$ is a constant which does not depend on $\Delta$ or $\vec{R}$. The above estimate shows that \eqref{eq_rare-event} holds for any $\vec{\beta}=(\beta_0,\beta_{\growth,b},\beta_{\growth,\sigma},\beta_{\conti,b},\beta_{\conti,\sigma},\beta_\ellip)$ with
\begin{equation*}
	\beta_0,\beta_{\growth,b}\in(0,2\kappa-1),\ \beta_{\conti,b}\in\left(0,\kappa-\frac{1}{2}\right),\ \beta_{\growth,\sigma},\beta_{\conti,\sigma},\beta_\ellip\in(0,\infty).
\end{equation*}
Concerning the constants $\beta_*$ and $\gamma_*$ defined by \eqref{eq_beta*} and \eqref{eq_gamma*}, we have
\begin{equation*}
\begin{split}
\beta_*&=
\min\left\{
\beta_0,
\frac{1}{1+\beta_\ellip^{-1}+\beta_{\conti,b}^{-1}},
\frac{1}{1+\beta_\ellip^{-1}+2\beta_{\conti,\sigma}^{-1}}
\right\}
\to
\frac{2\kappa-1}{2\kappa+1}
\end{split}
\end{equation*}
and
\begin{equation*}
\begin{split}
\gamma_*
&=
\min\left\{
\frac{1}{1+\beta_\ellip^{-1}+\beta_{\conti,b}^{-1}+\beta_{\growth,b}^{-1}},
\frac{1}{2}\cdot\frac{1}{1+\beta_\ellip^{-1}+\beta_{\conti,b}^{-1}+\beta_{\growth,\sigma}^{-1}},\right.\\
&\hspace{1.5cm}\left.
\frac{2}{1+\beta_\ellip^{-1}+2\beta_{\conti,\sigma}^{-1}+2\beta_{\growth,b}^{-1}+\beta_*^{-1}},
\frac{1}{1+\beta_\ellip^{-1}+2\beta_{\conti,\sigma}^{-1}+2\beta_{\growth,\sigma}^{-1}+\beta_*^{-1}}\right\}
\\&\to
\frac{2\kappa-1}{2(2\kappa+1)}
\end{split}
\end{equation*}
as $\beta_0,\beta_{\growth,b}\to2\kappa-1$, $\beta_{\conti,b}\to\kappa-\frac{1}{2}$ and $\beta_{\growth,\sigma},\beta_{\conti,\sigma},\beta_\ellip\to\infty$. Therefore, noting \cref{rem_gamma*}, by \cref{theo_main} (iii), we get the desired estimate. This completes the proof.
\end{proof}




\setcounter{theo}{0}
\setcounter{equation}{0}

\begin{appendix}
\section*{Appendix}
\phantomsection
\label{appendix}
\addcontentsline{toc}{section}{Appendix}


\section{Quantitative estimate for L\'evy--Prokhorov metric under transformation}\label{appendix_LP}

The following lemma provides a quantitative estimate for the L\'evy--Prokhorov metric under uniformly continuous transformation. This ensures that quantitative estimates for the weak approximation of the transformed SFDEs in \cref{subsubsec_reflected} and \cref{subsubsec_oscillator} are carried over to those of the original reflected SDEs and stochastic oscillator models.


\begin{lemm}\label{appendix_lemm_LP}
Let $(S,d_S)$ and $(T,d_T)$ be complete separable metric spaces, and let $f:S\to T$ be a uniformly continuous map. Let $\rho_f:[0,\infty)\to[0,\infty)$ be a strictly increasing, continuous function with $\rho_f(0)=0$ and $\rho_f(a)\geq a$ for any $a\in[0,1]$ such that
\begin{equation*}
	d_T\big(f(s_1),f(s_2)\big)\leq\rho_f\big(d_S(s_1,s_2)\big)
\end{equation*}
for all $s_1,s_2\in S$. Then, it holds that
\begin{equation}\label{appendix_eq_LP}
	d_{\LP,T}\big(\mu_1\circ f^{-1},\mu_2\circ f^{-1}\big)\leq\rho_f\big(d_{\LP,S}(\mu_1,\mu_2)\big)
\end{equation}
for all $\mu_1,\mu_2\in\cP(S)$, where $d_{\LP,S}:\cP(S)\times\cP(S)\to[0,1]$ and $d_{\LP,T}:\cP(T)\times\cP(T)\to[0,1]$ denote the L\'evy--Prokhorov metrics on $\cP(S)$ and $\cP(T)$ corresponding to $(S,d_S)$ and $(T,d_T)$, respectively.
\end{lemm}


\begin{proof}
Let $\mu_1,\mu_2\in\cP(S)$ be given. Since the left-hand side of \eqref{appendix_eq_LP} is not greater than $1$, we may assume that $\rho_f(d_{\LP,S}(\mu_1,\mu_2))<1$. Let $\rho_f(d_{\LP,S}(\mu_1,\mu_2))<\ep<1$. Since $\rho_f$ is strictly increasing and satisfies $\rho_f(a)\geq a$ for any $a\in[0,1]$, we have $d_{\LP,S}(\mu_1,\mu_2)<\rho_f^{-1}(\ep)\leq\ep$. Hence, by the coupling representation \eqref{eq_LP-coupling} of the L\'evy--Prokhorov metric, there exists a coupling $\fm\in\cP(S\times S)$ between the probability measures $\mu_1,\mu_2\in\cP(S)$ such that
\begin{equation*}
	\fm\left(\left\{(s_1,s_2)\in S\times S\relmiddle|d_S(s_1,s_2)>\rho_f^{-1}(\ep)\right\}\right)<\rho_f^{-1}(\ep)\leq\ep.
\end{equation*}
Define a continuous map $F:S\times S\to T\times T$ by $F(s_1,s_2):=(f(s_1),f(s_2))$ for each $(s_1,s_2)\in S\times S$, and consider the probability measure $\fm\circ F^{-1}\in\cP(T\times T)$. Clearly, $\fm\circ F^{-1}$ is a coupling between the probability measures $\mu_1\circ f^{-1},\mu_2\circ f^{-1}\in\cP(T)$. Furthermore,
\begin{align*}
	\big(\fm\circ F^{-1}\big)\left(\left\{(t_1,t_2)\in T\times T\relmiddle|d_T(t_1,t_2)>\ep\right\}\right)
	&=\fm\left(\left\{(s_1,s_2)\in S\times S\relmiddle|d_T\big(f(s_1),f(s_2)\big)>\ep\right\}\right)\\
	&\leq\fm\left(\left\{(s_1,s_2)\in S\times S\relmiddle|\rho_f\big(d_S(s_1,s_2)\big)>\ep\right\}\right)\\
	&=\fm\left(\left\{(s_1,s_2)\in S\times S\relmiddle|d_S(s_1,s_2)>\rho_f^{-1}(\ep)\right\}\right)\\
	&<\ep.
\end{align*}
Again by noting \eqref{eq_LP-coupling}, the above observations show that
\begin{equation*}
	d_{\LP,T}\big(\mu_1\circ f^{-1},\mu_2\circ f^{-1}\big)\leq\ep.
\end{equation*}
Since $\rho_f(d_{\LP,S}(\mu_1,\mu_2))<\ep<1$ is arbitrary, we obtain the desired estimate \eqref{appendix_eq_LP}.
\end{proof}


\section{A priori moment estimates for weak solutions of SFDEs}\label{appendix_LG}

The following lemma provides a standard a priori moment estimates for weak solutions of SFDEs under the one-sided linear-growth type conditions.


\begin{lemm}\label{appendix_lemm_LG}
Let data $\data$ be given. Suppose that we are given a weak solution $\weaksol$ to the SFDE \eqref{eq_SFDE} associated with $\data$.
\begin{itemize}
\item[(i)]
Suppose that there exists a constant $\widehat{C}\in(0,\infty)$ such that
\begin{equation}\label{appendix_eq_LG}
	\langle x(t),b(t,x)\rangle\leq\widehat{C}\big(1+\|x_t\|_\infty^2\big)\ \ \text{and}\ \ |\sigma(t,x)|^2\leq\widehat{C}\big(1+\|x_t\|_\infty^2\big)
\end{equation}
for any $(t,x)\in\Lambda[\supp\mu_0,D]$ and that $M_p(\mu_0):=\int_{\bR^n}|\xi|^p\,\mu_0(\diff\xi)<\infty$ for some $p\in[2,\infty)$. Then, there exists a constant $C_p\in(0,\infty)$, which depends only on $\widehat{C}$ and $p$, such that
\begin{equation*}
	\bE\big[\|X_T\|_\infty^p\big]\leq C_pe^{C_pT}\Big(1+M_p(\mu_0)\Big)
\end{equation*}
for any $T\in(0,\infty)$.
\item[(ii)]
Let $D=(0,\infty)^n$. Suppose that there exists a constant $\widehat{C}\in(0,\infty)$ such that
\begin{equation}\label{appendix_eq_LG-negative}
	\sum^n_{i=1}x_i(t)^{-3}b_i(t,x)\geq-\widehat{C}\big(1+\|x_t^{-1}\|_\infty^2\big)\ \ \text{and}\ \ \sum^n_{i=1}x_i(t)^{-4}\sum^d_{j=1}\sigma_{i,j}(t,x)^2\leq\widehat{C}\big(1+\|x_t^{-1}\|_\infty^2\big)
\end{equation}
for any $(t,x)\in\Lambda[\supp\mu_0,(0,\infty)^n]$ and that $\check{M}_p(\mu_0):=\int_{(0,\infty)^n}|\xi^{-1}|^p\,\mu_0(\diff\xi)<\infty$ for some $p\in[2,\infty)$. Then, there exists a constant $C_p\in(0,\infty)$, which depends only on $\widehat{C}$ and $p$, such that
\begin{equation*}
	\bE\big[\|X_T^{-1}\|_\infty^p\big]\leq C_pe^{C_pT}\Big(1+\check{M}_p(\mu_0)\Big)
\end{equation*}
for any $T\in(0,\infty)$.
\end{itemize}
\end{lemm}


\begin{proof}
First, we show assertion (i). Let $p\in[2,\infty)$ be fixed. In this proof, we denote by $C_p$ a positive constant depending only on $\widehat{C}$ and $p$ which may vary from line to line. Using It\^{o}'s formula, we have
\begin{align*}
	|X(t)|^p&=|X(0)|^p+\int^t_0\Big\{p|X(s)|^{p-2}\langle X(s),b(s,X)\rangle+\frac{p}{2}|X(s)|^{p-2}|\sigma(s,X)|^2\\
	&\hspace{3cm}+\frac{p(p-2)}{2}|X(s)|^{p-4}|\sigma(s,X)^\top X(s)|^2\1_{\bR^n\setminus\{0\}}(X(s))\Big\}\,\diff s\\
	&\hspace{1cm}+\int^t_0p|X(s)|^{p-2}\big\langle X(s),\sigma(s,X)\,\diff W(s)\big\rangle,\ \ t\in[0,\infty).
\end{align*}
By the assumption \eqref{appendix_eq_LG}, we obtain
\begin{equation}\label{appendix_eq_LG-estimate}
\begin{split}
	&|X(t)|^p\leq|X(0)|^p+C_pt+C_p\int^t_0\|X_s\|_\infty^p\,\diff s+p\int^t_0|X(s)|^{p-2}\big\langle X(s),\sigma(s,X)\,\diff W(s)\big\rangle,\ 
	t\in[0,\infty).
\end{split}
\end{equation}
For each $N\in\bN$, define $\tau_N:=\inf\{t\geq0\,|\,|X(t)|\geq N\}$ and $a_N(t):=\bE[\|X_{t\wedge\tau_N}\|_\infty^p]$, $t\in[0,\infty)$. Each $\tau_N$ is a stopping time such that $\tau_N\to\infty$ $\bP$-a.s. Thanks to \eqref{appendix_eq_LG-estimate}, by the Burkholder--Davis--Gundy inequality, second estimate in \eqref{appendix_eq_LG} and Young's inequality, we have
\begin{align*}
	a_N(t)&\leq M_p(\mu_0)+C_pt+C_p\int^t_0a_N(s)\,\diff s+C_p\bE\left[\left(\int^{t\wedge\tau_N}_0|X(s)|^{2p-2}(1+\|X_s\|_\infty^2)\,\diff s\right)^{1/2}\right]\\
	&\leq M_p(\mu_0)+C_pt+C_p\int^t_0a_N(s)\,\diff s+C_p\bE\left[\left(\int^t_0\|X_{s\wedge\tau_N}\|_\infty^p\,\diff s\right)^{1/2}\big(1+\|X_{t\wedge\tau_N}\|_\infty^p\big)^{1/2}\right]\\
	&\leq M_p(\mu_0)+C_pt+C_p\int^t_0a_N(s)\,\diff s+\frac{1}{2}\big(1+a_N(t)\big),
\end{align*}
and hence
\begin{equation*}
	a_N(t)\leq1+2M_p(\mu_0)+C_pt+C_p\int^t_0a_N(s)\,\diff s
\end{equation*}
for any $t\in[0,\infty)$. Hence, Gronwall's inequality yields
\begin{equation*}
	a_N(T)\leq C_pe^{C_pT}\big(1+M_p(\mu_0)\big)
\end{equation*}
for any $T\in(0,\infty)$. Then, letting $N\to\infty$ and using Fatou's lemma, we obtain the desired estimate in assertion (i).

Next, we show assertion (ii). By using It\^{o}'s formula, we see that $(X^{-1},W,\Omega,\cF,\bF,\bP)$ is a weak solution of the SFDE \eqref{eq_SFDE} with data $((0,\infty)^n,\check{\mu}_0,\check{b},\check{\sigma})$, where $\check{\mu}_0:=\mu_0\circ(\xi\mapsto \xi^{-1})^{-1}$, and
\begin{align*}
	&\check{b}_i(t,x):=-x_i(t)^2b_i(t,x^{-1})+x_i(t)^3\sum^d_{j=1}\sigma_{i,j}(t,x^{-1})^2,\ \ i\in\{1,\dots,n\},\\
	&\check{\sigma}_{i,j}(t,x):=-x_i(t)^2\sigma_{i,j}(t,x^{-1}),\ \ i\in\{1,\dots,n\},\ \ j\in\{1,\dots,d\},
\end{align*}
for $(t,x)\in\Lambda[\supp\check{\mu}_0,(0,\infty)^n]$. Clearly, the assumption \eqref{appendix_eq_LG-negative} implies that $\check{b}$ and $\check{\sigma}$ satisfy \eqref{appendix_eq_LG}. Hence, applying assertion (i) to $X^{-1}$ instead of $X$, we obtain the conclusion of (ii). This completes the proof.
\end{proof}


\section{Existence and uniqueness of the controlled Euler--Maruyama scheme}\label{appendix_controlledEM}

In this section, we show existence and uniqueness of the controlled Euler--Maruyama scheme $\hX^\pi$ defined as the solution of the non-standard SFDE \eqref{eq_controlledEM}. To do so, we need the following standard lemma.


\begin{lemm}\label{appendix_lemm_stopping}
Let $x,y:[0,\infty)\to\bR$ be right continuous functions, and define
\begin{equation*}
	\tau_x:=\inf\left\{t\geq0\relmiddle|x(t)=0\right\}\ \ \text{and}\ \ \tau_y:=\inf\left\{t\geq0\relmiddle|y(t)=0\right\}.
\end{equation*}
If $\tau_x\wedge\tau_y<\infty$ and $x(\tau_x\wedge\tau_y)=y(\tau_x\wedge\tau_y)$, then $\tau_x=\tau_y$.
\end{lemm}


\begin{proof}
Assume that $\tau_x<\tau_y$. Then, we have $x(\tau_x)=x(\tau_x\wedge\tau_y)=y(\tau_x\wedge\tau_y)=y(\tau_x)$. The definition of $\tau_y$ and the assumption $\tau_x<\tau_y$ yield that $y(\tau_x)\neq0$. However, the definition of $\tau_x$ and the right-continuity of $x$ yield that $x(\tau_x)=0$, leading a contradiction. Hence, we have $\tau_y\leq\tau_x$. Similarly, we can show that $\tau_x\leq\tau_y$. This completes the proof.
\end{proof}

Using the above standard fact, we show existence and uniqueness of the solution of a general version of \eqref{eq_controlledEM}.


\begin{lemm}\label{appendix_lemm_controlledEM}
Let $(\Omega,\cF,\bF,\bP)$ be a filtered probability space supporting a $d$-dimensional Brownian motion $W$ relative to $\bF$ and an $\bR^n$-valued $\cF_0$-measurable random variable $\xi$. Let $\bar{b}:\Lambda\to\bR^n$ and $\bar{\sigma}:\Lambda\to\bR^{n\times d}$ be progressively measurable maps, and let $f:\Omega\times[0,\infty)\times\bR^n\to\bR^n$ and $g:\Omega\times[0,\infty)\times\bR^n\to\bR$ be $\bF$-progressively measurable maps. Assume that $\bar{b},\bar{\sigma},f,g$ satisfy the following properties:
\begin{itemize}
\item
There exists an increasing sequence $(t_k)^\infty_{k=0}\subset[0,\infty)$ with $t_0=0$ and $\lim_{k\to\infty}t_k=\infty$ such that $\bar{b}(t,y)=\sum^\infty_{k=0}\bar{b}(t_k,y)\1_{[t_k,t_{k+1})}(t)$ and $\bar{\sigma}(t,y)=\sum^\infty_{k=0}\bar{\sigma}(t_k,y)\1_{[t_k,t_{k+1})}(t)$ for any $(t,y)\in\Lambda$;
\item
There exists a constant $L\in(0,\infty)$ such that $|f(\omega,t,\eta_1)-f(\omega,t,\eta_2)|\leq L|\eta_1-\eta_2|$ for any $\eta_1,\eta_2\in\bR^n$ and $(\omega,t)\in\Omega\times[0,\infty)$. Furthermore, it holds that $\int^T_0|f(\omega,t,0)|\diff t<\infty$ for any $T\in(0,\infty)$ and $\omega\in\Omega$;
\item
The map $g(\omega,\cdot,\cdot):[0,\infty)\times\bR^n\to\bR$ is continuous for any $\omega\in\Omega$.
\end{itemize}
Then, there exists a unique (up to $\bP$-indistinguishability) $\bR^n$-valued continuous $\bF$-adapted process $Y$ on $(\Omega,\cF,\bP)$ such that
\begin{equation}\label{appendix_eq_SFDE-stopped}
	\begin{dcases}
	\diff Y(t)=\bar{b}(t,Y)\,\diff t+\bar{\sigma}(t,Y)\,\diff W(t)+f(t,Y(t))\1_{[0,\tau)}(t)\,\diff t,\ \ t\in[0,\infty),\\
	Y(0)=\xi,\ \ \tau=\inf\left\{t\geq0\relmiddle|g(t,Y(t))=0\right\}.
	\end{dcases}
\end{equation}
\end{lemm}


\begin{proof}
First, we show that there exists a unique $\bR^n$-valued continuous $\bF$-adapted process $\widetilde{Y}$ on $(\Omega,\cF,\bP)$ such that
\begin{equation}\label{appendix_eq_SFDEtilde}
	\begin{dcases}
	\diff\widetilde{Y}(t)=\bar{b}(t,\widetilde{Y})\,\diff t+\bar{\sigma}(t,\widetilde{Y})\,\diff W(t)+f(t,\widetilde{Y}(t))\,\diff t,\ \ t\in[0,\infty),\\
	\widetilde{Y}(0)=\xi.
	\end{dcases}
\end{equation}
Noting the assumptions that
\begin{equation*}
	\bar{b}(t,y)=\sum^\infty_{k=0}\bar{b}(t_k,y)\1_{[t_k,t_{k+1})}(t)\ \ \text{and}\ \ \bar{\sigma}(t,y)=\sum^\infty_{k=0}\bar{\sigma}(t_k,y)\1_{[t_k,t_{k+1})}(t),
\end{equation*}
we can construct $(\widetilde{Y}(t))_{t\in(t_k,t_{k+1}]}$ for each $k\in\bN\cup\{0\}$ by the step-by-step argument. Indeed, assuming that the stopped continuous $\bF$-adapted process $\widetilde{Y}_{t_k}=\widetilde{Y}(t_k\wedge\cdot)$ is uniquely constructed for some $k\in\bN\cup\{0\}$, then $(\widetilde{Y}(t))_{t\in(t_k,t_{k+1}]}$ is constructed as the unique solution of
\begin{equation}\label{appendix_eq_SFDEtilde-k}
	\diff\widetilde{Y}(t)=\bar{b}\big(t_k,\widetilde{Y}\big)\,\diff t+\bar{\sigma}\big(t_k,\widetilde{Y}\big)\,\diff W(t)+f\big(t,\widetilde{Y}(t)\big)\,\diff t,\ \ t\in(t_k,t_{k+1}],
\end{equation}
with initial condition $\widetilde{Y}(t_k)$, which is given by the assumption of the induction. By progressive measurability of $\bar{b}$ and $\bar{\sigma}$, we have $\bar{b}(t_k,\widetilde{Y})=\bar{b}(t_k,\widetilde{Y}_{t_k})$ and $\bar{\sigma}(t_k,\widetilde{Y})=\bar{\sigma}(t_k,\widetilde{Y}_{t_k})$, and they are determined by the assumption of the induction. Thanks to the Lipschitz continuity of $f$, the equation \eqref{appendix_eq_SFDEtilde-k} is nothing but the standard SDE with Lipschitz coefficients, and hence it admits a unique continuous $\bF$-adapted solution $(\widetilde{Y}(t))_{t\in[t_k,t_{k+1}]}$. By induction, we see that \eqref{appendix_eq_SFDEtilde} admits a unique continuous $\bF$-adapted solution $\widetilde{Y}$.

Next, we define
\begin{equation*}
	\widetilde{\tau}:=\inf\left\{t\geq0\relmiddle|g(t,\widetilde{Y}(t))=0\right\},
\end{equation*}
which is an $\bF$-stopping time. Then, consider the following equation:
\begin{equation}\label{appendix_eq_SFDE-stopped'}
	\begin{dcases}
	\diff Y(t)=\bar{b}(t,Y)\,\diff t+\bar{\sigma}(t,Y)\,\diff W(t)+f(t,\widetilde{Y}(t))\1_{[0,\widetilde{\tau})}(t)\,\diff t,\ \ t\in[0,\infty),\\
	Y(0)=\xi.
	\end{dcases}
\end{equation}
Noting that the term $f(t,\widetilde{Y}(t))\1_{[0,\widetilde{\tau})}(t)$ is given, again by the step-by-step argument, we can construct $Y$ as follows:
\begin{align*}
	&Y(t)=Y(t_k)+\bar{b}(t_k,Y)(t-t_k)+\bar{\sigma}(t_k,Y)(W(t)-W(t_k))+\int^t_{t_k}f\big(s,\widetilde{Y}(s)\big)\1_{[0,\widetilde{\tau})}(s)\,\diff s,
	\ \ t\in(t_k,t_{k+1}],
\end{align*}
for each $k\in\bN\cup\{0\}$, with the initial condition $Y(0)=\xi$. We show that the process $Y$ is a solution to the SFDE \eqref{appendix_eq_SFDE-stopped}. To do so, define an $\bF$-stopping time $\tau$ by
\begin{equation*}
	\tau:=\inf\left\{t\geq0\relmiddle|g(t,Y(t))=0\right\}.
\end{equation*}
First, we show that
\begin{equation}\label{appendix_eq_Y=tildeY}
	Y(t\wedge\tau\wedge\widetilde{\tau})=\widetilde{Y}(t\wedge\tau\wedge\widetilde{\tau})\ \ \text{for any $t\in[0,\infty)$}.
\end{equation}
To do so, we show that $Y_{t_k\wedge\tau\wedge\widetilde{\tau}}=\widetilde{Y}_{t_k\wedge\tau\wedge\widetilde{\tau}}$ inductively with respect to $k\in\bN\cup\{0\}$. Clearly, the equality holds for $k=0$. Let $k\in\bN\cup\{0\}$ be fixed, and assume that $Y_{t_k\wedge\tau\wedge\widetilde{\tau}}=\widetilde{Y}_{t_k\wedge\tau\wedge\widetilde{\tau}}$. Then, on the event $\{t_k<\tau\wedge\widetilde{\tau}\}\in\cF_{t_k}$, we have, for any $t\in(t_k,t_{k+1}\wedge\tau\wedge\widetilde{\tau}]$,
\begin{align*}
	Y(t)&=Y(t_k)+\bar{b}(t_k,Y)(t-t_k)+\bar{\sigma}(t_k,Y)(W(t)-W(t_k))+\int^t_{t_k}f\big(s,\widetilde{Y}(s)\big)\,\diff s\\
	&=\widetilde{Y}(t_k)+\bar{b}\big(t_k,\widetilde{Y}\big)(t-t_k)+\bar{\sigma}\big(t_k,\widetilde{Y}\big)(W(t)-W(t_k))+\int^t_{t_k}f\big(s,\widetilde{Y}(s)\big)\,\diff s\\
	&=\widetilde{Y}(t).
\end{align*}
Hence, we get $Y_{t_{k+1}\wedge\tau\wedge\widetilde{\tau}}=\widetilde{Y}_{t_{k+1}\wedge\tau\wedge\widetilde{\tau}}$. By induction, we see that $Y_{t_k\wedge\tau\wedge\widetilde{\tau}}=\widetilde{Y}_{t_k\wedge\tau\wedge\widetilde{\tau}}$ holds for any $k\in\bN\cup\{0\}$, showing \eqref{appendix_eq_Y=tildeY}. In particular, it holds that $g(\tau\wedge\widetilde{\tau},Y(\tau\wedge\widetilde{\tau}))=g(\tau\wedge\widetilde{\tau},\widetilde{Y}(\tau\wedge\widetilde{\tau}))$ on $\{\tau\wedge\widetilde{\tau}<\infty\}$. Thus, applying \cref{appendix_lemm_stopping} to $x(t)=g(t,Y(t))$ and $y(t)=g(t,\widetilde{Y}(t))$, we see that $\tau=\widetilde{\tau}$. Combining this equality with \eqref{appendix_eq_SFDE-stopped'} and \eqref{appendix_eq_Y=tildeY}, we see that $Y$ satisfies \eqref{appendix_eq_SFDE-stopped}.

Uniqueness follows from a similar argument as above. Here, we provide a sketch of the proof. Let $Y'$ be another solution of \eqref{appendix_eq_SFDE-stopped} with stopping time $\tau'=\inf\{t\geq0\,|\,g(t,Y'(t))=0\}$. The step-by-step argument and the Lipschitz continuity of $f$ yield that $Y_{t_k\wedge\tau\wedge\tau'}=Y'_{t_k\wedge\tau\wedge\tau'}$ for any $k\in\bN\cup\{0\}$ a.s. Hence, we have $Y(t\wedge\tau\wedge\tau')=Y'(t\wedge\tau\wedge\tau')$ for any $t\in[0,\infty)$ a.s., and in particular $g(\tau\wedge\tau',Y(\tau\wedge\tau'))=g(\tau\wedge\tau',Y'(\tau\wedge\tau'))$ on $\{\tau\wedge\tau'<\infty\}$ a.s. Applying \cref{appendix_lemm_stopping} to $x(t)=g(t,Y(t))$ and $y(t)=g(t,Y'(t))$, we see that $\tau=\tau'$ a.s. Hence, $f(t,Y(t))\1_{[0,\tau)}(t)=f(t,Y'(t))\1_{[0,\tau')}(t)$ for any $t\in[0,\infty)$ a.s. Using this equality, again by the step-by-step argument as above, we see that $Y=Y'$ a.s. This completes the proof.
\end{proof}


\begin{rem}\label{appendix_rem_controlledEM}
Let \cref{assum} hold, and suppose that we are given a weak solution $\weaksol$ of the original SFDE \eqref{eq_SFDE}. Then, for each $T\in(0,\infty)$, $\pi\in\Pi_T$, $\lambda\in(0,\infty)$, $\Delta\in(0,1]$ and $\vec{R}=(R_{\growth,b},R_{\growth,\sigma},R_{\conti,b},R_{\conti,\sigma},R_\ellip)\in[1,\infty)^5$, applying \cref{appendix_lemm_controlledEM} to the $\cF_0$-measurable initial condition $\xi(\omega)=X(\omega,0)$, $\omega\in\Omega$, the path-dependent coefficients
\begin{equation*}
	\bar{b}(t,y)=b(\pi(t),\poly[y]),\ \bar{\sigma}(t,y)=\sigma(\pi(t),\poly[y]),\ \ (t,y)\in\Lambda,
\end{equation*}
and the $\bF$-progressively measurable maps
\begin{align*}
	&f(\omega,t,\eta)=\frac{\lambda}{\Delta}(X(\omega,t)-\eta)\1_{[0,T\wedge\zeta(\omega))}(t),\ g(\omega,t,\eta)=\dist\left(\eta,\bR^n\setminus B_{X(\omega,t)}(\Delta)\right),\\
	&\hspace{7cm}(\omega,t,\eta)\in\Omega\times[0,\infty)\times\bR^n,
\end{align*}
we see that there exists a unique (up to $\bP$-indistinguishability) $\bR^n$-valued continuous and $\bF$-adapted process $\hX^\pi=(\hX^\pi(t))_{t\in[0,\infty)}$ on $(\Omega,\cF,\bP)$ satisfying \eqref{eq_controlledEM}. This shows existence and uniqueness of the controlled Euler--Maruyama scheme.
\end{rem}


\section{Fundamental estimates for stochastic processes}\label{appendix_prob}

In this section, we prove some fundamental estimates for stochastic processes used in the proof of \cref{prop_X-hXpi}. The following lemma is concerned with the tail-probability of the modulus of continuity of a continuous local martingale, which is used for the estimate of the term $P_2$ appearing in \eqref{eq_prob}.


\begin{lemm}\label{appendix_lemm_mod-of-conti}
Let $M=(M^1,\dots,M^n)^\top$ be an $n$-dimensional continuous local martingale on a filtered probability space $(\Omega,\cF,\bF,\bP)$. For any $0<\delta\leq T<\infty$, $\kappa\in(0,\infty)$ and $\theta\in(0,\infty)$, it holds that
\begin{align*}
	&\bP\left(\max_{i\in\{1,\dots,n\}}\varpi(T,\langle M^i\rangle;\delta)\leq\kappa\ \ \text{and}\ \ \varpi(T,M;\delta)\geq2\sqrt{2n\kappa\left(\theta^2+\log\frac{T}{\delta}\right)}\right)
	\leq4n\exp\left(-\theta^2\right).
\end{align*}
\end{lemm}


\begin{proof}
Without loss of generality, we may assume that $M(0)=0$. Notice that
\begin{align*}
	&\bP\left(\max_{i\in\{1,\dots,n\}}\varpi(T,\langle M^i\rangle;\delta)\leq\kappa\ \ \text{and}\ \ \varpi(T,M;\delta)\geq2\sqrt{2n\kappa\left(\theta^2+\log\frac{T}{\delta}\right)}\right)\\
	&\leq\sum^n_{i=1}\bP\left(\varpi(T,\langle M^i\rangle;\delta)\leq\kappa\ \ \text{and}\ \ \varpi(T,M^i;\delta)\geq2\sqrt{2\kappa\left(\theta^2+\log\frac{T}{\delta}\right)}\right).
\end{align*}
Thus, it suffices to show the lemma for $n=1$.

Let $M$ be a one-dimensional continuous local martingale with $M(0)=0$. We show that
\begin{equation}\label{appendix_eq_mod-of-conti}
	\bP\left(\varpi(T,\langle M\rangle;\delta)\leq\kappa\ \ \text{and}\ \ \varpi(T,M;\delta)\geq2\sqrt{2\kappa\left(\theta^2+\log\frac{T}{\delta}\right)}\right)\leq4\exp\left(-\theta^2\right).
\end{equation}
By the Dambis--Dubins--Schwarz theorem, enlarging the probability space (which we again denote by $(\Omega,\cF,\bP)$) if necessary, we can find a one-dimensional Brownian motion $W$ such that $M=W\circ\langle M\rangle$. Set $m:=\max\{j\in\bN\,|\,j\leq\frac{T}{\delta}\}$. Notice that
\begin{equation*}
	\langle M\rangle(t)-\langle M\rangle(s)\leq\varpi(T,\langle M\rangle;\delta)\ \ \text{for any $0\leq s\leq t\leq T$ such that $t-s\leq\delta$},
\end{equation*}
and
\begin{equation*}
	\langle M\rangle(T)=\sum^{m-1}_{j=0}\big\{\langle M\rangle((j+1)\delta)-\langle M\rangle(j\delta)\big\}+\langle M\rangle(T)-\langle M\rangle(m\delta)\leq(m+1)\varpi(T,\langle M\rangle;\delta).
\end{equation*}
Hence, on the event $\{\varpi(T,\langle M\rangle;\delta)\leq\kappa\}$, we have
\begin{align*}
	\varpi(T,M;\delta)&=\sup_{\substack{0\leq s\leq t\leq T\\t-s\leq\delta}}|M(t)-M(s)|\\
	&=\sup_{\substack{0\leq s\leq t\leq T\\t-s\leq\delta}}\big|W\big(\langle M\rangle(t)\big)-W\big(\langle M\rangle(s)\big)\big|\\
	&\leq\sup_{\substack{0\leq s\leq t\leq(m+1)\kappa\\t-s\leq\kappa}}|W(t)-W(s)|\\
	&\leq\max_{j\in\{0,\dots,m-1\}}\sup_{s,t\in[j\kappa,(j+2)\kappa]}|W(t)-W(s)|\\
	&\leq2\max_{j\in\{0,\dots,m-1\}}\sup_{t\in[j\kappa,(j+2)\kappa]}|W(t)-W((j+1)\kappa)|\\
	&=2\sqrt{\kappa}\max_{j\in\{0,\dots,m-1\}}\left(\xi^1_j\vee\xi^2_j\right),
\end{align*}
where we set
\begin{align*}
	&\xi^1_j:=\frac{1}{\sqrt{\kappa}}\sup_{t\in[j\kappa,(j+1)\kappa]}\big|W(t)-W((j+1)\kappa)\big|\ \ \text{and}\\
	&\xi^2_j:=\frac{1}{\sqrt{\kappa}}\sup_{t\in[(j+1)\kappa,(j+2)\kappa]}\big|W(t)-W((j+1)\kappa)\big|
\end{align*}
for $j\in\{0,\dots,m-1\}$. Thus, we obtain
\begin{align*}
	&\bP\left(\varpi(T,\langle M\rangle;\delta)\leq\kappa\ \ \text{and}\ \ \varpi(T,M;\delta)\geq2\sqrt{2\kappa\left(\theta^2+\log\frac{T}{\delta}\right)}\right)\\
	&\leq\bP\left(\max_{j\in\{0,\dots,m-1\}}\left(\xi^1_j\vee\xi^2_j\right)\geq\sqrt{2\left(\theta^2+\log\frac{T}{\delta}\right)}\right)\\
	&\leq\sum^{m-1}_{j=0}\sum_{i\in\{1,2\}}\bP\left(\xi^i_j\geq\sqrt{2\left(\theta^2+\log\frac{T}{\delta}\right)}\right).
\end{align*}
By the time-homogeneity and the scaling property of the Brownian motion $W$, we see that $\xi^i_j\overset{d}{=}\sup_{t\in[0,1]}|W(t)|$ for any $j\in\{0,\dots,m-1\}$ and $i\in\{1,2\}$. Furthermore,
\begin{align*}
	&\sup_{t\in[0,1]}|W(t)|=\max\left\{\sup_{t\in[0,1]}W(t),\sup_{t\in[0,1]}(-W(t))\right\},\\
	&\sup_{t\in[0,1]}W(t)\overset{d}{=}\sup_{t\in[0,1]}(-W(t))\overset{d}{=}|W(1)|,\ \ \text{and}\\
	&\bP(|W(1)|\geq r)\leq\exp(-r^2/2)\ \ \text{for any $r>0$}.
\end{align*}
By these observations, we have
\begin{align*}
	&\bP\left(\varpi(T,\langle M\rangle;\delta)\leq\kappa\ \ \text{and}\ \ \varpi(T,M;\delta)\geq2\sqrt{2\kappa\left(\theta^2+\log\frac{T}{\delta}\right)}\right)\\
	&\leq2m\bP\left(\sup_{t\in[0,1]}|W(t)|\geq\sqrt{2\left(\theta^2+\log\frac{T}{\delta}\right)}\right)\\
	&\leq4m\bP\left(|W(1)|\geq\sqrt{2\left(\theta^2+\log\frac{T}{\delta}\right)}\right)\\
	&\leq4m\frac{\delta}{T}\exp\left(-\theta^2\right).
\end{align*}
Since $m=\max\{j\in\bN\,|\,j\leq\frac{T}{\delta}\}\leq\frac{T}{\delta}$, the estimate \eqref{appendix_eq_mod-of-conti} holds. This completes the proof.
\end{proof}

The following lemma, which is used for the estimate of the term $P_3$ appearing in \eqref{eq_prob}, is a slight refinement of \cite[Lemma B.1]{KuSc20} incorporating a logarithmic term.


\begin{lemm}\label{appendix_lemm_Kulik--Scheutzou}
Let $Z$ be a one-dimensional nonnegative It\^{o} process on a filtered probability space $(\Omega,\cF,\bF,\bP)$ represented by
\begin{equation*}
	\diff Z(t)=v(t)\,\diff t+\diff N(t),\ \ t\in[0,\infty),
\end{equation*}
for some progressively measurable process $v$ and one-dimensional continuous local martingale $N$ with $N(0)=0$. Suppose that there exist constants $T\in(0,\infty)$, $\kappa\in(0,\infty)$, $A\in[0,\infty)$ and $B\in(0,\infty)$ and a nonnegative random variable $\varsigma$ such that $\varsigma\leq T$ and
\begin{equation*}
	v(t)\leq-\kappa Z(t)+A\ \ \text{and}\ \ \frac{\diff\langle N\rangle(t)}{\diff t}\leq B\ \ \text{for any $t\in[0,\varsigma)$ $\bP$-a.s.}
\end{equation*}
Then, for any $\theta\in(0,\infty)$, it holds that
\begin{equation}\label{appendix_eq_Kulik--Scheutzou}
	\bP\left(\sup_{t\in[0,\varsigma]}\big(Z(t)-e^{-\kappa t}Z(0)\big)\geq\frac{A}{\kappa}+\sqrt{\frac{2B\big(\theta^2+\log(1\vee(T\kappa))\big)}{\kappa}}\right)\leq219\exp\left(-\theta^2\right).
\end{equation}
\end{lemm}


\begin{proof}
The arguments in the proof of Lemma B.1 in \cite{KuSc20} show that
\begin{equation*}
	Z(t)=e^{-\kappa t}Z(0)+\int^t_0e^{-\kappa(t-s)}\big(v(s)+\kappa Z(s)\big)\,\diff s+\int^t_0e^{-\kappa(t-s)}\,\diff N(s),\ \ t\in[0,\infty),
\end{equation*}
and
\begin{align*}
	&\sup_{t\in[0,\varsigma]}\int^t_0e^{-\kappa(t-s)}\big(v(s)+\kappa Z(s)\big)\,\diff s\leq\frac{A}{\kappa},\\
	&\sup_{t\in[0,\varsigma]}\int^t_0e^{-\kappa(t-s)}\,\diff N(s)\leq B^{1/2}\sup_{t\in[0,T]}\int^t_0e^{-\kappa(t-s)}\,\diff W(s),
\end{align*}
where $W$ is a one-dimensional Brownian motion defined on an enlarged probability space. The estimate \eqref{appendix_eq_Kulik--Scheutzou} then follows from \cref{appendix_lemm_OU-exp} below.
\end{proof}

The proof of the above lemma is based on the following estimate of the tail-probability of the supremum of a stochastic convolution (or the Ornstein--Uhlenbeck process), which is important in its own right.


\begin{lemm}\label{appendix_lemm_OU-exp}
Let $W$ be a one-dimensional Brownian motion on a probability space $(\Omega,\cF,\bP)$. For any $T\in(0,\infty)$ and $\kappa\in(0,\infty)$, it holds that
\begin{equation}\label{appendix_eq_OU-exp}
	\bE\left[\exp\left(\frac{\kappa}{2}\sup_{t\in[0,T]}\left|\int^t_0e^{-\kappa(t-s)}\,\diff W(s)\right|^2\right)\right]\leq73(2+T\kappa).
\end{equation}
In particular, for any $T\in(0,\infty)$, $\kappa\in(0,\infty)$ and $\theta\in(0,\infty)$, it holds that
\begin{equation}\label{appendix_eq_OU-prob}
	\bP\left(\sup_{t\in[0,T]}\left|\int^t_0e^{-\kappa(t-s)}\,\diff W(s)\right|\geq\sqrt{\frac{2\big(\theta^2+\log(1\vee(T\kappa))\big)}{\kappa}}\right)\leq219\exp\left(-\theta^2\right).
\end{equation}
\end{lemm}


\begin{proof}
Define $X(t):=\int^t_0e^{-\kappa(t-s)}\,\diff W(s)$, $t\in[0,\infty)$, and $f(x):=\exp(\frac{\kappa}{2}x^2)$, $x\in\bR$. We show that
\begin{equation}\label{appendix_eq_OU-exp'}
	\bE\left[\sup_{t\in[0,T]}f(X(t))\right]\leq73(2+T\kappa).
\end{equation}
Notice that $X$ is an Ornstein--Uhlenbeck process such that
\begin{equation*}
	\diff X(t)=-\kappa X(t)\,\diff t+\diff W(t),\ \ t\in[0,\infty),\ \ X(0)=0.
\end{equation*}
Also, notice that $f(0)=1$, $f'(x)=\kappa xf(x)$ and $f''(x)=(\kappa+\kappa ^2x^2)f(x)$. Thus, It\^{o}'s formula yields
\begin{align*}
	f(X(t))&=1+\int^t_0\Big\{-\kappa^2X(s)^2f(X(s))+\frac{1}{2}\big(\kappa+\kappa^2X(s)^2\big)f(X(s))\Big\}\,\diff s\\
	&\hspace{3cm}+\int^t_0\kappa X(s)f(X(s))\,\diff W(s),\ \ t\in[0,\infty).
\end{align*}
Equivalently, it holds that
\begin{align*}
	&f(X(t))+\frac{\kappa^2}{4}\int^t_0X(s)^2f(X(s))\,\diff s\\
	&=1+\frac{\kappa}{2}\int^t_0\left(1-\frac{\kappa}{2}X(s)^2\right)f(X(s))\,\diff s+\kappa\int^t_0X(s)f(X(s))\,\diff W(s),\ \ t\in[0,\infty).
\end{align*}
Since $(1-\frac{\kappa}{2}x^2)f(x)=(1-\frac{\kappa}{2}x^2)\exp(\frac{\kappa}{2}x^2)\leq1$ for any $x\in\bR$, we obtain
\begin{equation}\label{appendix_eq_OU-Ito}
	f(X(t))+\frac{\kappa^2}{4}\int^t_0X(s)^2f(X(s))\,\diff s\leq1+\frac{t\kappa}{2}+\kappa\int^t_0X(s)f(X(s))\,\diff W(s),\ \ t\in[0,\infty).
\end{equation}
Let $N\in\bN$, and define $\tau_N:=\inf\{t\geq0\,|\,|X(t)|\geq N\}$. Then $\tau_N$ is a stopping time such that $\lim_{N\to\infty}\tau_N=\infty$ a.s. Notice that the stopped stochastic integral $\int^{\cdot\wedge\tau_N}_0X(s)f(X(s))\,\diff W(s)$ is a martingale, and hence its expectation is zero. By taking the expectations in both sides of \eqref{appendix_eq_OU-Ito} with $t=T\wedge\tau_N$, we obtain
\begin{equation}\label{appendix_eq_OU-estimate}
	\frac{\kappa^2}{4}\bE\left[\int^{T\wedge\tau_N}_0X(s)^2f(X(s))\,\diff s\right]\leq1+\frac{T\kappa}{2}.
\end{equation}
Furthermore, taking the supremum for $t\in[0,T\wedge\tau_N]$ and then taking the expectations in both sides of \eqref{appendix_eq_OU-Ito}, we have, by the Burkholder--Davis--Gundy inequality\footnote{Following the proof of \cite[Theorem 3.28]{KaSh91}, in the Burkholder--Davis--Gundy inequality $\bE[\|M_T\|_\infty]\leq K_{1/2}\bE[\langle M\rangle(T)^{1/2}]$ for one-dimensional continuous local martingales starting from zero, we can take $K_{1/2}=6$.},
\begin{equation*}
	\bE\left[\sup_{t\in[0,T\wedge\tau_N]}f(X(t))\right]\leq1+\frac{T\kappa}{2}+6\kappa\bE\left[\left(\int^{T\wedge\tau_N}_0X(s)^2f(X(s))^2\,\diff s\right)^{1/2}\right].
\end{equation*}
Then, Young's inequality yields
\begin{align*}
	\bE\left[\sup_{t\in[0,T\wedge\tau_N]}f(X(t))\right]
	&\leq1+\frac{T\kappa}{2}+6\kappa\bE\left[\left(\int^{T\wedge\tau_N}_0X(s)^2f(X(s))\,\diff s\right)^{1/2}\left(\sup_{s\in[0,T\wedge\tau_N]}f(X(s))\right)^{1/2}\right]\\
	&\leq1+\frac{T\kappa}{2}+18\kappa^2\bE\left[\int^{T\wedge\tau_N}_0X(s)^2f(X(s))\,\diff s\right]+\frac{1}{2}\bE\left[\sup_{s\in[0,T\wedge\tau_N]}f(X(s))\right].
\end{align*}
Thanks to the stopping time $\tau_N$, we have $\bE[\sup_{s\in[0,T\wedge\tau_N]}f(X(s))]<\infty$. Hence, the above estimate shows that
\begin{equation*}
	\bE\left[\sup_{t\in[0,T\wedge\tau_N]}f(X(t))\right]\leq2+T\kappa+36\kappa^2\bE\left[\int^{T\wedge\tau_N}_0X(s)^2f(X(s))\,\diff s\right].
\end{equation*}
Combining this estimate with \eqref{appendix_eq_OU-estimate}, we obtain
\begin{equation*}
	\bE\left[\sup_{t\in[0,T\wedge\tau_N]}f(X(t))\right]\leq73(2+T\kappa).
\end{equation*}
By taking the limit $N\to\infty$ in this estimate and using Fatou's lemma, we obtain \eqref{appendix_eq_OU-exp'} and equivalently \eqref{appendix_eq_OU-exp}.

As for the estimate \eqref{appendix_eq_OU-prob}, by using Markov's inequality and \eqref{appendix_eq_OU-exp}, we have
\begin{align*}
	&\bP\left(\sup_{t\in[0,T]}\left|\int^t_0e^{-\kappa(t-s)}\,\diff W(s)\right|\geq\sqrt{\frac{2\big(\theta^2+\log(1\vee(T\kappa))\big)}{\kappa}}\right)\\
	&=\bP\left(\exp\left(\frac{\kappa}{2}\sup_{t\in[0,T]}\left|\int^t_0e^{-\kappa(t-s)}\,\diff W(s)\right|^2\right)\geq(1\vee(T\kappa))\exp\left(\theta^2\right)\right)\\
	&\leq\bE\left[\exp\left(\frac{\kappa}{2}\sup_{t\in[0,T]}\left|\int^t_0e^{-\kappa(t-s)}\,\diff W(s)\right|^2\right)\right](1\vee(T\kappa))^{-1}\exp\left(-\theta^2\right)\\
	&\leq73(2+T\kappa)(1\vee(T\kappa))^{-1}\exp\left(-\theta^2\right)\\
	&\leq219\exp\left(-\theta^2\right).
\end{align*}
This completes the proof.
\end{proof}

\end{appendix}


\section*{Acknowledgments}
\addcontentsline{toc}{section}{Acknowledgements}
The first author was supported by JSPS KAKENHI Grant Number 22K13958.
The second author was supported by JSPS KAKENHI Grant Number 21H00988 and 23K12988.


\end{document}